\documentclass[11pt]{book}
\def\thetitle{Planar Maps, Random Walks and Circle Packing}

\usepackage{amsmath,amssymb}
\usepackage{mathrsfs}

\usepackage[usenames,dvipsnames,svgnames,table]{xcolor}
\definecolor{CombinatoricaAqua}{HTML}{00698C}
\definecolor{CombinatoricaBlue}{HTML}{3A3293}
\definecolor{CombinatoricaBrown}{HTML}{66220C}
\definecolor{CombinatoricaRed}{HTML}{DF2A27}
\definecolor{HarvardCrimson}{rgb}{0.6471, 0.1098, 0.1882}

\makeatletter
\let\reftagform@=\tagform@
\def\tagform@#1{\maketag@@@
  {(\ignorespaces\textcolor{CombinatoricaBrown}{#1}\unskip\@@italiccorr)}}
\renewcommand{\eqref}[1]{\textup{\reftagform@{\ref{#1}}}}
\makeatother

\usepackage{imakeidx}
\makeindex
\newcommand{\defn}[1]{\textbf{#1}\index{#1}}

\usepackage[backref=page]{hyperref}
\hypersetup{
  unicode,
  pdfencoding=auto,
  pdfauthor={Asaf Nachmias},
  pdftitle={\thetitle},
  pdfsubject={05C81 (Primary), 05C80, 60G50 (Secondary)},
  pdfkeywords={random walk, planar maps, circle packing},
  colorlinks=true,
  citecolor=CombinatoricaBlue,
  linkcolor=CombinatoricaAqua,
  anchorcolor=CombinatoricaBrown,
  urlcolor=HarvardCrimson}

\usepackage{amsbsy}
\newcommand*\Bell{\ensuremath{\boldsymbol\ell}}

\usepackage{amsthm}
\usepackage{thmtools}

\usepackage[capitalize]{cleveref}


\declaretheoremstyle[
  spaceabove=\topsep, spacebelow=\topsep,
  headfont=\color{CombinatoricaBrown}\normalfont\bfseries,
  bodyfont=\itshape,
]{thm}
\declaretheoremstyle[
  spaceabove=\topsep, spacebelow=\topsep,
  headfont=\color{CombinatoricaBrown}\normalfont\bfseries,
  bodyfont=\normalfont,
]{dfn}
\declaretheoremstyle[
  spaceabove=0.5\topsep, spacebelow=0.5\topsep,
  headfont=\color{CombinatoricaBrown}\normalfont\bfseries,
  bodyfont=\normalfont,
]{rmk}

\declaretheorem[style=thm,parent=chapter]{theorem}
\declaretheorem[style=thm,sibling=theorem]{lemma}
\declaretheorem[style=thm,sibling=theorem]{corollary}
\declaretheorem[style=thm,sibling=theorem]{claim}

\declaretheorem[style=thm,sibling=theorem]{observation}

\declaretheorem[style=thm,sibling=theorem]{proposition}
\declaretheorem[style=rmk,sibling=theorem]{remark}

\declaretheorem[style=definition,sibling=theorem]{definition}

\declaretheorem[style=definition,sibling=theorem]{example}

\usepackage[nobysame,msc-links]{amsrefs}

\BibSpec{book}{%
    +{}  {\PrintPrimary}                {transition}
    +{,} { \textbf}                     {title} 
    +{.} { }                            {part}
    +{:} { \textit}                     {subtitle}
    +{,} { \PrintEdition}               {edition}
    +{}  { \PrintEditorsB}              {editor}
    +{,} { \PrintTranslatorsC}          {translator}
    +{,} { \PrintContributions}         {contribution}
    +{,} { }                            {series}
    +{,} { \voltext}                    {volume}
    +{,} { }                            {publisher}
    +{,} { }                            {organization}
    +{,} { }                            {address}
    +{,} { \PrintDateB}                 {date}
    +{,} { }                            {status}
    +{}  { \parenthesize}               {language}
    +{}  { \PrintTranslation}           {translation}
    +{;} { \PrintReprint}               {reprint}
    +{.} { }                            {note}
    +{.} {}                             {transition}
    +{}  {\SentenceSpace \PrintReviews} {review}
}

\makeatletter
\renewcommand{\PrintNames@a}[4]{%
  \PrintSeries{\name}
    {#1}
    {}{ and \set@othername}
    {,}{ \set@othername}
    {}{ and \set@othername}
    {#2}{#4}{#3}%
}
\makeatother

\usepackage{sectsty}
\sectionfont{\color{CombinatoricaBrown}}
\subsectionfont{\color{CombinatoricaBrown}}
\subsubsectionfont{\color{CombinatoricaBrown}}

\usepackage{titlesec}
\definecolor{gray75}{gray}{0.75}
\newcommand{\hsp}{\hspace{20pt}}
\titleformat{\chapter}[hang]
  {\Huge\bfseries}
  {\thechapter\hsp\textcolor{gray75}{::}\hsp}{0pt}
  {\huge\bfseries}

\usepackage{soul}
\soulregister\ref7

\usepackage[a4paper]{geometry}
\usepackage[textsize=tiny,backgroundcolor=black!10]{todonotes}
\presetkeys{todonotes}{noline}{}

\usepackage{mathtools}

\title{\thetitle}
\author{{\tiny Lecture notes by}\\Asaf Nachmias}

\usepackage{mleftright}
\mleftright

\newcommand{\ceil}[1]{\lceil{#1}\rceil}

\newcommand{\E}[0]{\mathbb{E}}
\newcommand{\pr}[0]{\mathbb{P}}
\newcommand{\PP}{\mathbf{P}}
\newcommand{\EE}{\mathbf{E}}

\DeclareMathOperator{\unif}{\mathsf{Unif}}

\newcommand{\CC}{\mathbb{C}}
\newcommand{\RR}{\mathbb{R}}
\renewcommand{\SS}{\mathbb{S}}
\newcommand{\ZZ}{\mathbb{Z}}
\newcommand{\NN}{\mathbb{N}}
\newcommand{\UU}{\mathbb{U}}

\newcommand{\eps}{\varepsilon}

\DeclareMathOperator{\diam}{diam}

\usepackage{graphicx}
\DeclareMathOperator{\carrier}{Carrier}
\newcommand{\reff}{\mathcal{R}_\mathrm{eff}}
\newcommand{\ceff}{\mathcal{C}_\mathrm{eff}}
\newcommand{\lr}{\leftrightarrow}
\newcommand{\energy}{\mathcal{E}}
\newcommand{\vect}[1]{\mathbf{#1}}
\newcommand{\ang}{\measuredangle}
\DeclareMathOperator{\gap}{gap}
\newcommand{\downto}{\searrow}

\DeclareMathOperator{\rad}{rad}
\DeclareMathOperator{\cent}{cent}
\DeclareMathOperator{\area}{area}
\newcommand{\origin}{\mathbf{0}}
\newcommand{\cev}[1]{\reflectbox{\ensuremath{\vec{\reflectbox{\ensuremath{#1}}}}}}
\newcommand{\Gdot}{\mathcal{G}_{\bullet}}
\newcommand{\Mdot}{\mathcal{M}_{\bullet}}
\newcommand{\dloc}{d_{\mathrm{loc}}}
\newcommand{\convd}{\stackrel{d}{\longrightarrow}}
\newcommand{\convl}{\xrightarrow{\mathrm{loc}}}
\newcommand{\convlpi}{\xrightarrow[\pi]{\mathrm{loc}}}
\DeclareMathOperator{\euc}{euc}

\newcommand{\T}{\mathcal{T}}
\newcommand{\A}{\mathcal{A}}
\newcommand{\sA}{\mathscr A}
\newcommand{\sB}{\mathscr B}

\newcommand{\cE}{\mathcal E}
\newcommand{\cF}{\mathcal F}

\newcommand{\cL}{\mathcal L}

\newcommand{\cR}{\mathcal R}

\usepackage{tikz}
\usetikzlibrary{arrows.meta,angles,quotes,decorations.pathreplacing,patterns,
  shapes.geometric,calc}
\tikzset{vx/.style=
  {circle, draw=white, fill=black, thick, inner sep=0pt, minimum width=5pt,
   label distance=0.1cm}}
\tikzset{lvxb/.style={vx, thin, draw=black, minimum width=10pt}}
\tikzset{lvxw/.style={lvxb,fill=white}}
\tikzset{lvxg/.style={lvxb,fill=black!25,draw=black!25}}

\usepackage{graphicx}
\usepackage{caption,subcaption}

\usepackage{lipsum}

\newcommand\blfootnote[1]{%
  \begingroup
  \renewcommand\thefootnote{}\footnote{#1}%
  \addtocounter{footnote}{-1}%
  \endgroup
}

\makeatletter
\renewcommand\@dotsep{10000}
\makeatother
\usepackage{enumitem}
\newlist{thmenum}{enumerate}{1} 
\setlist[thmenum]{label=\arabic*., ref=\thetheorem~(\arabic*)}
\crefalias{thmenumi}{theorem}

\usepackage{fancyvrb}

\begin{document}

\frontmatter
\thispagestyle{empty}
{\begingroup
\centering
\vfill
{\huge \thetitle}\\[0.1\textheight]
Lecture notes of the 48th Saint-Flour summer school, 2018

\centering
\begin{BVerbatim}
preliminary draft
\end{BVerbatim}

\vspace{2cm}
{\LARGE Asaf Nachmias}\\[0.05\textheight]
\vspace{-2cm}
{\includegraphics[scale=0.5]{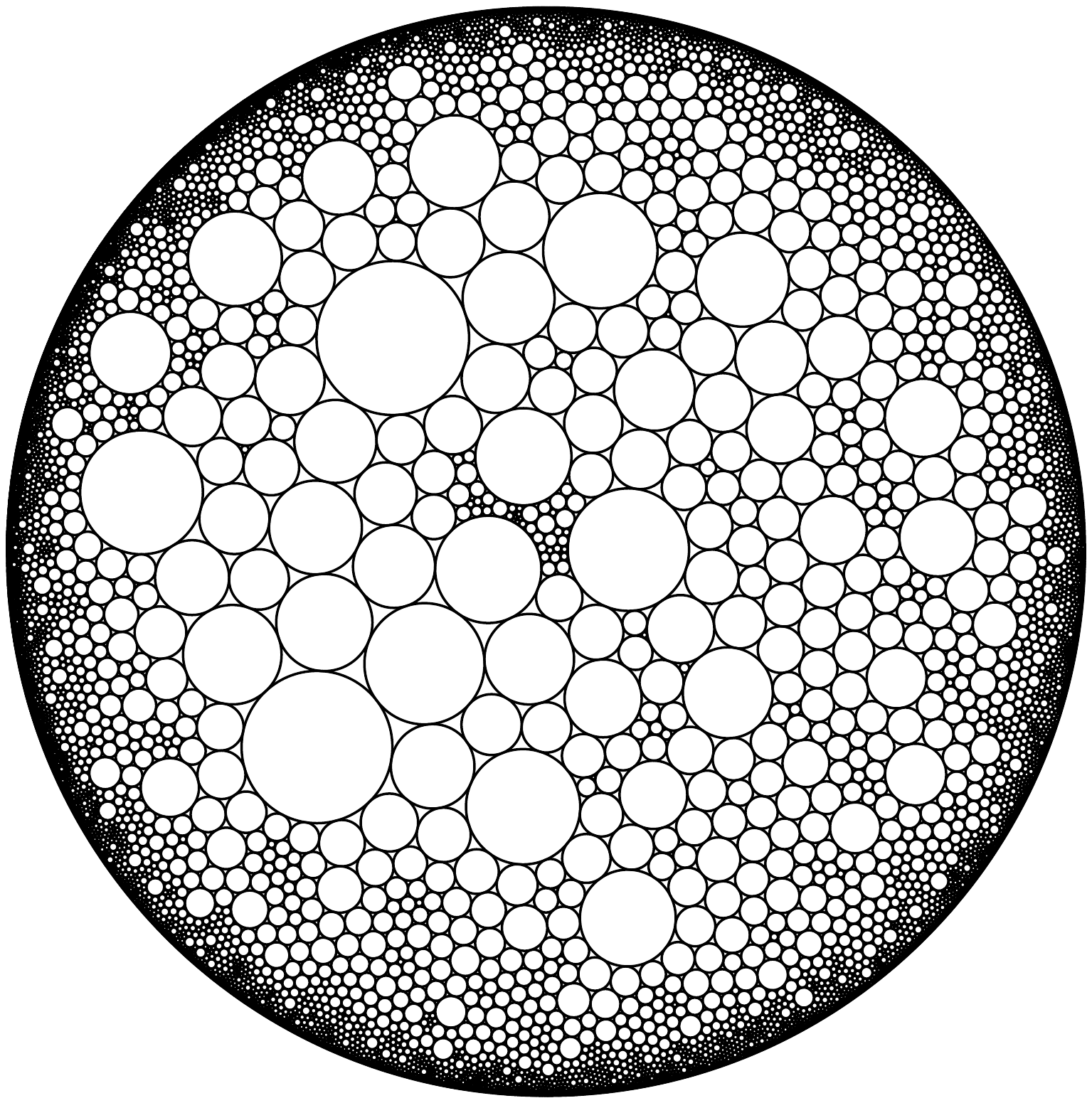}}\\
{\Large Tel Aviv University}\par
{\large\scshape 2018}\par
\vfill\null
\endgroup}



\chapter*{Preface}

These lecture notes are intended to accompany a single semester graduate course.
They are meant to be {\bf entirely self-contained}. All the theory required to prove the main results is presented and only basic knowledge in probability theory is assumed.

In \cref{chp:intro} we describe the main storyline of this text. It is meant to be light bedtime reading exposing the reader to the main results that will be presented and providing some background. \cref{chp:electric} introduces the theory of electric networks and discusses their highly useful relations to random walks. It is roughly based on Chapter 8 of Yuval Peres' excellent lecture notes \cite{PeresClimb}. We then discuss the circle packing theorem and present its proof in \cref{chp:cp}. \cref{chp:heschramm} discusses the beautiful theorem of He and Schramm \cite{HeSc} relating the circle packing type of a graph to recurrence and transience of the random walk on it. To the best of our knowledge, their work is the first to form connections between the circle packing theorem and probability theory. Next in \cref{chp:locallimit} we present the highly influential theorem of Benjamini and Schramm \cite{BeSc} about the almost sure recurrence of the simple random walk in planar graph limits of bounded degrees. The notion of a \emph{local limit} (also known as \emph{distributional limit} or \emph{Benjamini-Schramm limit}) of a sequence of finite graphs was introduced there for the first time to our knowledge (and also studied by Aldous-Steele \cite{AldousSteele} and Aldous-Lyons \cite{AldousLyons}); this notion is highly important in probability theory as well as other mathematical disciplines (see \cite{AldousLyons} and the references within). In \cref{chp:randommaps} we provide a theorem from which one can deduce the almost sure recurrence of the simple random walk on many models of random planar maps. This theorem was obtained by Ori Gurel-Gurevich and the author in \cite{GGN13}. \cref{chp:planarusf} discusses uniform spanning forests on planar maps and appeals to the circle packing theorem to show that the free uniform spanning forest on proper planar maps is almost surely connected, i.e., it is in fact a tree. This theorem was obtained by Tom Hutchcroft and the author in \cite{HutNach15b}. We close these notes in \cref{chp:related} with a description of some related contemporary developments in this field that are not presented in this text.

We have made an effort to add value beyond what is in the published papers. Our proof of the circle packing theorem in \cref{chp:cp} is inspired by Thurston's argument \cite{Th78} and Brightwell-Scheinerman \cite{BriSch93} but we have made what we think are some simplifications; the proof also employs a neat argument due to Ohad Feldheim and Ori Gurel-Gurevich (\cref{thm:howtodraw}) which makes the drawing part of the argument rather straightforward and avoids topological considerations that are used in the classical proofs. The original proof of the He-Schramm Theorem \cite{HeSc} is based on the notion of \emph{discrete extremal length} which is essentially a form of \emph{effective resistance} in electric networks (in fact, the \emph{edge} extremal length is precisely effective resistance, see \cite[Exercise 2.78]{LyonsPeres}). We find that our approach in \cref{chp:heschramm} using electric networks is somewhat more robust and intuitive to  probabilists. We obtain a quantitative version of the He-Schramm Theorem in \cref{chp:heschramm} as well as the Benjamini-Schramm Theorem \cite{BeSc} in \cref{chp:locallimit}, see \cref{theorem:loc:lim:recurrence}. These quantified versions are key to the proofs of \cref{chp:randommaps}. Lastly, some aspects of stationary random graphs are better explained here in \cref{chp:randommaps} than in the publication \cite{GGN13}.

\section*{Acknowledgments}

I would like to deeply thank Daniel Jerison, Peleg Michaeli and Matan Shalev for typing, editing and proofreading most of this text and for many comments, corrections and suggestions. I am  indebted to Tom Hutchcroft for his assistance in writing the introduction and for surveying related topics not included in these notes (\cref{chp:intro,chp:related}). I thank S\'ebastien Martineau, Pierre Petit, Dominik Schmid and Mateo Wirth for corrections and comments to this text. I am also grateful to the participants of the 48th Saint-Flour summer school and its organizers, Christophe Bahadoran, Arnaud Guillin and Hacene Djellout, for a very enjoyable summer school.

I am beholden to my collaborators on circle packing and planar maps related problems: Omer Angel, Martin Barlow, Itai Benjamini, Nicolas Curien, Ori Gurel-Gurevich, Tom Hutchcroft, Daniel Jerison, Gourab Ray, Steffen Rohde and Juan Souto. I have learned a lot from our work and conversations. Special thanks go to Ori Gurel-Gurevich for embarking together on this research endeavor beginning in 2011 at the University of British Columbia, Vancouver, Canada. Many of the ideas and methods presented in these notes were obtained in our joint work. 

I am also highly indebted to the late Oded Schramm whose mathematical work, originality and vision, especially in the topics studied in these notes, have been an enormous source of inspiration. It is no coincidence that his name appears on almost every other page here. It has become routine for my collaborators and I to ask ourselves \emph{``What would Oded do?''} hoping that reflecting on this question would cut right to the heart of matters. Steffen Rohde's wonderful survey \cite{RohdeSchramm} of Oded's work is very much recommended.

Lastly, I thank Shira Wilkof and baby Ada for muse and inspiration. %

\vspace{.5cm}
\hspace{.5cm} Asaf Nachmias\footnote{Department of Mathematical Sciences, Tel Aviv University, Tel Aviv, Israel.}, Tel Aviv, December 2018 \blfootnote{Email: \href{mailto:asafnach@post.tau.ac.il}{asafnach@post.tau.ac.il}}\blfootnote{This work is supported by ISF grant 1207/15 and ERC starting grant 676970.}

{\hypersetup{linkcolor=black}\tableofcontents}

\mainmatter

\chapter{Introduction}\label{chp:intro}
\section{The circle packing theorem}

A planar graph is a graph that can be drawn in the plane, with vertices represented by points and edges represented by non-crossing curves. There are many different ways of drawing any given planar graph and it is not clear what is a canonical method. One very useful and widely applicable method of drawing a planar graph is given by Koebe's 1936 \emph{circle packing theorem}\index{circle packing theorem} \cite{K36}, stated below. As we will see, various geometrical properties of the circle packing drawing (such as existence of accumulation points and their structure, bounds on the radii of circles and so on) encode important probabilistic information (such as the recurrence/transience of the simple random walk, connectivity of the uniform spanning forest and much more). This deep connection is especially fruitful to the study of random planar maps. Indeed, one of the main goals of these notes is to present a self-contained proof that the so-called \emph{uniform infinite planar triangulation} (UIPT) is almost surely recurrent \cite{GGN13}.


A \textbf{circle packing} is a collection of discs $P=\{C_v\}$ in the plane $\CC$ such that any two distinct discs in $P$ have disjoint interiors. That is, distinct discs in $P$ may be tangent, but may not overlap. 
Given a circle packing $P$, we define the \textbf{tangency graph} $G(P)$ of $P$ to be the graph with vertex set $P$ and with two vertices connected by an edge if and only if their corresponding circles are tangent. The tangency graph $G(P)$ can be drawn in the plane by drawing straight lines between the centers of tangent circles in $P$, and is therefore planar. It is also clear from the definition that $G(P)$ is \textbf{simple}, that is, any two vertices are connected by at most one edge and there are no edges beginning and ending at the same vertex. See \cref{fig:circlepacking}.

We call a circle packing $P$ a circle packing of a planar graph $G$ if $G(P)$ is isomorphic to $G$.

\begin{theorem}[Koebe '36]\label{thm:cp1}
Every finite simple planar graph $G$ has a circle packing. That is, there exists a circle packing $P$ such that $G(P)$ is isomorphic to $G$.
\end{theorem}

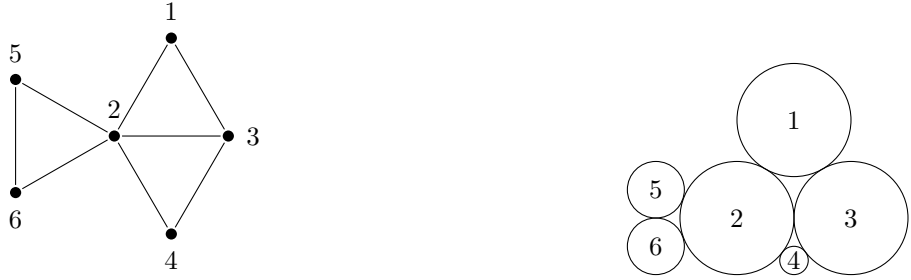
\begin{figure}[h]
  \centering
  \begin{subfigure}[b]{0.475\linewidth}
    \centering
    \begin{tikzpicture}[font=\small, scale=1.5]
      \node[vx,label=90:$1$]  (1) at (0.5,0.866) {};
      \node[vx,label=90:$2$]  (2) at (0,0) {};
      \node[vx,label=0:$3$]   (3) at (1,0) {};
      \node[vx,label=-90:$4$] (4) at (0.5,-0.866) {};
      \node[vx,label=90:$5$]  (5) at (-0.866,0.5) {};
      \node[vx,label=-90:$6$] (6) at (-0.866,-0.5) {};

      \draw (3) -- (1) -- (2) -- (3) -- (4) -- (2) -- (5) -- (6) -- (2);


    \end{tikzpicture}
  \end{subfigure}\hfill%
  \begin{subfigure}[b]{0.475\linewidth}
    \centering
    \begin{tikzpicture}[font=\small,scale=1.5]
      \node (1) at (0.5,0.866) {$1$};
      \node (2) at (0,0) {$2$};
      \node (3) at (1,0) {$3$};
      \node (4) at (0.5,-0.375) {$4$};
      \node (5) at (-0.7115,0.25) {$5$};
      \node (6) at (-0.7115,-0.25) {$6$};

      \draw (1) circle (0.5);
      \draw (2) circle (0.5);
      \draw (3) circle (0.5);
      \draw (4) circle (0.125);
      \draw (5) circle (0.25);
      \draw (6) circle (0.25);
    \end{tikzpicture}
  \end{subfigure}
  \caption{A planar graph and a circle packing of it.}
  \label{fig:circlepacking}
\end{figure}

One immediate consequence of the circle packing theorem is F\'ary's Theorem \cite{Fary}, which states that every finite simple planar graph can be drawn so that all the edges are represented by straight lines.

 The circle packing theorem was first discovered by Koebe \cite{K36}, who established it as a corollary to his work on the generalization of the Riemann mapping theorem to finitely connected domains; a brief sketch of Koebe's argument is given in \cref{fig:Koebesproof}. The theorem was rediscovered and popularized in the 70's by Thurston \cite{Th78}, who showed that it follows as a corollary to the work of Andreev on hyperbolic polyhedra (see also \cite{marden1990thurston}). 
  Thurston also initiated a popular program of understanding circle packing as a form of \emph{discrete complex analysis}, a viewpoint which has been highly influential in the subsequent development of the subject and which we discuss in more detail below  (see \cite{Smirnov} for a review of a different form of discrete complex analysis with many applications to probability). There are now many proofs of the circle packing theorem available including, remarkably, four distinct proofs discovered by Oded Schramm. In \cref{chp:cp} we will give an entirely combinatorial proof, which is adapted from the proof of Thurston \cites{Th78,marden1990thurston} and Brightwell and Scheinerman \cite{BriSch93}.

\begin{figure}
    \centering
    \begin{subfigure}[b]{0.2\textwidth}
        \includegraphics[width=\textwidth]{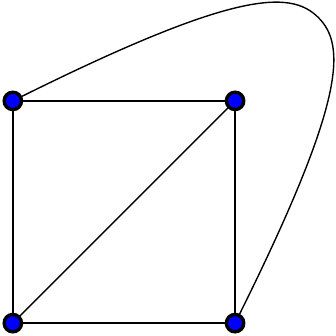}
        \caption{Step 1.}
    \end{subfigure}
    \qquad
    \begin{subfigure}[b]{0.2\textwidth}
        \includegraphics[width=\textwidth]{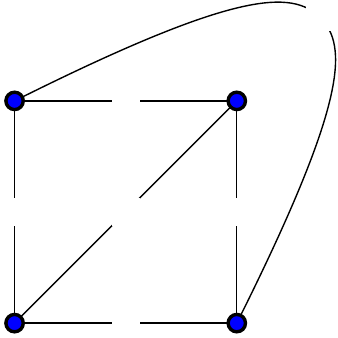}
        \caption{Step 2.}
    \end{subfigure}
    \qquad
    \begin{subfigure}[b]{0.2\textwidth}
        \includegraphics[width=\textwidth]{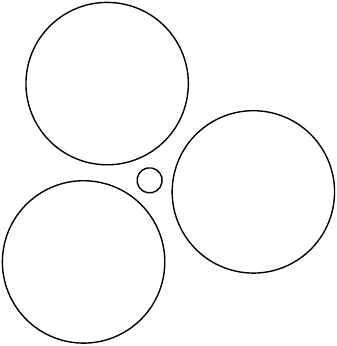}
        \caption{Step 3.}
    \end{subfigure}
    \qquad
        \begin{subfigure}[b]{0.2\textwidth}
        \includegraphics[width=\textwidth]{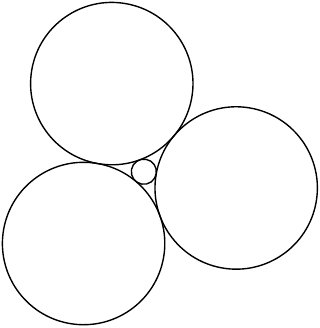}
        \caption{Step 4.}
    \end{subfigure}
    \caption{A sketch of how to obtain circle packings using Koebe's extension of the Riemann mapping theorem to finitely connected domains, which states that every domain $D \subseteq \CC \cup\{\infty\}$ with at most finitely many boundary components is conformally equivalent to a \emph{circle domain}, that is, a domain all of whose boundary components are circles or points. Step 1: We begin by drawing the finite simple planar graph $G$ in the plane in an arbitrary way. Step 2: If we remove the `middle $\eps$' of each edge, then the complement of the resulting drawing is a domain with finitely many boundary components. Step 3: Finding a conformal map from this domain to a circle domain gives an `approximate circle packing' of $G$. Step 4: Taking the limit as $\eps\downarrow 0$ can be proven to yield a circle packing of $G$. }\label{fig:Koebesproof}
\end{figure}

\subsection{Uniqueness} We cannot expect a uniqueness statement in \cref{thm:cp1} (see \cref{fig:circlepacking}; we may ``slide'' circles 5 and 6 along circle 2). However, when our graph is a \emph{finite triangulation}, circle packings enjoy uniqueness up to circle-preserving transformations.

\begin{definition}
  A planar \defn{triangulation} is a planar graph that can be drawn so that every face is incident to exactly three edges. In particular, when the graph is finite this property must hold for the outer face as well.
\end{definition}

\begin{claim}
\label{claim:rigidity}
If $G$ is a finite triangulation, then the circle packing whose tangency graph is isomorphic to $G$ is unique, up to M\"obius transformations and reflections in lines.
%
\end{claim}

The uniqueness of circle packing was first proven by Thurston, who noted that it follows as a corollary to Mostow's rigidity theorem. Since then, many different proofs have been found. In \cref{chp:cp} we will give a very short and elementary proof of uniqueness due to Oded Schramm that is based on the maximum principle. 


\subsection{Infinite planar graphs}
So far, we have only discussed the existence and uniqueness of circle packings of \emph{finite} planar triangulations. What happens with infinite triangulations? To address this question, we will need to introduce some more definitions.

\begin{figure}[t]
  \centering
  \begin{subfigure}[b]{0.475\linewidth}
    \centering
    \includegraphics[width=0.9\linewidth]{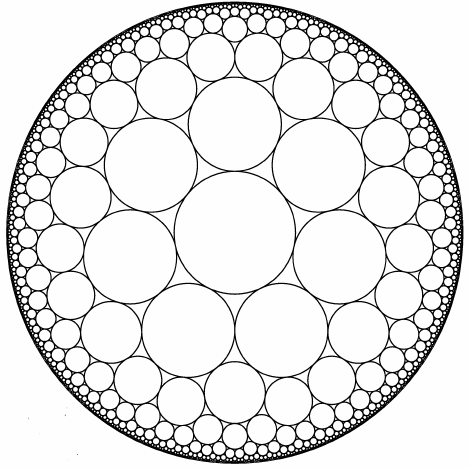}
  \end{subfigure}\hfill%
  \begin{subfigure}[b]{0.475\linewidth}
    \centering
    \includegraphics[width=0.9\linewidth]{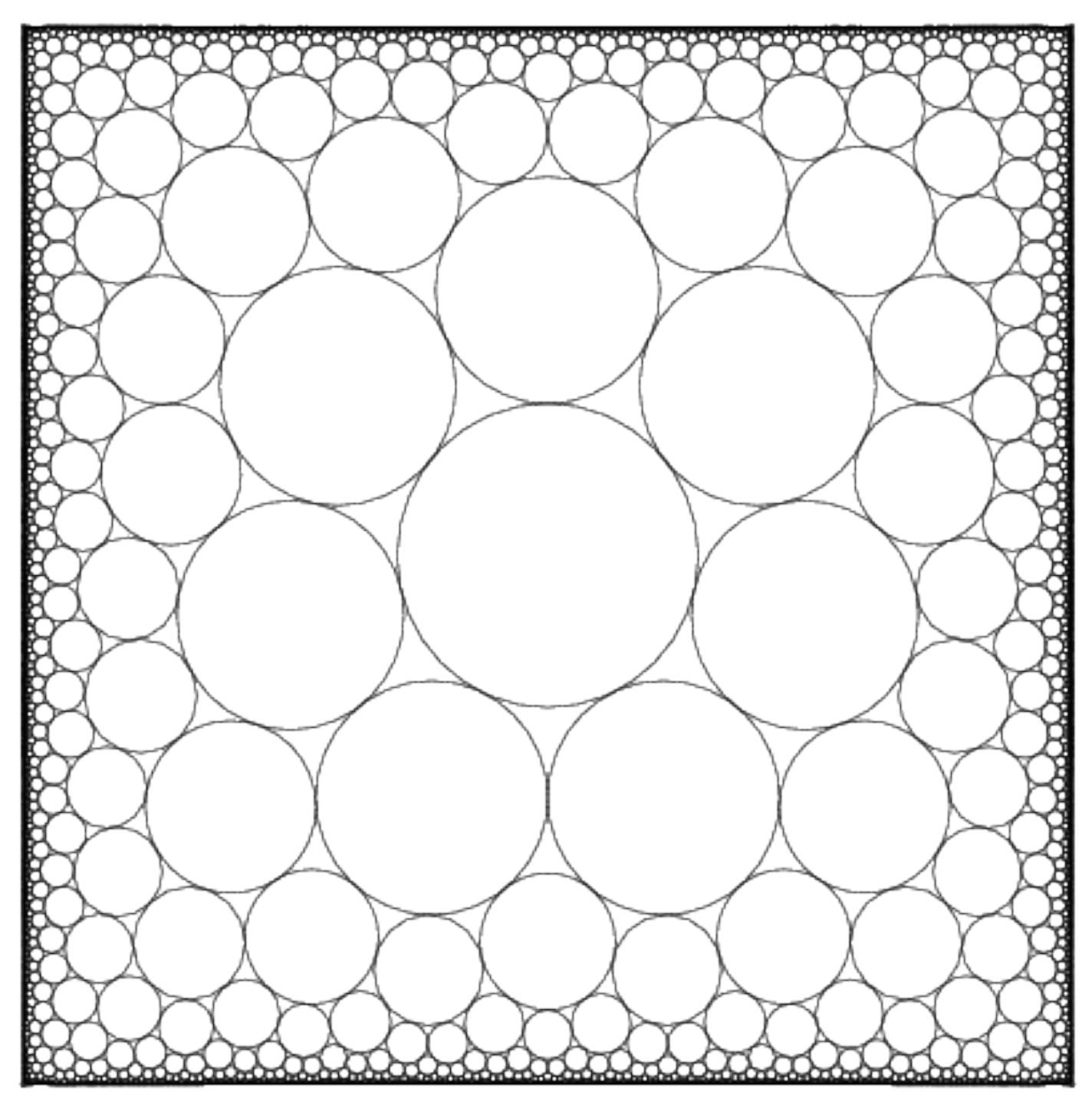}
  \end{subfigure}
  \caption{The 7-regular hyperbolic tessellation circle packed in a disc and in a square.}
  \label{fig:norigidity}
\end{figure}

\begin{definition}
We say that a graph $G$ is \defn{one-ended} 
 if the removal of any finite set of vertices leaves at most one infinite connected component.
\end{definition}

\begin{definition}
Let $P=\{C_v\}$ be a circle packing of a triangulation.
 We define the \defn{carrier} of $P$ to be the union of the closed discs bounded by the circles of $P$ together with the spaces bounded between any three circles that form a face (i.e., the interstices). We say that $P$ is \textbf{in} $D$ if its carrier is $D$.
\end{definition}

See \cref{fig:norigidity} for examples where the carrier is a disc or a square. The circle packing of the standard triangular lattice (see \cref{fig:CPtypes}) has the whole plane $\CC$ as its carrier. 
It is not too hard to see that if $G(P)$ is an infinite triangulation, then it is one-ended if and only if the carrier of $P$ is simply connected, see \cref{disk:one:ended}.

It can be shown via a compactness argument that any simple infinite planar triangulation can be circle packed in \emph{some} domain. Indeed, one can simply take subsequential limits of circle packings of finite subgraphs (the fact that such subsequential limits can be taken is a consequence of the ring lemma, \cref{lem:ring}). This is performed in \cref{cpt:inf}. However, this compactness argument does not give us any control of the domain we end up with as the carrier of our circle packing. The following theorems of He and Schramm \cite{HS93,HeSc} give us much better control; they can be thought of as discrete analogues of Poincar\'e-Koebe's uniformization theorem for Riemann surfaces.

\begin{theorem}[He and Schramm, '93] \label{thm:introheschramm}
  Any one-ended infinite triangulation can be circle packed such that the
  carrier is either the plane or the open unit disk, but not both. 
\end{theorem}

This theorem will be proved in \cref{chp:heschramm} (with the added assumption of finite maximal degree). The proofs in \cite{HS93,HeSc} are based on the notion of \emph{discrete extremal length}. We will present our own approach to the proof in \cref{chp:heschramm} based on a very similar notion of electric resistance discussed in \cref{chp:electric}. This approach is somewhat more appealing to a probabilist and allows for quantitative versions of the He-Schramm Theorem that will be used later for the study of random planar maps in \cref{chp:randommaps}.   

In view of \cref{thm:introheschramm}, we call an infinite one-ended simple planar triangulation \textbf{CP parabolic} if it can be circle packed in $\CC$, and call it \textbf{CP hyperbolic} if it can be circle packed in the open unit disk $\UU$.

\begin{theorem}[He and Schramm, '95]
\label{thm:HeSchramm_arbitrary_domain}
Let $T$ be a CP hyperbolic one-ended infinite planar triangulation and let $D \subsetneq \CC$ be a simply connected domain. Then there exists a circle packing of $T$ with carrier $D$. 
\end{theorem}

What about uniqueness? \cref{thm:HeSchramm_arbitrary_domain} shows that, in general, we have much more flexibility when choosing a circle packing of an infinite planar triangulation than we have in the finite case, see \cref{fig:norigidity} again. Indeed, it implies that the circle packing of a CP hyperbolic triangulation is \emph{not} determined up to M\"obius transformations and reflections, since, for example, we can circle pack the same triangulation in both the unit disc and the unit square, and these two packings are clearly not related by a M\"obius transformation. Fortunately, the following theorem of Schramm \cite{Schramm91} shows that we recover M\"obius rigidity if we restrict the packing to be in $\CC$ or $\UU$.

\begin{theorem}[Schramm '91]
\label{thm:Schramm_Rigidity}
 Let $T$ be a one-ended infinite planar triangulation. 
 \begin{itemize}
\item If $T$ is CP parabolic, then its circle packing in $\CC$ is unique up to dilations, rotations, translations and reflections.
\item If $T$ is CP hyperbolic, then its circle packing in $\UU$ is unique up to M\"obius transformations or reflections fixing $\UU$.
 \end{itemize}
\end{theorem}


\subsection{Relation to conformal mapping} 

A central motivation behind Thurston's popularization of circle packing was its role as a discrete analogue of conformal mapping. The resulting theory is somewhat tangential to the main thrust of these notes, but is worth reviewing for its beauty, and for the intuition it gives about circle packing. A  more detailed treatment of this and related topics is given in \cite{St05}.

Recall that a map $\phi: D \to D'$ between two domains $D,D' \subseteq \CC$ is conformal if and only if it is holomorphic and one-to-one. 
Intuitively, we can think of the latter condition as saying that $\phi$ maps infinitesimal circles to infinitesimal circles. Thus, it is natural to wonder, as Thurston did, whether conformal maps can be approximated by graph isomorphisms between circle packings of the corresponding domains, which \emph{literally} map circles to circles.

\begin{figure}[t]
\centering
\includegraphics[height=0.35\textwidth]{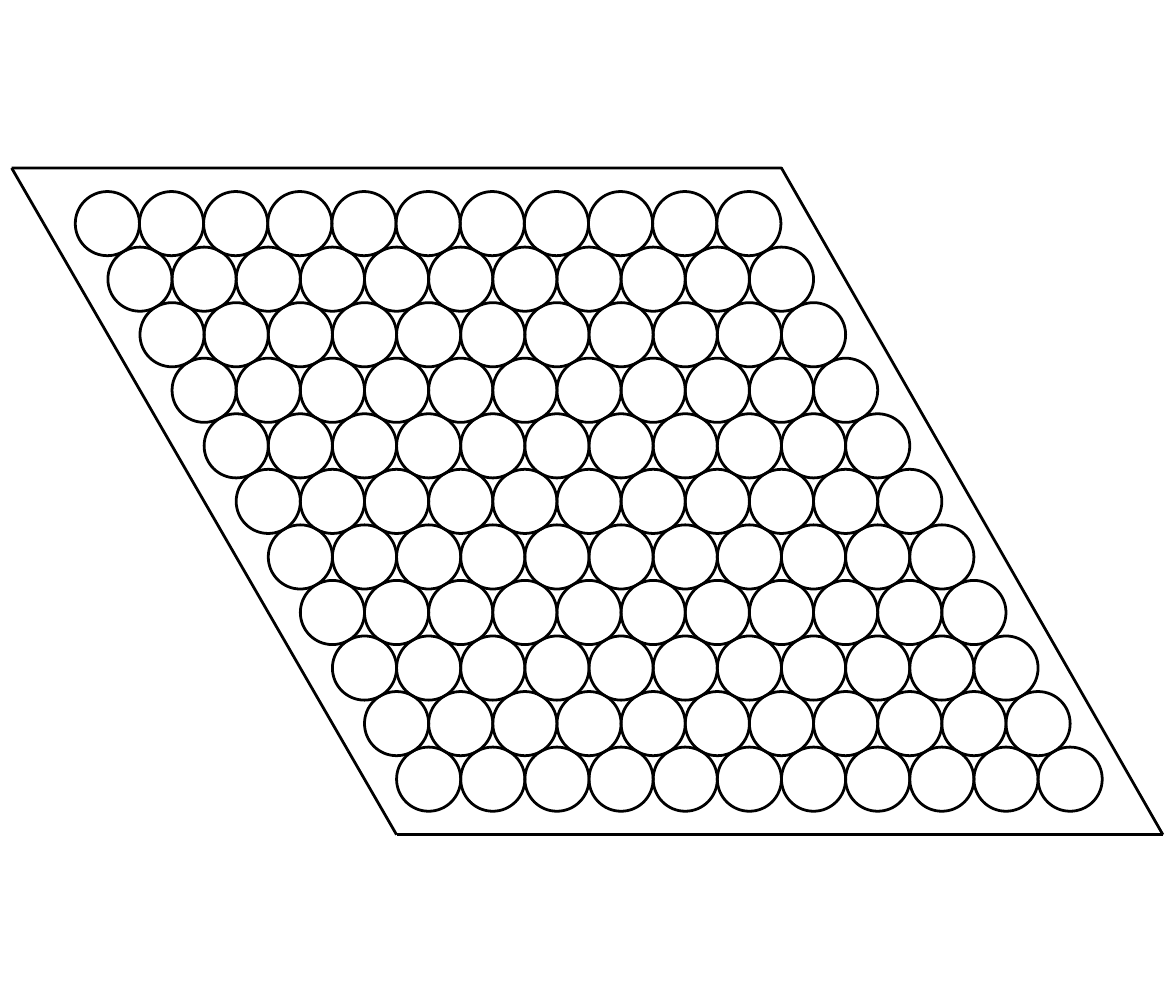}
\qquad
\includegraphics[trim = 5.47cm 7.33cm 3.25cm 7.65cm, clip, angle= 30, height=0.35\textwidth]{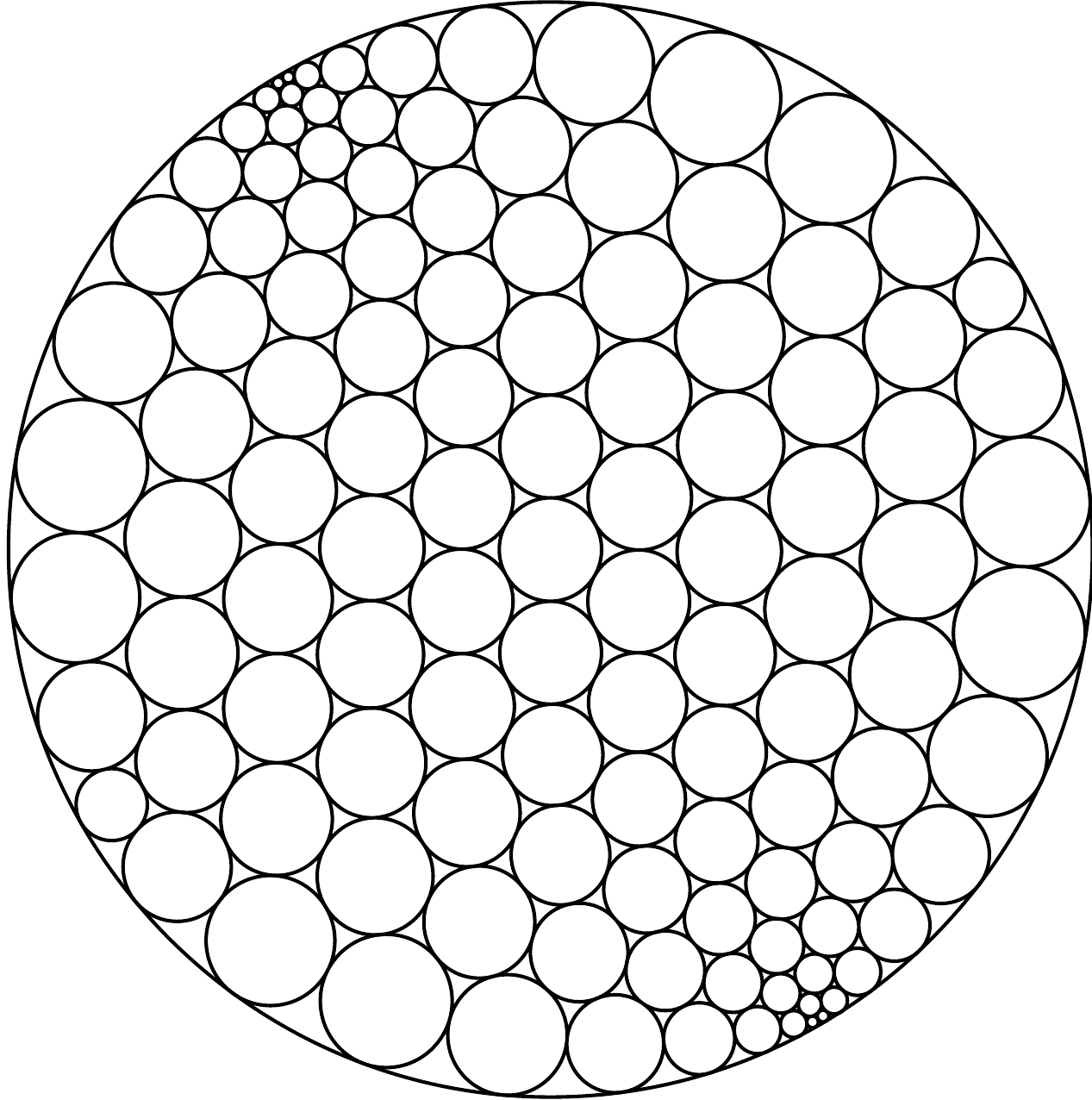}
\includegraphics[height=0.35\textwidth]{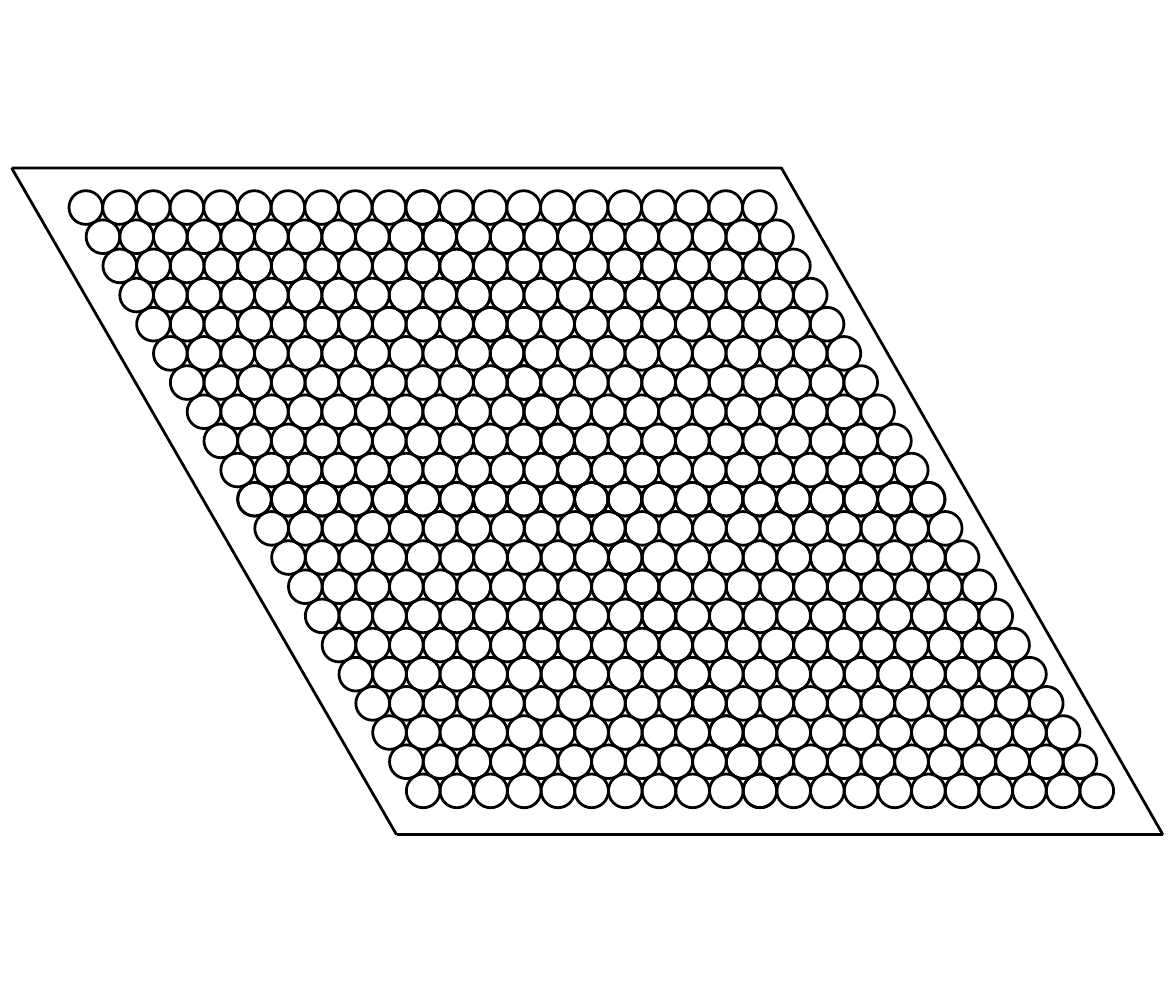}
\qquad
\includegraphics[
trim = 4.2cm 7.33cm 4.1cm 6.5cm, clip,
 angle= 30, height=0.35\textwidth]
 {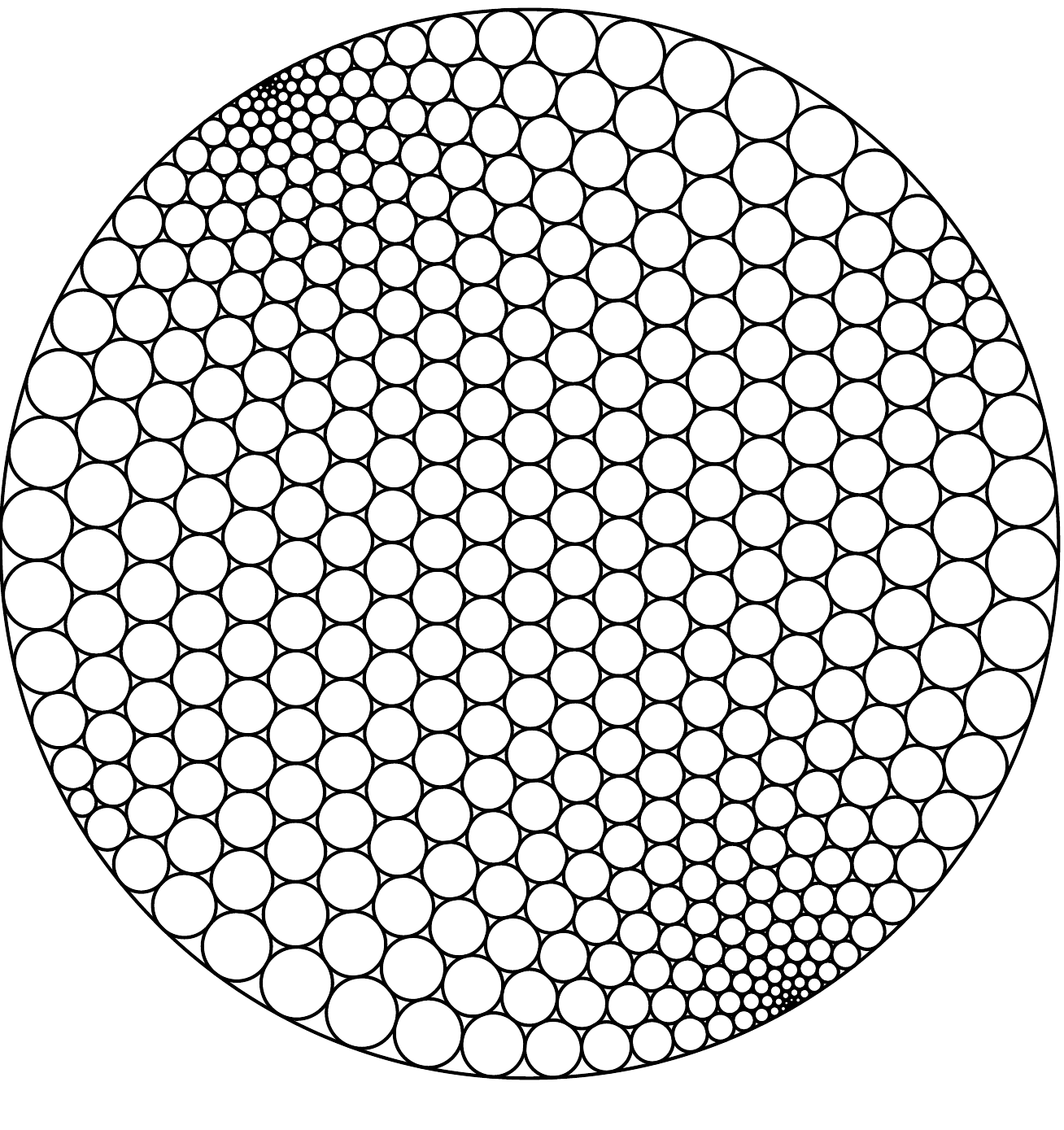}
\caption{Approximating the conformal map from a rhombus to the disc using circle packing, at two different degrees of accuracy.}
\label{fig:rodinsullivan}
\end{figure}

 For each $\eps>0$, let $\mathbb{T}_\eps =\{ \eps n + \eps \frac{1+\sqrt{3}i}{2}m: n,m\in \ZZ \} \subseteq \CC$ be the triangular lattice with lattice spacing $\eps$, which we make into a simple planar triangulation by connecting two vertices if and only if they have distance $\eps$ from each other. This triangulation is naturally circle packed in the plane by placing a disc of radius $\eps$ around each point of $\mathbb{T}_\eps$: this is known as the \textbf{hexagonal packing}. Now, let $D$ be a simply connected domain, and take $z_0$ to be a marked point in the interior of $D$. For each $\eps>0$ let $u_\eps$ be an element of $\mathbb{T}_\eps$ of minimal distance to $z_0$, and let $v_\eps=u_\eps+\eps$ and $w_\eps = u_\eps + (1+\sqrt{3}i)\eps/2$. For each $\eps>0$, let $T_\eps(D)$ be the subgraph of $\mathbb{T}_\eps$ induced by the vertices of distance at least $2\eps$ from $\partial D$ (i.e., the subgraph containing all such vertices and all the edges between them), and let $T'_\eps(D)$ be the component of $T_\eps(D)$ containing $u_\eps$. Finally, let $T''_\eps(D)$ be the triangulation obtained from $T'_\eps(D)$ by placing a single additional vertex $\partial_\eps$ in the outer face of $T'_\eps(D)$ and connecting this vertex to every vertex in the outer boundary of $T'_\eps(D)$.

Applying the circle packing theorem to $T''_\eps(D)$ and then applying a M\"obius transformation or a reflection if necessary, we obtain a circle packing $P_\eps$ of $T''_\eps(D)$ with the following properties:
\begin{itemize}
  \item The boundary vertex $\partial_\eps$ is represented by the unit circle,
  \item the vertex $u_\eps$ is represented by a circle centered at the origin, 
  \item the vertex $v_\eps$ is represented by a circle centered on the real line, and
  \item the vertex $w_\eps$ is represented by a circle centered in the upper half-plane.
\end{itemize}
The function sending each vertex of $T'_\eps(D)$ to the center of the circle representing it in $P_\eps$ can be extended piecewise on each triangle by an affine extension. Call the resulting function $\phi_\eps$.

The following theorem was conjectured by Thurston and proven by Rodin and Sullivan \cite{RS87}.

\begin{theorem}[Rodin and Sullivan '87]
\label{thm:RodinSullivan}
Let $\phi$ be the unique conformal map from $D$ to $\UU$ with $\phi(z_0)=0$ and $\phi'(z_0)>0$. Then $\phi_\eps$ converge to $\phi$ as $\eps \downarrow 0$, uniformly on compact subsets of $D$. 
\end{theorem}

See \cref{fig:rodinsullivan}. The key to the proof of \cref{thm:RodinSullivan} was to establish that the hexagonal packing is the only circle packing of the triangular lattice, which is now a special case of \cref{thm:Schramm_Rigidity}. 

Various strengthenings and generalizations of \cref{thm:RodinSullivan} have been established in the works \cite{MR1639851,MR1638772,MR1395721,MR1230272,MR1094463,MR1244888}.

\section{Probabilistic applications}
\label{sec:intro_prob}

Why should we be interested in circle packing as probabilists? At a very heuristic level, when we uniformize the \emph{geometry} of a triangulation by applying the circle packing theorem, we also uniformize the \emph{random walk} on the triangulation, allowing us to compare it to a standard reference process that we understand very well, namely Brownian motion. Indeed, since Brownian motion is conformally invariant and circle packings satisfy an approximate version of conformality, it is not unreasonable to expect that the random walk on a circle packed triangulation will behave similarly to Brownian motion. This intuition turns out to be broadly correct, at least when the triangulation has bounded degrees, although it is more accurate to say that the random walk behaves like a \emph{quasi-conformal image} of Brownian motion, that is, the image of Brownian motion under a function that distorts angles by a bounded amount.

Although it is possible to make the discussion in the paragraph above precise, in these notes we will be interested primarily in much coarser information that can be extracted from circle packings, namely \emph{effective resistance estimates} for planar graphs. This fundamental topic is thoroughly discussed in \cref{chp:electric}. One of the many definitions of the effective resistance $\reff(A \lr B)$ between two disjoint sets $A$ and $B$ in a finite graph is 
\[
\reff(A \lr B) = \sum_{v\in A} \deg(v) \pr_v(\tau_B < \tau^+_A),
\]
where $\pr_v$ is the law of the simple random walk started at $v$, $\tau_B$ is the first time the walk hits $B$, and $\tau_A^+$ is the first positive time the walk visits $A$.  Good enough control of effective resistances allows one to understand most aspects of the random walk on a graph. We can also define effective resistances on infinite graphs, although issues arise with boundary conditions. An infinite graph is recurrent if and only if the effective resistance from a vertex to infinity is infinite. 

The effective resistance can also be computed via either of two variational principles: \emph{Dirichlet's principle} and \emph{Thomson's principle}, see \cref{sec:energy}. The first expresses the effective resistance as a \emph{supremum} of energies of a certain set of \emph{functions}, while the second expresses the effective resistance as an \emph{infimum} of energies of a certain set of \emph{flows}. Thus, we can bound effective resistances from above by constructing flows, and from below by constructing functions. A central insight is that we can \emph{use the circle packing} to construct these functions and flows. This idea leads fairly easily to various statements such as the following:

\begin{itemize}
  \item The effective resistance across a Euclidean annulus of fixed modulus is at most a constant. If the triangulation has bounded degrees, then the resistance is at least a constant.
  \item The effective resistance between the left and right sides of a Euclidean square is at most a constant. If the triangulation has bounded degrees, then the resistance is at least a constant.
\end{itemize}

See for instance \cref{hs:1:lem:1}. We will use these ideas to prove the following remarkable theorem of He and Schramm \cite{HeSc}, which pioneered the connection between circle packing and random walks.

\begin{theorem}[He and Schramm, '95]
Let $T$ be a one-ended infinite triangulation. If $T$ has bounded degrees, then it is CP parabolic if and only if it is recurrent for simple random walk, that is, if and only if the simple random walk on $T$ visits every vertex infinitely often almost surely.
\end{theorem}

This has been extended to the multiply-ended cases in \cite{MR3607800}, see also \cref{chp:related}, item 4.

\subsection{Recurrence of distributional limits of random planar maps}

Random planar maps is a widely studied field lying at the intersection of probability, combinatorics and statistical physics. It aims to answer the vague question ``what does a typical random surface look like?''

 We provide here a very quick account of this field, referring the readers to the excellent lecture notes \cite{LeGallMiermontBuzios} by Le Gall and Miermont, and the many references within for further reading. The enumerative study of planar maps (answering questions of the form ``how many simple triangulations on $n$ vertices are there?'')\ began with the work of Tutte in the 1960's \cite{Tutte62} who enumerated various classes of finite planar maps, in particular triangulations. Cori and Vauquelin \cite{CV}, Schaeffer \cite{Sc} and Chassaing and Schaeffer \cite{CS} have found beautiful bijections between planar maps and labeled trees and initiated this fascinating topic in enumerative combinatorics. The bijections themselves are model dependent and extremely useful since many combinatorial and metric aspects of random planar maps can be inferred from them. This approach has spurred a new line of research: limits of large random planar maps. 

Two natural notions of such limits come to mind: scaling limits and local limits. In the first notion, one takes a random planar map $M_n$ on $n$ vertices, scales the distances appropriately (in most models the correct scaling turns out to be $n^{-1/4}$), and aims to show that this random metric space converges in distribution in the Gromov-Hausdorff sense. The existence of such limits was suggested by Chassaing and Schaeffer \cite{CS}, Le Gall \cite{legall07}, and Marckert and Mokkadem \cite{MM06}, who coined the term \emph{the Brownian map} for such a limit. The recent landmark work of Le Gall \cite{LeGall2013} and Miermont \cite{Miermont2013} establishes the convergence of random $p$-angulations for $p=3$ and all even $p$ to the Brownian map. 

The study of local limits of random planar maps, initiated by Benjamini and Schramm \cite{BeSc}, while bearing many similarities, is independent of the study of scaling limits. 
The \emph{local} limit of a random planar map $M_n$ on $n$ vertices is an infinite random rooted graph $(U,\rho)$ with the property that neighborhoods of $M_n$ around a random vertex converge in distribution to neighborhoods of $U$ around $\rho$. The infinite random graph $(U,\rho)$ captures the local behavior of $M_n$ around typical vertices. We develop this notion precisely in \cref{chp:locallimit}. 

In their pioneering work, Angel and Schramm \cite{AS03} showed that the local limit of a uniformly chosen random triangulation on $n$ vertices exists and that it is a one-ended infinite planar triangulation. They termed the limit as the \emph{uniform infinite planar triangulation} (UIPT). The uniform infinite planar quadrangulation (UIPQ), that is, the local limit of a uniformly chosen random quadrangulation (i.e., each face has $4$ edges) on $n$ vertices, was later constructed by Krikun \cite{Krikun}. 

The questions in this line of research concern the almost sure properties of this limiting geometry. It is a highly fractal geometry that is drastically different from $\ZZ^2$. Angel \cite{Angel03} proved that the volume of a graph-distance ball of radius $r$ in the UIPT is almost surely of order $r^{4+o(1)}$ and that the boundary component separating this ball from infinity has volume $r^{2+o(1)}$ almost surely. For the UIPQ this is proved in \cite{CS}. 


Due to the various combinatorial techniques of generating random planar maps, many of the metric properties of the UIPT/UIPQ are firmly understood. Surface properties of these maps are somewhat harder to understand using enumerative methods. Recall that a non-compact simply connected Riemannian surface is either conformally equivalent to the disc or the whole plane and that this is determined according to whether Brownian motion on the surface is transient or recurrent. Hence, the behavior of the simple random walk on the UIPT/UIPQ is considered here as a ``surface property'' (see also \cite{GillRohde}).

As mentioned earlier, one of the main objectives of these notes is to answer the question of the almost sure recurrence/transience of the simple random walk on the UIPT/UIPQ. We provide a general statement, \cref{thm:lim:rec} of these notes, to which a corollary is 

\begin{theorem}[\cite{GGN13}] \label{thm:uiptrecurrent}
The UIPT and UIPQ are almost surely recurrent.
\end{theorem}

The proof heavily relies on the circle packing theorem and can be viewed as an extension of the remarkable theorem of Benjamini and Schramm \cite{BeSc} stating that the local limit of finite planar maps with finite maximum degree is almost surely recurrent. The maximum degree of the UIPT is unbounded and so one cannot apply \cite{BeSc}. A combination of the techniques presented in \cref{chp:heschramm,chp:locallimit,chp:randommaps} is required to overcome this difficulty.

Recently, there have been terrific new developments studying further surface properties of the UIPT/UIPQ. Lee \cite{Lee17} has given an exciting new proof of \cref{thm:uiptrecurrent} based on a spectral analysis and an embedding theorem for planar maps due to \cite{KLPT11}. His proof also yields that the spectral dimension of the UIPT/UIPQ is at most $2$ and applies to local limits of sphere-packable graphs in higher dimensions as well. Gwynne and Miller \cite{GM17} provided the converse bound showing that the spectral dimension of the UIPT equals $2$ and calculated other exponents governing the behavior of the random walk. Their results are based on the deep work of Gwynne, Miller and Sheffield \cite{GMS17} (see also \cref{chp:related}, item 9). 

\chapter{Random walks and electric networks}\label{chp:electric}

An extremely useful tool and viewpoint for the study of random walks is Kirchhoff's theory of electric networks. Our treatment here roughly follows \cite[Chapter 8]{PeresClimb}, we also refer the reader to \cite{LyonsPeres} for an in-depth comprehensive study.

\begin{definition}
  A \defn{network} is a connected graph $G=(V,E)$  endowed with positive edge weights, $\{c_e\}_{e\in E}$ (called \textbf{conductances}\index{conductance}).  The reciprocals $r_e=1/c_e$ are
  called \textbf{resistances}\index{resistance}.
\end{definition}

In sections \ref{sec:voltages}---\ref{sec:energy} below we discuss finite networks. We extend our treatment to infinite networks in Section \ref{sec:infinitegraphs}.

\section{Harmonic functions and voltages}\label{sec:voltages}

Let $G=(V,E)$ be a finite network. In physics classes it is taught that when we impose specific voltages at fixed vertices $a$ and $z$, then current flows through the network according to certain laws (such as the series and parallel laws). An immediate consequence of these laws is that the function from $V$ to $\RR$ giving the voltage at each vertex is harmonic at each $x \in V \setminus \{a,z\}$.

\begin{definition}
  A function $h:V\to\RR$ is \textbf{harmonic} \index{harmonic function} at a vertex $x$ if
  \begin{equation}\label{eq:harmon}
    h(x) = \frac{1}{\pi_x}\sum_{y:y\sim x}c_{xy}h(y)
    \qquad \textrm{where} \qquad
    \pi_x:=\sum_{y:y\sim x}c_{xy}.
  \end{equation}
\end{definition}

Instead of starting with the physical laws and proving that voltage is harmonic, we now take the axiomatically 
equivalent approach of defining voltage to be a harmonic function and deriving the laws as corollaries.

\begin{definition} Given a network $G=(V,E)$ and two distinct vertices $a,z \in V$, 
  a \defn{voltage} is a function $h:V\to\RR$ that is harmonic at any
  $x\in V\setminus\{a,z\}$.
\end{definition}

We will show in \cref{voltage:exist} and \cref{voltage:unique} that for any $\alpha, \beta \in \RR$, there is a unique voltage $h$ such that $h(a) = \alpha$ and $h(z) = \beta$ (this assertion is true only when the network is finite).

\begin{claim}\label{harmon:linear}
  If $h_1,h_2$ are harmonic at $x$ then so is any linear combination of
  $h_1,h_2$.
\end{claim}

\begin{proof}
  Let $\bar{h}=\alpha h_1+\beta h_2$ for some $\alpha,\beta\in\RR$.
  It holds that
  \begin{equation*}
    \bar{h}(x)
    = \alpha h_1(x) + \beta h_2(x)
    = \frac{1}{\pi_x}\sum_{y:y\sim x}c_{xy} \alpha h_1(y)
      + \frac{1}{\pi_x}\sum_{y:y\sim x}c_{xy} \beta h_2(y)
    = \frac{1}{\pi_x}\sum_{y:y\sim x}c_{xy}\bar{h}(y).\qedhere
  \end{equation*}
\end{proof}

\begin{claim}\label{maximum}
  If $h:V\to\RR$ is harmonic at all the vertices of a finite network, then it is constant.
\end{claim}

\begin{proof}
  Let $M=\sup_{x}h(x)$ be the maximum value of $h$. Let $A=\{x\in V:h(x)=M\}$.
  Since $G$ is finite, $A\neq \emptyset$. Given $x \in A$, we have that $h(y) \leq h(x)$ for all neighbors $y$ of $x$. By harmonicity, $h(x)$ is the weighted average of the values of $h(y)$ at the neighbors; but this can only happen if all neighbors of $x$ are also in $A$. Since $G$ is connected
  we obtain that $A=V$ implying that $h$ is constant.
\end{proof}

%

We now show that a voltage is determined by its boundary values, i.e., by its
values at $a,z$.

\begin{claim}\label{voltage:boundary}
  If $h$ is a voltage satisfying $h(a)=h(z)=0$, then $h\equiv 0$.
\end{claim}

\begin{proof}
  Put $M=\max_{x}h(x)$ (which is attained since $G$ is finite) and let
  $A=\{x\in V:h(x)=M\}$. As before, by harmonicity, if $x\in
  A\setminus\{a,z\}$ then all of its neighbors are also in $A$. Since $G$
  is connected, there exists a simple path from $x$ to either $a$ or $z$ such that only its endpoint is in $\{a,z\}$. Since $h(a)=h(z)=0$ we learn that $M=0$, that is, $h$ is
  non-positive. Similarly, one proves that $h$ is non-negative, thus
  $h\equiv 0$.
\end{proof}


\begin{corollary}\label{voltage:unique}[Voltage uniqueness]
  For every $\alpha,\beta\in\RR$, if $h, h'$ are voltages satisfying
  $h(a)=h'(a)=\alpha$ and $h(z)=h'(z)=\beta$, then $h\equiv h'$. 
\end{corollary}
\begin{proof}
  By \cref{harmon:linear}, the function $h-h'$ is a voltage, taking the value
  $0$ at $a$ and $z$, hence by \cref{voltage:boundary} we get $h \equiv h'$.
\end{proof}

\begin{claim}\label{voltage:exist}
  For every $\alpha,\beta\in\RR$, there exists a voltage $h$ satisfying
  $h(a)=\alpha$, $h(z)=\beta$.
\end{claim}

\begin{proof}[Proof 1]
  We write $n=|V|$. Observe that a voltage $h$ with $h(a)=\alpha$ and
  $h(z)=\beta$ is defined by a system of $n-2$ linear equations of the form
  \eqref{eq:harmon} in $n-2$ variables (which are the values $h(x)$ for $x\in
  V\setminus\{a,z\}$). \Cref{voltage:unique} guarantees that the matrix
  representing that system has empty kernel, hence it is invertible.
\end{proof}

We present an alternative proof of existence based on the random walk on the
network. Consider the Markov chain $\{X_n\}$ on the state space $V$ with transition probabilities
\begin{equation}\label{wRW:def}
  p_{xy} := \pr(X_{t+1}=y\mid X_t=x) = \frac{c_{xy}}{\pi_x}.
\end{equation}
This Markov chain is a \defn{weighted random walk} (note that if $c_{xy}$ are
all $1$ then the described chain is the so-called \defn{simple random walk}).
We write $\pr_x$ and $\E_x$ for the probability and expectation, respectively,
conditioned on $X_0=x$. For a vertex $x$, define the \defn{hitting time} of $x$
by
\begin{equation*}
  \tau_x := \min\{t\ge 0\mid X_t=x\}.
\end{equation*}

\begin{proof}[Proof 2]
We will find a voltage $g$ satisfying $g(a)=0$ and $g(z)=1$ by setting
\begin{equation*}
  g(x) = \pr_x(\tau_z<\tau_a).
\end{equation*}
Indeed, $g$ is harmonic at $x\ne a,z$, since by the law of total probability
and the Markov property we have
\begin{equation*}
  g(x) = \frac{1}{\pi_x}\sum_{y:y\sim x} c_{xy} \pr_x(\tau_z<\tau_a\mid X_1=y)
  = \frac{1}{\pi_x}\sum_{y:y\sim x} c_{xy} \pr_y(\tau_z<\tau_a)
  = \frac{1}{\pi_x}\sum_{y:y\sim x} c_{xy} g(y).
\end{equation*}
For general boundary conditions $\alpha, \beta$ we define $h$ by
\begin{equation*}
  h(x) = g(x)\cdot(\beta-\alpha)+\alpha \, .
\end{equation*}
By \cref{harmon:linear}, $h$ is a voltage, and clearly $h(a)=\alpha$ and
$h(z)=\beta$, concluding the proof.
\end{proof}

This proof justifies the equality between simple random walk probabilities and voltages that was discussed at the start of this chapter: since the function $x \mapsto \pr_x( \tau_z < \tau_a )$ is harmonic on $V \setminus \{a,z\}$ and takes values $0,1$ at $a,z$ respectively, it must be equal to the voltage at $x$ when voltages $0,1$ are imposed at $a,z$.

\begin{claim}\label{voltage:maxprin} If $h$ is a voltage with $h(a) \leq
h(z)$, then $h(a) \leq h(x) \leq h(z)$ for all $x \in V$. 

Furthermore, if $h(a) < h(z)$ and $x \in V\setminus\{a,z\}$ is a vertex such that $x$ is in the connected component of $z$ in the graph $G \setminus \{a\}$, and $x$ is in the connected component of $a$ in the graph $G \setminus \{z\}$, then $h(a) < h(x) < h(z)$.  \end{claim} 

\begin{proof} This follows
directly from the construction of $h$ in Proof 2 of \cref{voltage:exist} and
the uniqueness statement of \cref{voltage:unique}. Alternatively, one can
argue as in the proof of \cref{voltage:boundary} that if $M = \max_x h(x)$ and
$m = \min_x h(x)$, then the sets $A = \{ x \in V : h(x) = M \}$ and $B = \{ x
\in V : h(x) = m \}$ must each contain at least one element of $\{a,z\}$.

To prove the second assertion, we note that by \cref{voltage:exist} and \cref{voltage:unique} it is enough to check when $h$ is the voltage with boundary values $h(a)=0$ and $h(z)=1$. In this case, the condition on $x$ guarantees that the probabilities that the random walk started at $x$ visits $a$ before $z$ or visits $z$ before $a$ are positive. By proof 2 of \cref{voltage:exist} we find that $h(x) \in (0,1)$. 
\end{proof}

\section{Flows and currents}\label{sec:flows:and:currents}

For a graph $G=(V,E)$, denote by $\vec{E}$ the set of edges of $G$, each
endowed with the two possible orientations.  That is, $(x,y)\in \vec{E}$ iff
$\{x,y\}\in E$ (and in that case, $(y,x)\in\vec{E}$ as well).

\begin{definition}
  A \textbf{flow from $a$ to $z$}\index{flow} in a network $G$ is a function
  $\theta:\vec{E}\to\RR$ satisfying
  \begin{enumerate}
    \item For any $\{x,y\}\in E$ we have $\theta(xy)=-\theta(yx)$ (\defn{antisymmetry}), and
    \item $\forall x\not \in \{a,z\}$ we have $\sum_{y:y\sim x}\theta(xy)=0$ (\defn{Kirchhoff's node law}\index{node law}).
  \end{enumerate}
\end{definition}

\begin{claim}\label{flow:linear}
  If $\theta_1,\theta_2$ are flows then, so is any linear combination of
  $\theta_1,\theta_2$.
\end{claim}

\begin{proof}
  Let $\bar{\theta}=\alpha \theta_1+\beta \theta_2$ for some
  $\alpha,\beta\in\RR$.
  It holds that
  \begin{equation*}
    \bar{\theta}(xy)
    = \alpha \theta_1(xy) + \beta \theta_2(xy)
    = -\alpha \theta_1(yx) - \beta \theta_2(yx)
    = -\bar{\theta}(yx),
  \end{equation*}
  and for $x\ne a,z$,
  \begin{equation*}
    \sum_{y:y\sim x}\bar{\theta}(xy)
    = \alpha\sum_{y:y\sim x} \theta_1(xy)
     +\beta\sum_{y:y\sim x} \theta_2(xy)=0.\qedhere
  \end{equation*}
\end{proof}

\begin{definition}
  Given a voltage $h$, the \defn{current flow} $\theta=\theta_h$ associated
  with $h$ is defined by $\theta(xy) = c_{xy}(h(y)-h(x))$.
\end{definition}

In other words, the voltage difference across an edge is the product of the current flowing along the edge with the resistance of the edge. This is known as \textbf{Ohm's law}. According to this definition, the current flows from vertices with lower voltage to vertices with higher voltage. We will use this convention throughout, but the reader should be advised that some other sources use the opposite convention.

\begin{claim}
  The current flow associated with a voltage is indeed a flow.
\end{claim}

\begin{proof}
  The current flow is clearly antisymmetric by definition.  To show that it
  satisfies the node law, observe that for $x\ne a,z$, since $h$ is harmonic,
  \begin{equation*}
    \sum_{y:y\sim x}\theta(xy)
    = \overbrace{\sum_{y:y\sim x}c_{xy}h(y)}^{=\pi_x h(x)}
    - \overbrace{\sum_{y:y\sim x}c_{xy}h(x)}^{=\pi_x h(x)} = 0.\qedhere
  \end{equation*}
\end{proof}

\begin{claim}\label{claim:cyclelaw}
  The current flow associated with a voltage $h$ satisfies \defn{Kirchhoff's cycle law}, that is, for every
  directed cycle $\vec{e}_1,\ldots,\vec{e}_m$,
  \begin{equation*}
    \sum_{i=1}^r r_{e_i}\theta(\vec{e}_i) = 0.
  \end{equation*}
\end{claim}

\begin{proof}
  Write $\vec{e}_i=(x_{i-1},x_i)$, and observe that $x_0=x_m$.  We have that
  \begin{equation*}
    \sum_{i=1}^m r_{e_i}\theta(\vec{e}_i)
    = \sum_{i=1}^m r_{x_{i-1}x_i}c_{x_{i-1}x_i}(h(x_i)-h(x_{i-1}))
    = \sum_{i=1}^m (h(x_i)-h(x_{i-1}))
    =0.\qedhere
  \end{equation*}
\end{proof}

For examples of a flow which does not satisfy the cycle law and a current flow,
see \cref{fig:flows}.

\begin{claim}\label{voltage:of:flow}
  Given a flow $\theta$ which satisfies the cycle law, there exists a voltage
  $h=h_\theta$ such that $\theta$ is the current flow associated with $h$. Furthermore, this voltage is unique up to an additive constant.
\end{claim}

\begin{proof}
  For every vertex $x$, let $\vec{e}_1,\ldots,\vec{e}_k$ be a path from $a$ to
  $x$, and define
  \begin{equation}\label{eq:hbypath}
    h(x) = \sum_{i=1}^k r_{e_i}\theta(\vec{e}_i).
  \end{equation}
  Note that since $\theta$ satisfies the cycle law, the right hand side of
  \eqref{eq:hbypath} does not depend on the choice of the path, hence $h(x)$ is
  well defined.  Let $x\in V$, and consider a
  given path $\vec{e}_1,\ldots,\vec{e}_k$ from $a$ to $x$ (if $x=a$ we take the empty path). To evaluate $h(y)$
  for $y\sim x$, consider the path $\vec{e}_1,\ldots,\vec{e}_k,xy$ from $a$ to
  $y$, so $h(y)=h(x) + r_{xy}\theta(xy)$.  It follows that
  $h(y)-h(x)=r_{xy}\theta(xy)$, hence $\theta(xy)=c_{xy}(h(y)-h(x))$, meaning
  that $\theta$ is indeed the current flow associated with $h$. 

  Since $\theta(xy)=c_{xy}(h(y)-h(x))$ for any $x \sim y$, the node law of immediately implies that $h$ is a voltage. To show that $h$ is unique up to an additive constant, suppose that $g:V\to\RR$ is another voltage such that $r_{xy} \theta(xy)=g(y)-g(x)$. It follows that $g(y)-g(y)=g(x)-g(x)$ for any $x \sim y$. Since $G$ is connected it follows that $g-h$ is the constant function on $V$.
\end{proof}

\begin{figure}[t]
  \centering
  \begin{subfigure}[b]{0.475\linewidth}
    \centering
    \begin{tikzpicture}[font=\small, scale=3]

      \node[vx,label=180:$a$] (a) at (-0.866,0) {};
      \node[vx,label=0:$z$]   (z) at (0.866,0) {};
      \node[vx]               (1) at (0,0.5) {};
      \node[vx]               (2) at (0,-0.5) {};

      \draw[-{Latex[length=3mm]}] (a) -- (1) node [midway,above,sloped] {$1$};
      \draw[-{Latex[length=3mm]}] (a) -- (2) node [midway,above,sloped] {$1$};
      \draw[-{Latex[length=3mm]}] (1) -- (2) node [midway,above,sloped] {$1$};
      \draw[-{Latex[length=3mm]}] (1) -- (z) node [midway,above,sloped] {$0$};
      \draw[-{Latex[length=3mm]}] (2) -- (z) node [midway,above,sloped] {$2$};

    \end{tikzpicture}
  \end{subfigure}\hfill%
  \begin{subfigure}[b]{0.475\linewidth}
    \centering
    \begin{tikzpicture}[font=\small, scale=3]

      \node[vx,label=180:$a$] (a) at (-0.866,0) {};
      \node[vx,label=0:$z$]   (z) at (0.866,0) {};
      \node[vx]               (1) at (0,0.5) {};
      \node[vx]               (2) at (0,-0.5) {};

      \draw[-{Latex[length=3mm]}] (a) -- (1) node [midway,above,sloped] {$1/2$};
      \draw[-{Latex[length=3mm]}] (a) -- (2) node [midway,above,sloped] {$1/2$};
      \draw[-{Latex[length=3mm]}] (1) -- (2) node [midway,above,sloped] {$0$};
      \draw[-{Latex[length=3mm]}] (1) -- (z) node [midway,above,sloped] {$1/2$};
      \draw[-{Latex[length=3mm]}] (2) -- (z) node [midway,above,sloped] {$1/2$};

    \end{tikzpicture}
  \end{subfigure}
  \caption{On the left, a flow of strength $2$ in which the cycle law is
    violated. On the right, the unit (i.e., strength $1$) current flow.}
  \label{fig:flows}
\end{figure}
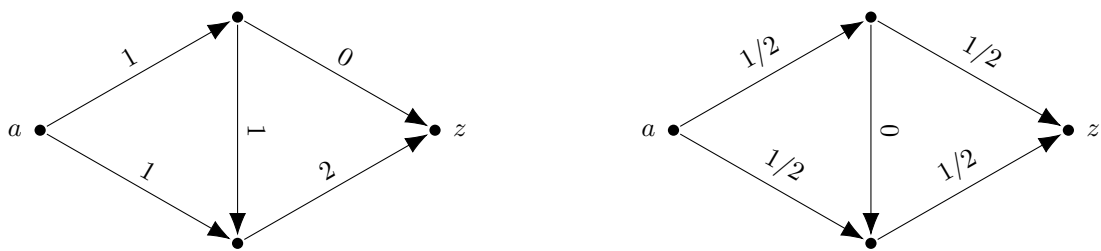

\begin{definition}
  The \textbf{strength} \index{strength of a flow} of a flow $\theta$ is
  \begin{equation*}
    \|\theta\|=\sum_{x:x\sim a}\theta(ax).
  \end{equation*}
\end{definition}

\begin{claim}\label{flow:in:flow:out}
  For every flow $\theta$,
  \begin{equation*}
    \sum_{x:x\sim z}\theta(xz) = \|\theta\|.
  \end{equation*}
\end{claim}

\begin{proof} We have that
  \begin{align*}
    0 &= \sum_{x\in V}\sum_{y:y\sim x}\theta(xy)\\
      &= \sum_{x\in V\setminus\{a,z\}}\sum_{y:y\sim x}\theta(xy)
        + \sum_{y:y\sim a}\theta(ay) + \sum_{y:y\sim z}\theta(zy)\\
      &= \sum_{y:y\sim a}\theta(ay) + \sum_{y:y\sim z}\theta(zy)
  \end{align*}
  where the first equality is due to antisymmetry, and the third equality is
  due to the node law. The claim follows again by antisymmetry.
\end{proof}

\begin{claim}\label{claim:cycleflows}
  If $\theta_1,\theta_2$ are flows satisfying the cycle law and
  $\|\theta_1\|=\|\theta_2\|$, then $\theta_1=\theta_2$.
\end{claim}

\begin{proof}
  Let $\bar{\theta}=\theta_1-\theta_2$.  According to \cref{flow:linear},
  $\bar{\theta}$ is a flow.  It also satisfies the cycle law, as for every
  cycle $\vec{e}_1,\ldots,\vec{e}_m$,
  \begin{equation*}
    \sum_{i=1}^m r_{e_i}\bar{\theta}(\vec{e}_i)
    = \sum_{i=1}^m r_{e_i}\theta_1(\vec{e}_i)
     -\sum_{i=1}^m r_{e_i}\theta_2(\vec{e}_i) = 0.
  \end{equation*}

  Observe in addition that $\|\bar{\theta}\|=\|\theta_1\|-\|\theta_2\|=0$.
  Now, let $h=h_{\bar{\theta}}$ be the voltage defined in
  \cref{voltage:of:flow}, chosen so that $h(a)=0$.  Note that it is harmonic
  at $a$, since
  \begin{align*}
    \frac{1}{\pi_a}\sum_{x:x\sim a}c_{ax}h(x)
    &= \frac{1}{\pi_a}\sum_{x:x\sim a}c_{ax}(h(a)+r_{ax}\bar\theta(ax))\\
    &= \frac{1}{\pi_a}\sum_{x:x\sim a}c_{ax}h(a)
      + \frac{1}{\pi_a}\sum_{x:x\sim a}\bar\theta(ax)
    = h(a)+ \frac{\|\bar{\theta}\|}{\pi_a} = h(a).
  \end{align*}
  Similarly, using \cref{flow:in:flow:out} it is also harmonic at $z$. Since
  $h$ is harmonic everywhere, it is constant by \cref{maximum}, and thus
  $h\equiv 0$, hence $\bar{\theta}\equiv 0$ and so $\theta_1=\theta_2$.
\end{proof}

This last claim prompts the following useful definition. 
\begin{definition}\label{def:unitcurrentflow}
  The \defn{unit current flow} from $a$ to $z$ is the unique current flow
  from $a$ to $z$ of strength $1$.
\end{definition}

\section{The effective resistance of a network}\label{sec:effectiveresistance}

Suppose we are given a voltage $h$ on a network $G$ with fixed vertices $a$ and $z$. Scaling $h$ by a constant multiple causes the associated current flow to scale by the same multiple, while adding a constant to $h$ does not change the current flow at all. Therefore, the strength of the current flow is proportional to the difference $h(z) - h(a)$. 
\begin{claim} \label{claim:ohm}
  For every non-constant voltage $h$ and a current flow $\theta$ corresponding
  to $h$, the ratio
  \begin{equation}\label{eq:reff}
    \frac{h(z)-h(a)}{\|\theta\|}
  \end{equation}
  is a positive constant which does not depend on $h$.
\end{claim}

\begin{proof}
  Let $h_1,h_2$ be two non-constant voltages, and let $\theta_1,\theta_2$ be their
  associated current flows. For $i=1,2$, let $\bar{h}_i=h_i/\|\theta_i\|$
  and let $\bar{\theta}_i$ be the current flow associated with $\bar{h}_i$ (note that since $h_i$ is non-constant $\|\theta_i\|\neq 0$). Thus, $\|\bar{\theta}_i\|=1$. By \cref{claim:cycleflows} we get $\bar{\theta}_1=\bar{\theta}_2$ and
  therefore $\bar{h}_1=\bar{h}_2+c$ for some constant $c$ by \cref{voltage:of:flow}. It follows that
  $\bar{h}_1(z)-\bar{h}_1(a)=\bar{h}_2(z)-\bar{h}_2(a)$.

  To see that this constant is positive, it is enough to check one particular choice of a voltage. By \cref{voltage:exist}, let $h$ be the voltage with $h(a)=0$ and $h(z)=1$. By \cref{voltage:maxprin} and since $G$ is connected, we have that $h(x) > 0$ for at least one neighbor $x$ of $a$. Thus, the corresponding current flow $\theta$ has $\|\theta\|>0$ making \eqref{eq:reff} positive.
\end{proof}

\cref{claim:ohm} is the mathematical manifestation of Ohm's law which states that the voltage difference across an electric circuit is proportional to the current through it. The constant of proportionality is usually called the \emph{effective resistance} of the circuit.

\begin{figure}[t]
  \centering
  \begin{subfigure}[t]{0.475\linewidth}
    \centering
    \begin{tikzpicture}[font=\small, scale=1]

      \node[vx,label=-90:$a$,label=90:$0$] (a) at (0,0) {};
      \node[vx,label=90:$1$]               (1) at (1,0) {};
      \node[vx,label=90:$2$]               (2) at (2,0) {};
      \node[vx,label=90:$3$]               (3) at (3,0) {};
      \node[vx,label=90:$4$]               (4) at (4,0) {};
      \node[vx,label=-90:$z$,label=90:$5$] (z) at (5,0) {};

      \draw (a) -- (1) -- (2) -- (3) -- (4) -- (z);

    \end{tikzpicture}
    \caption{\label{fig:reff:path} For the voltage depicted, the voltage difference
    between $a$ and $z$ is $5$, and the current flow's strength
    is $1$, hence the effective resistance is $5/1=5$.}
  \end{subfigure}\hfill%
  \begin{subfigure}[t]{0.475\linewidth}
    \centering
    \begin{tikzpicture}[font=\small, scale=3]

      \node[vx,label=180:$a$,label=90:$0$] (a) at (-0.866,0) {};
      \node[vx,label=0:$z$,label=90:$1$]   (z) at (0.866,0) {};
      \node[vx,label=90:$1/2$]               (1) at (0,0.1) {};
      \node[vx,label=-90:$1/2$]               (2) at (0,-0.1) {};

      \draw (1) -- (a) -- (2) -- (1) -- (z) -- (2);

    \end{tikzpicture}
    \caption{\label{fig:reff:diamond} For the voltage depicted, the voltage difference
    between $a$ and $z$ is $1$, and the current flow's strength is
    $1$, hence the effective resistance is $1/1=1$.}
  \end{subfigure}
  \caption{Examples for effective resistances of two networks with unit edge conductances.}
  \label{fig:reff}
\end{figure}

\begin{definition}
The number defined in \eqref{eq:reff} is called the \defn{effective resistance} between $a$ and $z$ in the network, and is denoted $\reff(a \lr z)$. We call its reciprocal the \defn{effective conductance} between $a$ and $z$ and is denoted $\ceff(a \lr z):=\reff(a\lr z)^{-1}$.
\end{definition}

For examples of computing the effective resistances of networks, see \cref{fig:reff}.

\paragraph{Notation}
In most cases we write $\reff(a\lr z)$ and suppress the notation of which
network we are working on. However, when it is important to us what the
network is, we will write $\reff(a\lr z ; G)$ for the effective resistance in
the network $G$ with unit edge conductances and $\reff(a\lr z ; (G,\{r_e\}))$
for the effective resistance in the network $G$ with edge resistances
$\{r_e\}_{e\in E}$.  Furthermore, given disjoint subsets $A$ and $Z$ of vertices in a graph
$G$, we write $\reff(A \lr Z)$ for the effective resistance
between $a$ and $z$ in the network obtained from the original network by
identifying all the vertices of $A$ into a single vertex $a$, and all the vertices of $Z$ into a single vertex $z$.

\paragraph{Probabilistic interpretation}

For a vertex $x$ we write $\tau_x^+$ for the stopping time 
\begin{equation}\label{eq:tauplusdef} \tau_x^+ = \min \{ t \geq 1 \mid X_t = x \} \, ,\end{equation}
where $X_t$ is the weighted random walk on the network, as defined in \eqref{wRW:def}. Note that if $X_0 \neq x$
then $\tau_x = \tau_x^+$ with probability $1$. 

\begin{claim}\label{effresprob}
$$ \reff(a\lr z) = \frac{1}{\pi_a \pr_a(\tau_z<\tau_a^+)} \, .$$
\end{claim}
\begin{proof}
Consider the voltage $h$ satisfying $h(a)=0$ and $h(z)=1$, and let
$\theta$ be the current flow associated with $h$.  Due to uniqueness of $h$
(\cref{voltage:unique}) we have
that for $x\ne a,z$,
\begin{equation*}
  h(x) = \pr_x(\tau_z<\tau_a),
\end{equation*}
hence
\begin{align*}
  \pr_a(\tau_z<\tau_a^+)
  &= \frac{1}{\pi_a}\sum_{x\sim a} c_{ax}\pr_x(\tau_z < \tau_a)\\
  &= \frac{1}{\pi_a}\sum_{x\sim a} c_{ax} h(x)\\  
  &= \frac{1}{\pi_a}\sum_{x\sim a} \theta(ax) 
  = \frac{\|\theta\|}{\pi_a} = \frac{1}{\pi_a\reff(a\lr z)} \, . \qedhere
\end{align*}
\end{proof}

\paragraph{Network Simplifications}
Sometimes a network can be replaced by a simpler network, without changing
the effective resistance between a pair of vertices.
\begin{claim}\label{parallellaw}[Parallel law]\index{parallel law}
  Conductances add in parallel. Suppose $e_1, e_2$ are parallel edges between
  a pair of vertices, with conductances $c_1$ and $c_2$, respectively. If we
  replace them with a single edge $e'$ with conductance $c_1+c_2$, then the
  effective resistance between $a$ and $z$ is unchanged.
\end{claim}

A demonstration of the parallel law appears in \cref{fig:parallel}.

\begin{proof}
Let $G'$ be the graph where $e_1$ and $e_2$ are replaced with $e'$ with
conductance $c_1+c_2$. Then it is immediate that if $h$ is any voltage
function on $G$, then it remains a voltage function on the network $G'$. The claim follows.
\end{proof}

\begin{figure}[t]
  \centering
  \begin{subfigure}[t]{0.475\linewidth}
    \centering
    \begin{tikzpicture}[font=\small, scale=4,baseline=(u.base)]]
      \node[vx,label=180:$u$]   (u) at (0,0) {};
      \node[vx,label=0:$v$]   (v) at (1,0) {};
      
      \draw -- (u) to [bend left=10] node[midway,above] {$c_1$} (v);
      \draw -- (u) to [bend right=10] node[midway,below] {$c_2$} (v);
    \end{tikzpicture}
  \end{subfigure}\hfill%
  \begin{subfigure}[t]{0.475\linewidth}
    \centering
    \begin{tikzpicture}[font=\small, scale=4, baseline=(u.base)]
      \node[vx,label=180:$u$]   (u) at (0,0) {};
      \node[vx,label=0:$v$]   (v) at (1,0) {};
      
      \draw -- (u) to node[midway,above] {$c_1 + c_2$} (v);
    \end{tikzpicture}
  \end{subfigure}
  \caption{Demonstrating the parallel law. Two parallel edges are replaced by a single edge.}\label{fig:parallel}
\end{figure}
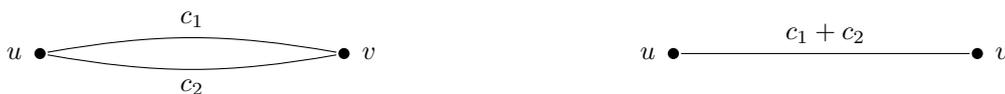

\begin{claim}\label{serieslaw}[Series law]\index{series law}
  Resistances add in series. Suppose that $u \not \in \{a,z\}$ is a vertex of
  degree $2$ and that $e_1=(u,v_1)$ and $e_2=(u,v_2)$ are the two edges
  touching $u$ with edge resistances $r_1$ and $r_2$, respectively. If we
  erase $u$ and replace $e_1$ and $e_2$ by a single edge $e'=(v_1,v_2)$ of
  resistance $r_1+r_2$, then the effective resistance between $a$ and $z$ is
  unchanged.
\end{claim}

The series law is depicted in \cref{fig:series:law}.

\begin{proof}
  Denote by $G'$ the graph in which $u$ is erased and $e_1$ and $e_2$ are
  replaced by a single edge $(v_1,v_2)$ of resistance $r_1+r_2$. Let $\theta$
  be a current flow from $a$ to $z$ in $G$, and define a flow $\theta'$ from
  $a$ to $z$ in $G'$ by putting $\theta'(e)=\theta(e)$ for any $e \neq
  e_1,e_2$ and $\theta'(v_1,v_2)=\theta(v_1,u)$. Since $u$ had degree $2$, it
  must be that $\theta(v_1,u)=\theta(u,v_2)$. Thus $\theta'$
  satisfies the node law at any $x\not\in\{a,z\}$ and 
  $\|\theta\|=\|\theta'\|$. Furthermore, since
  $\theta$ satisfies the cycle law, so does $\theta'$. We conclude $\theta'$ is a current flow of the same strength as $\theta$ and the voltage difference
  they induce is the same.
\end{proof}

\begin{figure}[t]
  \centering
  \begin{subfigure}[baseline]{0.475\linewidth}
    \centering
    \begin{tikzpicture}[font=\small, scale=4]
    \node[vx,label=180:$v_1$]   (u) at (0,0) {};
    \node[vx,label=90:$u$]   (z) at (0.5,0) {};
    \node[vx,label=0:$v_2$]   (v) at (1,0) {};
    
    \draw -- (u) to node[midway,above] {$r_1$} (z);
    \draw -- (z) to node[midway,above] {$r_2$} (v);
    \end{tikzpicture}
  \end{subfigure}
  \begin{subfigure}[baseline]{0.475\linewidth}
    \centering
    \begin{tikzpicture}[font=\small, scale=4]
      \node[vx,label=180:$v_1$]   (u) at (0,0) {};
      \node[vx,label=0:$v_2$]   (v) at (1,0) {};

      \draw -- (u) to node[midway,above] {$r_1 + r_2$} (v);
    \end{tikzpicture}
  \end{subfigure}
  \caption{\label{fig:series:law} An example of a network $G$ where edges in
    series are replaced by a single edge.}
\end{figure}
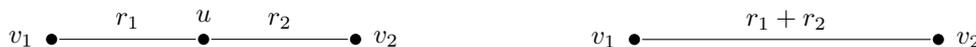

The operation of \defn{gluing} a subset of vertices $S \subset V$ consists of
identifying the vertices of $S$ into a single vertex and keeping all the edges 
and their conductances. In this process we may generate parallel edges or loops.

\begin{claim}\label{gluinglaw}[Gluing]\index{gluing}
  Gluing vertices of the same voltage does not change the effective resistance between $a$ and $z$.
\end{claim}
\begin{proof} This is immediate since the voltage on the glued graph is still harmonic. 
\end{proof}

\paragraph{Example: Spherically Symmetric Tree}
Let $\Gamma$ be a \defn{spherically symmetric tree}, that is, a rooted tree
where all vertices at the same distance from the root have the same number of
children. Denote by $\rho$ the root of the tree, and let $\{d_n\}_{n\in \NN}$
be a sequence of positive integers.  Every vertex at distance $n$ from the root $\rho$
has $d_n$ children.  Denote by $\Gamma_n$ the set of all vertices of height
$n$.  We would like to calculate $\reff(\rho \lr \Gamma_n)$.  Due to the
tree's symmetry, all vertices at the same level have the same voltage and
therefore by \cref{gluinglaw} we can identify them. Our simplified network
has now one vertex for each level,  denoted by $\{v_i\}_{i\in \NN}$ (where $\rho = v_0$), with
$|\Gamma_{n+1}|$ edges between $v_{n}$ and $v_{n+1}$. Using the parallel law
(\cref{parallellaw}), we can reduce each set of $|\Gamma_n|$ edges to a single edge
with resistance $\frac{1}{|\Gamma_n|}$, then, using the series law
(\cref{serieslaw}) we get
\begin{equation*}
  \reff(\rho \lr \Gamma_n)
  = \sum_{i=1}^{n} \frac{1}{|\Gamma_i|}
  = \sum_{i=1}^{n} \frac{1}{d_0\cdots d_{i-1}} \, ,
\end{equation*}
see Figure \ref{fig:parserlaw}.

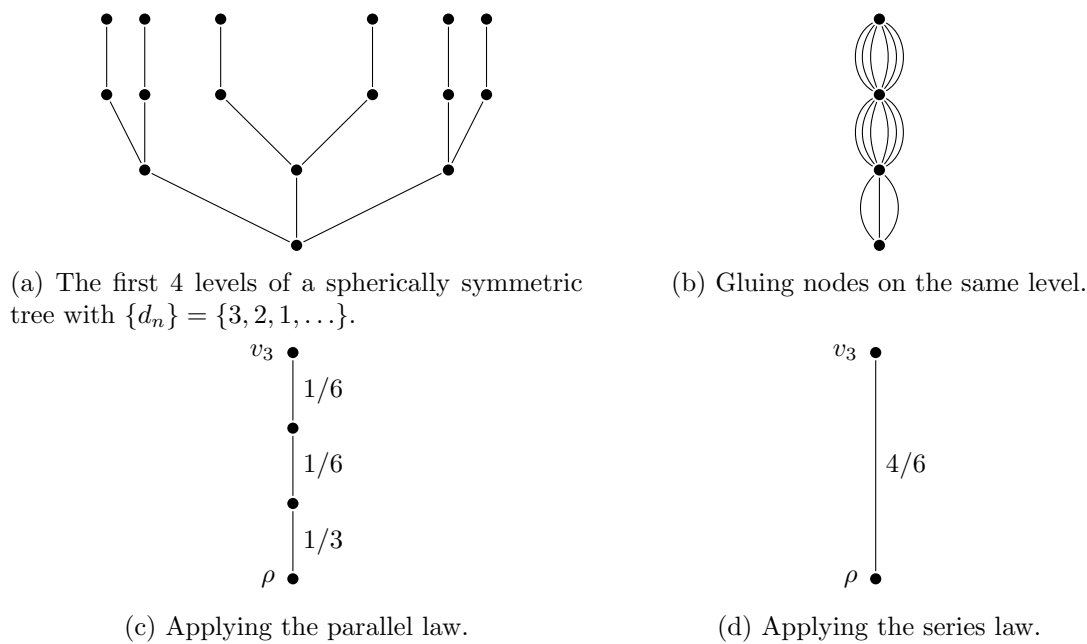
\begin{figure}[ht]
  \centering
  \begin{subfigure}[t]{0.475\linewidth}
    \centering
    \begin{tikzpicture}[font=\small, scale=3]
      \node[vx]   (1) at (0,0) {};
      \node[vx]   (11) at (-0.666,0.333) {};
      \node[vx]   (12) at (0,0.333) {};
      \node[vx]   (13) at (0.666,0.333) {};
      \node[vx]   (131) at (0.833,0.666) {};          \node[vx]   (132) at (0.666,0.666) {};              \node[vx]   (121) at (0.333,0.666) {};
      \node[vx]   (122) at (-0.333,0.666) {};         \node[vx]   (111) at (-0.666,0.666) {};             \node[vx]   (112) at (-0.833,0.666) {};
      \node[vx]   (1311) at (0.833,1) {};         \node[vx]   (1321) at (0.666,1) {};             \node[vx]   (1211) at (0.333,1) {};
      \node[vx]   (1221) at (-0.333,1) {};
      \node[vx]   (1111) at (-0.666,1) {};
      \node[vx]   (1121) at (-0.833,1) {};

      \draw (1) -- (11) (1) -- (12) (1) -- (13)
      (11) --(111) (11)--(112) (12)--(121) (12)--(122) (13)--(131) (13)--(132) (131)--(1311) (132)--(1321) (121)--(1211) (122)--(1221) (111)--(1111) (112)--(1121);


    \end{tikzpicture}
    \caption{The first $4$ levels of a spherically symmetric tree with
      $\{d_n\} = \{3,2,1,\ldots\}$.}
  \end{subfigure}
  \begin{subfigure}[t]{0.475\linewidth}
    \centering
    \begin{tikzpicture}[font=\small, scale=3]
      \node[vx]   (1) at (0,0) {};
      \node[vx]   (2) at (0,0.333) {};
      \node[vx]   (3) at (0,0.666) {};
      \node[vx]   (4) at (0,1) {};

      \draw -- (1) to [bend left=45] (2);
      \draw -- (1) to [bend right=45] (2);
      \draw -- (1) to (2);
      \draw -- (2) to [bend left=20] (3);
      \draw -- (2) to [bend left=40] (3);
      \draw -- (2) to [bend left=60] (3);
      \draw -- (2) to [bend right=20] (3);
      \draw -- (2) to [bend right=40] (3);
      \draw -- (2) to [bend right=60] (3);
      \draw -- (3) to [bend left=20] (4);
      \draw -- (3) to [bend left=40] (4);
      \draw -- (3) to [bend left=60] (4);
      \draw -- (3) to [bend right=20] (4);
      \draw -- (3) to [bend right=40] (4);
      \draw -- (3) to [bend right=60] (4);
    \end{tikzpicture}
    \caption{Gluing nodes on the same level.}
  \end{subfigure}
  \begin{subfigure}[b]{0.475\linewidth}
    \centering
    \begin{tikzpicture}[font=\small, scale=3]
      \node[vx,label=180:$\rho$]   (1) at (0,0) {};
      \node[vx]   (2) at (0,0.333) {};
      \node[vx]   (3) at (0,0.666) {};
      \node[vx,label=180:$v_3$]   (4) at (0,1) {};

      \draw -- (1) to node[right,midway] {1/3} (2);
      \draw -- (2) to node[right,midway] {1/6} (3);
      \draw -- (3) to node[right,midway] {1/6} (4);
    \end{tikzpicture}
    \caption{Applying the parallel law.}
  \end{subfigure}
  \begin{subfigure}[b]{0.475\linewidth}
    \centering
    \begin{tikzpicture}[font=\small, scale=3]
      \node[vx,label=180:$\rho$]   (1) at (0,0) {};
      \node[vx,label=180:$v_3$]   (4) at (0,1) {};

      \draw -- (1) to node[right,midway] {4/6} (4);
    \end{tikzpicture}
    \caption{Applying the series law.}
  \end{subfigure}
  \caption{Using network simplifications.}
  \label{fig:parserlaw}
\end{figure}

By \cref{effresprob} we learn that
\begin{equation}\label{eq:returnprobtree}
  \pr_\rho(\tau_{n} < \tau_\rho^+)
    = \frac{1}{d_0 \sum_{i=1}^n \frac{1}{d_0 \cdots d_{i-1}}} \, ,
\end{equation}
where $\tau_n$ is the hitting time of $\Gamma_n$ for the random walk on
$\Gamma$. Observe that 
\begin{equation*}
  \pr_\rho \left(\tau_{n} < \tau_\rho^+ \text{ for all } n \right)
    = \pr_\rho \left( X_t \text{ never returns to }\rho \right) \, ,
\end{equation*}
so by \eqref{eq:returnprobtree} we reach an interesting dichotomy. If
$\sum_{i=1}^\infty \frac{1}{d_1 \cdots d_i} = \infty$, then the random walker
returns to $\rho$ with probability $1$, and hence returns to $\rho$ infinitely
often almost surely. If $\sum_{i=1}^\infty \frac{1}{d_1 \cdots d_i} <
\infty$, then with positive probability the walker never returns to $\rho$, and
hence visits $\rho$ only finitely many times almost surely. 

The former graph is called a \defn{recurrent} graph and the latter is called
\defn{transient}. We will get back to this dichotomy in
\cref{sec:infinitegraphs}.

\subsection{The commute time identity}

The following lemma shows that the effective resistance between $a$ and $z$ is proportional to the expected time it takes the random walk starting at $a$ to visit $z$ and then return to $a$, in other words, the expected \emph{commute time} between $a$ and $z$. We will use this lemma only in \cref{chp:randommaps} so the impatient reader can skip this section and return to it later.

\begin{lemma}[Commute time identity]\label{commute}\index{commute time identity}
  Let $G=(V,E)$ be a finite network and $a\neq z$ two vertices. Then 
  \begin{equation*}
  \E_a[\tau_z] + \E_z[\tau_a] = 2 \reff(a\lr z) \sum_{e\in E} c_e
  \end{equation*}
\end{lemma}
\begin{proof}
  We denote by $G_z:V\times V \to \RR$ the Green function
  \begin{equation*}
  G_z(a,x) = \E_a[\text{number of visits to $x$ before $z$}]
  \end{equation*}
  and note that
  \begin{equation*}
  \E_a[\tau_z] = \sum_{x\in V} G_z(a,x).
  \end{equation*}
  It is straightforward to show that the function $\nu(x) = G_z(a,x)/\pi_x$ is harmonic in $V\setminus\{a,z\}$. Also, we have that $G_z(a,z) = 0$ and $G_z(a,a) = \frac{1}{\pr_a(\tau_z < \tau_a)} = \pi_a\reff(a\lr z)$.
  Thus, $\nu$ is a voltage function with boundary conditions $\nu(z) = 0$ and $\nu(a) = \reff(a\lr z)$ which satisfies
  \begin{equation*}
  \E_a[\tau_z] = \sum_{x\in V}\nu(x)\pi_x \, .
  \end{equation*}
  Similarly, the same analysis for $\E_z[\tau_a]$ yields the same result, with the voltage function $\eta$ which has boundary conditions $\eta(z) = \reff(a\lr z)$ and $\eta(a)=0$. Therefore, $\eta(x) = \nu(a) - \nu(x)$ for all $x \in V$ since both sides are harmonic functions in $V\setminus\{a,z\}$ that receive the same boundary values. This implies that \begin{equation*}
  \E_z[\tau_a] = \sum_{x\in V}\pi_x\left(\nu(a)-\nu(x)\right).
  \end{equation*}
  Summing these up gives \begin{equation*}
  \E_a[\tau_z] + \E_z[\tau_a] = \sum_{x\in V} \pi_x \nu(a) =  2 \sum_{e\in E} c_e \reff(a\lr z).\qedhere
  \end{equation*}
\end{proof}

\section{Energy}\label{sec:energy}

So far we have seen how to compute the effective resistance of a network via harmonic functions and current flows. However, in typical situations it is hard to find a flow satisfying the circle law. Luckily, an extremely useful property of the effective resistance is that it can be represented by a variational problem. Our Physics intuition asserts that the \emph{energy} of the unit current flow is minimal among all unit flows from $a$ to $z$. The notion of energy can be made precise and will allow us to obtain valuable monotonicity properties. For instance, removing any edge from an electric network can only increase its effective resistance. Hence, any recurrent graph remains recurrent after removing any subset of edges from it. Two variational problems govern the effective resistance, Thomson's principle, which is typically used to bound the effective resistance from above, and Dirichlet's principle, allowing to bound it from below.

\begin{definition}
  The \textbf{energy} \index{energy of a flow} of a flow $\theta$ from $a$ to $z$, denoted by $\energy(\theta)$, is defined to be
  \begin{equation*}
    \energy(\theta) := \frac{1}{2}\sum_{\vec{e}\in\vec{E}}r_{\vec{e}}\, \theta(\vec{e})^2 = \sum_{e\in E} \theta(e)^2 r_e.
  \end{equation*}
\end{definition}
Note that in the second sum we sum over undirected edges, but since
$\theta(xy)^2 = \theta(yx)^2$, this is well defined.

\begin{theorem}[Thomson's Principle]\label{thomson:principle}
  \index{Thomson's principle}
  \begin{equation*}
    \reff(a\lr z) = \inf\{\energy(\theta) : \|\theta\| = 1,
      \theta \text{ is a flow from } a\to z\}
  \end{equation*}
  and the unique minimizer is the unit current flow.
\end{theorem}

\begin{proof}
  First, we will show that the energy of the unit current flow is the
  effective resistance. Let $I$ be the unit current flow, and $h$ the
  corresponding (\cref{voltage:of:flow}) voltage function.
  \begin{align*}
    \energy(I) &= \frac{1}{2}\sum_{x\in V}\sum_{y:y\sim x}r_{xy}I(xy)^2
    = \frac{1}{2}\sum_{x\in V}\sum_{y:y\sim x} r_{xy}
      \left(\frac{h(y)-h(x)}{r_{xy}}\right) I(xy) \\
    &= \frac{1}{2}\sum_{x\in V}\sum_{y:y\sim x} \left(h(y)-h(x)\right) I(xy) 
    = \frac{1}{2}\sum_{x\in V}\sum_{y:y\sim x}h(y)I(xy) - \frac{1}{2}\sum_{x\in V}\sum_{y:y\sim x}h(x)I(xy).
  \end{align*}
  
  Observe that in the second term of the right hand side, for every $x\ne a,z$
  the sum over all $y\sim x$ is $0$ due to the node law, hence the entire term
  equals $\frac{1}{2}(h(a)-h(z))$.  From antisymmetry of $I$, the first term
  on the right hand side equals $-\frac{1}{2}(h(a)-h(z))$, hence the right hand
  side equals altogether $h(z)-h(a)=\reff(a\lr z)$.
  
  We will now show that every other flow $J$ with $\|J\| = 1$ has
  $\energy(J) \geq \energy(I)$. Let $J$ be such flow and write $J = I + (J
  - I)$. Set $\theta = J-I$ and note that $\|\theta\| = 0$. We have
  \begin{align*}
    \energy(J) &=
    \frac{1}{2}\sum_{x\in V}\sum_{y:y\sim x} r_{xy}(I(xy)+\theta(xy))^2 \\
    &= \frac{1}{2}\sum_{x\in V}\sum_{y:y\sim x}r_{xy}I(xy)^2
      + \frac{1}{2}\sum_{x\in V}\sum_{y:y\sim x}r_{xy}\theta(xy)^2
      + \sum_{x\in V}\sum_{y:y\sim x}r_{xy}\theta(xy)I(xy) \\
    &\geq \energy(I) + \energy(\theta) + \sum_{x\in V}\sum_{y:y\sim x}r_{xy}\theta(xy)I(xy).
  \end{align*}
  Now,
  \begin{align*}
    \sum_{x\in V}\sum_{y:y\sim x}r_{xy}\theta(xy)I(xy)
    &= \sum_{x\in V}\sum_{y:y\sim x}r_{xy}\theta(xy)\frac{(h(y)-h(x))}{r_{xy}}\\
    &= \sum_{x\in V}\sum_{y:y\sim x}\theta(xy)\left(h(y)-h(x)\right) \\
    &= 2\cdot \|\theta \|\cdot \left(h(z)-h(a)\right) = 0,
  \end{align*}
  where the last inequality follows from the same reasoning as before. We
  conclude that $\energy(J) \geq \energy(I)$ as required and that equality
  holds if and only if $\energy(\theta)=0$, that is, if and only if $J = I$.
\end{proof}

\begin{corollary}[Rayleigh's Monotonicity Law]\label{rayleigh}
  \index{Rayleigh's Monotonicity Law}
  If $\{r_e\}_{e\in E}$ and $\{r'_e\}_{e\in E}$ are edge resistances on the
  same graph $G$ so that $r_e \leq r'_e$ for all edges $e\in E$, then
  \begin{equation*}
    \reff(a\lr z ; (G,\{r_e\})) \leq \reff(a\lr z ; (G,\{r'_e\})).
  \end{equation*}
\end{corollary}

\begin{proof}
  Let $\theta$ be a flow on $G$, then
  \begin{equation*}
    \sum_{e\in E} r_{e}\theta(e)^2
    \leq \sum_{e \in E} r'_{e}\theta(e)^2.
  \end{equation*}
  This inequality is preserved while taking infimum over all flows with strength $1$. Applying
  \cref{thomson:principle} finishes the proof.
\end{proof}

\begin{corollary}\label{gluing:monotone}
  Gluing vertices cannot increase the effective resistance between $a$ and $z$.
\end{corollary}
\begin{proof}
  Denote by $G$ the original network and by $G'$ the network obtained from
  gluing a subset of vertices. Then every flow $\theta$ on $G$ (viewed as a function on the edges) is a flow on $G'$. Hence the infimum in
  \cref{thomson:principle} taken over flows in $G'$ is taken over a larger
  subset of flows.
\end{proof}

\begin{definition}
  The \textbf{energy} \index{energy of a function} of a \emph{function} $h:V\to\RR$, denoted by $\energy(h)$, is 
  defined to be
  \begin{equation*}
    \energy(h) := \sum_{\{x,y\}\in E} c_{xy}(h(x)-h(y))^2.
  \end{equation*}
\end{definition}

Compare the following lemma with Thomson's principle 
(\cref{thomson:principle}).
\begin{lemma}[Dirichlet's principle]\label{lem:dirichlet}
  \index{Dirichlet's principle}
  Let $G$ be a finite network with source $a$ and sink $z$.  Then
  \begin{equation*}
    \frac{1}{\reff(a\lr z)} = \inf\big\{\energy(h):
    h:V\to\RR, \, h(a)=0, \, h(z)=1 \big\}.
  \end{equation*}
\end{lemma}

\begin{proof}
  The infimum is obtained when $h$ is \emph{the} harmonic function taking $0$
  and $1$ at $a,z$ respectively.  The reason is that if there exists
  $v\ne a,z$ with
  \begin{equation}\label{eq:dirichlet}
    h(v)\ne\sum_{u\sim v}\frac{c_{vu}}{\pi_v} h(u),
  \end{equation}
  then we can change the value of $h$ at $v$ to be the right hand side of 
  \eqref{eq:dirichlet} and the energy will only decrease. One way to see this is that if $X$ is a random variable then the value $\E(X)$ minimizes the 
  function $f(x)=\E\left( (X-x)^2 \right)$.
  
  Let $h$ be that harmonic function and let $I$ be its current flow, so
  $I(xy)=c_{xy}(h(y)-h(x))$.  Write $\hat{I}=\reff(a\lr z)\cdot I$, so
  $\|\hat{I}\|=1$. By Thomson's principle,
  \begin{equation*}
    \reff(a\lr z) = \energy(\hat{I})
    = \sum_{e\in E}r_e\hat{I}(e)^2
    = \sum_{\{x,y\}\in E}r_{xy}\reff(a\lr z)^2 c_{xy}^2 (h(y)-h(x))^2,
  \end{equation*}
  hence
  \begin{equation*}
    \frac{1}{\reff(a\lr z)} = \energy(h).\qedhere
  \end{equation*}
\end{proof}
\section{Infinite graphs}\label{sec:infinitegraphs}

Let $G=(V,E)$ be an infinite connected graph with edge
resistances $\{r_e\}_{e\in E}$. We assume henceforth that this network is \defn{locally finite}, that is, for any vertex $x\in V$ we have $\sum_{y : y \sim x} c_{xy} < \infty$. Let $\{G_n\}$ be a sequence of finite subgraphs of $G$ such that
$\bigcup_{n\in \NN} G_n = G$ and $G_n \subset G_{n+1}$; we call such a sequence an \defn{exhaustive sequence} of $G$. Identify all vertices
of $G \setminus G_n$ with a single vertex $z_n$.

\begin{claim}
  Given an exhaustive sequence $\{G_n\}$ of $G$, the limit
  \begin{equation} \label{eq:resistancetoinfinity}
    \lim_{n\to \infty} \reff(a\lr z_n ; G_n \cup \{z_n\})
  \end{equation}
  exists.
\end{claim}
\begin{proof}
  The graph $G_{n} \cup \{z_n\}$ can be obtained
  from $G_{n+1} \cup \{z_{n+1}\}$ by gluing the vertices in $G_{n+1}\setminus G_n$ with
  $z_{n+1}$ and labeling the new vertex $z_n$. By \cref{gluing:monotone}, the effective resistance $\reff(a\lr z_n ; G_n \cup \{z_n\})$ is increasing in $n$. 
\end{proof}

\begin{claim}\label{claim:sandwiching}
  The limit in \eqref{eq:resistancetoinfinity} does not depend on the choice of exhaustive sequence $\{G_n\}$.
\end{claim}

\begin{proof}
  Indeed, let $\{G_n\}$ and $\{G'_n\}$ be two exhaustive sequences of $G$. We can find subsequences $\{i_k\}_{k \geq 1}$ and $\{j_k\}_{k\geq 1}$ such that
  \begin{equation*}
    G_{i_1} \subseteq G'_{j_1} \subseteq G_{i_2} \subseteq \ldots
  \end{equation*}
 Since $\{ G_{i_1}, G'_{j_1}, G_{i_2}, \ldots \}$ is itself an exhaustive sequence of G, the limit of effective resistances for this sequence exists and equals the limits of effective resistances for the subsequences $\{ G_{i_k} \}$ and $\{ G'_{j_k} \}$. In turn, these are equal to the limits of effective resistances for the original sequences $\{ G_n \}$ and $\{ G'_n \}$, respectively.
 \end{proof}

\begin{definition} In an infinite network, the \text{effective resistance} \index{effective resistance in an infinite network}  from a vertex $a$ and $\infty$ is
  \begin{equation*}
    \reff(a\lr \infty) := \lim_{n\to \infty} \reff(a\lr z_n ; G_n \cup \{z_n\}) \, .
  \end{equation*}
\end{definition}

We are now able to address the question of recurrence versus transience of a graph systematically. Recall the definition of $\tau_x^+$ in \eqref{eq:tauplusdef}. In an infinite network we define $\tau_a^+=\infty$ when there is no time $t$ such that $X_a=x$. 

\begin{definition}
  A network $(G,\{r_e\}_{e\in E})$ is called \textbf{recurrent} \index{recurrent network} if
  $\pr_a(\tau_a^+=\infty) = 0$, that is, if the probability of the random walker started at $a$ never returning to $a$ is $0$. Otherwise, it is called \textbf{transient} \index{transient network}.
\end{definition}

Observe that since $G$ is connected, if $\pr_a(\tau_a^+=\infty)=0$ for one vertex $a$, then it holds for all vertices in the network. As we have seen, if $n$ is large enough so that $a\in G_n$, then
\begin{equation*}
  \reff(a\lr z_n ; G_n \cup \{z_n\})
  = \frac{1}{\pi_a \cdot \pr_a\left(\tau_{G\setminus G_{n}} < \tau_a^+\right)} \,\,.
\end{equation*}
Since $\bigcap_n \{\tau_{G\setminus G_{n}} < \tau_a^+\} = \{\tau_a^+=\infty\}$ we have
\begin{equation*}
  \reff(a\lr \infty) = \frac{1}{\pi_a \cdot \pr_a(\tau_a^+=\infty)} \,\, ,
\end{equation*}
with the convention that $1/0=\infty$.

\begin{definition}
  Let $G$ be an infinite network. A function $\theta:E(G)\to \RR$ is a
  \defn{flow from $a$ to $\infty$} if it is anti-symmetric and satisfies the
  node law on each vertex $v\neq a$.
\end{definition}

The following follows easily from \cref{thomson:principle}, we omit the proof.

\begin{theorem}[Thomson's principle for infinite networks]
  \label{thomson:principle:infinite}
  Let $G$ be an infinite network, then
  \begin{equation*}
  \reff(a\lr\infty) = \inf\{\energy(\theta)
    : \theta \text{ is a flow from } a\to\infty \text{ of strength } 1\}.
  \end{equation*}
\end{theorem}

\begin{corollary}\label{cor:transiencecondition}
  Let $G$ be an infinite graph. The following are equivalent:
  \begin{enumerate}
    \item $G$ is transient.
    \item There exists a vertex $a\in V$ such that $\reff(a\lr\infty) < \infty$. Hence all vertices satisfy this.
    \item There exists a vertex $a\in V$ (and hence all vertices) and a unit flow $\theta$ from $a$ to $\infty$ with $\energy(\theta) <\infty$. Hence all vertices satisfy this.
  \end{enumerate}
\end{corollary}

We will now develop a useful method for bounding effective resistances from below. This will lead us to a popular sufficient criterion for recurrence in \cref{nash:williams:recurrence}.

\begin{definition}
  A \defn{cutset} $\Gamma \subseteq E(G)$ separating $a$ from $z$ is a set of edges such that every path from $a$ to $z$ must use an edge from $\Gamma$.
\end{definition}

\begin{claim}\label{flow:cutset} Let $\theta$ be a flow from $a$ to $z$ in a finite network, and let $\Gamma$ a cutset separating $a$ from $z$. Then
  \begin{equation*}
    \sum_{e\in \Gamma} |\theta(e)| \geq \|\theta\|.
  \end{equation*}
\end{claim}
\begin{proof}
  Denote by $Z$ the set of vertices separated from $a$ by $\Gamma$. Denote by $G'$ the network where $Z$ is identified to a single vertex $x$ and all edges having both endpoints in $Z$ are removed.
  Now, the restriction of $\theta$ to the edges of the new network is a flow from $a$ to $x$. By \cref{flow:in:flow:out}, we have $\sum_{y:y\sim x}\theta(yx) = \|\theta\|$. Also, all edges incident to $x$ must be in $\Gamma$, since otherwise $x$ is not separated from $a$ by $\Gamma$.
  Therefore 
  \begin{equation*}
   \sum_{e\in \Gamma} |\theta(e)| \geq \sum_{y:y\sim x}\theta(yx) = \|\theta\|.\qedhere
   \end{equation*}
\end{proof}

\begin{theorem}[Nash-Williams inequality]\index{Nash-Williams inequality}
  Let $\{\Gamma_n\}$ be disjoint cutsets separating $a$ from $z$ in a finite network. Then
  \begin{equation*}
    \reff(a\lr z) \geq \sum_n \left(\sum_{e\in \Gamma_n} c_e\right)^{-1}.
  \end{equation*}
\end{theorem}

\begin{proof}
    Let $\theta$ be a flow from $a$ to $z$ with $\|\theta\|=1$. From
    Cauchy-Schwarz, for each $n$ we have
    \begin{equation*}
      \left(\sum_{e\in \Gamma_n} \sqrt{r_e}\sqrt{c_e}|\theta(e)|\right)^2
      \leq \sum_{e\in \Gamma_n} c_e \sum_{e\in \Gamma_n} r_e \theta(e)^2 \, .
    \end{equation*}
    Also, since $\Gamma_n$ is a cutset, the flow passing through $\Gamma_n$ is at least $\|\theta\|$, by \cref{flow:cutset}. So
    \begin{equation*}
      \left(\sum_{e\in \Gamma_n} \sqrt{r_e}\sqrt{c_e}|\theta(e)|\right)^2
      \geq \|\theta\|^2 = 1.
    \end{equation*}
    Combining them, we get that
    \begin{equation*}
      \sum_{e\in \Gamma_n} r_e \theta(e)^2
      \geq \frac{1}{\sum_{e\in \Gamma_n} c_e}.
    \end{equation*}
    Summing over all $n$ gives
    \begin{equation*}
      \energy(\theta) \geq \sum_n \sum_{e\in \Gamma_n} r_e \theta(e)^2
      \geq \sum_n \left(\sum_{e\in \Gamma_n} c_e\right)^{-1}.
    \end{equation*}
    Applying Thomson's principle (\cref{thomson:principle}) yields the result.
\end{proof}

Consider now an infinite network $G=(V,E)$. We say that $\Gamma \subset E$ is a cutset separating $a$ from $\infty$ if any infinite simple path from $a$ must intersect $\Gamma$.

\begin{corollary} \label{nash:williams:recurrence} In any infinite network, 
  if there exists a collection $\{\Gamma_n\}$ of disjoint cutsets separating $a$ from $\infty$ such that
  \begin{equation*}
    \sum_n \left(\sum_{e\in \Gamma_n} c_e\right)^{-1} = \infty,
  \end{equation*}
  then the network is recurrent.
\end{corollary}

\begin{example}[$\ZZ^2$ is recurrent]
  Define $\Gamma_n$ as the set of vertical edges $\{ (x,y), (x,y+1) \}$ with $|x| \leq n$ and $\min \{ |y|, |y+1| \} = n$ along with the horizontal edges $\{ (x,y), (x+1,y) \}$ with $|y| \leq n$ and $\min \{ |x|, |x+1| \} =n$, see Figure \ref{fig:z2}. Then $\{\Gamma_n\}$ is a collection of disjoint cutsets
  separating $0$ from $\infty$. Also, $|\Gamma_n| = 4(2n+1)$ and therefore $\sum_n \left(\sum_{e\in \Gamma_n} c_e\right)^{-1} = \infty$. We deduce by \cref{nash:williams:recurrence} that $\ZZ^2$ is recurrent. 
\end{example}

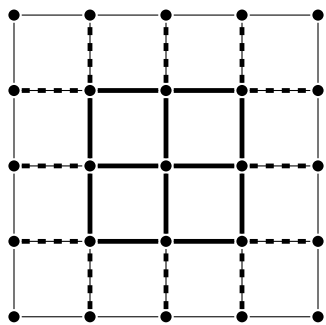
\begin{figure}[t]
    \centering
    \begin{subfigure}[b]{0.475\linewidth}
        \centering
        \begin{tikzpicture}[font=\small]
        \newcount\py
        \newcount\py
        \foreach \x in {0,...,4}
        \foreach \y in {0,...,4}
        {\node[vx] (\x\y) at (\x,\y) {};}

        \draw -- (12) to (13);
        
        \foreach \x in {0,...,4}
        \foreach \y in {0,...,3}
        {
            \pgfmathtruncatemacro{\py}{\y+1};
            \draw -- (\x\y) to (\x\py);
        }
        
        \foreach \x in {0,...,3}
        \foreach \y in {0,...,4}
        {
            \pgfmathtruncatemacro{\px}{\x+1};
            \draw -- (\x\y) to (\px\y);
        }
        
        \foreach \x in {1,...,3}
        \foreach \y in {1,...,2}
        {
            \pgfmathtruncatemacro{\py}{\y+1};
            \draw[line width=1.8pt] -- (\x\y) to (\x\py);
        }
        
        \foreach \x in {1,...,2}
        \foreach \y in {1,...,3}
        {
            \pgfmathtruncatemacro{\px}{\x+1};
            \draw[line width=1.8pt] -- (\x\y) to (\px\y);
        }
        
        \foreach \x in {1,...,3}
        \foreach \y in {0,3}
        {
            \pgfmathtruncatemacro{\py}{\y+1};
            \draw[line width=1.8pt,dashed] -- (\x\y) to (\x\py);
        }
        \foreach \y in {1,...,3}
        \foreach \x in {0,3}
        {
            \pgfmathtruncatemacro{\px}{\x+1};
            \draw[line width=1.8pt,dashed] -- (\x\y) to (\px\y);
        }
        
        \end{tikzpicture}
    \end{subfigure}
    \caption{A part of $\ZZ^2$: the edges in $\{-1,0,1\}^2$ are drawn in bold. $\Gamma_1$ is dashed.}
    \label{fig:z2}
\end{figure}

\begin{remark}
    There are recurrent graphs for which there exists $M<\infty$ such that for every collection
    $\{\Gamma_n\}$ of disjoint cutsets, $\sum_n \left(\sum_{e\in \Gamma_n}
    c_e\right)^{-1} \leq M$. Therefore, the Nash-Williams inequality is not
    sharp. See Example 1.2 in \cite{LyonsPeres}.
\end{remark}

\section{Random paths}\index{method of random paths}\label{sec:randompaths}

We now present the method of random paths, which is one of the most useful methods for generating unit flows on a network and bounding their energy. In fact, it is possible to show that the electric flow can be represented by such a random path. Suppose $G$ is a network with fixed vertices $a,z$ and $\mu$ is a probability measure on the set of 
paths from $a$ to $z$.

\begin{claim}\label{claim:randompath}
  For a path $\gamma$ sampled from $\mu$, let
  \begin{equation*}
    \theta_\gamma(\vec{e}) =
    \text{(\# of times $\vec{e}$ was traversed by $\gamma$)
     $-$ (\# of times $\cev{e}$ was traversed by $\gamma$)},
  \end{equation*}
where by $\vec{e}$ and $\cev{e}$ we mean the two orientations of an edge $e$ of $G$. Set
  \begin{equation*}
    \theta(\vec{e}) = \E{\theta_\gamma(\vec{e})}.
  \end{equation*}
  Then $\theta$ is a flow from $a$ to $z$ with $\|\theta\|=1$.
\end{claim}

\begin{proof}
  $\theta$ is antisymmetric since $\theta_\gamma$ is antisymmetric for every 
  $\gamma$,
  and it satisfies the node law since $\theta_\gamma$ satisfies the node law.
  Similarly, the ``strength'' of $\theta_\gamma$ (i.e., $\sum_{x\sim 
  a}\theta_\gamma(ax)$) is $1$,
  hence $\|\theta\|=1$.
\end{proof}

An example of the use of this method is the following classical result. 
\begin{theorem}\label{thm:z3transient}
  $\ZZ^3$ is transient.
\end{theorem}

\begin{proof}
  For $R>0$ denote by $B_R= \{(x,y,z):x^2+y^2+z^2 \leq R^2\}$ the ball of radius $R$ in $\RR^3$. Put $V_R=B_R\cap\ZZ^3$ and let $\partial V_R$ be the external vertex boundary of $V_R$, that is, the set of vertices not in $V_R$ which belong to an edge with an endpoint in $V_R$. 

We construct a random path $\mu$ from $\origin$ to $\partial V_R$ by choosing a uniform random
point $\vect{p}$ in $\partial B_R = \{(x,y,z) : x^2 + y^2 + z^2 = R^2\}$, drawing a straight line between $\origin$ and $\vect{p}$ in $\RR^3$, considering the set of distance at most $10$ in $\RR^3$ from the line, and then choosing (in some arbitrary fashion) a path in $\ZZ^3$ which is contained inside this set. The non-optimal constant $10$ was chosen in order to guarantee that such a discrete path exists for any point $\vect{p} \in \partial B_R$.

By \cref{claim:randompath}, the measure $\mu$ corresponds to a flow from $\origin$ to $\partial V_R$. To estimate the energy of this flow, we note that if $\vec{e}$ is an edge at distance $r \leq R$ from the origin, then the probability that it is traversed by a path drawn by $\mu$ is $O(r^{-2})$. Furthermore, there are $O(r^2)$ such edges. Hence the energy of the flow is at most 
\begin{equation*}
    \energy = O \big ( \sum_{r=1}^R r^2 \cdot (r^{-2})^2 \big ) \leq C \, , 
  \end{equation*}
for some constant $C<\infty$ which does not depend on $R$. By \cref{claim:randompath} and \cref{thomson:principle} we learn that $\reff(\origin\lr \partial V_R)\le C$ for all $R$, and so by \cref{cor:transiencecondition} we deduce that $\ZZ^3$ is transient. 
\end{proof}

\section{Exercises}

\begin{enumerate}
\item Let $G_z(a,x)$ be the Green's function, that is,
$$G_z(a,x) = \E_a\big [ \# \mbox{visits to $x$ before visiting $z$} \big ] \, .$$
Show that the function $h(x) = G_z(a,x) / \pi(x)$ is a voltage.



\item Show that the effective resistance satisfies the triangle inequality. That is, for any three vertices $x,y,z$ we have 
\begin{equation} \label{exercise:triangle} \reff(x \lr z) \leq \reff(x \lr y) + \reff(y \lr z) \, .\end{equation}
\item Let $a,z$ be two vertices of a finite network and let $\tau_a, \tau_z$ be the first visit time to $a$ and $z$, respectively, of the weighted random walk. Show that for any vertex $x$
$$ \PP_x( \tau_a < \tau_z ) \leq {\reff(x \lr \{a,z\} ) \over \reff(x \lr a)} \, .$$


\item Consider the following tree $T$. At height $n$ it has $2^n$ vertices (the root is at height $n=0$) and if $(v_1, \ldots, v_{2^n})$ are the vertices at level $n$ we make it so that $v_k$ has $1$ child at level $n+1$ and if $1 \leq k \leq 2^{n-1}$ and $v_k$ has $3$ children at level $n+1$ for all other $k$.
\begin{itemize} \item[(a)] Show that $T$ is recurrent.
\item [(b)] Show that for any disjoint edge cutsets $\Pi_n$ we have that $\sum_{n} |\Pi_n|^{-1} < \infty$. (So, the Nash-Williams criterion for recurrence is not sharp)
\end{itemize}

\item \begin{itemize}
\item[(a)] Let $G$ be a finite planar graph with two distinct vertices $a\neq z$ such that $a,z$ are on the outer face. Consider an embedding of $G$ so that $a$ is the left most point on the real axis and $z$ is the right most point on the real axis. Split the outer face of $G$ into two by adding the ray from $a$ to $-\infty$ and the ray from $z$ to $+\infty$. Consider the dual graph $G^*$ of $G$ and write $a^*$ and $z^*$ for the two vertices corresponding to the split outer face of $G$. Assume that all edge resistances are $1$. Show that
    $$ \reff(a \lr z; G) = {1 \over \reff(a^* \lr z^*; G^*)} \, .$$
\item[(b)] Show that the probability that a simple random walk on $\mathbb{Z}^2$ started at $(0,0)$ has probability $1/2$ to visit $(0,1)$ before returning to $(0,0)$.
\end{itemize}




\end{enumerate}

\chapter{The circle packing theorem}\label{chp:cp}
\section{Planar graphs, maps and embeddings}\label{sec:planarstuff}

\begin{definition}
  A graph $G=(V,E)$ is \textbf{planar}\index{planar graph} if it can be
  \textbf{properly drawn}\index{proper drawing} in the plane, that is,
  there exists a mapping sending distinct vertices to distinct points of
  $\RR^2$ and edges to continuous curves between the corresponding vertices so that no two curves intersect, except at the   vertices they share. We call such a mapping a \textbf{proper drawing}\index{drawing} of $G$.
\end{definition}

\begin{remark}
  A single planar graph has infinitely many drawings. Intuitively, some may seem similar to one another, while others seem different. For example,
  \begin{center}
    \begin{tikzpicture}[font=\small, scale=1]
      \node[vx]               (1) at (0,0) {};
      \node[vx]               (2) at (1,0) {};
      \node[vx]               (3) at (0.5,0.866) {};

      \draw (1) -- (2) -- (3) -- (1);

      \node at (2,0.433) {$\equiv$};

      \node[vx]               (a) at (3,0) {};
      \node[vx]               (b) at (4,0) {};
      \node[vx]               (c) at (3.5,0.866) {};

      \node (a0) at (3.00,0.00) {};
      \node (a1) at (3.10,-0.04) {};
      \node (a2) at (3.20,0.00) {};
      \node (a3) at (3.30,0.04) {};
      \node (a4) at (3.40,0.00) {};
      \node (a5) at (3.50,0.04) {};
      \node (a6) at (3.60,0.09) {};
      \node (a7) at (3.70,0.04) {};
      \node (a8) at (3.80,0.09) {};
      \node (a9) at (3.90,0.04) {};
      \node (a10) at (4.00,0.00) {};

      \node (b0) at (4.00,0.00) {};
      \node (b1) at (3.99,0.11) {};
      \node (b2) at (3.90,0.17) {};
      \node (b3) at (3.89,0.28) {};
      \node (b4) at (3.80,0.35) {};
      \node (b5) at (3.79,0.46) {};
      \node (b6) at (3.70,0.52) {};
      \node (b7) at (3.69,0.63) {};
      \node (b8) at (3.60,0.69) {};
      \node (b9) at (3.59,0.80) {};
      \node (b10) at (3.50,0.87) {};

      \node (c0) at (3.50,0.87) {};
      \node (c1) at (3.49,0.76) {};
      \node (c2) at (3.40,0.69) {};
      \node (c3) at (3.31,0.63) {};
      \node (c4) at (3.22,0.56) {};
      \node (c5) at (3.13,0.50) {};
      \node (c6) at (3.12,0.39) {};
      \node (c7) at (3.11,0.28) {};
      \node (c8) at (3.02,0.22) {};
      \node (c9) at (3.01,0.11) {};
      \node (c10) at (3.00,0.00) {};

      \draw plot [smooth] coordinates
        {(a) (a0) (a1) (a2) (a3) (a4) (a5) (a6) (a7) (a8) (a9) (a10) (b)};
      \draw plot [smooth] coordinates
        {(b) (b0) (b1) (b2) (b3) (b4) (b5) (b6) (b7) (b8) (b9) (b10) (c)};
      \draw plot [smooth] coordinates
        {(c) (c0) (c1) (c2) (c3) (c4) (c5) (c6) (c7) (c8) (c9) (c10) (a)};

      \node[vx]               (a*) at (3,0) {};
      \node[vx]               (b*) at (4,0) {};
      \node[vx]               (c*) at (3.5,0.866) {};

    \end{tikzpicture}
  \end{center}
  while
  \begin{center}
    \begin{tikzpicture}[font=\small, scale=1]
      \node[vx]               (1a) at (0,1) {};
      \node[vx]               (2a) at (1,1) {};
      \node[vx]               (3a) at (1,0) {};
      \node[vx]               (4a) at (0,0) {};
      \node[vx]               (5a) at (1.866,0.5) {};

      \draw (2a) -- (1a) -- (4a) -- (3a) -- (5a) -- (2a) -- (3a);

      \node at (2.866,0.5) {$\not\equiv$};

      \node[vx]               (1b) at (3.866,1) {};
      \node[vx]               (2b) at (4.866,1) {};
      \node[vx]               (3b) at (4.866,0) {};
      \node[vx]               (4b) at (3.866,0) {};
      \node[vx]               (5b) at (4,0.5) {};

      \draw (2b) -- (1b) -- (4b) -- (3b) -- (5b) -- (2b) -- (3b);

    \end{tikzpicture}
  \end{center}
\end{remark}

The following definition gives a precise sense to the above intuitive
equivalence / non-equivalence of drawings.

\begin{definition}
  A \defn{planar map} is a graph endowed with a cyclic permutation of the
  edges incident to each vertex, such that there exists a proper drawing in
  which the clockwise order of the curves touching the image of a vertex
  respects that cyclic permutation.
\end{definition}

The combinatorial structure of a planar map allows us to define faces directly (that is, without mentioning the drawing). Consider each edge of the graph as directed in both ways, and say that a directed edge $\vec{e}$ {\bf precedes} $\vec{f}$ (or, equivalently, $\vec{f}$ {\bf succeeds} $\vec{e}$), if there exist vertices  $v,x,y$ such that $\vec{e}=(x,v)$, $\vec{f}=(v,y)$, and $y$ is the successor
of $x$ in the cyclic permutation $\sigma_v$; see \Cref{fig:precedes}.

We say that $\vec{e},\vec{f}$ {\bf belong to the same face} if there exists a finite directed path $\vec{e}_1,\ldots,\vec{e}_m$ in the graph with $\vec{e}_i$ preceding $\vec{e}_{i+1}$ for $i=1,\ldots, m-1$ and such that either $\vec{e}=\vec{e}_1$ and $\vec{f}=\vec{e}_m$, or $\vec{f}=\vec{e}_1$ and $\vec{e}=\vec{e}_m$. This is readily seen to be an equivalence relation and we call each equivalence class a \defn{face}. Even though a face is a set of directed edges, we frequently ignore the orientations and consider a face the set of corresponding undirected edges. Each (undirected) edge is henceforth incident to either one or two faces.

\begin{figure}[t]
  \begin{subfigure}[b]{0.475\linewidth}
    \centering
    \begin{tikzpicture}[font=\small, scale=2,rotate=-10]
      \node[vx] (o) at (0:0) {};
      \node[vx] (0) at (0:1) {};
      \node[vx] (1) at (72:1) {};
      \node[vx] (2) at (144:1) {};
      \node[vx] (3) at (216:1) {};
      \node[vx] (4) at (288:1) {};
    
      \draw[gray,dashed] (0) -- (o) -- (1);
      \draw[gray,dashed] (o) -- (2);
    
      \draw[-{Latex[length=3mm]}] (4) -- (o) node [midway,right] {$\vec{e}$};
      \draw[-{Latex[length=3mm]}] (o) -- (3) node [midway,above] {$\vec{f}$};
    \end{tikzpicture}
    \caption{\label{fig:precedes:a} $\vec{e}$ precedes $\vec{f}$.}
  \end{subfigure}\hfill%
  \begin{subfigure}[b]{0.475\linewidth}
    \centering
    \begin{tikzpicture}[font=\small, scale=2,rotate=0]
      \node[vx,label=-90:$y$] (o) at (0:0) {};
      \node[vx,label=90:$x$] (0) at (90:1) {};
      \node[vx] (1) at (210:1) {};
      \node[vx] (2) at (330:1) {};
    
      \draw[gray,dashed] (0) -- (1) -- (2) -- (0);
    
      \draw[-{Latex[length=3mm]}] (0) -> (o);
      \draw[-{Latex[length=3mm]}] (o) -- (0);
    \end{tikzpicture}
    \caption{\label{fig:precedes:b} $(x,y)$ precedes $(y,x)$.}
  \end{subfigure}
  \caption{\label{fig:precedes}Examples for the edge precedence relation.}
\end{figure}
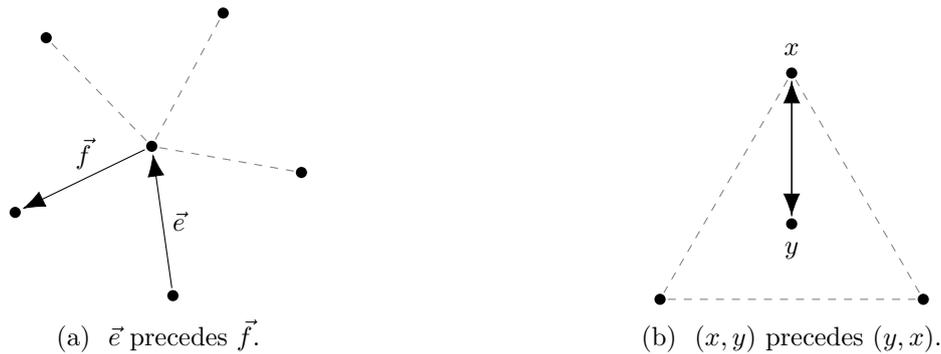

When the map is finite an equivalent definition of a face is the set of edges that bound a connected component 
of the complement of the drawing, that is, of $\RR^2$ with the images of the vertices and edges removed. This definition is not suitable for infinite planar maps since there may be a complicated set of accumulation points. Given a proper drawing of a finite planar map, there is a unique unbounded connected component of the complement of the drawing; the edges that bound it are called the \defn{outer face} and all other faces are called \textbf{inner faces} \index{inner face}. However, for any face in a finite map there is a drawing so that this face bounds the unique unbounded connected component, and because of this we shall henceforth refer to the outer face as an arbitrarily chosen face of the map.



We will use the following classical formula.
\begin{theorem}[Euler's formula]\label{euler}\index{Euler's formula}
  Suppose $G$ is a planar graph with $n$ vertices, $m$ edges and $f$ faces.
  Then
  \begin{equation*}
    n - m + f = 2.
  \end{equation*}
\end{theorem}

We now state the main theorem we will discuss and use throughout this course. Its proof is presented in the next section.

\begin{theorem}[The circle packing theorem, \cite{K36}]\label{thm:cp} \index{circle packing theorem}
  Given any finite simple planar map $G=(V,E)$, $V=\{v_1,\ldots,v_n\}$, there
  exist $n$ circles in $\RR^2$, $C_1,\ldots,C_n$, with disjoint interiors,
  such that $C_i$ is tangent to $C_j$ if and only if $\{i,j\}\in E$. Furthermore, for every
  vertex $v_i$, the clockwise order of the circles tangent to $C_i$ agrees
  with the cyclic permutation of $v_i$'s neighbors in the map.
\end{theorem}

\begin{figure}[t]
  \centering
  \begin{subfigure}[b]{0.475\linewidth}
    \centering
    \begin{tikzpicture}[font=\small, scale=1.5]
      \node[vx,label=90:$1$]  (1) at (0,1) {};
      \node[vx,label=90:$2$]  (2) at (1,1) {};
      \node[vx,label=-90:$3$] (3) at (1,0) {};
      \node[vx,label=-90:$4$] (4) at (0,0) {};
      \node[vx,label=0:$5$]   (5) at (1.866,0.5) {};

      \draw (2) -- (1) -- (4) -- (3) -- (5) -- (2) -- (3);

      \node (c1) at (0,-1) {$1$};
      \node (c2) at (1,-1) {$2$};
      \node (c3) at (1,-2) {$3$};
      \node (c4) at (0,-2) {$4$};
      \node (c5) at (1.866,-1.5) {$5$};

      \draw (c1) circle (0.5);
      \draw (c2) circle (0.5);
      \draw (c3) circle (0.5);
      \draw (c4) circle (0.5);
      \draw (c5) circle (0.5);
    \end{tikzpicture}
  \end{subfigure}\hfill%
  \begin{subfigure}[b]{0.475\linewidth}
    \centering
    \begin{tikzpicture}[font=\small,scale=1.5]
      \node[vx,label=90:$1$]  (1) at (0,1) {};
      \node[vx,label=90:$2$]  (2) at (1,1) {};
      \node[vx,label=-90:$3$] (3) at (1,0) {};
      \node[vx,label=-90:$4$] (4) at (0,0) {};
      \node[vx,label=90:$5$]  (5) at (0.134,0.5) {};

      \draw (2) -- (1) -- (4) -- (3) -- (5) -- (2) -- (3);

      \node (c1) at (0,-1) {$1$};
      \node (c2) at (1,-1) {$2$};
      \node (c3) at (1,-2) {$3$};
      \node (c4) at (0,-2) {$4$};
      \node (c5) at (0.625,-1.5) {$5$};

      \draw (c1) circle (0.5);
      \draw (c2) circle (0.5);
      \draw (c3) circle (0.5);
      \draw (c4) circle (0.5);
      \draw (c5) circle (0.125);
    \end{tikzpicture}
  \end{subfigure}
  \caption{\label{fig:cp} Two distinct planar maps (of the same graph) with corresponding circle packings.}
\end{figure}
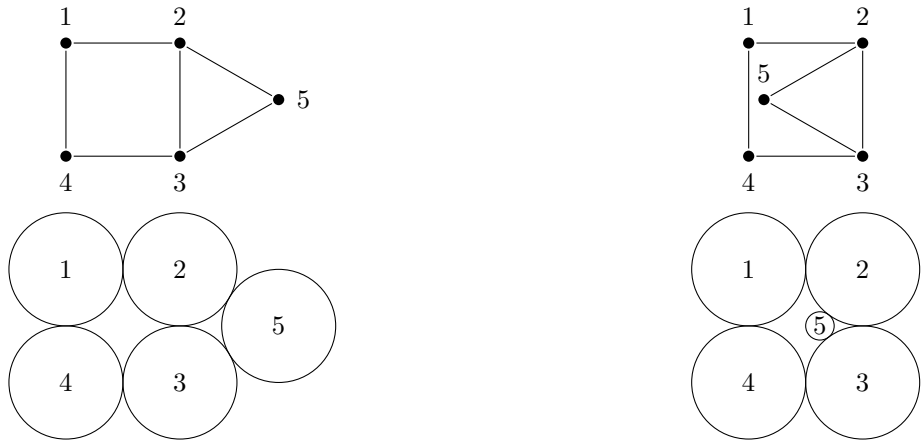


\Cref{fig:cp} gives examples for embeddings of maps which respect the cyclic
orderings of neighbors, as guaranteed to exist according to the theorem. 

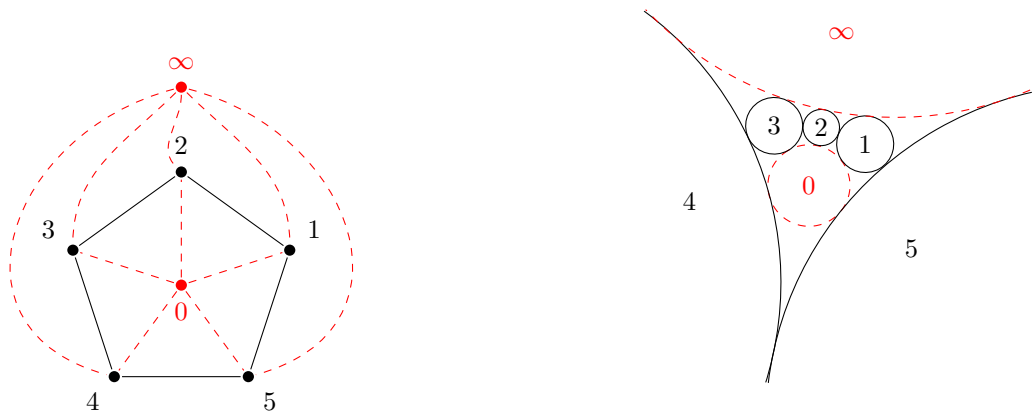
\begin{figure}[t]
  \centering
  \begin{subfigure}[b]{0.475\linewidth}
    \centering
    \begin{tikzpicture}[font=\small, scale=1.5]
      \node[vx,label=18:$1$]  (1) at (0.951,0.309) {};
      \node[vx,label=90:$2$]  (2) at (0,1) {};
      \node[vx,label=162:$3$] (3) at (-0.951,0.309) {};
      \node[vx,label=234:$4$] (4) at (-0.588,-0.809) {};
      \node[vx,label=306:$5$] (5) at (0.588,-0.809) {};

      \node[vx,fill=red,label={[red]-90:$0$}]      (0) at (0,0) {};
      \node[vx,fill=red,label={[red]90:$\infty$}]  (i) at (0,1.75) {};

      \draw (1) -- (2) -- (3) -- (4) -- (5) -- (1);
      \draw[red,dashed] (0) -- (1);
      \draw[red,dashed] (0) -- (2);
      \draw[red,dashed] (0) -- (3);
      \draw[red,dashed] (0) -- (4);
      \draw[red,dashed] (0) -- (5);
      \draw[red,dashed] (i) to[out=-45,in=90,looseness=1] (1);
      \draw[red,dashed] (i) to[out=-90,in=135,looseness=1] (2);
      \draw[red,dashed] (i) to[out=225,in=90,looseness=1] (3);
      \draw[red,dashed] (i) to[out=195,in=165,looseness=1.5] (4);
      \draw[red,dashed] (i) to[out=-15,in=15,looseness=1.5] (5);

    \end{tikzpicture}
    \caption{\label{fig:cp:tri1}
      A planar map and a triangulation.}
  \end{subfigure}\hfill%
  \begin{subfigure}[b]{0.475\linewidth}
    \centering
    \begin{tikzpicture}[font=\small,scale=1.5,rotate=120]
       \clip (3.29,1.07) circle (2);



\node[red] (0) at (3.13,1.22) {$0$};
\draw[red,dashed] (0) circle (0.36);
\node (3) at (3.74,1.22) {$3$};
\draw (3) circle (0.25);
\node (2) at (3.52,0.87) {$2$};
\draw (2) circle (0.16);
\node (1) at (3.2,0.61) {$1$};
\draw (1) circle (0.25);
\node (5) at (-0.0,0.0) {$5$};
\node (5l) at (2.2,0.72) {$5$};
\draw (5) circle (3.0);
\node (4) at (4.01,4.46) {$4$};
\node (4l) at (3.53,2.2) {$4$};
\draw (4) circle (3.0);
\node[red] (i) at (5.87,-1.24) {$\infty$};
\node[red] (il) at (4.15,0.3) {$\infty$};
\draw[red, dashed] (i) circle (3.0);

    \end{tikzpicture}
    \caption{\label{fig:cp:tri2} A circle packing of the triangulation.}
  \end{subfigure}
  \caption{\label{fig:cp:tri}
    Circle packing of a triangulation of a planar map.}
\end{figure}

First note that it suffices to prove the theorem for {\bf triangulations}\index{triangulation}, that is, simple planar maps in which every face has precisely three edges. Indeed, in any planar map we may add a
single vertex inside each face and connect it to all vertices bounding that
face. The obtained map is a triangulation, and after applying the circle packing theorem for
triangulations, we may remove the circles corresponding to the added
vertices, obtaining a circle packing of the original map which respects its
cyclic permutations.  This is depicted in \cref{fig:cp:tri}. 

Thus, it suffices to prove \cref{thm:cp} for finite triangulations. In this case an important uniqueness statement also holds. 

\begin{theorem}\label{thm:cptriangulation} Let $G=(V,E)$ be a finite triangulation on vertex set $V=\{v_1,\ldots,v_n\}$ and assume that $\{v_1,v_2,v_3\}$ form a face. Then for any three positive numbers $\rho_1, \rho_2, \rho_3$, there exists a circle packing $C_1,\ldots,C_n$ as in \cref{thm:cp} with the additional property that $C_1, C_2, C_3$ are mutually tangent, form the outer face, and have radii $\rho_1, \rho_2,\rho_3$, respectively. Furthermore, this circle packing is unique, up to translations and rotations of the plane.
\end{theorem}

\section{Proof of the circle packing theorem}

We here prove \cref{thm:cptriangulation} which implies \cref{thm:cp} as explained above. We therefore assume from now on that our map is a triangulation.  Denote by
$n$, $m$ and $f$ the number of vertices, edges and faces of the map
respectively, and observe that $3f=2m$, since each edge is counted in
exactly two faces, and each face is bounded by exactly three edges.
Therefore, by Euler's formula (\cref{euler}), we have that
\begin{equation*}
  2 = n - m + f = n - \frac{3}{2}f + f = n - \frac{1}{2}f,
\end{equation*}
thus
\begin{equation}\label{eq:num:innerfaces}
  f=2n-4.
\end{equation}

We assume the vertex set is $\{v_1, \ldots, v_n\}$, that $\{v_1,v_2,v_3\}$ is the outer face and that $\rho_1, \rho_2, \rho_3$ are three positive numbers that will be the radii of the outer circles $C_1, C_2, C_3$ eventually. Denote by $F^\circ$ the set of faces of the map except the outer face, and for a subset of
vertices $A$ let $F(A)$ be the set of inner faces with at least one vertex in
$A$.  We write $F(v)$ when we mean $F(\{v\})$.  

Given a vector
$\vect{r}=(r_1,\ldots,r_n)\in(0,\infty)^n$, an inner face $f$ bounded by the
vertices $v_i,v_j,v_k$, and a distinguished vertex $v_j$, we associate a
number $\alpha_f^{\vect{r}}(v_j)=\ang v_iv_jv_k\in(0,\pi)$ which is the angle
of $v_j$ in the triangle $\triangle v_iv_jv_k$ created by connecting the
centers of three mutually tangent circles $C_i,C_j,C_k$ of radii $r_i,r_j$ and $r_k$ (that is, in a triangle with side lengths $r_i+r_j$, $r_j+r_k$ and $r_k+r_i$).  This number can be calculated using the cosine formula
\begin{equation*}
  \cos(\ang v_iv_jv_k) = 1 - \frac{2r_ir_k}{(r_i+r_j)(r_j+r_k)} \, ,
\end{equation*}
however, we will not use this formula directly. For every $j\in \{1,\ldots,n\}$ we define
\begin{equation*}
  \sigma_{\vect{r}}(v_j) = \sum_{f\in F(v_j)}\alpha_f^{\vect{r}}(v_j)
\end{equation*}
to be the \emph{sum of angles} at $v_i$ with respect to $\vect{r}$.  Let $\theta_1, \theta_2, \theta_3$ be the angles formed at the centers of three mutually tangent circles $C_1, C_2, C_3$ of radii $\rho_1, \rho_2, \rho_3$. Equivalently, these are the angles of a triangle with edge lengths $r_1+r_2$, $r_2+r_3$ and $r_3+r_1$. If the vector $\vect{r}$ was indeed the vector of radii of a circle packing of the map satisfying \cref{thm:cptriangulation}, then we would have
\begin{equation} \label{eq:defr}
  \sigma_{\vect{r}}(v_i)=
  \begin{cases}
    \theta_i & i\in\{1,2,3\} \, ,\\
     2\pi & \text{otherwise} \, ,
  \end{cases}
\end{equation}
and additionally $(r_1,r_2,r_3)=(\rho_1,\rho_2,\rho_3)$. The proof is split into three parts:
\begin{enumerate}
  \item Show that there exists a vector $\vect{r}\in(0,\infty)^n$ satisfying \eqref{eq:defr}; 
  \item Given such $\vect{r}$, show that a circle packing with these radii exists and that $(r_1,r_2,r_3)$ is a positive multiple of $(\rho_1,\rho_2,\rho_3)$; furthermore, this circle packing is unique up to translations and rotations.
  \item Show that $\vect{r}$ is unique up to scaling all entries by a constant factor.
\end{enumerate}

\subsection{Proof of \cref{thm:cptriangulation}, step 1: Finding the radii vector $\vect{r}$ } 

\begin{observation}\label{obs:sum:sigma}
  For every $\vect{r}$,
  \begin{equation*}
    \sum_{i=1}^n\sigma_{\vect{r}}(v_i) = \left|F^\circ\right|\pi = (2n-5)\pi.
  \end{equation*}
\end{observation}
\begin{proof} Follows immediately since each inner face $f$ bounded by the vertices $v_i, v_j, v_k$ contributes the three angles $\alpha_f^{\mathbf{r}}(v_i)$, $\alpha_f^{\mathbf{r}}(v_j)$ and $\alpha_f^{\mathbf{r}}(v_k)$ which sum to $\pi$. By \eqref{eq:num:innerfaces}, there are $2n-5$ inner faces.
\end{proof}

We now set
\begin{equation} \label{eq:deltadef}
  \delta_{\vect{r}}(v_i)=
  \begin{cases}
    \sigma_{\vect{r}}(v_i) - \theta_i
      & j \in \{1,2,3\}\, ,\\
    \sigma_{\vect{r}}(v_j) - 2\pi & \text{otherwise}.
  \end{cases}
\end{equation}
Using this notation, our goal is to find $\vect{r}$ for which
$\delta_{\vect{r}}\equiv 0$.  It follows from \cref{obs:sum:sigma} that for
every $\vect{r}$,
\begin{equation}\label{eq:sum:delta}
  \sum_{i=1}^n\delta_{\vect{r}}(v_i)
  = \sum_{i=1}^n\sigma_{\vect{r}}(v_i) - \theta_1 - \theta_2 - \theta_3 - (n-3)\cdot 2\pi
  = 0 \, . 
\end{equation}
We define
\begin{equation*}
  \cE_{\vect{r}}=\sum_{i=1}^n \delta_{\vect{r}}(v_i)^2.
\end{equation*}
We would like to find $\vect{r}$ for which
$\cE_{\vect{r}}=0$. We will use the following geometric observation; see \cref{fig:angles:obs}.
\begin{observation}\label{obs:geometric}
  Let $\vect{r}=(r_1,\ldots,r_n)$ and $\vect{r}'=(r_1',\ldots,r_n')$, and let
  $f\in F^\circ$ be bounded by $v_i,v_j,v_k$.
  \begin{itemize}
    \item If $r_i'\le r_i$, $r_k'\le r_k$ and $r_j'\ge r_j$, then
      $\alpha_f^{\vect{r}'}(v_j) \le \alpha_f^{\vect{r}}(v_j)$.
    \item If $r_i'\ge r_i$, $r_k'\ge r_k$ and $r_j'\le r_j$, then
          $\alpha_f^{\vect{r}'}(v_j) \ge \alpha_f^{\vect{r}}(v_j)$.
    \item $\alpha_f^{\vect{r}}(v_j)$ is continuous in $\vect{r}$.
  \end{itemize}
\end{observation}
\begin{proof} A proof using the cosine formula is routine and is omitted.
\end{proof}

\begin{figure}[t]
  \centering
  \begin{subfigure}[b]{0.475\linewidth}
    \centering
    \begin{tikzpicture}[font=\small, scale=0.5, rotate=75]
      \node[vx] (ci) at (1.342,2.683) {};
      \node[vx] (cj) at (0,0) {};
      \node[vx] (ck) at (-2.236,4.472) {};

      \draw (ci) circle (1);
      \draw (cj) circle (2);
      \draw (ck) circle (3);
      
      \draw (cj) -- (ci);
      \draw (cj) -- (ck);
      \pic [draw, color=red, <->, "$\alpha$", angle eccentricity=1.5]
      {angle = ci--cj--ck};
    \end{tikzpicture}
  \end{subfigure}\hfill%
  \begin{subfigure}[b]{0.475\linewidth}
    \centering
    \begin{tikzpicture}[font=\small, scale=0.5, rotate=75]
      \node[vx] (ci) at (0.851,2.615) {};
      \node[vx] (cj) at (0,0) {};
      \node[vx] (ck) at (-1.469,4.517) {};

      \draw (ci) circle (0.5);
      \draw (cj) circle (2.25);
      \draw (ck) circle (2.5);
      
      \draw (cj) -- (ci);
      \draw (cj) -- (ck);
      \pic [draw, color=red, <->, "$\alpha'$", angle eccentricity=1.5]
      {angle = ci--cj--ck};
    \end{tikzpicture}
  \end{subfigure}
  \caption{\label{fig:angles:obs}
    When the radius of a circle corresponding to a vertex increases, while the
    radii of the circles corresponding to its two neighbors in a given face decrease, the vertex's angle in the corresponding triangle decreases (see \cref{obs:geometric}).}
\end{figure}
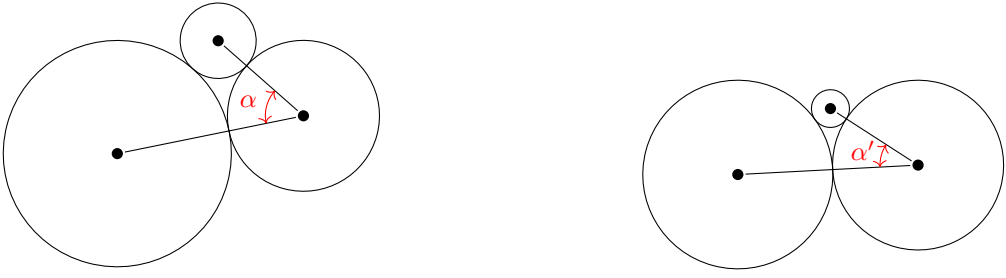

We now define an iterative algorithm, whose input and output are both vectors
of radii normalized to have $\ell_1$ norm $1$.  We start with the vector
$\vect{r}^{(0)}=\left(\frac{1}{n},\ldots,\frac{1}{n}\right)$, and given
$\vect{r}=\vect{r}^{(t)}$ we construct $\vect{r}'=\vect{r}^{(t+1)}$.  Write
$\delta=\delta_{\vect{r}}$ and $\delta'=\delta_{\vect{r}'}$, and similarly
$\cE=\cE_{\vect{r}}$ and $\cE'=\cE_{\vect{r}'}$.  We begin by
ordering the set of reals $\{\delta(v_i)\mid 1 \leq i \leq n\}$.  If $\delta\equiv 0$
we are done; otherwise, we may choose $s\in\RR$ such that the set
$S=\{v\mid \delta(v)>s\}\ne\varnothing$ and its complement $V \setminus S$ are non-empty and such that the gap
\begin{equation*}
  \gap_\delta(S):=\min_{v\in S}\delta(v) - \max_{v\notin S}\delta(v)>0
\end{equation*}
is maximal over all such $s$. See \cref{fig:gap} for illustration.

Once we choose $S$, a step of the algorithm consists of two steps:
\begin{enumerate}
  \item For some $\lambda\in(0,1)$ to be chosen later, we set
    \begin{equation*}
      (\vect{r}_\lambda)_i =
          \begin{cases}
            r_i & v_i\in S,\\
            \lambda r_i & v_i\notin S.
          \end{cases}
    \end{equation*}
  \item We normalize $\vect{r}_\lambda$ so that the sum of entries is $1$,
    letting $\bar{\vect{r}}_\lambda$ be the normalized vector. Note that this step does not change the vector $\delta$.
\end{enumerate}
We will choose an appropriate $\lambda$ that will decrease all values of
$\delta(v)$ for $v\in S$, increase all values of $\delta(v)$ for
$v\notin S$, and will close the gap.  This will be made formal in the following
two claims.

\begin{claim}\label{claim:A}
  For every $\lambda\in(0,1)$, setting $\vect{r}'=\bar{\vect{r}}_\lambda$,
  we have that $\delta'(v)\le\delta(v)$ for any $v\in S$, and 
  $\delta'(v)\ge \delta(v)$ for any $v\notin S$.
\end{claim}

\begin{claim}\label{claim:B}
  There exists $\lambda\in(0,1)$ such that setting
  $\vect{r}'=\bar{\vect{r}}_\lambda$ gives that $\gap_{\delta'}(S)=0$.
\end{claim}

\begin{proof}[Proof of \cref{claim:A}]
  Consider $v_j\notin S$ and an inner face $v_i,v_j,v_k$.
  \begin{description}
    \item[Case I] $v_i,v_k\notin S$.  In this case, the radii of $C_i,C_j,C_k$ 
      are all multiplied by the same number $\lambda$, so
      $\alpha_f^{\vect{r}'}(v_j) = \alpha_f^{\vect{r}}(v_j)$.
    \item[Case II] $v_i,v_k\in S$.  In this case, the radii of $C_i,C_k$ remain
      unchanged and the radius of $C_j$ decreases, thus by
      \cref{obs:geometric},
      $\alpha_f^{\vect{r}'}(v_j) \ge \alpha_f^{\vect{r}}(v_j)$.
    \item[Case III] $v_i\notin S$, $v_k\in S$.  In this case the radii of 
      $C_i,C_j$ are multiplied by $\lambda$ and the radius of $C_k$ is unchanged. The angles of $\triangle v_i v_j v_k$ remain unchanged if we multiply all radii by $\lambda^{-1}$, thus we could just as easily have left $C_i$, $C_j$ unchanged and increased the radius of $C_k$. By \cref{obs:geometric}, we get that
      $\alpha_f^{\vect{r}'}(v_j) \ge \alpha_f^{\vect{r}}(v_j)$.
  \end{description}
  It follows that $\delta'(v)\ge\delta(v)$ for any $v\notin S$. An identical argument shows that
 $\delta'(v)\le \delta(v)$ for all $v\in S$.
\end{proof}

In order to prove \cref{claim:B}, we present another claim.

\begin{claim}\label{claim:C}
  \begin{equation*}
    \lim_{\lambda\downto 0} \sum_{v\notin S} \delta_{\vect{r}_\lambda}(v) > 0.
  \end{equation*}
\end{claim}

\begin{proof}[Proof of \cref{claim:B} using \cref{claim:C}]
The function $\lambda \mapsto \gap_{\delta_{\mathbf{r}_\lambda}}(S)$ is continuous on $(0,1]$ by the third bullet of \cref{obs:geometric}, and its value at $\lambda = 1$ is $\gap_\delta(S) > 0$. \cref{claim:C} says that if $\mu > 0$ is small enough, then
$$
\sum_{v \notin S} \delta_{\mathbf{r}_\mu}(v) > 0 \, ,
$$
from which it follows that $\max_{v \notin S} \delta_{\mathbf{r}_\mu}(v) > 0$. By (3.4), we also have
\[
\sum_{v \in S} \delta_{\mathbf{r}_\mu}(v) < 0, 
\]
meaning that $\min_{v \in S} \delta_{\mathbf{r}_\mu}(v) < 0$ and therefore $\gap_{\delta_{\mathbf{r}_\mu}}(S) < 0$. By continuity, there exists $\lambda \in (\mu, 1)$ such that $\mathrm{gap}_{\delta_{\mathbf{r}_\lambda}}(S) = 0$.
\end{proof}

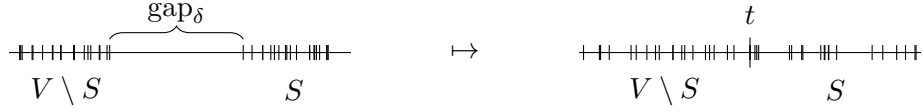
\begin{figure}[t]
  \centering
  \begin{tikzpicture}
    \draw (-6,0) -- (-1.5,0);
    \node at (0,0) {$\mapsto$};
    \draw (1.5,0) -- (6,0);
    
    \node at (-5.25,-0.5) {$V\setminus S$};
    \foreach\x in {-4.71,-5.149,-5.314,-5.009,-5.321,-4.804,-4.961,-4.675,-5.696,-5.685,-5.558,-4.812,-5.14,-5.846,-5.429,-5.86,-5.146,-5.827,-4.924,-5.422} {
      \draw (\x,-0.1) -- (\x,0.1);};
    \node at (-2.25,-0.5) {$S$};
    \foreach\x in {-2.044,-2.344,-2.662,-1.913,-2.802,-2.362,-2.662,-1.822,-2.298,-1.954,-1.977,-2.044,-2.555,-2.508,-2.217,-1.803,-1.803,-1.987,-2.449,-2.918} {
      \draw (\x,-0.1) -- (\x,0.1);};

    \draw
    [decorate,decoration={brace,amplitude=5pt,raise=5pt},yshift=0pt]
(-4.675,0) -- (-2.918,0) node [black,midway,yshift=15pt] {$\gap_\delta$};

    \node at (2.625,-0.5) {$V\setminus S$};
    \foreach\x in {2.556,2.852,3.164,2.507,3.539,2.393,3.453,2.999,2.253,2.735,1.768,2.891,1.898,3.218,3.278,2.737,2.185,1.784,1.56,3.75} {
      \draw (\x,-0.1) -- (\x,0.1);};
    \node at (4.875,-0.5) {$S$};
    \foreach\x in {4.681,4.43,4.298,4.89,3.815,5.796,5.497,4.269,5.935,4.736,3.839,5.357,4.787,4.782,4.727,4.445,5.68,3.864,5.921,3.75} {
      \draw (\x,-0.1) -- (\x,0.1);};
      
    \draw (3.75,-0.2) -- (3.75,0.2);
    \node at (3.75,0.5) {$t$};
  \end{tikzpicture}
  \caption{Left: Finding the maximum gap between two consecutive values of $\delta$,
  and splitting the set of values into $S$ and its complement. Right: Moving from $\vect{r}$ to $\vect{r}'$ closes the gap between $S$ and $V\setminus S$.}
  \label{fig:gap}
\end{figure}

\begin{proof}[Proof of \cref{claim:C}]
  We first show that for each face $f\in F(V\setminus S)$ bounded by 
  $v_i,v_j,v_k$, the sum of angles at the vertices belonging to 
  $V\setminus S$ converges to $\pi$ as $\lambda\downto 0$. We show this by the following case analysis. The statements in cases II and III can be justified by drawing a picture or appealing to the cosine formula.
  \begin{description}
    \item[Case I] If $v_i,v_j,v_k\notin S$ then since the face is a
      triangle,
      $\alpha_f^{\vect{r}_\lambda}(v_i)
        +\alpha_f^{\vect{r}_\lambda}(v_j)
        +\alpha_f^{\vect{r}_\lambda}(v_k)= \pi$ for all $\lambda \in (0,1)$.
    \item[Case II] If $v_i,v_j\notin S$ but $v_k\in S$ then
      $\lim_{\lambda\downto 0}\alpha_f^{\vect{r}_\lambda}(v_k)=0$,
      hence $\lim_{\lambda\downto 0}\alpha_f^{\vect{r}_\lambda}(v_i)
        + \alpha_f^{\vect{r}_\lambda}(v_j) = \pi$.
    \item[Case III] If $v_i\notin S$ but $v_j,v_k\in S$ then
      $\lim_{\lambda\downto 0}\alpha_f^{\vect{r}_\lambda}(v_j)
        + \alpha_f^{\vect{r}_\lambda}(v_k)=0$,
      hence $\lim_{\lambda\downto 0}\alpha_f^{\vect{r}_\lambda}(v_i) = \pi$.
  \end{description}
  
  It follows that
  \begin{equation}\label{sumsigma:limit}
    \lim_{\lambda\downto 0}\sum_{v \notin S} \sigma_{\vect{r}_\lambda}(v)
    = |F(V\setminus S)|\pi.
  \end{equation}

For convenience, set
\[
\theta(v_i) = \begin{cases} \theta_i & 1 \leq i \leq 3, \\ 2\pi & \text{otherwise}, \end{cases}
\]
so that $\delta_\vect{r}(v) = \sigma_\vect{r}(v) - \theta(v)$ for all $v \in V$. Then
\begin{equation} \label{V-minus-S}
\lim_{\lambda \searrow 0} \sum_{v \notin S} \delta_{\vect{r}_\lambda}(v) = |F(V \setminus S)| \pi - \sum_{v \notin S} \theta(v).
\end{equation}
Let $\bar{F} = F^\circ \setminus F(V \setminus S)$, so every face in $\bar{F}$ contains only vertices of $S$. We will show that
\begin{equation} \label{F-bar-ineq}
|\bar{F}| \pi < \sum_{v \in S} \theta(v).
\end{equation}
If \eqref{F-bar-ineq} holds, then we can add the negative quantity $|\bar{F}| \pi - \sum_{v \in S} \theta(v)$ to the right side of \eqref{V-minus-S}, obtaining $|F^\circ| \pi - \sum_{v \in V} \theta(v) = (2n-5)\pi - (2n-5)\pi = 0$. It follows that \eqref{V-minus-S} is strictly positive, proving the claim. Thus it suffices to show \eqref{F-bar-ineq}.

In the rest of the proof, we fix an embedding of $G$ in the plane with $(v_1,v_2,v_3)$ as the outer face. Let $G[S]$ be the subgraph of $G$ induced by $S$. Partition $S$ into equivalence classes, $S = S_1 \cup \cdots \cup S_k$, where two vertices are equivalent if they are in the same connected component of $G[S]$. Then $G[S] = G[S_1] \cup \cdots \cup G[S_k]$. Let $\bar{F}_j$ be the set of faces in $\bar{F}$ that appear as faces of $G[S_j]$, so that we have the disjoint union $\bar{F} = \bar{F}_1 \cup \cdots \cup \bar{F}_k$.

Since $S$ is nonempty, it is enough to show that for all $1 \leq j \leq k$,
\begin{equation} \label{Fj-bar-ineq}
|\bar{F}_j| \pi < \sum_{v \in S_j} \theta(v).
\end{equation}
Let $m_j$ and $f_j$ denote the number of edges and faces, respectively, of $G[S_j]$. Observe that $|\bar{F}_j| \leq f_j - 1$. If $|\bar{F}_j| = 0$, then \eqref{Fj-bar-ineq} is trivial. If $|\bar{F}_j| \geq 1$, then $G[S_j]$ has at least one inner face, and since it is a simple graph, every face must have degree at least 3. (The degree of a face is the number of directed edges that make up its boundary.) Because the sum of the degrees of all the faces equals twice the number of edges, we have $2m_j \geq 3f_j$. Euler's formula now gives
\[
|S_j| + f_j - 2 = m_j \geq \frac{3}{2} f_j,
\]
and hence $f_j \leq 2|S_j| - 4$. Thus, the left side of \eqref{Fj-bar-ineq} satisfies
\[
|\bar{F}_j| \pi \leq (2|S_j| - 5) \pi.
\]
If $S_j$ contains all of $v_1,v_2,v_3$, then the right side of \eqref{Fj-bar-ineq} is
\[
\theta_1 + \theta_2 + \theta_3 + (|S_j| - 3)\cdot  2\pi = (2|S_j| - 5) \pi.
\]
Otherwise, at least one of the $\theta_i$ is replaced by $2\pi$ and so the right side of \eqref{Fj-bar-ineq} is strictly greater than the left side. In fact, \eqref{Fj-bar-ineq} holds except when $v_1,v_2,v_3 \in S_j$ and $|\bar{F}_j| = f_j - 1 = 2|S_j| - 5$. We now show that this situation cannot occur.

The equality $|\bar{F}_j| = f_j - 1$ means that every inner face of $G[S_j]$ is an element of $\bar{F}_j$ and therefore a face of $G$. Since $v_1,v_2,v_3 \in S_j$, the outer face of $G[S_j]$ is $(v_1,v_2,v_3)$, which is the same as the outer face of $G$. So, every face of $G[S_j]$ is also a face of $G$. But this is impossible: if we choose any $v \in V \setminus S$, then $v$ must lie in some face of $G[S_j]$, which then cannot be a face of $G$. Therefore, it cannot be true that $v_1,v_2,v_3 \in S_j$ and also $|\bar{F}_j| = f_j - 1$, so we conclude that \eqref{Fj-bar-ineq} always holds.
\end{proof}

We now analyse the algorithm.  Let $\lambda\in(0,1)$ be the one guaranteed by
\cref{claim:B}, and set $\vect{r}'=\bar{\vect{r}}_\lambda$.

\begin{claim}\label{leastsquares}
  $\cE' \le \cE\left(1-\frac{1}{2n^3}\right)$.
\end{claim}

\begin{proof}
  As depicted in \cref{fig:gap}, define
  \begin{equation*}
    t = \min_{v\in S}\delta'(v) = \max_{v\notin S}\delta'(v).
  \end{equation*}
  By \eqref{eq:sum:delta} we have that $\sum_{i=1}^n\delta(v_i)=\sum_{i=1}^n\delta'(v_i)=0$, hence
  \begin{equation*}
    \cE-\cE' = \sum_{i=1}^n \delta(v_i)^2
      - \sum_{i=1}^n \delta'(v_i)^2
    = \sum_{i=1}^n (\delta(v_i)-\delta'(v_i))^2
      + 2\sum_{i=1}^n (t-\delta'(v_i))(\delta'(v_i)-\delta(v_i)).
  \end{equation*}
  
  If $v\in S$, then $t \le \delta'(v) \le \delta(v)$ and if $v\notin S$, then $t \ge \delta'(v) \ge \delta(v)$. Thus, in both cases $(t-\delta'(v))(\delta'(v)-\delta(v))\ge 0$.
  Taking $u\in S$ and $v\notin S$ with $\delta'(u)=\delta'(v)=t$, we have that
  \begin{equation*}
    \cE-\cE' \ge (\delta(u)-t)^2 + (\delta(v)-t)^2
    \ge \frac{(\delta(u)-\delta(v))^2}{2}
    \ge \frac{\gap_{\delta}(S)^2}{2} \, .
  \end{equation*}
  Since $\gap_{\delta}(S)$ was chosen to be the maximal gap we may bound,
  \begin{equation*}
    \gap_{\delta}(S) \ge
    \frac{1}{n}\left(\max_{v\in V}\delta(v) - \min_{v\in V}\delta(v)\right) \, .
  \end{equation*}
  For every $v\in V$,
  \begin{equation*}
    \max_{w\in V}\delta(w) - \min_{w\in V}\delta(w) \ge |\delta(v)|,
  \end{equation*}
  and thus
  \begin{equation*}
    n\left(\max_{v\in V}\delta(v) - \min_{v\in V}\delta(v)\right)^2
      \ge \sum_{i=1}^n \delta(v_i)^2 = \cE \, .
  \end{equation*}
  Hence
  \begin{equation*}
    \cE-\cE' \ge \frac{1}{2n^2}
      \left( \max_{v\in V}\delta(v) - \min_{v\in V}\delta(v) \right)^2
    \ge \frac{1}{2n^2}\cdot\frac{\cE}{n},
  \end{equation*}
  and we conclude that
  \begin{equation*}
    \cE' \le \cE\left(1-\frac{1}{2n^3}\right).\qedhere
  \end{equation*}
\end{proof}

Write $\cE^{(t)}=\cE_{\vect{r}^{(t)}}$.
By iterating the described algorithm, we obtain from \cref{leastsquares} that
\begin{equation*}
  \cE^{(t)} \le \cE^{(0)}\left(1-\frac{1}{2n^3}\right)^t \longrightarrow 0 \qquad \hbox{as } t \to \infty \, .
\end{equation*}
By our normalization $\left\|\vect{r}^{(t)}\right\|_{\ell_1}=1$. Thus, by compactness, there exists a subsequence $\{t_k\}$ and a 
vector $\vect{r}^\infty$ such that $\vect{r}^{(t_k)}\to \vect{r}^\infty$ as $k \to \infty$.
From continuity of $\cE$ we have that $\cE\left(\vect{r}^\infty\right)=0$, meaning that \eqref{eq:defr} is satisfied.  For 
$\vect{r}^\infty$ to be feasible as a vector of radii, we also have to argue
that it is positive (the fact that no coordinates are $\infty$ follows since $|| \vect{r}^\infty||_{\ell_1}=1$).

\begin{claim}\label{r:positive}
  $\vect{r}^\infty_i>0$ for every $i$.
\end{claim}

\begin{proof}

Let $S = \{ v_i \in V : \vect{r}_i^\infty > 0 \}$. Because of the normalization of $\vect{r}$, we know that $S$ is nonempty. Assume for contradiction that $S \subsetneq V$. We repeat the exact same argument used in the proof of \cref{claim:C} showing first by case analysis that
\[
\lim_{t \to \infty} \sum_{v \notin S} \sigma_{\vect{r}^{(t)}}(v) = |F(V \setminus S)| \pi
\]
and then deducing that
\[
\lim_{t \to \infty} \sum_{v \notin S} \delta_{\vect{r}^{(t)}}(v) > 0.
\]
This contradicts that $\lim_{t \to \infty} \cE^{(t)} = 0$, so we conclude that $S = V$.
\end{proof}

\subsection{Proof of \cref{thm:cptriangulation}, step 2: Drawing the circle packing described by $\vect{r}^\infty$}

Given the vector of radii $\vect{r}^\infty$ satisfying \eqref{eq:defr}, we now show that the corresponding circle packing can be drawn uniquely up to translations and rotations. In fact, we provide a slightly more general statement which is due to Ori Gurel-Gurevich and Ohad Feldheim [personal communications, 2018].

Let $G=(V,E)$ be a finite planar triangulation on vertex set $\{v_1,\ldots,v_n\}$ and assume that $\{v_1,v_2,v_3\}$ is the outer face. A vector of positive real numbers $\Bell=\{\ell_e\}_{e \in E}$ indexed by the edge set $E$ is called \emph{feasible} if for any face enclosed by edges $e_1,e_2,e_3$, the lengths $\ell_{e_1}, \ell_{e_2}, \ell_{e_3}$ can be made to form a triangle. In other words, these lengths satisfy three triangle inequalities,
$$ \ell_i + \ell_j > \ell_k \qquad \{i,j,k\} = \{1,2,3\} \, .$$

Given a feasible edge length vector $\Bell$ we may again use the cosine formula to compute, for each face $f$, the angle at a vertex of the triangle formed by the three corresponding edge lengths. We denote these angles, as before, by $\alpha^{\Bell}_f(v)$ where $v$ is a vertex of $f$. Similarly, we define 
$$ \sigma_{\Bell}(v) = \sum_{f \in F(v)} \alpha^{\Bell}_f(v)$$
to be the sum of angles at a vertex $v$. 

\begin{theorem}\label{thm:howtodraw} Let $G$ be a finite triangulation and $\Bell$ a feasible vector of edge lengths. Assume that $\sigma_{\Bell}(v)=2\pi$ for any internal vertex $v$. Then there is a drawing of $G$ in the plane so that each edge $e$ is drawn as a straight line segment of length $\ell_e$ and no two edges cross. Furthermore, this drawing is unique up to translations and rotations.
\end{theorem}

It is easy to use the theorem above to draw the circle packing given the radii vector $\vect{r}^\infty$ satisfying \eqref{eq:defr}. Indeed, given $\vect{r}^\infty$ we set $\Bell$ by putting $\ell_e = \vect{r}^\infty_i + \vect{r}^\infty_j$ for any edge $e=\{v_i,v_j\}$ of the graph. Condition \eqref{eq:defr} implies that $\Bell$ is feasible. We now apply \cref{thm:howtodraw} and obtain the guaranteed drawing and draw a circle $C_i$ of radii $\vect{r}^\infty_i$ around $v_i$ for all $i$. \cref{thm:howtodraw} guarantees that the for any edge $\{v_i,v_k\}$ the distance between $v_i,v_j$ is precisely $\vect{r}^\infty_i + \vect{r}^\infty_j$ and thus $C_i$ and $C_j$ are tangent. Conversely, assume that $v_i, v_j$ do not form an edge. To each vertex $v$ let $A_v$ be the union of faces touching $v_i$ in the drawing of \cref{thm:howtodraw}. Since $v_i$ and $v_j$ are not adjacent we have that $A_{v_i}$ and $A_{v_j}$ have disjoint interiors. Furthermore, $C_i \subset {\rm {Int}}(A_i)$ since the boundary of $A_i$ is composed of lines between consecutive neighbors of $v_i$ and each of these lines are contained in the corresponding circles. By the same token $C_j \subset {\rm {Int}}(A_j)$ and we conclude that $C_i$ and $C_j$ are not tangent.

Lastly, we note that by \eqref{eq:defr} the outer boundary of the polygon we drew is a triangle with angles $\theta_1, \theta_2, \theta_3$ and hence $(r_1,r_2,r_3)$ is a positive multiple of $(\rho_1,\rho_2,\rho_3)$. Step 2 of the proof of \cref{thm:cp} is now concluded.

\begin{proof}[Proof of \cref{thm:howtodraw}] We prove this by induction on the number of vertices $n$. The base case $n=3$ is trivial since the feasibility of $\Bell$ guarantees that the edge lengths of the three edges of the outer face can form a triangle. Any two triangles with the same edge lengths can be rotated and translated to be identical, so the uniqueness statement holds for $n=3$.

Assume now that $n>3$ so that there exists an internal vertex $v$. Denote by $v_1, \ldots, v_m$ the neighbors of $v$ ordered clockwise. We begin by placing $v$ at the origin and drawing all the faces to which $v$ belongs, see \cref{fig:howtodraw}, left. That is, we draw the edge $\{v,v_1\}$ as a straight line interval of length $\ell_{\{v,v_1\}}$ on the positive $x$-axis emanating from the origin and proceed iteratively: for each $1 < i \leq m$ we draw the edge $\{v,v_i\}$ as a straight line interval of length $\ell_{\{v,v_i\}}$ emanating from the origin ($v$) at a clockwise angle of $\alpha^{\Bell}_f(v)$ from the previous drawn line segment of $\{v,v_{i-1}\}$, where $f=\{v,v_{i-1},v_i\}$. This determines the location of $v_1, \ldots, v_m$ in the plane and allows us to ``complete'' the triangles by drawing the straight line segments connecting $v_i$ to $v_{i+1}$, each of length $\ell_{\{v_i,v_{i+1}\}}$ where $1 \leq i \leq m$ (where $v_{m+1}=v_1$). Denote these edges by $e_1, \ldots, e_m$.

\begin{figure} \centering
\includegraphics[scale=0.75]{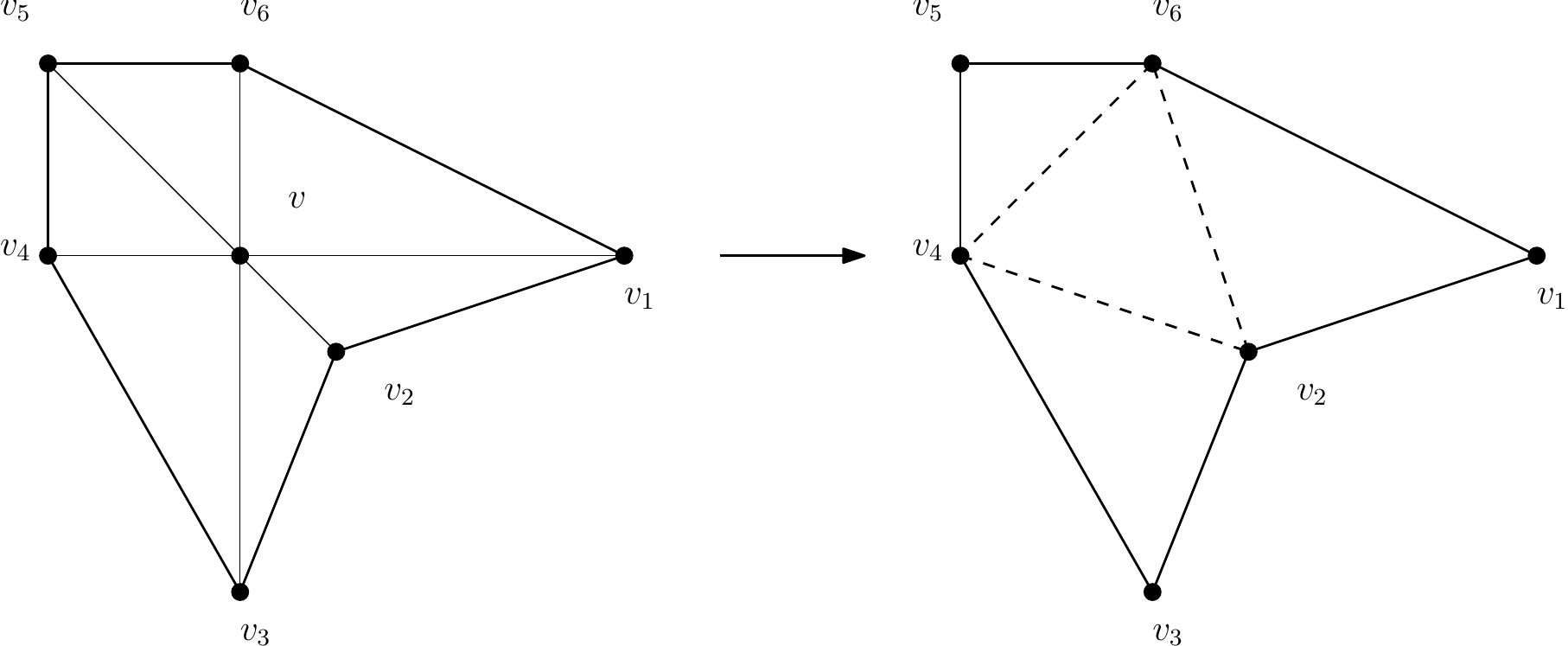}
\caption{On the left, we first draw the polygon surrounding $v$. On the right, we then erase $v$ and the edge emanating from it, replacing it with diagonals that triangulate the polygon while recording the lengths of the diagonals in $\Bell'$. The latter is the input to the induction hypothesis.}
\label{fig:howtodraw}
\end{figure}

Since $\sigma_{\Bell}(v)=2\pi$ we learn that these $m$ triangles have disjoint interiors and that the edges $e_1, \ldots, e_m$ form a closed polygon containing the origin in its interior. It is a classical fact \cite{Meisters75} that every closed polygon can be triangulated by drawing some diagonals as straight line segments in the interior of the polygon. We fix such a choice of diagonals and use it to form a new graph $G'$ on $n-1$ vertices and $|E(G)|-3$ edges by erasing $v$ and the $m$ edges emanating from it and adding the new $m-3$ edges corresponding to the diagonals we added. Furthermore, we generate a new edge length vector $\Bell'$ corresponding to $G'$ by assigning the new edges lengths corresponding to the Euclidean length of the drawn diagonals and leaving the other edge lengths unchanged. See \cref{fig:howtodraw}, right.

It is clear that $\Bell'$ is feasible and that the angle sum at each internal vertex of $G'$ is $2\pi$. Therefore we may apply the induction hypothesis and draw the graph $G'$ according to the edge lengths $\Bell'$. This drawing is unique up to translations and rotations by induction. Note that in this drawing of $G'$, the polygon corresponding to $e_1, \ldots, e_m$ must be the exact same polygon as before, up to translations and rotations, since it has the same edge lengths and the same angles between its edges. Since it is the same polygon, we can now erase the diagonals in this drawing and place a new vertex in the same relative location where we drew $v$ previously, along with the straight line segments connecting it to $v_1, \ldots, v_m$. Thus we have obtained the desired drawing of $G$. The uniqueness up to translations and rotations of this drawing follows from the uniqueness of the drawing of $G'$ and the fact that the location of $v$ is uniquely in that drawing.
\end{proof}

\subsection{Proof of \cref{thm:cptriangulation}, step 3: Uniqueness}

\begin{theorem} [Uniqueness of Circle Packing] \label{thm:cp:uniqueness}
Given a simple finite triangulation with outer face $v_1,v_2,v_3$ and three radii $\rho_1,\rho_2, \rho_3$, the circle packing with $C_{v_1},C_{v_2},C_{v_3}$ having radii $\rho_1,\rho_2,\rho_3$ is unique up to translations and rotations.
\end{theorem}

\begin{proof} We have already seen in step 2 that given the radii vector $\vect{r}$ the drawing we obtain is unique up to translations and rotations. Thus, we only need to show the uniqueness of $\vect{r}$ given $\rho_1, \rho_2, \rho_3$.

To that aim, suppose that $\vect{r}^a$ and $\vect{r}^b$ are two vectors satisfying \eqref{eq:defr}. Since the outer face in both vectors correspond to a triangle of angles $\theta_1, \theta_2, \theta_3$ we may rescale so that $\vect{r}^a_i=\vect{r}^b_i=\rho_i$ for $i=1,2,3$. After this rescaling, assume by contradiction that $\vect{r}^a \neq \vect{r}^b$ and let $v$ be the interior vertex which maximizes ${\vect{r}_v^a}/{\vect{r}_v^b}$. We can assume without loss of generality that this quantity is strictly larger than $1$, as otherwise we can swap $\vect{r}^a$ and $\vect{r}^b$.

Now we claim that for each $f = (v, u_1, u_2) \in F(v)$, we have $\alpha_f^{\vect{r}^a}(v) \leq \alpha_f^{\vect{r}^b}(v)$, with equality if and only if the ratios $\vect{r}_{u_i}^a / \vect{r}_{u_i}^b$, for $i=1,2$, are both equal to $\vect{r}_v^a / \vect{r}_v^b$. This is a direct consequence of \cref{obs:geometric}. Indeed, scale all the radii in $\vect{r}^b$ by a factor of $\vect{r}_v^a / \vect{r}_v^b$ to get a new vector $\vect{r}'$ such that $\vect{r}_v^a = \vect{r}'_v$ and $\vect{r}_u^a \leq \vect{r}'_u$ for all $u \neq v$. The second bullet point in \cref{obs:geometric} implies that $\alpha_f^{\vect{r}^a}(v) \leq \alpha_f^{\vect{r}'}(v) = \alpha_f^{\vect{r}^b}(v)$. As well, if either $\vect{r}_{u_1}^a < \vect{r}'_{u_1}$ or $\vect{r}_{u_2}^a < \vect{r}'_{u_2}$, then the cosine formula  yields the strict inequality $\alpha_f^{\vect{r}^a}(v) < \alpha_f^{\vect{r}'}(v)$. Thus, $\alpha_f^{\vect{r}^a}(v) = \alpha_f^{\vect{r}^b}(v)$ only if $\vect{r}_{u_i}^a / \vect{r}_{u_i}^b = \vect{r}_v^a / \vect{r}_v^b$ for $i = 1,2$.

Now, since $\alpha_f^{\vect{r}^a}(v) \leq \alpha_f^{\vect{r}^b}(v)$ for each $f \in F(v)$, while $\sigma_{\vect{r}^a}(v) = \sigma_{\vect{r}^b}(v) = 2\pi$, the equality $\alpha_f^{\vect{r}^a}(v) = \alpha_f^{\vect{r}^b}(v)$ must hold for each $f$. Therefore, each neighbor $u$ of $v$ satisfies $\vect{r}_u^a / \vect{r}_u^b = \vect{r}_v^a / \vect{r}_v^b$. Because the graph is connected, the ratio $\vect{r}_u^a / \vect{r}_u^b$ must be constant for all vertices $u \in V(G)$. But this contradicts that $\vect{r}_v^a / \vect{r}_v^b > 1$ while $\vect{r}_{v_i}^a / \vect{r}_{v_i}^b = 1$ for $i = 1,2,3$. We conclude that $\vect{r}^a = \vect{r}^b$.
\end{proof}

\chapter{Parabolic and hyperbolic packings}\label{chp:heschramm}

\section{Infinite planar maps}

In this chapter we discuss countably infinite \defn{locally finite} (that is, the vertex degrees are finite) connected simple graphs. In a similar fashion to the previous chapter, an infinite planar graph is a connected infinite graph such that there exists a drawing of it in the plane. We recall that a \emph{drawing} is a correspondence sending vertices to points of $\RR^2$ and edges to continuous curves between the corresponding vertices such that no two edges cross. An \defn{infinite planar map} is an infinite planar graph equipped with a set of cyclic permutations $\{\sigma_v : v\in V\}$ of the neighbors of each vertex $v$, such that there exists a drawing of the graph which respects these permutations, that is, the clockwise order of edges emanating from a vertex $v$ coincides with $\sigma_v$. 

Unlike the finite case, one cannot define faces as the connected components of the plane with the edges removed since the drawing may have a complicated set of accumulation points. This is the reason that we have defined faces in \cref{sec:planarstuff} combinatorially, that is, based solely on the edge set and the cyclic permutation structure. This definition makes sense in both the finite and infinite case.

A (finite or infinite) planar map is a {\bf triangulation}\index{triangulation} if each of its faces has exactly $3$ edges. Given a drawing of a triangulation, the Jordan curve theorem implies that the edges of each face bound a connected component of the plane minus the edges. We will often refer to the faces as these connected components. A triangulation is called a \defn{plane triangulation} if there exists a drawing of it such that every point of the plane is contained in either a face or an edge and any compact subset of the plane intersects at most finitely many edges and vertices. The term \defn{disk triangulation} is also used in the literature and means that there exists a drawing in the open unit disk in $\RR^2$ such that every point of the disk is contained in either a face or an edge and any compact subset of the disk intersects at most finitely many edges and vertices. Since the plane and the open disk are homeomorphic, we deduce that these two definitions are equivalent. For example, take the product of the complete graph $K_3$ on $3$ vertices with an infinite ray $\NN$ and add a diagonal edge in each face that has $4$ edges; then this is a plane triangulation. However, the product of $K_3$ with a bi-infinite ray $\ZZ$ together with the same diagonals is a triangulation but not a plane triangulation, since it cannot be drawn in the plane without an accumulation point.

It turns out that there is a combinatorial criterion for a triangulation to be a plane/disk triangulation. We say that an infinite graph is \defn{one-ended} if the removal of any finite set of its vertices leaves exactly one infinite connected component.

\begin{lemma}\label{disk:one:ended}
  An infinite triangulation is a plane triangulation if and only if it is 
  one-ended.
\end{lemma}
\begin{proof}
Suppose $G=(V,E)$ is a plane triangulation and consider a drawing of the graph with no accumulation points in the plane such that every point of the plane belongs to either an edge or a face. Let $A\subseteq V$ be a finite set of vertices and take $B \subset \RR^2$ to be a ball around the origin which contains every vertex of $A$, every edge touching a vertex of $A$ and every face incident to such an edge. Let $u\neq v$ be two vertices drawn outside of $B$ and take a continuous curve $\gamma$ between them in $\RR^2 \setminus B$.  By definition of $B$, this path only touches faces and edges that are not incident to the vertices of $A$ and hence one can trace a discrete path from $u$ to $v$ in the graph that ``follows'' $\gamma$ and avoids $A$. Since $B$ intersects only finitely many edges and vertices, we learn that $G \setminus A$ has a unique infinite component.

Conversely, assume now that $G$ is one-ended and consider a drawing of $G$ in the plane. By the stereographic projection we project the drawing to the unit sphere $\SS^2$ in $\RR^3$. Denote by $\mathcal{I}$ the complement in $\SS^2$ of the union of all faces and edges. Since $G$ is an infinite triangulation this union is an open set, hence $\mathcal{I}$ is a closed set and its boundary $\partial \mathcal{I}$ is precisely the set of accumulation points of the drawing. Since $\mathcal{I}$ is closed, each connected component of $\mathcal{I}$ must be closed as well and hence contain at least one accumulation point. Since $G$ is one-ended $\mathcal{I}$ cannot have more than one connected component, since otherwise we would be able to separate the two components by a finite set of edges and obtain two infinite connected components. Now choose a point  $p \in \mathcal{I}$ and rotate the sphere so that $p$ is the north pole. Project back the rotated sphere to the plane and consider the drawing in the plane. In this drawing the union of all faces and edges must be a simply connected set. By the Riemann mapping theorem this set is homeomorphic to the whole plane, and we deduce that the triangulation is a plane triangulation. 
\end{proof}


  

\section{The Ring Lemma and infinite circle packings}

The circle packing theorem \cref{thm:cp} is stated for finite planar maps. However, it is not hard to argue that any infinite map also has a circle packing. To this aim we will prove what is known as Rodin and Sullivan's \emph{Ring Lemma} \cite{RS87}; we will use it again many times throughout these notes. Given circles $C_0,C_1,\ldots,C_M$ with disjoint interiors, we say that $C_1, \ldots, C_M$ completely surround $C_0$ if they are all tangent to $C_0$ and $C_i$ is tangent to $C_{i+1}$ for $i=1,\ldots, M$ (where $C_{M+1}$ is set to be $C_1$).

\begin{lemma}[Ring Lemma, Rodin \& Sullivan '87 \cite{RS87}]\label{lem:ring}
  \index{Ring Lemma}
  For every integer $M>0$ there exists $A>0$ such that if $C_0$ is a circle completely surrounded by $M$ circles $C_1,\ldots,C_M$, and $r_i$ is the radius of $C_i$ for every $i=0,1,\ldots,M$, then $r_0/r_i\le A$ for every
  $i=1,\ldots,M$.
\end{lemma}

\begin{proof}
We may scale the picture so that $r_0=1$. Assume that the radius of $C_2$ is small and consider the circles $C_1$ and $C_3$ to its left and right. It cannot be that both $C_1$ and $C_3$ have large radii compared to $C_2$ since in this case they will intersect; see \cref{fig:ringlemma}. Hence, one of them has to be small as well. Assume without loss of generality that it is $C_3$. By similar reasoning, one of $C_1$ and $C_4$ has to be small. We continue this argument this way and get a path of circles of small radii; thus, for the circles $C_1, \ldots, C_M$ to completely surround $C_0$ we learn that $M$ must be large. \end{proof}

\begin{figure}
  \centering
    \begin{tikzpicture}[font=\small, scale=0.3]
      \node[vx,label=-90:$C_0$, inner sep=0.8pt, minimum size=0.5pt] (c0) at (0,0) {};
      \node[vx,label=-90:$C_2$,inner sep=0.8pt, minimum size=0.5pt] (c2) at (90:4.4) {};
      \node[vx,label=90:$C_3$, inner sep=0.8pt, minimum size=0.5pt] (c3) at (67.23:7) {};
      \node[vx,label=90:$C_1$, inner sep=0.8pt, minimum size=0.5pt] (c1) at (112.77:7) {};
      
      \draw (c0) circle (4);
      \draw (c2) circle (0.4);
      \draw (c3) circle (3);
      \draw (c1) circle (3);

    \end{tikzpicture}
    \caption{$C_2$ is small, but both $C_1$ and $C_3$ are large.}
\label{fig:ringlemma}
\end{figure}
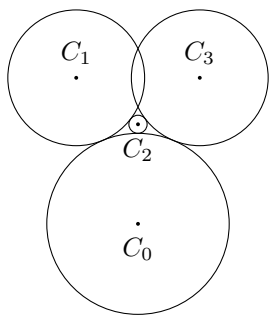

For a circle packing $P$ and a vertex $v$, denote by $C_v$ the circle corresponding to $v$, by $\cent(v)$ the center of that circle, and by $\rad(v)$ its radius. We write $G(P)$ for the tangency graph of the packing $P$, that is, the graph in which each vertex is a circle of $P$ and two such circles form an edge when they are tangent. 

\begin{claim}\label{cpt:inf}
  Let $G$ be an infinite planar map. Then there exists a circle packing $P$ such that $G(P)$ is isomorphic to $G$ as planar maps.
\end{claim}

\begin{proof} If $G$ is not a triangulation, then it is always possible to add in each face new vertices and edges touching them so the resulting graph is a planar triangulation (in an infinite face we have to put infinitely many vertices). After circle packing this new graph, we can remove all the circles corresponding to the added vertices and remain with a circle packing of $G$. Thus, we may assume without loss of generality that $G$ is a triangulation.

Fix a vertex $x$, and let $G_n$ be the graph distance ball of radius $n$ around $x$. Apply the circle packing theorem to $G_n$ to obtain a packing $P_n$, and scale and translate it so that $\rad(x)=1$ and $\cent(x)$ is the origin. 

  Consider a neighbor $y$ of $x$. By the
  Ring Lemma (\cref{lem:ring}), there exists a constant $A=A(x,y)>0$ such that $A^{-1}\le \rad(y) \le A$. By compactness there exists a subsequence of packings $P_{n_k}$ for which
  $\rad_{n_k}(y)$ and $\cent_{n_k}(y)$ both converge. By taking further subsequences for the rest of $x$'s neighbors, and then for the rest of the graph's vertices, it follows by a  diagonalization argument that there exists a subsequence such that the radii and centers of all vertices converge. The limiting packing $P_{\infty}$ satisfies that $G(P_{\infty})$ is isomorphic to $G$.
\end{proof}

\section{Statement of the He-Schramm theorem}

Given a circle packing $P$ of a graph $G$, we obtain a drawing of $G$ as follows: plot each vertex as the center of its corresponding circle in $P$ and connect adjacent vertices by straight lines. It is immediate that this is a drawing of $G$. When $G$ is a triangulation and $P$ is a circle packing of $G$, we define the \defn{carrier} of $P$, denoted $\carrier(P)$, to be the open subset of the plane obtained by taking the union of all faces (seen as open subsets of the plane) and all edges. When $P$ is a circle packing of an infinite one-ended triangulation, the argument in \cref{disk:one:ended} shows that $\carrier(P)$ is simply connected.

 We say that {\bf $G$ is circle packed in $\RR^2$} when $\carrier(P)=\RR^2$. Denote by $\UU$ the disk $\{z \in \RR^2 : |z|<1\}$; we say that {\bf $G$ is circle packed in $\UU$} when $\carrier(P)=\UU$. See \cref{fig:CPtypes}.

\begin{figure}[t]
  \centering
  \begin{subfigure}[b]{0.475\linewidth}
    \centering
    \includegraphics[width=0.9\linewidth]{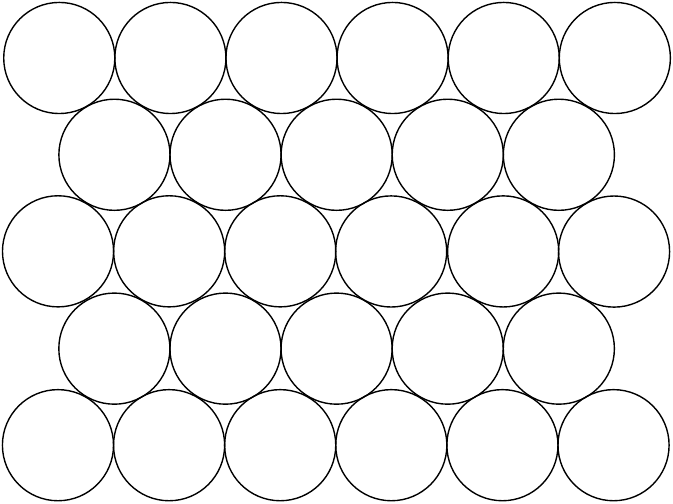}
  \end{subfigure}\hfill%
  \begin{subfigure}[b]{0.475\linewidth}
    \centering
    \includegraphics[width=0.9\linewidth]{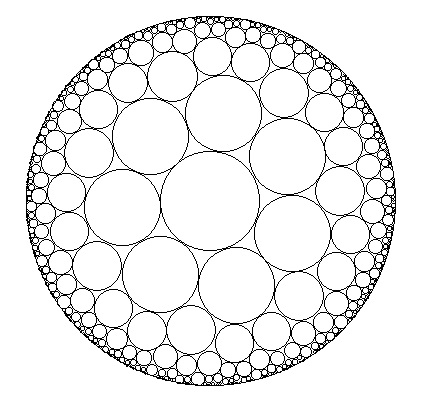}
  \end{subfigure}
  \caption{Two circle packings with carriers $\RR^2$ (left) and $\UU$ (right).}
  \label{fig:CPtypes}
\end{figure}

Let $G$ be a plane triangulation. Then $G$ can be drawn in the plane $\RR^2$ or alternatively in the disk $\UU$ (since they are homeomorphic), but can it be \emph{circle packed} both in  $\RR^2$ and in $\UU$? A celebrated theorem of He and Schramm \cite{HeSc} states that this cannot be done: each plane triangulation can be circle packed in either the plane or the disk, but not both. In fact, the combinatorial property of $G$ that determines on which side of the dichotomy we are is the recurrence or transience of the simple random walk on $G$ (assuming also that $G$ has bounded degrees, that is, $\sup_{x \in V(G)} \deg(x) < \infty$). This is the content of the He-Schramm theorem, which we are now ready to state.

\begin{theorem}[He, Schramm '95 \cite{HeSc}]\label{thm:hs}
  \index{He-Schramm's theorem}
  Let $G$ be an infinite plane triangulation with bounded degrees.
  
  \begin{thmenum}
    \item\label{hs:rec:RR} If $G$ is recurrent, then there exists a circle packing $P$ of $G$ such that $\carrier(P) = \RR^2$.
    \item\label{hs:trans:UU} If $G$ is transient, then there exists a circle packing $P$ of $G$ such that $\carrier(P) = \UU$.
    \item\label{hs:RR:rec} If $P$ is a circle packing of $G$ with $\carrier(P) = \RR^2$, then $G$ is recurrent.
    \item\label{hs:UU:trans} If $P$ is a circle packing of $G$ with
    $\carrier(P) = \UU$, then $G$ is transient.
  \end{thmenum}
\end{theorem}

\begin{corollary}
  Any bounded degree plane triangulation can be circle-packed in $\RR^2$ or $\UU$, but not both.
\end{corollary}

\begin{remark} In fact, it is proved in \cite{HeSc} that the corollary above holds without the assumption of bounded degree. Furthermore, in \cite{HeSc} \cref{hs:rec:RR} and \cref{hs:UU:trans} are proved without the bounded degrees assumption, but the other two statements require this assumption.
\end{remark}

The following example demonstrates why the bounded degree condition is necessary for \cref{hs:trans:UU} and \cref{hs:RR:rec}.

\begin{example} Let $P$ be a triangular lattice circle packing (as in \cref{fig:transient:unbounded}), and let $C_0,C_1,C_2,\ldots$ be an infinite horizontal path of circles in $P$ going (say) to the right. In the upper face shared by $C_n$ and $C_{n+1}$, draw $2^n$ circles which form a vertical path and each of them tangent both to $C_n$ and $C_{n+1}$; the last circle of these is also tangent to the upper neighbor of $C_n$ and $C_{n+1}$. See \cref{fig:transient:unbounded}. 

The resulting graph is a plane triangulation and the carrier of the packing is $\RR^2$. However, it is an easy exercise to verify that the tangency graph of this circle packing is transient.

\end{example}

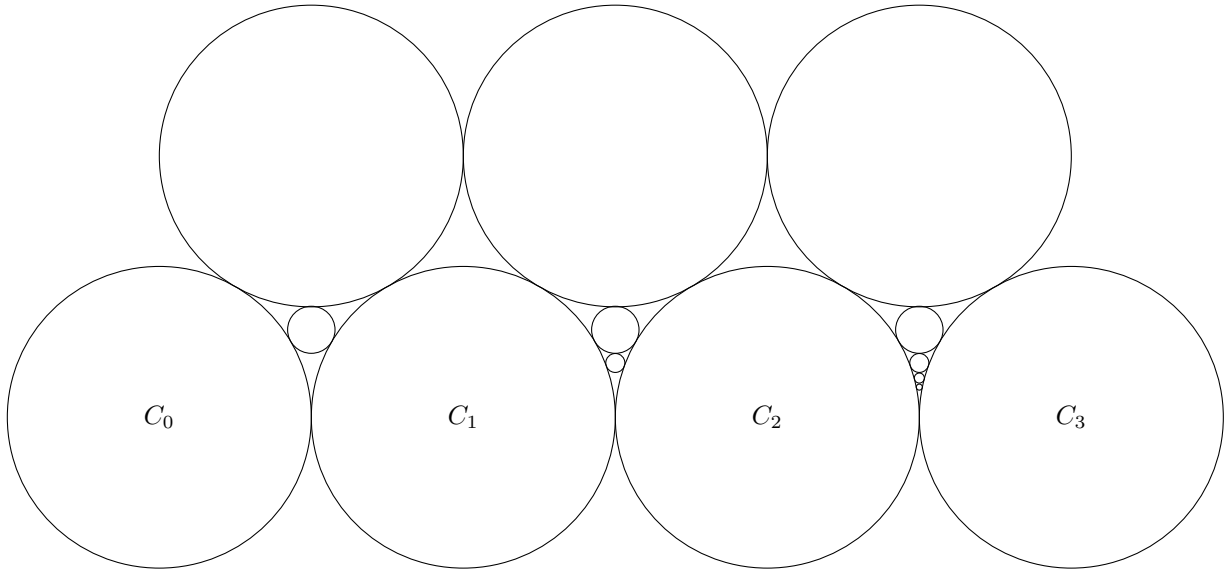
\begin{figure}[ht]
  \centering
    \begin{tikzpicture}[font=\small,scale=4]
    \node (1) at (-1,0) {$C_0$};
    
    \node (2) at (0,0) {$C_1$};
    \node (3) at (1,0) {$C_2$};
    \node (4) at (2,0) {$C_3$};
    \node (5) at (-0.5,0.866) {};
    \node (6) at (0.5,0.866) {};
    \node (7) at (1.5,0.866) {};
    \node (8) at (-0.5,0.29) {};
    \node (9) at (0.5,0.18) {};
    \node (10) at (0.5,0.29) {};
    \node (11) at (1.5,0.1) {};
    \node (12) at (1.5,0.13) {};
    \node (13) at (1.5,0.18) {};
    \node (14) at (1.5,0.29) {};
    
    \draw (1) circle (0.5);
    \draw (2) circle (0.5);
    \draw (3) circle (0.5);
    \draw (4) circle (0.5);
    \draw (5) circle (0.5);
    \draw (6) circle (0.5);
    \draw (7) circle (0.5);
    \draw (8) circle (0.078);
    \draw (9) circle (0.0314);
    \draw (10) circle (0.078);
    \draw (11) circle (0.01);
    \draw (12) circle (0.0166);                 \draw (13) circle (0.0314);
    \draw (14) circle (0.078);
    \end{tikzpicture}
    \caption{\label{fig:transient:unbounded} Unbounded degree transient triangulation circle packed in $\RR^2$.}
\end{figure}




In the rest of this chapter we prove \cref{thm:hs}. We begin by proving parts $3$ and $4$, in which a circle packing is given and one uses its geometry to deduce estimates about the effective resistance. Afterwards we prove parts $1$ and $2$, in which we use electrical estimates to deduce facts about the geometry of the circle packing.  

\section{Proof of the He-Schramm Theorem}

\subsection{Proof of \cref{hs:RR:rec}}

Denote the circle packing $P=\{C_v\}_{v\in V}$ where $V$ is the vertex set of $G$ and $C_v$ denotes the circle corresponding to the vertex $v$. Write $\Delta$ for the maximum degree of $G$ and fix a vertex $v_0$. By scaling and translating we may assume that $C_{v_0}$ is a radius $1$ circle around the origin. Given an open set $D$, we denote by $V_D\subseteq V$ the set of vertices $v$ for which the center of $C_v$ is in $D$.  For a real number $R>0$, let $V_R=V_{B(\origin,R)}$ where $B(\origin,R)$ is the Euclidean ball of radius $R$ around the origin.

\begin{lemma}\label{hs:1:lem:1}
  There exist $C=C(\Delta)>1$ and $c=c(\Delta)>0$ such that for every $R\ge 1$ 
  we have
  \begin{description}
    \item[(i)] There are no edges between $V_R$ and $V \setminus V_{CR}$, and
    \item[(ii)] $\reff\left(V_R\lr V \setminus V_{CR} \right)\ge c$.
  \end{description}
\end{lemma}

\begin{proof} We begin with part (i). For every $v\in V_R$ it holds that $\rad(v)\le R$ since $C_{v_0}$ is centered at the origin.  By the Ring Lemma 
  (\cref{lem:ring}), there exists $A=A(\Delta)$ such that $\rad(u)\le AR$ for every $u\sim v$ , and therefore $|\cent(u)|\le (A+2)R$.  Hence (i) holds  with $C=A+2$. \newline

  To prove part (ii) we define
  \begin{equation*}
    h(v) =
    \begin{cases}
      0 & v\in V_R,\\
      1 & v\in V \setminus V_{CR},\\
      \frac{|\cent(v)| - R}{(C-1)R} & \text{otherwise}.
    \end{cases}
  \end{equation*}
Recall from \cref{lem:dirichlet} that $\reff\left(V_R\lr V \setminus V_{CR} \right)\ge 
  \energy(h)^{-1}$.
  By the triangle inequality, for an edge $\{x,y\}$ with both endpoints in $V_{CR}\setminus V_R$ we have
  \begin{equation*}
    |h(x)-h(y)| \le \frac{|\cent(x)-\cent(y)|}{(C-1)R}
    = \frac{\rad(x)+\rad(y)}{(C-1)R},
  \end{equation*}
and it is straightforward to check that the same bound holds also when one of the edge's endpoints is in $V_R$ or $V\setminus V_{CR}$. Thus, using the Ring Lemma's (\cref{lem:ring}) constant $A=A(\Delta)$ from part (i),
  \begin{equation*}
    \energy(h) \le \sum_{x \in V_{CR} \setminus V_R} \sum_{y: y \sim x} \frac{( (A+1)\rad(x) )^2}{(C-1)^2R^2}
    \le \frac{\Delta(A+1)^2}{\pi(C-1)^2R^2}\cdot \sum_{x\in V_{CR \setminus V_R}} \area(C_x),
  \end{equation*}
where $\area(C_x)$ is the area that $C_x$ encloses (that is, $\pi \rad(x)^2$). We have that $\sum_x\area(C_x)\le \area(B(\origin,2CR)) = 4\pi C^2R^2$, hence 
if $C=A+2$, then
  \begin{equation*}
    \energy(h) \le 4\Delta C^2,
  \end{equation*}
  and the result follows for $c=(4\Delta C^2)^{-1}$.
\end{proof}

\begin{proof}[Proof of \cref{hs:RR:rec}]
Consider the unit current flow $I$ from $v_0$ to $\infty$ and fix any $R\geq 1$. Restricting this flow to the edges which have at least one endpoint in the annulus $V_{CR} \setminus V_R$ gives a unit flow from $V_R$ to $V \setminus V_{CR}$, by part (i) of \cref{hs:1:lem:1}. Hence, by part (ii) of that lemma and by Thomson's principle (\cref{thomson:principle}), the energy contributed to $\energy(I)$ from these edges is at least $c$. In the same manner, the edges which have at least one endpoint in the annulus $V_{C^{2k+1}R} \setminus V_{C^{2k}R}$ contribute at least $c$ to $\energy(I)$. Part (i) of \cref{hs:1:lem:1} implies that all these edge sets are disjoint, hence $\energy(I)=\infty$ and we learn that $G$ is recurrent (\cref{cor:transiencecondition}).
\end{proof}

\subsection{Proof of \cref{hs:UU:trans}}

We will use the given circle packing of $G$ to create a random path to infinity with finite energy. This gives transience by \cref{claim:randompath}. This proof strategy is similar to that of \cref{thm:z3transient}.

\begin{proof}[Proof of \cref{hs:UU:trans}]
Let $v_0$ be a fixed vertex of the graph, and apply a M\"obius transformation to make the circle of $P$ corresponding to $v_0$ be centered at the origin $\origin$. We now use \cref{claim:randompath} to construct a flow $\theta$ from $v_0$ to $\infty$ by choosing a uniform random point $\vect{p}$ on $\partial\UU$, taking the straight line from $\origin$ to $\vect{p}$ and considering the set of all circles in the packing $P$ that intersect this line in the order that they are visited; this set forms an infinite simple path in the graph which starts at $v_0$.

To bound the energy of the flow, we claim that there exists some constant $C$ (which may depend on the graph $G$ and the packing $P$) such that the probability that the random path uses the vertex $v$ is bounded above by $C\rad(v)$. Indeed, since there are only finitely many vertices with centers at distance at most $1/2$ from $\origin$, we may assume that the center of $v$ is of distance at least $1/2$ from $\origin$. In this case, in order for $v$ to be included in the random path the circle of $v$ must intersect the line between $\origin$ and $\vect{p}$. By the Ring Lemma (\cref{lem:ring}) the neighbors of $v$ have circles of radii comparable to $\rad(v)$ and so the probability of the line touching them is at most $C \rad(v)$. Since the vertex degree is bounded by $\Delta$, we find that
  \begin{equation*}
    \energy(\theta)
    \le C \Delta \sum_{v \in V} \rad(v)^2 \leq C \Delta \pi \, ,
  \end{equation*}
and we deduce by \cref{cor:transiencecondition} that $G$ is transient.
\end{proof}



\subsection{Proof of \cref{hs:rec:RR}}
We apply \cref{cpt:inf} to obtain a circle packing $P$ of $G$. We claim that $\carrier(P)=\RR^2$. Fix some vertex $v$ and rescale and translate so that $P(v)$ is the the unit circle $\partial \UU$. Assume by contradiction that $\carrier(P) \neq \RR^2$ and let $p \in \RR^2 \setminus \carrier(P)$ be a point not in the carrier. Rotate the packing so that $p=R$ for some real number $R>1$. Let $U\in [-1,1]$ and consider the circle $C_U = \{ z : |z-p| = R-U \}$. We traverse $C_U$ from the point $U$ counterclockwise and consider all the circles of $P$ which intersect $C_U$. The circles of $P$ we obtain this way is a simple path in the graph $G$ starting from $v$. The argument in \cref{disk:one:ended} shows that $\carrier(P)$ is simply connected, and since $p \not \in \carrier(P)$ it cannot be that $C_U$ is contained in $\carrier(P)$. Thus, as we traverse $C_U$ counterclockwise we must hit the boundary of $\carrier(P)$. We conclude that the path in $G$ we obtained in this manner is an infinite simple path starting at $v$. 

We now let $U$ be a uniform random variable in $[-1,1]$ and let $\mu$ denote the probability measure on random infinite paths starting at $v$ we obtained as described above. Let $\theta$ be the flow induced by $\mu$ as in \cref{claim:randompath}. We wish to bound the energy $\energy(\theta)$. Consider a vertex $w\in G$ and its circle $P(w)$ and let $B$ be the Euclidean ball of radius $R+1$ around $p$. If $P(w)$ does not intersect B, it cannot be included in the random path by our construction. If it does intersect this ball, then the probability that the random path intersects it is bounded above by its radius. Thus, 
$$ \energy(\theta) \leq \Delta \sum_{w : P(w) \cap B \neq \emptyset} \rad(w)^2 \, ,$$
where $\Delta$ is the maximal degree of $G$. We learn that $\energy(\theta)$ is bounded above by a constant multiple of the area of all circles of $P$ that intersect $B$. Since $p \not \in \carrier(P)$, by the Ring Lemma (\cref{lem:ring}), any circle of $P$ that intersects $B$ cannot have radius more than $AR$ for some large $A \geq R$ (since otherwise, all the circles surrounding this vertex will have radius more than $R+1$, contradicting the fact that $p \not \in \carrier(P)$). We learn that all the circles counted in the sum above are contained in the Euclidean ball of radius $(A+1)R+1$ around $p$. Since these circles has disjoint interiors, the sum of their area is bounded above by the area of the Euclidean ball above. We conclude that $\energy(\theta)<\infty$, hence $G$ is transient by \cref{cor:transiencecondition} and we have reached a contradiction. \qed
%


\newcommand{\Beuc}{B_{\mathrm{euc}}}

\subsection{Proof of \cref{hs:trans:UU}}

We will use the following simple corollary of the circle packing theorem, \cref{thm:cp}.
\begin{claim}\label{cp:no:outer:pack}
  Let $G$ be a finite simple planar map such that all faces have three edges except for one face (which we can think of as the outer face). Then, there is a circle packing $P$ of $G$ such that all circles of the outer face are internally tangent to $\partial \UU$ and all other circles of $P$ are contained in $\UU$.
\end{claim}
\begin{proof}
Denote by $v_1, \ldots, v_m$ the vertices of the outer face in clockwise order. Add a new vertex $v^*$ to the graph and connect it to $v_1, \ldots, v_m$ according to their order. We obtain a triangulation $G^*$. Apply \cref{thm:cp} to obtain a circle packing $P=\{C_v\}_{v \in V(G^*)}$. By translating and dilating we may assume that $C_{v^*}$ is centered at the origin and has radius $1$. Apply the map $z \mapsto \frac{1}{z}$ on this packing. Since this map preserves circles, the image of the circles $\{C_v\}_{v \in V(G^*)\setminus\{v^*\}}$ under this map is precisely the desired circle packing. 
\end{proof}

Furthermore, we will require an auxiliary general estimate. Given a circle packing $P$ and a set of vertices $A$, we write $\diam_P(A)$ for the Euclidean diameter of the union of all circles in $P$ corresponding to the vertices of $A$. 

\begin{lemma}\label{hs:lemma:diam:res}
Let $P$ be a circle packing in $\UU$ of a finite triangulation except the outer face with maximum degree $\Delta$, such that the circle of the vertex $v_0$ is centered at the origin and has radius $r_0$. Assume that $r_0 \in (r_{\min},r_{\max})$ for some constants $0<r_{\min}<r_{\max}<1$. Then there exists a constant $c=c(r_{\min}, r_{\max}, \Delta)>0$ such that for any connected set $A$ of vertices,
\begin{equation}\label{eq:reffupper} \reff(v_0 \lr A) \geq c \log \frac{1}{\diam_P(A)} \, .\end{equation}
If in addition all circles of the outer face are tangent to $\partial \UU$ and $A$ contains a vertex of the outer face, then
\begin{equation}\label{eq:refflower} \reff(v_0 \lr A) \leq  c^{-1} \log \frac{1}{\diam_P(A)\wedge{1 \over 2}} \, .\end{equation}
\end{lemma}
\begin{proof} Write $\eps = \diam_P(A)$ and let $z(A)$ denote the union of all circles corresponding to the vertices of $A$. We begin with the proof of \eqref{eq:reffupper}, which goes along similar lines to the proof of \cref{hs:1:lem:1}. Let $z_0\in \RR^2$ be such that $z(A) \subset \{|z-z_0|\leq \eps\}$. For any $r>0$ denote by $V_r$ the set of circles with centers inside $\{|z-z_0| \leq r\}$, so that $z(A) \subset V_\eps$. Repeating the proof of \cref{hs:1:lem:1} shows that there exists a constant $C=C(\Delta)>0$ such that
\begin{enumerate}
\item[(i)] There are no edges between $V_r$ and $V \setminus V_{Cr}$, and
\item[(ii)] $\reff(V_r \lr V \setminus V_{Cr}) \geq C^{-1}$, as long as $V_r$ and $V\setminus V_{Cr}$ are non-empty. 
\end{enumerate}
Regarding this proof, we note that even though for some values of $r$ the set $\{|z-z_0|\leq r\}$ is not contained in $\UU$ (unlike the proof of \cref{hs:1:lem:1} when the carrier is all of $\RR^2$). However, this only works to our benefit. The proof of \eqref{eq:reffupper} now proceeds similarly to the proof of \cref{hs:RR:rec}. By the Ring Lemma (\cref{lem:ring}), the Euclidean distance between the circle corresponding to $v_0$ and $A$ is at least some constant (which depends on $\Delta, r_{\min}, r_{\max}$) so that $v_0 \not \in V_{C^K \eps}$ for some $K = \Omega(\log(1/\eps))$. For each $k=0,2,4,\ldots, K$ the sets of edges which have at least one endpoint in the annulus $V_{C^{k+1} \eps} \setminus V_{C^k \eps}$ are disjoint by (i). By (ii), each of these sets of edges contribute at least $C^{-1}$ to the energy of the unit current flow from $A$ to $v_0$, concluding the proof of \eqref{eq:reffupper} using Thomson's principle (\cref{thomson:principle}). \newline

For the proof of \eqref{eq:refflower} we construct a unit flow from $v_0$ to $A$ that has energy $O(\log(1/\eps))$. The construction is in the same spirit as the proof of \cref{hs:UU:trans}, but there are some technical difficulties. Since $A$ contains a vertex that is tangent to $\partial \UU$, we choose $z_0\in \partial \UU$ that belongs to a circle of $A$. By rotating the packing we may assume that $z_0 = e^{i \eps /4}$. 

We now treat two cases separately. In the first case we assume that there exists $z_1$ in $z(A)$ such that $\arg(z_1)\in [0,\eps/2]$ and $|z_1|\leq 1-\eps/2$ such that the path in $z(A)$ from $z_0$ to $z_1$ remains in the sector $\arg(z) \in [0,\eps/2]$. Consider the points 
$$ x_0=-r_0 \qquad x_1=r_0 \qquad y_1 = 1-\eps/3 \qquad y_0=1 \, ,$$
and note that $x_0,x_1$ are the two leftmost and rightmost points on the circle of $v_0$. Let $C_0$ and $C_1$ be the upper half plane semi-circles in which $x_0,y_0$ and $x_1,y_1$ are antipodal points, respectively. The choice of $y_0, y_1$ is made so that the path between $z_0$ to $z_1$ in $z(A)$ must cross the region bounded by $C_0,C_1$ and the intervals $[x_0,x_1], [y_1,y_0]$, by our assumption on $z_1$ as long as $\eps$ is small enough. See \cref{fig:circularenergycalc}, left. 

For each $t\in[0,1]$ write $C_t$ for the upper half plane semi-circle in which $ty_1+(1-t)y_0$ and $tx_1+(1-t)x_0$ are antipodal points, so that $C_t$ continuously interpolates between $C_0$ and $C_1$. See \cref{fig:circularenergycalc}, left. Choose $t\in[0,1]$ uniformly at random and consider the random path $\gamma$ which traces $C_t$ from left to right. This random path starts at the circle of $v_0$ and must hit the path between $z_0$ and $z_1$ by our previous discussion. Hence, the circles of $P$ that intersect $\gamma$ must contain a path in the graph from $v_0$ to $A$. By \cref{claim:randompath} we obtain a flow $I$ from $v_0$ to $A$ whose energy $\energy(I)$ we now bound.

\begin{figure}[t]
  \centering
  \begin{subfigure}[b]{0.475\linewidth}
    \centering
    \includegraphics{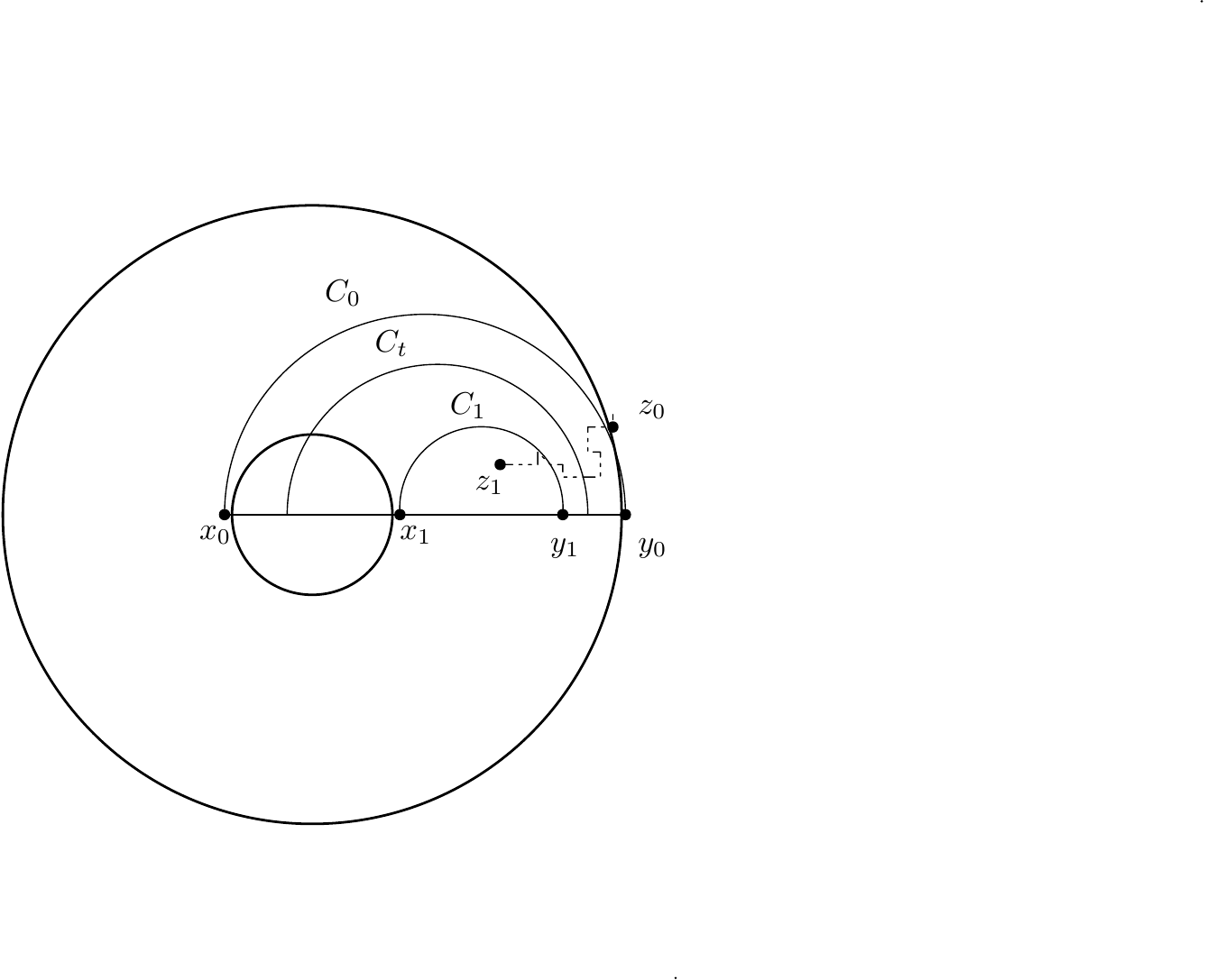}
  \end{subfigure}\hfill%
  \begin{subfigure}[b]{0.475\linewidth}
    \centering
    \includegraphics{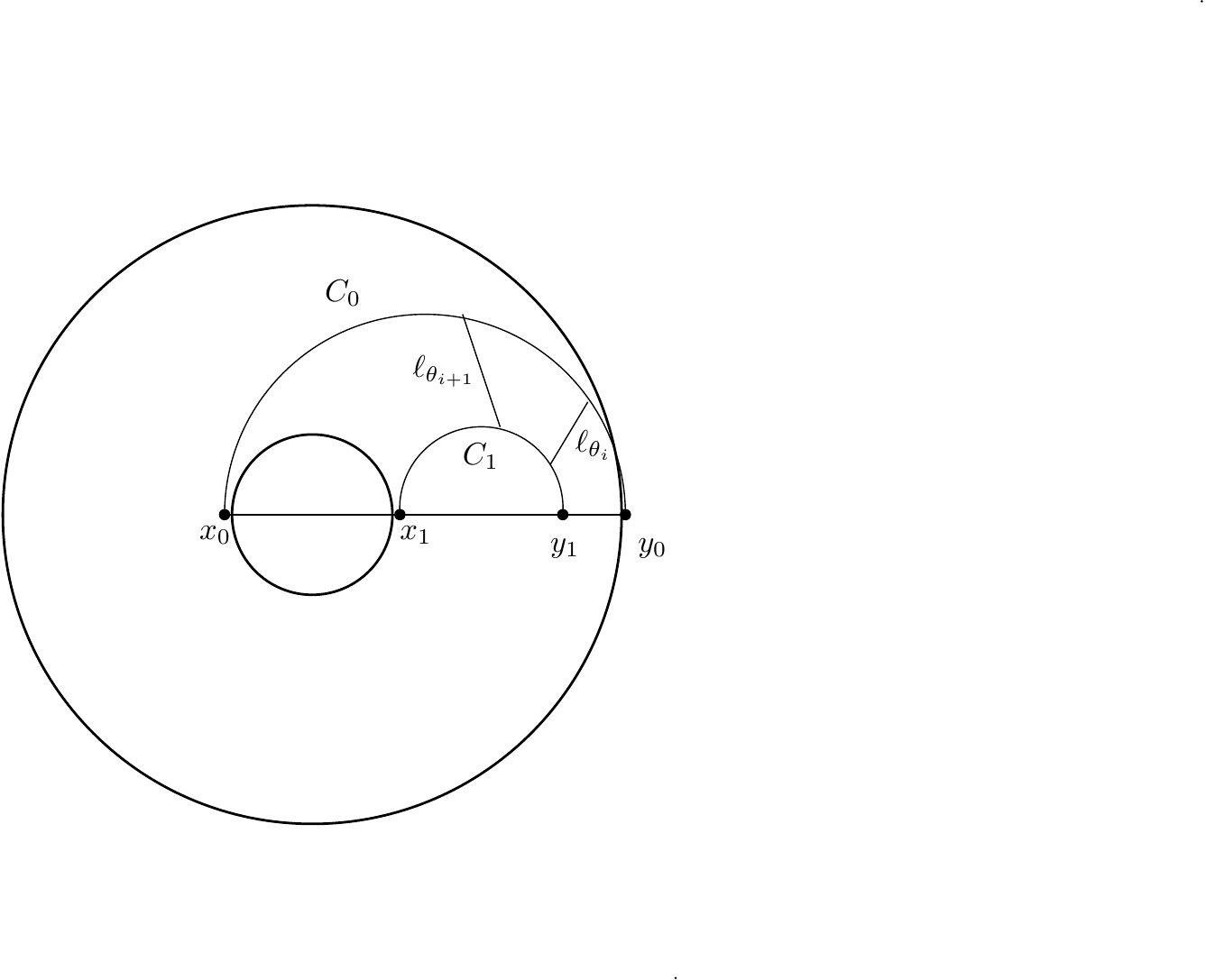}
  \end{subfigure}
  \caption{Left: For any $t\in[0,1]$ the semi-circle $C_t$ must intersect the path in $A$ between $z_0$ and $z_1$. Right: The quadrilateral $Q_i$ is bounded between $\ell_{\theta_i}$, $\ell_{\theta_{i+1}}$, $C_0$ and $C_1$.}
  \label{fig:circularenergycalc}
\end{figure}

For an angle $\theta\in[0,\pi]$ we denote by $w_\theta(t)$ the point at angle $\theta$ on the semi-circle $C_t$. It is an exercise to see that the set of points $\{w_\theta(t) : t\in[0,1]\}$ form a straight line interval $\ell_\theta$. Furthermore, when $t$ is chosen uniform in $[0,1]$, the intersection of $C_t$ and $\ell_\theta$ is a uniformly chosen point on $\ell_\theta$. Fix some constant $A>1$ and set $\theta_0 = 0$ and $\theta_i = A^{i-1} \eps$ for $i=1,\ldots,K$ where $K=\Theta(\log(1/\eps))$ such that $\theta_K=\pi$. We will obtain the bound $\energy(I)=O(\log(1/\eps))$ by bounding from above by a constant the contribution to $\energy(I)$ coming from edges which intersect the quadrilateral $Q_i$ of $\RR^2$ bounded by $\ell_{\theta_i}, \ell_{\theta_{i+1}}, C_0,C_1$; see \cref{fig:circularenergycalc}, right. The random path $\gamma$ restricted to $Q_i$ can be sampled by choosing a uniform random point on $\ell_{\theta_i}$, setting $t\in[0,1]$ to be the unique number such that $C_t$ intersects $\ell_{\theta_i}$ at the chosen point, and tracing the part of $C_t$ from $\ell_{\theta_i}$ to $\ell_{\theta_{i+1}}$. The lengths of the four curves bounding $Q_i$ are all of order $A^i \eps$ and so we deduce that if $v$ corresponds to a circle of radius $O(A^i \eps)$ which intersects $Q_i$, then the probability that it is visited by $\gamma$ is $O(\rad(v)/A^i \eps)$. Since the sum of $\rad(v)^2$ over such $v$'s is at most the area of $Q_i$, up to a multiplicative constant since some of these circles need not be contained in $Q_i$, and so it has order $A^{2i} \eps^2$. Since the degrees are bounded, we deduce that the contribution to the energy from edges touching such $v$'s is $O(1)$. Lastly, if $v$ corresponds to a larger circle, then we bound its probability of being visited by $\gamma$ by $1$ and note that there can only be $O(1)$ many such $v$'s whose circles intersects $Q_i$. Thus the contribution from these is another $O(1)$. Since there are $O(\log(1/\eps))$ such $i$, we learn that $\energy(I) = O(\log(1/\eps))$ finishing our proof in this case using Thomson's principle (\cref{thomson:principle}).

In the second case, we assume that there exists $z_1 \in z(A)$ such that $\arg(z_1) \not \in [0,\eps/2]$ and $|z_1|\geq 1-\eps$. It is clear that since $\diam_P(A)=\eps$ either the first or the second case must occur. Denote $z_0'=|z_1|e^{i\eps/4}$ and let $x_0, x_1$ be antipodal points on the circle of $v_0$ such that the straight line between $x_0$ and $x_1$ is parallel to the straight line between $z_0'$ and $z_1$. Consider the trapezoid on the vertices $z_0', z_1, x_0, x_1$. We choose a uniform random point $t \in [0,1]$ and stretch a straight line from $t x_0 + (1-t)x_1$ to $t z_1 + (1-t)z_0'$. We then add to it a straight line from $t z_1 + (1-t)z_0'$ to $w \in \partial \UU$ where $\arg(w) = \arg(t z_1 + (1-t)z_0')$. For any $t\in [0,1]$, this path $\gamma$ starts from the circle of $v_0$ and must hit the path between $z_0$ and $z_1$ in $z(A)$. Thus, the set of all circles which intersect $\gamma$ must contain a path in the graph that starts at $v_0$ and ends at $A$; this random choice of $\gamma$ gives us as usual a unit flow from $v_0$ to $A$. See \cref{fig:trapezoidenergycalc}. By repeating the same argument as in the previous case (that is, splitting the trapezoid into $O(\log (1/\eps))$ many trapezoids of constant aspect ratio), we see that the contribution to the energy of the flow induced by $\gamma$ of the edges in the trapezoid is $O(\log(1/\eps))$. Furthermore, the same argument gives that the edges in the quadrilateral formed by the vertices $z_0, z_0', z_1$ and $e^{i \arg(z_1)}$ contribute at most a constant to the energy, concluding our proof by Thomson's principle in this case as well. 
\end{proof}

\begin{figure}[t]
  \centering
  \includegraphics[scale=0.8]{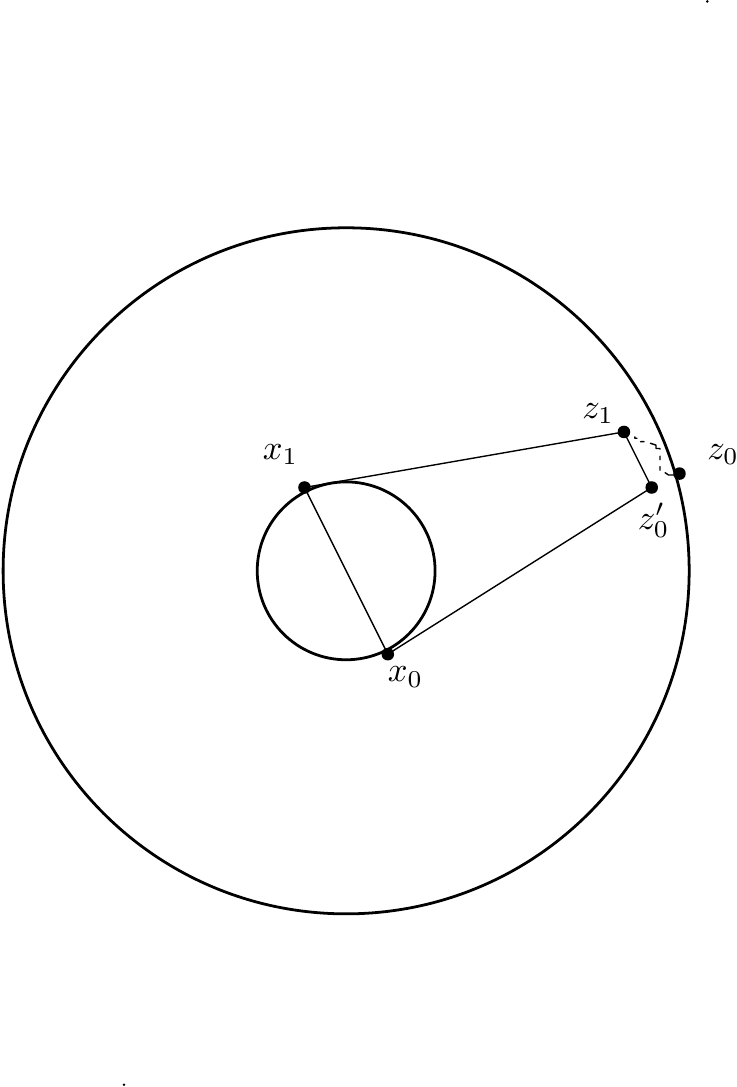}

  \caption{The resistance across the trapezoid on vertices $x_0, x_1, z_0', z_1$ is $O(\log(1/\eps))$ when $|z_0'-z_1|=\Theta(\eps)$.}
  \label{fig:trapezoidenergycalc}
\end{figure}

\begin{proof}[Proof of \cref{hs:trans:UU}] As usual we denote by $d_G(u,v)$ the graph distance between the vertices $u,v$ of $G$. Fix some $v_0 \in V$ and let \begin{align*}
  A_j &= \{v: d_G(v_0,v) \leq j\}, \\
  V_j &= A_j \cup \{\text{finite components of }V\setminus A_j\}, \\ 
  E_j &= \{\text{edges induced by }V_j\}.
  \end{align*} 
Observe that since $G$ is one-ended the finite map $\left(V_j, E_j\right)$ is a triangulation except the outer 
face which we denote by $\partial V_j$. We apply \cref{cp:no:outer:pack} to pack $(V_j,E_j)$ inside the unit disk $\UU$ such that the circles of $\partial V_j$ are tangent to $\UU$. By applying a M\"obius transformation from $\UU$ onto $\UU$, we may assume that the circle corresponding to $v_0$ is centered at the origin $\origin$. We denote this packing by $P_j$ and let $r_0^j$ be the radius of $v_0$ in $P_j$. 

Since $G$ is transient it follows that there exists some $c=c(\Delta)>0$ such that $r_0^j \geq c$ for all $j$ by \cref{cor:transiencecondition}. Indeed, if $r_0^j \leq \eps$, we learn by \cref{hs:1:lem:1} and the proof of \cref{hs:RR:rec} that $\reff(v_0 \to \infty) \geq c' \log(\eps^{-1})$ for some $c'=c'(\Delta)>0$.

As we did in \cref{cpt:inf}, we now take a subsequence in which the centers and radii of all vertices converge. Denote the resulting limiting packing by $P_\infty$. This packing has all circles inside $\UU$ and we therefore deduce that $\carrier(P_\infty) \subseteq \UU$. It is a priori possible that $\carrier(P_\infty)$ is some strict subset of $\UU$, i.e., that all the circles stabilize inside some strict subset of $\UU$. We now argue however that this is not possible.

Let $Z$ be the set of accumulation points of $\carrier(P_\infty)$; it suffices to show that $Z\subset \partial \UU$ (since any simply connected open set $G \subset \UU$ for which $\partial G \subset \partial \UU$ must equal $\UU$). Since $Z$ is a compact set, let $z \in Z$ minimize $|z|$ among all $z\in Z$; it suffices to show that $z\in \partial \UU$. Fix $\varepsilon>0$ and put \begin{equation*} 
U_\eps(z)= \big \{ v \in G : |\cent_{P_\infty}(v) - z| \leq \eps \big \} \, .
\end{equation*}

The set $U_\eps(z)$ may not be connected (graph-wise), yet by our choice of $z$ it is clear that $U_\eps(z)$ has an infinite connected component. Indeed, one can draw a straight line from the origin to $z$ without intersecting $Z$ and consider the set of all circles intersecting this line; from some point onwards the vertices corresponding to these circles will reside in $U_\eps(z)$.

Therefore, let $W_\eps(z)$ be an infinite connected component of the graph spanned on $U_\eps(z)$. Let $J=J(z,\eps)$ be the first integer such that $V_J \cap W_\eps(z) \neq \emptyset$. Since the $V_j$'s are increasing sets and $W_\eps(z)$ is an infinite connected set, we have that $\partial V_j \cap W_\eps(z) \neq \emptyset$ for all $j \geq J$. Consider now a connected component $A$ of the graph spanned on the vertices $V_j \cap W_\eps(z)$. 
Denote by $P_\infty^{j}$ the finite circle packing obtained from $P_\infty$ by taking only the circles of $V_j$. 

Since $A \subset W_\eps(z)$, it follows that $\diam_{P_\infty^{j}}(A) \leq 4\eps$. By \cref{hs:lemma:diam:res}, \cref{eq:reffupper}, applied to the set $A$ in the packing $P_\infty^{j}$, we deduce that $\reff(v_0 \lr A; V_j) \geq c \log(1/\eps)$. By Rayleigh's monotonicity (\cref{rayleigh}) we have that $\reff(v_0 \lr A; (V_j,E_j)) \geq c \log(1/\eps)$. Since $A$ is a connected component of $V_j \cap W_\eps(z)$ and since $W_\eps(z)$ is an infinite connected set of vertices in $G$, it follows that $A$ must contain a vertex of $\partial V_j$. Thus, we may apply \cref{hs:lemma:diam:res}, \cref{eq:refflower}, to the set $A$ in the packing $P_j$, we get that there exists some $c=c(G)>0$ such that
\begin{equation}\label{eq:diamA} \diam_{P_j}(A) \leq \eps^c \, .\end{equation}

Choose some $v_J \in \partial V_J \cap W_\eps(z)$ so that the circle corresponding to $v_J$ in $P_\infty$ touches $\partial \UU$ and $|\cent_{P_\infty}(v_J) - z|\leq \eps$. For each $j > J$ choose some $v_j \in \partial V_j \cap W_\eps(z)$ so that $v_j$ and $v_J$ are in the same connected component $A$ of $V_j \cap W_\eps(z)$. Since the circle of $v_j$ in $P_j$ touches $\partial \UU$ we learn by \eqref{eq:diamA} that the distance of the circle of $v_J$ in $P_j$ from $\partial \UU$ is at most $\eps^c$ for all $j \geq J$. Since the circle corresponding to $v_J$ in $P_\infty$ is the limit of its circles in $P_j$ we deduce that the distance of $\cent_{P_\infty}(v_J)$ from $\partial \UU$ is at most $\eps^c$. We deduce that distance of $z$ from $\partial \UU$ is at most $\eps + \eps^c$. Since $\eps$ was arbitrary we learn that $z\in \partial \UU$, as required.
\end{proof}

\section{Exercises}

\begin{enumerate}


\item Let $G$ be a finite simple planar map such that all of its faces have 3 edges except for the outer face which is a simple cycle. Show that there exists a circle packing of $G$ such that all the circles are inside the unit disc $\{z: |z|\leq 1\}$ and all the circles corresponding to the vertices of the outer face are tangent to the unit circle $\{z: |z|=1\}$.

\item Let $G$ be a triangulation of the plane with maximal degree at most $6$. Prove that the simple random walk on $G$ is recurrent.



\item Let $G$ be a plane triangulation that can be circle packed in the unit disc $\{z : |z|<1\}$. Show that the simple random walk on $G$ is transient. (Note that $G$ may have \emph{unbounded} degrees)

\item[4.(*)] Let $P$ be a circle packing of a finite simple planar map with degree bounded by $D$ such that all of its faces are triangles except for the outerface. Assume that the carrier of $P$ is contained in $[-11,11]^2$, contains $[-10,10]^2$ and that all circles have radius at most $1$. Let $h$ be the harmonic function taking the value $1$ on all vertices with centers left of the line $\{-10\}\times \mathbb{R}$, taking the value $0$ on all vertices with centers right of the line $\{10\}\times \mathbb{R}$, and is harmonic anywhere else. Assume $x$ and $y$ are two vertices such that their centers are contained in $[-1,1]^2$ and that the Euclidean distance between these centers is at most $\epsilon>0$. Show that
$$ |h(x) - h(y)| \leq {C \over \log(1/\epsilon)} \, ,$$
for some constant $C=C(D)>0$ independent of $\epsilon$. [Hint: assume $h(x)<h(y)$ and consider the sets $A=\{v:h(v)\leq h(x)\}$ and $B = \{v: h(v) \geq h(y)\}$].




\end{enumerate}

\chapter{Planar local graph limits}\label{chp:locallimit}

\section{Local convergence of graphs and maps}\label{sec:locconv}

In order to study large random graphs it is mathematically appealing and natural to introduce an infinite limiting object and study its properties. In their seminal paper, Benjamini and Schramm \cite{BeSc} introduced the notion of locally convergent graph sequences, which we now describe.

We will consider random variables taking values in the space $\Gdot$ of rooted locally finite connected graphs viewed up to root preserving graph isomorphisms. That is, $\Gdot$ is the space of pairs $(G,\rho)$ where $G$ is a graph (finite or infinite) and $\rho \in V(G)$ is a vertex of $G$, where two elements $(G_1,\rho_1), (G_2,\rho_2)$ are considered equivalent if there is a graph isomorphism between them (that is, a bijection $\varphi:V(G_1)\to V(G_2)$ such that $\varphi(\rho_1)=\varphi(\rho_2)$ and $\{v_1,v_2\}\in E(G_1)\iff \{\varphi(v_1),\varphi(v_2)\}\in E(G_2)$). In a similar fashion we define $\Mdot$ to be the set of equivalence classes of rooted maps; in this case we require the graph isomorphism in addition to preserve the cyclic permutations of the neighbors of each vertex, that is, it is a {\bf map isomorphism}. Let us describe the topology on $\Gdot$ and $\Mdot$. For convenience we discuss $\Gdot$ but every statement in the following holds for $\Mdot$ as well.

Given an element $(G,\rho)$ of $\Gdot$ the finite graph $B_{G}(\rho,R)$ is the the subgraph of $(G,\rho)$ rooted at $\rho$ spanned by the vertices of distance at most $R$ from $\rho$. We provide $\Gdot$ with the \defn{local metric}
\begin{equation*}
  \dloc((G_1,\rho_1),(G_2,\rho_2)) = 2^{-R},
\end{equation*}
where $R$ is the largest integer for which
$B_{G_1}(\rho_1,R)$ and $B_{G_2}(\rho_2,R)$ are isomorphic as graphs. This is a separable topological space (the finite graphs form a countable base for the topology) and is easily seen to be complete, i.e., it is a Polish space. The distances are bounded by $1$ but the space is not compact. Indeed, the sequence $G_n$ of stars with $n$ leaves emanating from the root $\rho$ has no converging subsequence. 

%

Since $\Gdot$ is a Polish space, we can discuss convergence in distribution of a sequence of random variables $\{X_n\}_{n=1}^\infty$ taking values in $\Gdot$. We say that $X_n$ 
\textbf{converges in distribution}\index{convergence in distribution} to a random variable $X$,
and denote it by $X_n\convd X$, if for every bounded continuous function $f:\Gdot\to \RR$ we have that $\E(f(X_n))\to\E(f(X))$. We will be focused here on the particular situation in which $X_n$ is a \emph{finite} rooted random graph $(G_n,\rho_n)$ such that given $G_n$, the root $\rho_n$ is uniformly distributed among the vertices of $G_n$. It is a very common setting and justifies the following definition.

\begin{definition} \label{def:locconv}
  Let $\{G_n\}$ be a sequence of (possibly random) finite graphs.  We say that $G_n$ \textbf{converges locally}\index{local convergence}\index{local limit} to a (possibly infinite) random rooted graph $(U,\rho)\in\Gdot$, and denote it by $G_n\convl(U,\rho)$, if for every integer $r\ge 1$,
  \begin{equation*}
    B_{G_n}(\rho_n,r) \convd B_U(\rho,r),
  \end{equation*}
  where $\rho_n$ is a uniformly chosen vertex from $G_n$.
\end{definition}
\noindent It is straightforward to see that this definition is equivalent to saying that the random variables $(G_n,\rho_n)$ converge in distribution to $(U,\rho)$.

\subsection{Examples}
\begin{itemize}
\item The sequence $\{G_n\}$ of paths of length $n$ converges locally to the graph $(\ZZ,0)$ (note that the root vertex can be chosen to be any vertex of $\ZZ$ since $(\ZZ,i)$ and $(\ZZ,j)$ are equivalent for all $i,j \in \ZZ$).
\item The sequence $\{G_n\}$ of the $n \times n$ square grid converges locally to the graph $(\ZZ^2,\origin)$ (again the root can be chosen to be any vertex of $\ZZ^2$).
\item Let $\lambda>0$ be fixed and let $\{G(n,\frac{\lambda}{n})\}$ be the sequence of random graphs obtained from the complete graph $K_n$ by retaining each edge with probability $\frac{\lambda}{n}$ and erasing it otherwise, independently for all edges. This is known as the Erd\"os-R\'enyi random graph. One can verify that this sequences converges locally to a branching process with progeny distribution Poisson$(\lambda)$.
\item If $G_n$ is the binary tree of height $n$ then its local limit is
\emph{not} the infinite binary tree with any vertex distribution.  Instead, it is the following so-called \defn{canopy tree} depicted in \cref{fig:canopy} and the root is at distance $k \geq 0$ from the leaves with probability $2^{-k-1}$. Note that the distance of the root from the leaves determines the isomorphism class of the rooted graph. 
It is easy to see that the canopy tree is not the infinite binary tree (for example, it has leaves); in fact, it is recurrent.
\item Consider $G_n$ to be a path of length $n$, glued via one of its leaves
into a $\sqrt{n}\times\sqrt{n}$ grid.  The local limit of $G_n$ is $(U,\rho)$,
where $(U,\rho)$ is $(\ZZ,0)$ with probability $1/2$, and $(\ZZ^2,\origin)$ 
otherwise.
\end{itemize}

\begin{figure}[t]
  \centering
    \begin{tikzpicture}[font=\small, scale=1.5]
      \input{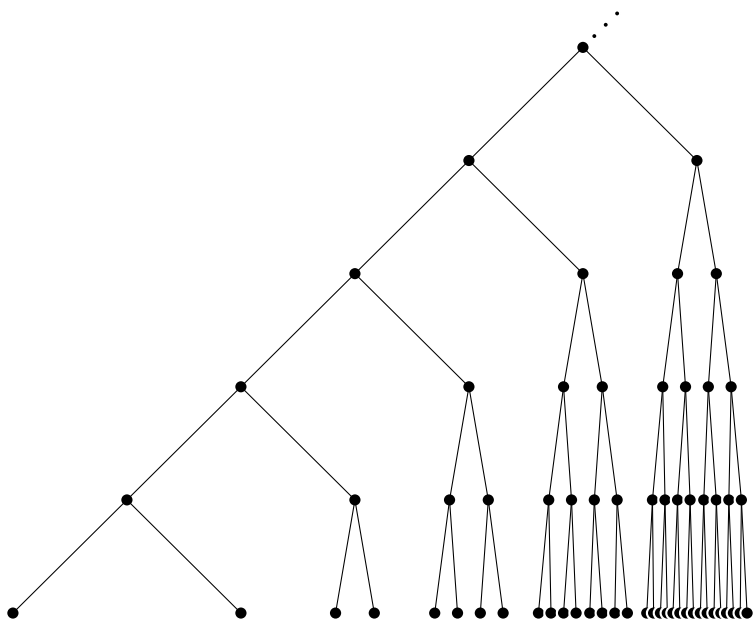}
    \end{tikzpicture}
  \caption{\label{fig:canopy} A part of the \textbf{canopy tree}.}
\end{figure}

Our goal in this chapter is to prove the following pioneering result.

\begin{theorem}[Benjamini-Schramm \cite{BeSc}]\label{thm:BeScLimit} Let $M<\infty$ and let $G_n$ be finite planar maps (possibly random) with degrees almost surely bounded by $M$ such that $G_n\convl(U,\rho)$. Then $(U,\rho)$ is a.s.\ recurrent.
\end{theorem}




For instance, a local limit of planar maps cannot give the $3$-regular tree as its limit (this tree however can be obtained as a local limit of random $3$-regular graphs). The bounded degree assumption is necessary for this theorem. Indeed, suppose we start with a binary tree of height $n$ and replace each edge $(u,v)$ that is at distance $k\geq 0$ from the leaves by $2^k$ parallel edges. By the same reasoning of the local convergence of binary trees to the canopy tree, the modified graph sequence converges locally to a modified canopy tree in which an edge at distance $k$ from the leaves is replaced with $2^k$ parallel edges. Using the parallel law it is immediate to see that this graph is transient, and that the effective resistance from a leaf to $\infty$ is at most $2$ (in fact it equals $2$). See \cref{fig:canopy:transient}.


\begin{figure}[t]
  \centering
  \begin{tikzpicture}[font=\small, scale=1.5]
  \input{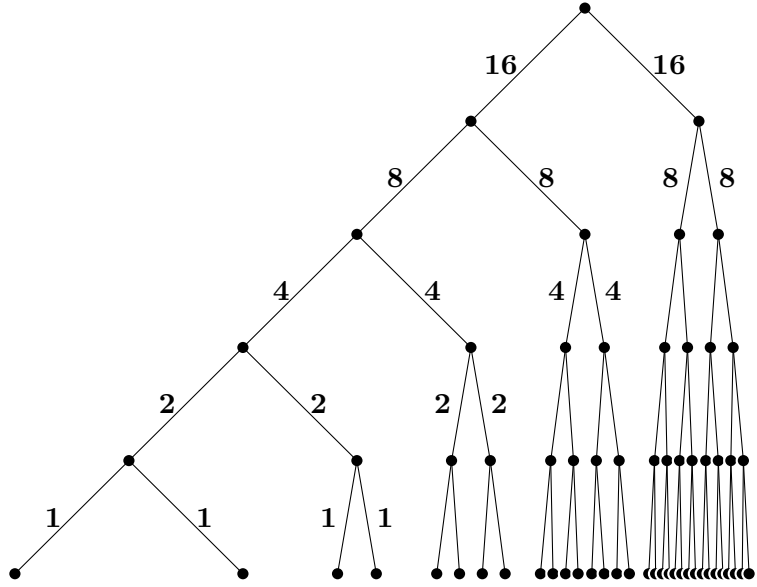}
  \end{tikzpicture}
  \caption{\label{fig:canopy:transient} A part of a \textbf{transient canopy tree}\index{transient canopy tree}. Numbers on edges are conductances of those edges after applying the parallel law.}
\end{figure}

\section{The Magic Lemma}\index{magic lemma}
Suppose $C\subseteq\RR^2$ is finite.  For each $w\in C$, define
\begin{equation*}
  \rho_w = \min\{|v-w|:v\in C\setminus \{w\}\}.
\end{equation*}
We call $\rho_w$ the \defn{isolation radius} of $w$.
Given $\delta\in(0,1)$, $s\ge 2$ and $w\in C$, we say that $w$ is 
\textbf{$(\delta,s)$-supported} if in the disk of radius $\delta^{-1}\rho_w$ 
around $w$ there are at least $s$ points of $C$ outside any given disk of 
radius $\delta\rho_w$. Formally, $w$ is 
\textbf{$(\delta,s)$-supported} if
\begin{equation*}
  \inf_{p\in\RR^2}
  \left|
    C\cap B\left(w,\delta^{-1}\rho_w\right)\setminus B(p,\delta\rho_w)
  \right| \ge s.
\end{equation*}

The proof of \cref{thm:BeScLimit} is based on the following lemma, which has been dubbed ``the Magic Lemma''.

\begin{lemma}[\cite{BeSc}]\label{lem:magic}
  There exists $A>0$ such that for every $\delta\in(0,1/2)$, every finite
  $C\subseteq\RR^2$ and every $s\ge 2$, the number of $(\delta,s)$-supported
  points in $C$ is at most
  \begin{equation*}
    \frac{A|C|\delta^{-2}\ln(\delta^{-1})}{s}.
  \end{equation*}
\end{lemma}

\begin{remark}
  We prove the lemma for $\RR^2$, but it holds for $\RR^d$ or any other doubling metric space. In fact, a metric space for which the lemma holds must be doubling; see \cite{Gill}.
\end{remark}

\subsection{Proof of \cref{lem:magic}}
Let $k\ge 3$ be an integer (later we will take $k=k(\delta)$).  Let $G_0$ be 
a tiling of $\RR^2$ by $1\times 1$ squares, rooted at some point $p$, and 
for every $n\in\ZZ$, let $G_n$ be a tiling of $\RR^2$ by $k^n\times k^n$ such that each square of $G_n$ is tiled by $k\times k$ squares of $G_{n-1}$. We may choose $p$ so that none of the points of $C$ lies on the edge of a square.
  
We say that a square $S\in G_n$ is \textbf{$s$-supported} if for every 
smaller square $S'\in G_{n-1}$ we have that
$|C\cap (S\setminus S')|\ge s$.

\begin{claim}\label{claim:mag:squares}
  For any $s\geq 2$ the total number of $s$-supported squares, in $G=\bigcup_{n\in\ZZ} G_n$, is 
  at most $2|C|/s$.
\end{claim}
\begin{proof}  
  Define a ``flow'' $f:G\times G\to\RR$ as follows:
  \begin{equation*}
    f(S',S) = \begin{cases}
      \min(s/2,|S'\cap C|) & S'\subseteq S, S'\in G_n, S\in G_{n+1},\\
      -f(S,S') & S\subseteq S', S\in G_n, S'\in G_{n+1},\\
      0 & \text{otherwise}.
    \end{cases}
  \end{equation*}
  Let us make two initial observations. First we have that
  \begin{equation}\label{magic:obs1}
  \sum_{S'\in G} f(S',S)\ge 0 \, ,
  \end{equation}
  by splitting into the two cases depending on whether there exists a square $S' \subseteq S$ such that $f(S',S)=s/2$ or not. Secondly, if $S$ is a $s$-supported square 
  \begin{equation}\label{magic:obs2}
  \sum_{S'\in G} f(S',S) \ge \frac{s}{2} \, ,
  \end{equation}
  by splitting into cases depending on whether the number of squares $S' \subseteq S$ such that $f(S',S)=s/2$ is at most one or at least two. 
  
  Let $a\in\ZZ$ be such that each square in $G_a$ contains at most $1$ point 
  of $C$ so there are no $s$-supported squares in $\bigcup_{n\leq a} G_n$.  It easily follows from the definition of $f$ that
  \begin{equation}\label{magic:sum:S:a}
    \sum_{S'\in G_a} \sum_{S\in G_{a+1}} f(S',S) = |C|,
  \end{equation}
  and that for every $b\in\ZZ$
  \begin{equation}\label{magic:sum:S:b}
    \sum_{S'\in G_b}\sum_{S\in G_{b+1}} f(S',S) \ge 0.
  \end{equation}

  Now, using \eqref{magic:sum:S:a} and \eqref{magic:sum:S:b},
  \begin{align*}
    \sum_{n=a+1}^b \sum_{S\in G_n}\sum_{S'\in G} f(S',S)
    &= \sum_{n=a+1}^b \sum_{S\in G_n} \left(  
      \sum_{S'\in G_{n-1}} f(S',S) + \sum_{S'\in G_{n+1}} f(S',S)
    \right)\\
    &= \sum_{S\in G_{a+1}}\sum_{S'\in G_a} f(S',S)
      +\sum_{S\in G_b}\sum_{S'\in G_{b+1}} f(S',S)
      \le |C|.
  \end{align*}
\noindent Therefore, using \eqref{magic:obs1} and \eqref{magic:obs2}, we deduce that there are at most $2|C|/s$ $s$-supported squares in $\bigcup_{n> a} G_n$. Sending $b\to \infty$ finishes the proof.
\end{proof}

The above claim is very close to the statement of \cref{lem:magic} which we are pursuing. However, we need to move from squares to circles. We use a technique called \defn{random padded partitions}.

We choose $k=\ceil{20\delta^{-2}}$ and let $\beta\sim\unif([0,\ln{k}])$.
Let $G_0$ be a tiling with side length $e^\beta$, based at the origin.
Suppose we have defined $G_n$ as a tiling of squares of side length $e^\beta k^n$; then $G_{n+1}$ is a tiling of squares of side length $e^\beta k^{n+1}$ that is based uniformly at one of the $k^2$ possible points of $G_n$. Because the desired statement is invariant under translation and dilation of $C$, we may assume that $C$ does not intersect the edges of $G_n$ (for every $n$) and that $\rho_w\ge k$ for every $w\in C$. We call a point $w\in C$ a \textbf{city} in a square $S\in G$ if:
\begin{itemize}
\item the side length of $S$ is in the interval 
$[4\delta^{-1}\rho_w,5\delta^{-1}\rho_w]$, and
\item the distance from $w$ to the center of $S$ is at most 
$\delta^{-1}\rho_w$.
\end{itemize}

\begin{claim}
  The probability that any given $w\in C$ is a city is
  $\Omega(\ln^{-1}(\delta^{-1}))$.
\end{claim}

\begin{proof}
  For the first item to hold, $\beta$ needs to satisfy that there 
  exists $n\in\ZZ$ such that
  $e^\beta k^n\in[4\delta^{-1}\rho_w,5\delta^{-1}\rho_w]$,
  or $\beta+n\ln{k} \in \ln(\delta^{-1}\rho_w) + [\ln4,\ln5]$. Since $\beta\in\unif([0,\ln{k}])$, the probability for that is 
  $(\ln(5/4))/\ln{k}$, which is $\Omega(\ln^{-1}(\delta^{-1}))$ when $\delta\in(0,1/2)$.

  As for the second item, it holds with positive probability (independent of
  $\delta$) over the $k^2$ basing choices for $G_n$, given that $\beta$ satisfies the 
  requirement posed by the first item.
\end{proof}

\begin{claim}
  If $w$ is a city in $S$ and is $(\delta,s)$-supported, then $S$ is 
  $s$-supported.
\end{claim}

\begin{proof}
  If $S\in G_n$, any little square $S' \in G_{n-1}$ has side length at most \begin{equation*}\frac{\delta^2}{20} \cdot \frac{5\rho_w}{\delta} = \frac{\delta\rho_w}{4}.
  \end{equation*} Hence, it is contained in a  ball of radius
  $\delta\rho_w$. Thus, for every $S' \in G_{n-1}$ with $S'\subseteq S$ there exists a point $p$ such that \begin{equation*}
  |C\cap (S\setminus S')|\ge \left|
  C\cap \left( B\left(w,\delta^{-1}\rho_w\right)\setminus B(p,\delta\rho_w)\right)
  \right| \geq s.\qedhere
  \end{equation*}
\end{proof}

Now note that the expected number of pairs $(w,S)$ such that $S$ is
$s$-supported, $w$ is $(\delta,s)$-supported, and $w$ is a city, is at least
$c\ln^{-1}(\delta^{-1})N$, where $N$ is the number of $(\delta,s)$-supported 
points.  Also, no more than $c\delta^{-2}$ points of $C$ can be cities in a 
single square $S$.  It follows from \cref{claim:mag:squares} that
\begin{equation*}
  N \le \frac{A|C|\delta^{-2}\ln(\delta^{-1})}{s},
\end{equation*}
concluding the proof of \cref{lem:magic}. \qed 

\section{Recurrence of bounded degree planar graph limits}\label{graphlimits:recurrence}

\cref{thm:BeScLimit} follows immediately from the following theorem which gives a quantitative estimate on the growth of the resistance in local limits of bounded degree planar maps. In particular, it states that the resistance grows logarithmically in the Euclidean distance of the corresponding circle packing. 

\begin{theorem} \label{theorem:loc:lim:recurrence}
  Let $(U, \rho)$ be a local limit of finite planar maps with maximum degree at most $D$. Then, almost surely, there exist a constant $c>0$ and a sequence $\{B_k\}_{k\geq 1}$ of subsets of $U$ such that for each $k$ we have
  \begin{enumerate}
    \item $\left|B_k\right| \leq c^{-1}k$, and
    \item $\reff(\rho \lr U\setminus B_k) \geq c \log k$.
  \end{enumerate}
  In particular, $(U,\rho)$ is almost surely recurrent.
\end{theorem}

We write $B_{\euc}(p,r)$ for the Euclidean ball of radius $r$ around a point $p\in \RR^2$. As before, for a subset $O \subset \RR^2$ and a given circle packing we write $V_O$ for the set of vertices in which the centers of the corresponding circles are in $O$. In order to prove \cref{theorem:loc:lim:recurrence}, we will need the following immediate corollary of the Magic Lemma (\cref{lem:magic}):

\begin{corollary}\label{cor:magic}
  Let $G$ be a finite simple planar triangulation, and $P$ a circle packing of $G$. Let $\rho$ be a uniform random vertex and $P'$ a dilation and translation of $P$ such that the circle of $\rho$ is a unit circle centered at the origin $\origin$. Then, there exists a universal constant $A>0$ such that in the packing $P'$, for every real $r\geq 2$ and integer $s \geq 2$ 
  \begin{equation*}
  \pr\left(\forall p \in \RR^2 \quad \left|V_{B_{\euc}(\origin,r)\setminus B_{\euc}(p, 
  \frac{1}{r})} \right| \geq s\right) \leq \frac{Ar^2 \log r}{s} \, .
  \end{equation*}
\end{corollary}
\begin{proof}
  Apply the Magic Lemma with $\delta = \frac{1}{r}$ and $s=s$, with the
  centers of circles of $P'$ as the point set $C$. Note that there exists a constant $C>0$ such that for all $w \in V$ the isolation radius of $w$,  $\rho_w$, satisfies $\rad(C_w) \leq \rho_w \leq C\rad(C_w)$ (without appealing to the Ring Lemma).
\end{proof}

The following lemma provides the main estimate needed to prove \cref{theorem:loc:lim:recurrence}. Once it has been shown, \cref{theorem:loc:lim:recurrence} will follow by a Borel-Cantelli argument.

\begin{lemma}\label{local:limit:subset:lemma}
  Let $G$ be a finite simple planar map with maximum degree at most $D$ and let $\rho$ 
  be a uniform random vertex of $G$. Then, there exists a constant $C=C(D)<\infty$ 
  such that for all $k\geq 1$,
  \begin{equation*}
  \pr\left(\exists B\subseteq V, \,\, \left|B\right| \leq C k, \,\, \reff(\rho \lr 
  V\setminus B) \geq C^{-1} \log k \right) \geq 1 - Ck^{-\frac{1}{3}}\log k.
  \end{equation*}
\end{lemma}
\begin{proof}
  We first assume that $G$ is a triangulation and consider a circle packing of it where the circle of $\rho$ is a unit circle centered at the origin $\origin$. Applying \cref{cor:magic} with $r = 
  k^{\frac{1}{3}}, s = k$, we have that with probability at least $1 - 
  Ak^{-\frac{1}{3}} \log(k)/3$, there exists $p\in\RR^2$ with 
  $$\left|V_{B_{\euc}(\origin,r)\setminus B_{\euc}(p,\frac{1}{r})} \right| < k.$$
  Now, if $|V_{B_{\euc}(p,\frac{1}{r})}| \leq 1$, we set $B = 
  V_{B_{\euc}(\origin,r)}$. We then have $\left|B_k\right| \leq k$ and by applying 
  \cref{hs:1:lem:1} we get that $\reff(\rho \lr V\setminus B) \geq C^{-1} \log(k)$.
  Else, if $|V_{B_{\euc}(p,\frac{1}{r})}| \geq 2$ then we take $B =  
  V_{B_{\euc}(\origin,r)} \setminus V_{B_{\euc}(p, \frac{1}{r})}$. By the Ring 
  Lemma, there exists a $c'=c'(D)>0$ such that $\left|p\right| \geq 1+c'$. 
  Since $|V_{B_{\euc}(p,\frac{1}{r})}| \geq 2$, we have a vertex in that set with radius at 
  most $r^{-1}$. Therefore, $B_{\euc}(p, \frac{2}{r})$ contains at least one 
  full circle $C_v$. Hence, by scaling and translating such that $C_v = 
  \UU$, we get (again, by \cref{hs:1:lem:1}) that
  \begin{equation*}
  \reff\left(V_{B_{\euc}(p,\frac{2}{r})} \lr V\setminus 
  V_{B_{\euc}(p,c'/2)}\right) \geq c_2 \log k \, ,
  \end{equation*}
  for some other constant $c_2=c_2(D)>0$. Since $\rho \in V\setminus 
  V_{B_{\euc}(p,c'/2)}$ we obtain 
  \begin{equation*}
  \reff\left(\rho \lr V_{B_{\euc}(p,\frac{2}{r})}\right) \geq c_2 \log k \, .
  \end{equation*} 
  By \cref{hs:1:lem:1} we also have that
  \begin{equation*}
  \reff\left(\rho \lr V\setminus V_{B_{\euc}(\origin,r)}  \right) \geq c_3 \log k \, ,
  \end{equation*}
  for some $c_3=c_3(D)>0$. By \cref{effresprob} this means that
  \begin{equation*}
  \pr_\rho\left(\tau_{V\setminus V_{B_{\euc}(\origin,r)}} < \tau_\rho^+\right) \leq 
  \frac{1}{c_2\log(k)} \quad \textrm{and} \quad \pr_\rho\left(\tau_{ V_{B_{\euc}(p,\frac{2}{r})}} < 
  \tau_\rho^+\right) \leq \frac{1}{c_3\log(k)} \, .
  \end{equation*}
  By the union bound
  \begin{equation*}
  \pr_\rho\left(\tau_{V\setminus B} < \tau_\rho^+\right) \leq 
  \frac{2}{\min(c_2, c_3)\log(k)} \, ,
  \end{equation*}
  hence by \cref{effresprob} again
  \begin{equation*}
  \reff\left(\rho \lr V\setminus B\right) \geq \min(c_2,c_3) D^{-1} \log(k) /2 \, ,
  \end{equation*} concluding the proof when $G$ is a triangulation.

  If $G$ is not a triangulation, we would like to add edges to make it a triangulation while making sure that the maximal degree does not increase too much. We also have to ensure that the graph remains simple which may require us to add some additional vertices as well. Let $f$ be a 
  face of $G$ with vertices $v_1,\ldots,v_{k}$. 
  Suppose first that there are no edges between non-consecutive vertices of 
  the face. In this case, we draw the edges in a zig-zag fashion, as in \cref{fig:zigzag:1}.
  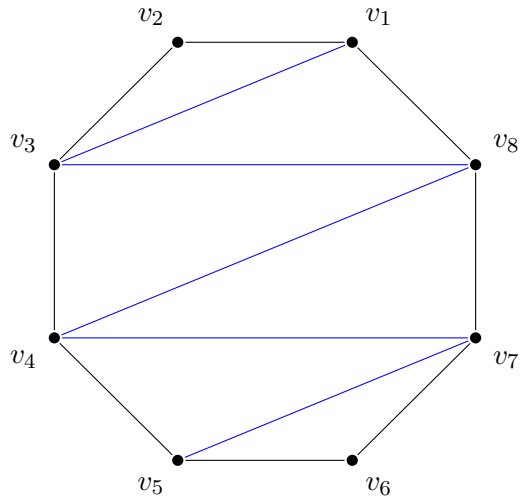
\begin{figure}[h]
    \centering
    \begin{tikzpicture}
    \foreach\i in {1,...,8} {
      \path let \n1={22.5+45*\i} in
        node (v\i) at (\n1:3) [vx,label={\n1:$v_\i$}] {};
    }
    \foreach\i in {1,...,8} {
      \draw let \n1={int(mod(\i,8)+1)} in (v\i) -- (v\n1);
    }
    \draw[color=blue] (v1) -- (v3) -- (v8) -- (v4) -- (v7) -- (v5);
    \end{tikzpicture}
    \caption{Adding diagonals to a face in a zigzag fashion}
    \label{fig:zigzag:1}
  \end{figure}
  
  In the case where edges between non-consecutive vertices of the face exist,
  we draw a cycle $u_1,\ldots,u_k$ inside $f$. Then, we connect $u_i$ to $v_i$ 
  and $v_{i+1}$ for each $i < k$ and $u_k$ to $v_k$ and $v_1$. Finally, we 
  triangulate the inner face created by the new cycle by zig-zagging as in 
  the previous case (see \cref{fig:zigzag:2}). 

  Since each vertex of the original graph is a member of at most $D$ faces and for each face at most $2$ edges are added, the maximal degree of the resulting graph is at most $3D$. Similarly, the number of vertices in the resulting graph is at most $D$ times the number of vertices in the original graph hence the probability of a random vertex being a vertex of the original graph is at least $D^{-1}$. If this occurs then it is 
  straightforward to see that the existence of a subset of vertices $B$ in the new graph which satisfies the required conditions implies the existence of such a set in the old graph, concluding our reduction to the triangulation case and finishing our proof. 
  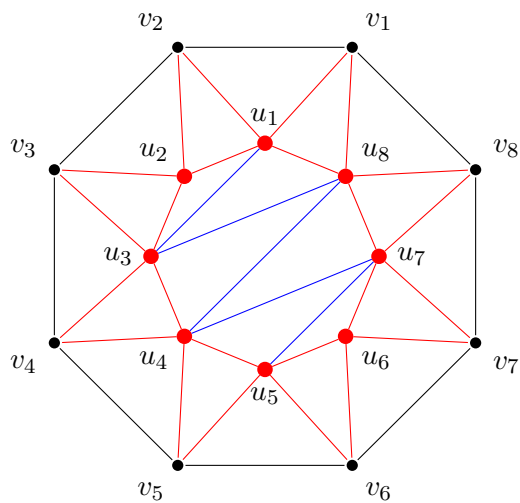
\begin{figure}[h]
    \centering
    \begin{tikzpicture}
    \foreach\i in {1,...,8} {
      \path let \n1={22.5+45*\i} in
        node (v\i) at (\n1:3) [vx,label={\n1:$v_\i$}] {};
      \path let \n1={45+45*\i} in
        node (u\i) at (\n1:1.5) [vx,red,label={\n1:$u_\i$}] {};
    }
    \foreach\i in {1,...,8} {
      \draw let \n1={int(mod(\i,8)+1)} in (v\i) -- (v\n1);
      \draw[red] let \n1={int(mod(\i,8)+1)} in (u\i) -- (u\n1);
      \draw[red] let \n1={int(mod(\i,8)+1)} in (u\i) -- (v\i);
      \draw[red] let \n1={int(mod(\i,8)+1)} in (u\i) -- (v\n1);
    }
    \draw[color=blue] (u1) -- (u3) -- (u8) -- (u4) -- (u7) -- (u5);
    \end{tikzpicture}
    \caption{Drawing an inner cycle and triangulating the new inner face.}    \label{fig:zigzag:2}
  \end{figure}
\end{proof}
We are ready to deduce \cref{theorem:loc:lim:recurrence}.
\begin{proof}[Proof of \cref{theorem:loc:lim:recurrence}] Assume that $G_n$ are finite planar maps with maximum degree at most $D$ such that $G_n \convl (U,\rho)$. If $\{G_n\}$ are not simple graphs we erase loops and merge parallel edges into a single edge to obtain the sequence $\{G_n'\}$. It is immediate that $G_n' \convl (U',\rho')$ where $(U',\rho')$ is distributed as $(U,\rho)$ after removing from $U$ all loops and merging parallel edges into a single edge. Since the maximum degree is bounded, $U'$ is recurrent if and only if $U$ is recurrent. Thus we may assume that $G_n$ are simple graphs so the previous estimates may be used.

  Denote by $\mathcal{A}_k$ the event \begin{equation*}
  \mathcal{A}_k = \left\{\exists B \subseteq U, \,\, \left|B\right| \leq Ck, \,\,
  \reff\left(\rho \lr V\setminus B \right)\ge c\log{k}\right\} \, ,
  \end{equation*}
  where $C=C(D)<\infty$ is the constant from \cref{local:limit:subset:lemma}. Therefore 
  $\pr(\mathcal{A}_k^c) \leq c^{-1}k^{-\frac{1}{3}}\log(k)$. Looking at the sequence 
  $\{\mathcal{A}_{2^j}\}_{j\geq 1}$, by Borel-Cantelli, almost surely there exists $j_0$ such that for all 
  $j\geq j_0$ the event $\mathcal{A}_{2^j}$ holds. Thus we have proved the required assertion for $k$ which is a power of $2$. To prove this for all $k$ sufficiently large, let $B_{2^j}$ be the set guaranteed to exist in the definition of $\mathcal{A}_{2^j}$, and take $B_k = B_{2^j}$ for the unique $j$ for which $2^j \leq k < 2^{j+1}$. It is immediate that these sets satisfy the assertion of the theorem, concluding our proof.
\end{proof}

\section{Exercises}

\begin{enumerate}

\item For a graph $G$, let $G^2$ be the graph on the same vertex set as $G$ so that vertices $u,v$ form an edge if and only if the graph distance in $G$ between $u$ and $v$ is at most $2$. Show that if $G$ has uniformly bounded degrees, then $G$ is recurrent if and only if $G^2$ is recurrent.

\item Construct an example of a local limit $(U, \rho)$ of finite planar graphs such that $U$ is almost surely recurrent, but $U^2$ is almost surely transient. 

\item Let $G(n,p)$ be the random graph on $n$ vertices drawn such that each of the ${n \choose 2}$ possible edges appears with probability $p$ independently of all other edges. Let $c>0$ be a constant, show that $G(n,c/n)$ converges locally to a branching process with progeny distribution Poisson$(c)$.

\item Fix an integer $k\geq 1$. Construct an example of a sequence of finite simple planar maps $G_n$ such that $G_n$ converge locally to $(U,\rho)$ with the property that $\E[\deg^k(\rho)] < \infty$ and $U$ is almost surely transient.

\item (*) Suppose that $G_n$ is a sequence of finite trees converging locally to $(U, \rho)$. Show that $U$ is almost surely recurrent. (Note that the degrees may be \emph{unbounded})






\end{enumerate}

\chapter{Recurrence of random planar maps} \label{chp:randommaps}

Our main goal in this chapter is to remove the bounded degrees assumption in \cref{thm:BeScLimit} and replace it with the assumption that the degree of the root has an exponential tail.

\begin{theorem}\label{thm:lim:rec} [\cite{GGN13}]
  Let $G_n$ be a sequence of (possibly random) planar graphs such that $G_n \convl (U,\rho)$ and there exist $C,c > 0$ such that $\pr(\deg(\rho) \geq k) \leq Ce^{-ck}$ for every $k$. Then $U$ is almost surely recurrent.
\end{theorem}

As discussed in \cref{sec:intro_prob}, the last theorem is immediately applicable in the setting of random planar maps. It is well known that the degree of the root in the UIPT and the UIPQ has an exponential tail. See \cite[Lemma 4.1 and 4.2]{AS03} or \cite{GaoRichmond} for the UIPT and \cite[Proposition 9]{BC13} for the UIPQ.

\begin{corollary} [\cite{GGN13}]
  The UIPT/UIPQ are almost surely recurrent.
\end{corollary}

\section{Star-tree transform} \label{sec:startree}
We present here a transformation which maps any planar map $G$ to a planar map $G^*$ with maximal 
degree of $3$. We call this transformation $G \mapsto G^*$  the {\bf star-tree transform}\index{star-tree transform}. Recall that a \defn{balanced rooted tree} is a finite rooted tree in which every non-leaf vertex has precisely two children and the distance of the leaves from the root differs by at most $1$. The transformation is performed as follows.
\begin{enumerate}
  \item Subdivide each edge $e$ by adding a new vertex $w_e$ of degree two in the ``middle''. See \cref{fig:startree:b}. Denote the resulting graph by $G'$.
  \item For every vertex $v\in V(G)$, replace all edges incident to $v$ in $G'$ by a   
    balanced binary tree rooted at $v$, whose leaves are the neighbors of $v$ 
    in $G'$. We perform this in a fashion which preserves the cyclic order of these neighbors and thus preserves planarity. Furthermore, add two extra vertices and attach them to the root. Denote this tree by $T_v$.  See \cref{fig:startree:d}. 


    
\end{enumerate}
\begin{figure}[ht]
  \centering
  \begin{subfigure}[t]{0.475\linewidth}
    \centering
    \begin{tikzpicture}[font=\small, scale=3]
    \node[lvxb,label={210:$u$}] (u) at (0,0) {};
    \node[lvxb,label={30:$v$}] (v) at (30:1) {};
    
    \draw (u)--(v);
  \end{tikzpicture}
    
  \caption{An original edge of $G$.}
  \end{subfigure}
  \begin{subfigure}[t]{0.475\linewidth}
    \centering
    \begin{tikzpicture}[font=\small, scale=3]
      \node[lvxb,label={210:$u$}] (u) at (0,0) {};
      \node[lvxw,label={120:$w$}] (w) at (30:0.5) {};
      \node[lvxb,label={30:$v$}]  (v) at (30:1) {};
    
      \draw (u)--(w)--(v);
    \end{tikzpicture}
    \caption{Subdividing an edge.}
    \label{fig:startree:b}
  \end{subfigure}
  \begin{subfigure}[b]{0.475\linewidth}
  \centering
  \begin{tikzpicture}[font=\small, scale=3]
    \node[lvxb,label={90:$v$}]  (v) at (0,0) {};
    \foreach\i in {1,...,6} {
      \path let \n1={-\i*60+60} in
        node[lvxw,label={\n1:$w_\i$}] (w\i) at (\n1:0.5) {};
      \draw (v) -- (w\i);
    }
  \end{tikzpicture}
  \caption{The ``star'' of a vertex in $G'$.}
  \end{subfigure}
  \begin{subfigure}[b]{0.475\linewidth}
    \centering
    \begin{tikzpicture}[font=\small, scale=1]
      \node[lvxb,label={90:$v$}] (v) at (0,0) {}; 
      \node[fill] (e1) at (-1,0) {};
      \node[fill] (e2) at (1,0) {};
      \node[lvxg] (g1) at (-1,1) {};
      \node[lvxg] (g2) at (1,1) {};
      \node[lvxg] (g11) at (-2,2) {};
      \node[lvxg] (g12) at (-1,2) {};

      \node[lvxw,label={90:$w_1$}] (w1) at (-3,3) {};
      \node[lvxw,label={90:$w_2$}] (w2) at (-2,3) {};
      \node[lvxw,label={90:$w_3$}] (w3) at (-1,3) {};
      \node[lvxw,label={90:$w_4$}] (w4) at (0,3) {};
      \node[lvxw,label={90:$w_5$}] (w5) at (1,2) {};
      \node[lvxw,label={90:$w_6$}] (w6) at (2,2) {};
    
      \draw (v) -- (g1) -- (g11) -- (w1);
      \draw (g1) -- (g11) -- (w2);
      \draw (g1) -- (g12) -- (w3);
      \draw (g1) -- (g12) -- (w4);
      \draw (v) -- (g2) -- (w5);
      \draw (v) -- (g2) -- (w6);
      \draw (v) -- (e1);
      \draw (v) -- (e2);
    \end{tikzpicture}
    \caption{Transforming the star of $v$ into a tree $T_v$.}
    \label{fig:startree:d}
  \end{subfigure}
  \caption{The star-tree transform}
  \label{fig:startree}
\end{figure}
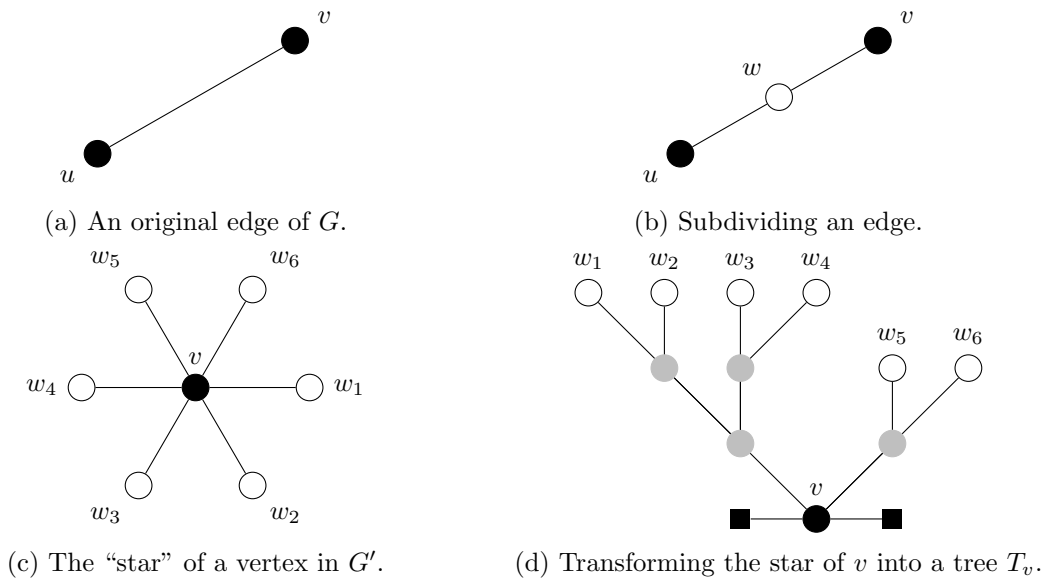

Denote the resulting graph by $G^*$. Note that each edge $e$ in $G^*$ corresponds to precisely one vertex $v$ of $G$ such that $e$ belongs to $T_v$. 

\begin{lemma}\label{lem:startree}
  Let $G$ be a planar map and $G^*$ its star-tree transform. We set 
  edge resistances on $G^*$ by putting $R_e = 1/d_G(v)$, where $v$ is the vertex of $G$ for which $e\in T_v$ and $d_G(v)$ is the degree of $v$ in $G$. If the network $(G^*,R_e)$ is 
  recurrent, then $G$ is recurrent as well. 
\end{lemma}

\begin{proof}

It is clear that from the point of view of recurrence versus transience, the two edges leading to the two ``extra'' neighbors of each root do not matter and can be removed. Hence for the rest of the proof we write $T_v$ for the previously defined tree with these two edges removed. The purpose of these extra edges will become apparent later in the proof of \cref{thm:lim:rec}. 

Assume $G$ is transient and let $a\in V(G)$ be some vertex. There is a flow $\theta$ from $a$ to $\infty$ such that $\energy(\theta) < \infty$. We will construct a flow $\theta^*$ on $(G^*,R_e)$ from $a$ to $\infty$ with finite energy, showing that $(G^*,R_e)$ is transient, giving the theorem.

We first provide some notation. We denote by $A$ the set of vertices that were added to form $G'$ in the first step of the star-tree transform (that is, the white vertices in \cref{fig:startree}). Each vertex $w\in A$ is a leaf of precisely two trees $T_u$ and $T_v$, where $\{u,v\}$ was the edge of $G$ that $w$ divided. We call $u$ and $v$ the {\bf tree roots} of $w$. We denote by $B$ the set of vertices that were added to $G^*$ in the second step of the star-tree transform, that is, the gray vertices in \cref{fig:startree:d}. The vertices of $V(G)$ are the black discs in \cref{fig:startree}. Each vertex of $x\in V(G)\cup B$ is a member of a single tree $T_v$; we call $v$ the {\bf tree root} of $x$. Lastly, for any $x\in V(G)\cup B$ we denote by $C_x \subset A$ the set of leaves of $T_v$, where $v$ is the tree root of $x$, for which the path from the leaf to the root of $T_v$ goes through $x$; in other words, $C_x$ is the set of leaves of $T_v$ which are the ``descendants'' of $x$. If $x\in A$, then we set $C_x=\{x\}$.


To define $\theta^*$, let $e=(x,y)$ be an edge of $T_v$. Assume that $x$ is closer to the root of $T_v$ than $y$ in graph distance. We set
$$\theta^*(e) = \sum_{w \in C_y} \theta(v, v_w)\, ,$$ where $v_w$ is the tree root of $w$ that is \emph{not} $v$. The construction of $\theta^*$ is depicted in \cref{fig:startreeflow}.
  
\begin{figure}[ht]
  \centering
  \begin{subfigure}[t]{0.475\linewidth}
    \centering
    \begin{tikzpicture}[font=\small, scale=3]
    \node[lvxb,label={210:$u$}] (u) at (0,0) {};
    \node[lvxb,label={30:$v$}] (v) at (30:1) {};
    
    \draw (u) -- (v) node [midway,above] {$\theta_1$};
  \end{tikzpicture}
  \caption{An original edge of $G$ which has flow $\theta_1$.}
  \end{subfigure}
  \begin{subfigure}[t]{0.475\linewidth}
    \centering
    \begin{tikzpicture}[font=\small, scale=3]
      \node[lvxb,label={210:$u$}] (u) at (0,0) {};
      \node[lvxw,label={-60:$w$}] (w) at (30:0.5) {};
      \node[lvxb,label={30:$v$}]  (v) at (30:1) {};
    
      \draw (u)--(w) node [midway,above] {$\theta_1$};
      \draw (w)--(v) node [midway,above] {$\theta_1$};
    \end{tikzpicture}
    \caption{The flow passes through the divided edge.}
    \label{fig:startree:flow:b}
  \end{subfigure}
  \begin{subfigure}[b]{0.475\linewidth}
  \centering
  \begin{tikzpicture}[font=\small, scale=3]
    \node[lvxb,label={90:$v$}]  (v) at (0,0) {};
    \foreach\i in {1,...,6} {
      \path let \n1={-\i*60+60} in
        node[lvxw,label={\n1:$w_\i$}] (w\i) at (\n1:0.5) {}
        node[] (t\i) at (\n1-15:0.25) {$\theta_\i$};
      \draw (v) -- (w\i);
    }
  \end{tikzpicture}
  \caption{The flow going out from a vertex of $G$ in $G'$.}
  \end{subfigure}
  \begin{subfigure}[b]{0.475\linewidth}
    \centering
    \begin{tikzpicture}[font=\small, scale=1]
      \node[lvxb,label={90:$v$}] (v) at (0,0) {}; 
      \node[lvxg] (g1) at (-1,1) {};
      \node[lvxg] (g2) at (1,1) {};
      \node[lvxg] (g11) at (-2,2) {};
      \node[lvxg] (g12) at (-1,2) {};

      \node[lvxw,label={90:$w_1$}] (w1) at (-3,3) {};
      \node[lvxw,label={90:$w_2$}] (w2) at (-2,3) {};
      \node[lvxw,label={90:$w_3$}] (w3) at (-1,3) {};
      \node[lvxw,label={90:$w_4$}] (w4) at (0,3) {};
      \node[lvxw,label={90:$w_5$}] (w5) at (1,2) {};
      \node[lvxw,label={90:$w_6$}] (w6) at (2,2) {};
      
      \node (t1) at (-2.7,2.5) {$\theta_1$};
      \node (t2) at (-1.8,2.5) {$\theta_2$};
      \node (t3) at (-1.2,2.5) {$\theta_3$};
      \node (t4) at (-0.2,2.5) {$\theta_4$};
      \node (t12) at (-2.2,1.5) {$\theta_1+\theta_2$};
      \node (t34) at (-0.3,1.5) {$\theta_3+\theta_4$};
      \node (t5) at (0.8,1.5) {$\theta_5$};
      \node (t6) at (1.8,1.5) {$\theta_6$};
      \node (t56) at (1.3,0.5) {$\theta_5+\theta_6$};
      \node (t1234) at (-2,0.5) {$\theta_1+\theta_2+\theta_3+\theta_4$};

      \draw (v) -- (g1) -- (g11) -- (w1);
      \draw (g1) -- (g11) -- (w2);
      \draw (g1) -- (g12) -- (w3);
      \draw (g1) -- (g12) -- (w4);
      \draw (v) -- (g2) -- (w5);
      \draw (v) -- (g2) -- (w6);
    \end{tikzpicture}
    \caption{The division of the flow in $T_v$.}
    \label{fig:startree:flow:d}
  \end{subfigure}
  \caption{The construction of the flow $\theta^*$ from $\theta$.}
  \label{fig:startreeflow}
\end{figure}
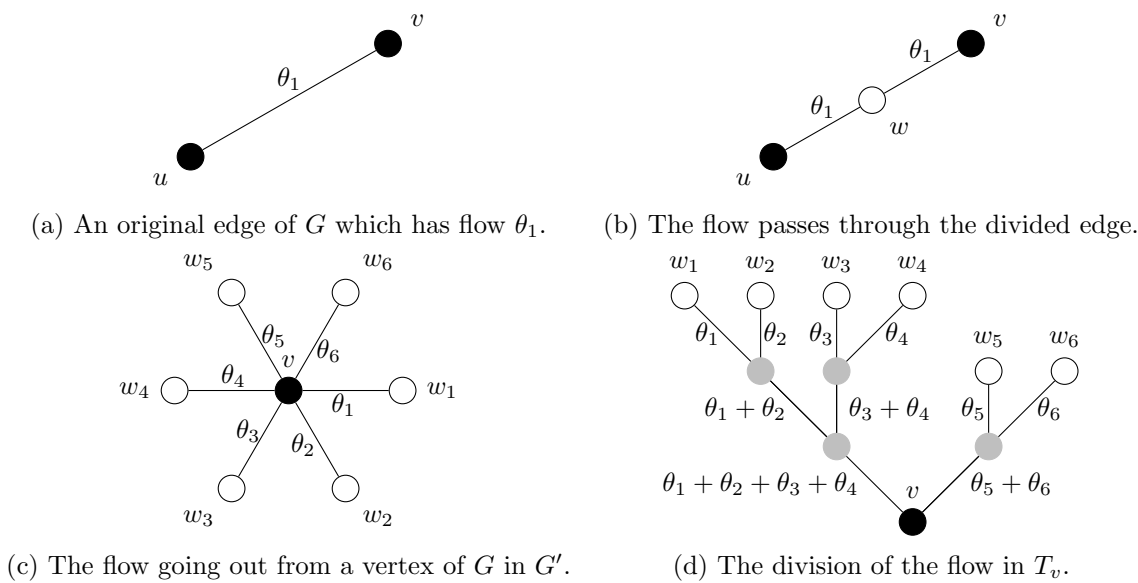

  We will now show that $\energy(\theta^*) \leq 4 \energy(\theta)$ where the energy of $\theta^*$ is taken in the network $(G^*, R_e)$. Let $v\in V(G)$ and write $h$ for the height of $T_v$, that is, $h$ is the maximal graph distance from a leaf of $T_v$ to its root. Note that since the tree is balanced, the distances from the leaves to the root vary by at most $1$. Let $e = (x,y)$ be an edge of $T_v$ and assume that $x$ is closer than $y$ to the root of $T_v$ and that the graph distance of $y$ from the root is $\ell \in \{1,\ldots, h\}$. By the construction of $\theta^*$, the contribution of $e$ to $\energy(\theta^*)$ is 
  $$R_e\theta^*(e)^2 = \frac{1}{d_G(v)} \left(\sum_{w\in 
  C_y}\theta(v,v_w)\right)^2.$$
  Since $y$ is at distance $\ell$ from the root, $|C_y| \leq 2^{h-\ell}$. Hence by Cauchy-Schwarz
  $$R_e\theta^*(e)^2 \leq \frac{2^{h-\ell}}{d_G(v)}\sum_{w\in 
  C_y}\theta(v,v_w)^2 \, .$$
  Summing over all edges in $T_v$ at distance $\ell$ from the root, we go over each leaf of $T_v$ precisely once. Thus,
  $$ \sum_{\substack{e=(x,y)\in T_v \\ d_{G^*}(y,v) = \ell}} R_e\theta^*(e)^2 \leq 
  \frac{2^{h-\ell}}{d_G(v)} \sum_{w\in C_v}\theta(v,v_w)^2 \, .$$
  We now sum over all edges in $T_v$ by summing over $\ell \in \{1,\ldots,h\}$. We get
  $$ \sum_{e\in T_v} R_e\theta^*(e)^2 \leq \frac{2^h}{d_G(v)}\sum_{w\in 
  C_v}\theta(v,v_w)^2 \leq 2 \sum_{w\in C_v}\theta(v,v_w)^2 \, ,$$
 since $h \leq \log_2(d_G(v)) + 1$. Lastly, we sum this over all $v\in V(G)$ and note that each term of the form $\theta(v,v_w)^2$ in the last sum appears twice. Hence, 
  $$ \energy(\theta^*) \leq 4\energy(\theta) \, ,$$
concluding our proof.
\end{proof}

\section{Stationary random graphs and markings}

\subsection{Stationary random graphs}

Recall that \cref{thm:lim:rec} and the entire setup of \cref{chp:locallimit} is adapted to the case when $G_n$ is itself random. The reason is that in \cref{def:locconv} we consider the graph distance ball $B_{G_n}(\rho_n,r)$ as a random variable in the probability space $(\Gdot,\dloc)$, where $\rho_n$ conditioned on $G_n$ is a uniformly chosen random vertex. 

Let us emphasize that this is {\bf not} the same as drawing a sample of $\{G_n\}$ and claiming that almost surely $G_n\convl(U,\rho)$. For example, let $G_n$ be a path of length $n$ with probability $1/2$ and an $n \times n$ square grid with probability $1/2$, independently for all $n$. In this case $G_n \convl (U,\rho)$ where $U=\ZZ$ with probability $1/2$ and $U=\ZZ^2$ with probability $1/2$, however, almost surely on the sequence $\{G_n\}$, the local limit of $G_n$ does not exist.

In many cases it is useful to take a random root drawn from the stationary distribution on $G_n$, that is, the probability distribution on vertices giving each vertex $v$ probability $\deg_{G_n}(v)/2|E(G_n)|$. In a similar fashion to \cref{def:locconv}, we define this type of local convergence. 

\begin{definition} \label{def:locconvpi}
  Let $\{G_n\}$ be a sequence of (possibly random) finite graphs.  We say that $G_n \convlpi (U,\rho)$ where $(U,\rho)$ is a random rooted graph, if for every integer $r\ge 1$,
  \begin{equation*}
    B_{G_n}(\rho_n,r) \convd B_U(\rho,r),
  \end{equation*}
  where $\rho_n$ is a randomly chosen vertex from $G_n$ distributed according to the stationary distribution on $G_n$. We call such a limit a \defn{stationary local limit}. 
\end{definition}

Let us remark that $G_n \convl (U,\rho)$ does not imply that $G_n \convlpi (U',\rho')$ for some $(U',\rho')$. Indeed, let $G_n$ be a path of length $n$ attached to a complete graph on $\sqrt{n}$ vertices. Then the local limit of $G_n$ is $\ZZ$, however the limit according to a stationary random root does not exist. 




The reason for taking the $\convlpi$ limit rather than the uniform limit as before is that the random walk on the limit $(U, \rho)$ starting from $\rho$ is then stationary.

\begin{claim}
  Assume that $G_n \convlpi (U,\rho)$. Conditioned on $(U,\rho)$, let $X_1$ be a uniformly chosen neighbor of $\rho$. Then $(U,X_1)$ is equal in law to $(U,\rho)$. Similarly, if $\{X_n\}_{n\geq 0}$ is the simple random walk on $(U,\rho)$, then for each $n \geq 0$ the law of $(U,X_n)$ coincides with the law of $(U,\rho)$.
\end{claim}
\begin{proof} If $H$ is a finite graph and $v$ is a vertex chosen from the stationary distribution, then it is immediate that a uniformly chosen random neighbor of $v$ is distributed according to the stationary distribution. Thus for any fixed $r>0$ we have that $B_{G_n}(\rho_n,r)$ has the same distribution as $B_{G_n}(X_1,r)$ where $\rho_n$ is drawn from the stationary distribution on $G_n$ and $X_1$ is a uniform neighbor of $\rho_n$. The claim follows now by definition.
\end{proof}

\begin{definition} \label{def:stationaryrg}
A random rooted graph $(G,\rho)$ is called a \defn{stationary random graph} if $(G,X_1)$ has the same distribution as $(G,\rho)$, where the vertex $X_1$ is a uniform neighbor of $\rho$ (conditioned on $(G,\rho)$). 
\end{definition}

We would like to develop a simple abstract framework that will allow us to comfortably move from $\convl$ convergence to $\convlpi$ convergence and vice versa. This is straightforward when $\{G_n\}$ are a sequence of \emph{deterministic} graphs with uniformly bounded average degree but is less obvious when $G_n$ themselves are random. For this we need to \emph{degree bias} our random graphs.

\begin{definition}\label{def:degbias} \index{degree biasing and unbiasing} Denote by $\pr$ the law of a random rooted graph $(G,\rho)$ and assume that $\E \deg(\rho) < \infty$. The probability measure $\mu$ on $(\Gdot, \dloc)$ defined by
$$ \mu(\A) := {1 \over \E \deg(\rho)} \sum_{k \geq 1} k \,\,\pr(\A \cap \deg(\rho) = k) \, ,$$
for any event $\A \subset (\Gdot, \dloc)$ is called the {\bf degree biasing} of $\pr$. 
Similarly, the probability measure $\nu$ defined by 
$$ \nu(\A) = {1 \over \E[\deg(\rho)^{-1}]} \sum_{k \geq 1} {\pr(\A \cap \deg(\rho) = k) \over k} \, ,$$
is called the {\bf degree unbiasing} of $\pr$. Note that to define $\nu$ we do not need to require that $\E \deg(\rho) < \infty$.
\end{definition} 

\begin{lemma} Assume that $(G,\rho)$ is a random rooted graph such that $G$ is almost surely finite, that the distribution of $\rho$ given $G$ is uniform and that $\E \deg(\rho) < \infty$. Then the degree biasing of $(G,\rho)$ is a stationary random graph.

Conversely, assume that $(G^\pi,\rho^\pi)$ is a stationary random graph such that $G^\pi$ is almost surely finite. Then the degree unbiasing of it $(G,\rho)$ is such that $G$ is almost surely finite and $\rho$ condition on $G$ is uniformly distributed.
\end{lemma}
\begin{proof} We will prove only the first statement and the second is similar. Denote by $(G^\pi,\rho^\pi)$ a random variable drawn according to the degree biasing of $(G,\rho)$. Let $H$ be a fixed finite graph and denote by $\deg_H(v)$ the degree of a vertex $v$ in $H$. By definition we have that
\begin{equation}\label{eq:stationarymidstep} \pr ( (G^\pi,\rho^\pi) = (H,v) ) = {\deg_H(v) \cdot \pr( (G,\rho)=(H,v) ) \over \E \deg(\rho)} \, .\end{equation}
Let $X_1$ be a uniformly chosen neighbor of $\rho^\pi$. Then by \eqref{eq:stationarymidstep}
$$ \pr ( (G^\pi, X_1) = (H,u)) = \sum _{v \, : \, \{u,v\}\in E(H)} {\pr ( (G^\pi,\rho^\pi) = (H,v) ) \over \deg_H(v)} = {\sum _{v \, : \, \{u,v\}\in E(H)} \pr ( (G,\rho) = (H,v) ) \over \E \deg(\rho)} \, .$$
Since $\rho$ is uniformly distributed given $G$ the quantity $\pr ( (G,\rho) = (H,v) )$ is the same for all $v$. So
$$\pr ( (G^\pi, X_1) = (H,u)) = {\deg_H(u) \pr ( (G,\rho) = (H,u) ) \over \E \deg(\rho)}$$
so by \eqref{eq:stationarymidstep} the required assertion follows. \end{proof}

\begin{corollary}\label{cor:switch} Assume that $\{G_n\}$ is a sequence of random finite graphs such that $G_n \convl (U,\rho)$ and denote by $\rho_n$ a uniformly chosen vertex of $G_n$ and by $(G_n^\pi,\rho_n^\pi)$ the degree biasing of $(G_n,\rho_n)$. Assume further that $\E \deg (\rho) < \infty$ and that $\E \deg (\rho_n) \to \E \deg (\rho)$. 

Then $G_n^\pi \convlpi (U^\pi,\rho^\pi)$ where $(U^\pi,\rho^\pi)$ is the degree biasing of $(U,\rho)$. Furthermore, $(U,\rho)$ and $(U^\pi,\rho^\pi)$ are absolutely continuous with respect to each other.

\noindent Conversely, assume that $\{G_n^\pi\}$ is a sequence of random finite graphs such that $G_n^\pi \convlpi (U^\pi,\rho^\pi)$, denote by $\rho_n^\pi$ a random vertex of $G_n$ drawn according to the stationary distribution and by $(G_n,\rho_n)$ the degree unbiasing of $(G_n^\pi,\rho_n^\pi)$.

Then $G_n \convl (U,\rho)$ where $(U,\rho)$ is the degree unbiasing of $(U^\pi,\rho^\pi)$. Furthermore, $(U,\rho)$ and $(U^\pi,\rho^\pi)$ are absolutely continuous with respect to each other.
\end{corollary}
\begin{proof} Indeed, let $(H,v)$ be a finite rooted graph and $r>0$ a fixed integer. Then 
$$ \pr ( B_{G_n^\pi}(\rho, r) =  (H,v) ) = { \deg_H(v) \pr (B_{G_n}(\rho_n,r) = (H,v)) \over \E \deg (\rho_n)} \, .$$  
Since $G_n \convl (U,\rho)$ and $\E \deg (\rho_n) \to \E \deg (\rho)$ we obtain that 
$$ \lim_{n \to \infty} \pr ( B_{G_n^\pi}(\rho, r) =  (H,v) ) = {\deg_H(v) \pr (B_{U}(\rho,r) = (H,v)) \over \E \deg (\rho)} = \pr ( B_{U^\pi}(\rho^\pi,r) = (H,v)) \, ,$$
where the last equality is by definition. The absolute continuity of $(U,\rho)$ and $(U^\pi, \rho^\pi)$ follows immediately from the definition. 

The second statement follows by the same proof. Note that we by the dominated convergence theorem we have that $\E [\deg(\rho_n^\pi)^{-1}] \to \E [ \deg(\rho^\pi)^{-1}]$. 
\end{proof}

We end this subsection by addressing the somewhat technical issue of verifying the condition $\E \deg(\rho_n) \to \E\deg(\rho)$ in \cref{cor:switch}. It is not guaranteed guaranteed just by requiring $\sup \E \deg (\rho_n) <\infty$ as we see in the example of a path of length $n$ together with $\sqrt{n}$ loops attached to $\sqrt{n}$ arbitrary vertices of the path; in this example $\deg (\rho) =2$ almost surely, and $\E \deg (\rho_n) = 3 + o(1)$. However, we now show that it is always possible to ``truncate'' the finite graphs $G_n$ by removing edges touching vertices of large degrees so that the limit is unchanged and the average degrees converge to the expected degree of the limit. Given a finite graph $G$ and an integer $k\geq 1$ we denote by $G \wedge k$ the graph obtained from $G$ by erasing all the edges touching vertices of degree at least $k$.

\begin{lemma}\label{lem:truncate} Let $\{G_n\}$ be a sequence of random finite graphs such that $G_n \convl (U,\rho)$ and $\E \deg(\rho) < \infty$. Then there exists a sequence $k(n)\to\infty$ such that 
$$ G_n \wedge k(n) \convl (U,\rho) \, ,$$
and 
$$ \E \deg(\rho_n) \to \E \deg(\rho) \, ,$$
where $\rho_n$ is a uniformly chosen vertex of $G_n$.
\end{lemma}
\begin{proof} We first show that for \emph{any} sequence $k(n) \to \infty$ we have that $G_n \wedge k(n) \convl (U,\rho)$. Indeed, since $G_n \convl (U,\rho)$ we have that for any fixed integer $r \geq 1$
$$ \pr \Big ( \max \big \{ \deg (v) : v \in B_{G_n}(\rho_n, r) \big \} \geq k(n) \Big ) \to 0 \, .$$
If $\max \{ \deg (v) : v \in B_{G_n}(\rho_n, r+1) \} < k(n)$, then $B_{G_n}(\rho_n,r) = B_{G_n\wedge k(n)}(\rho_n,r)$. Since $G_n$ and $G_n \wedge k(n)$ have the same set of vertices we deduce that for any fixed $r\geq 1$ and any rooted graph $(H,v)$  
$$ \pr \big ( B_{G_n\wedge k(n)}(\rho_n,r) =(H,v)) \to \pr( B_U(\rho,r) = (H,v)) \, .$$

Secondly, since $\deg(\rho_n)$ converges in distribution to $\deg(\rho)$ we have that there exists a sequence $k(n)\to \infty$ such that $\E \deg(\rho_n)\wedge k(n) \to \E \deg (\rho)$. Indeed, by dominated convergence we have that $\displaystyle \E [\deg(\rho) \wedge k] \to_{k \to \infty} \E \deg(\rho)$. Furthermore, by bounded convergence for any fixed $k$ we have $\E \deg(\rho_n) \wedge k \to_{n \to \infty} \E [\deg(\rho) \wedge k]$. Hence for any $\eps>0$ there exists $k$ and $n_0$ such that for all $n \geq n_0$ we have that $|\E[\deg(\rho_n) \wedge k] - \E \deg(\rho)| \leq \eps$.
\end{proof}

\subsection{Markings} \index{Markings}

Given a locally convergent sequence of (possibly random) graphs $G_n$, we wish to apply the star-tree transform on them to create a sequence $G_n^*$ and take its local limit of that while ``remembering'', in light of \cref{lem:startree}, the original degrees of $G_n$. The approach is a rather straightforward extension of the abstract setting of \cref{sec:locconv}, see also \cite{AldousLyons}. We consider the space of triples $(G,\rho, M)$ where $G=(V,E)$ is a graph, $\rho\in V$ is a vertex and $M:E\to \RR$ is a function assigning real values to the edges. We endow the space with a metric by setting the distance between $(G_1,\rho_1,M_1)$ and $(G_2,\rho_2,M_2)$ to be $2^{-R}$ where $R$ is the maximal value such that there exists a rooted graph isomorphism $\varphi$ between $B_{G_1}(\rho_1,R)$ and $B_{G_2}(\rho_2,R)$ such that $|M_1(e) - M_2(\varphi(e))|\leq R^{-1}$ for all edges $e\in E(G)$ both of whose end points are in $B_{G_1}(\rho_1,R)$. It is easy to check that this space is again a Polish space, so again we may define convergence in distribution of random variables taking values in this space. 

We say that such a random triplet $(U,\rho,M)$ is {\bf stationary} if conditioned on $(U,\rho,M)$ a uniformly chosen random neighbor $X_1$ of $\rho$ satisfies that $(U,\rho,M)$ has the same law as $(U,X_1,M)$ in the space of isomorphism classes of rooted graphs with markings (that is, rooted isomorphisms that preserve the markings). Given a marking $M$ we extend it to $M:E(U)\cup V(U)\to \RR$ by setting $M(v) = \max_{e: v\in e} M(e)$ for any $v\in V(U)$. We say that $(U,\rho,M)$ has an {\bf exponential tail} if for some $A<\infty$ and $\beta>0$ we have that $\pr(M(\rho)\geq s)\leq A e^{-\beta s}$ for all $s\geq 0$. 

In the following lemma we consider a stationary triplet $(U,\rho,M)$ that has an exponential tail and compare the hitting probabilities of certain sets when we endow the graphs with two sets of edge resistances: the first are the usual unit resistances, and in the second we may change the edge resistances arbitrarily but only on edges with high $M$ values. We tailored the lemma this way in order to show that $(G^*,R_e)$ from \cref{lem:startree} is recurrent.

\begin{lemma}\label{lem:comparerw} Let $(U,\rho,M)$ be a stationary, bounded degree rooted random graph with markings which has an exponential tail. Conditioned on $(U, \rho, M)$, fix some finite set $B\subset U$. Let $\PP_\rho$ denote the unit-resistance random walk on $U$ starting from $\rho$ and let $\PP'_{\rho}$ denote the random walk on $U$ with edge resistances $R'_e$ satisyfing that $R'_e=1$ whenever $M(e)\leq 21 \beta^{-1} \log|B|$. Then almost surely on $(U,\rho,M)$ there exists $K<\infty$ such that for any finite subset $B\subset U$ with $|B|\geq K$ we have
$$ \big | \PP_\rho(\tau_{U\setminus B} < \tau^+_\rho) - \PP'_\rho(\tau_{U \setminus B} < \tau^+_\rho) \big | \leq {1 \over |B|} \, .$$
\end{lemma}
\begin{proof}
For every integers $T,s\ge 1$ we set
  \begin{equation*}
    \mathcal{A}_{T,s} = \left\{
      \PP_\rho(\exists t< T :  M(X_t)\ge s) \le T^3 e^{-\beta s/2}
    \right\} \, .
  \end{equation*}
Since $(U,\rho,M)$ is stationary and has an exponential tail for any $t\geq 0$ we have
  \begin{equation*}
    \E\big [\PP_\rho(M(X_t)\ge s) \big ] \le A e^{-\beta s} \, ,
  \end{equation*}
hence by the union bound
  \begin{equation*}
    \E\big [\PP_\rho(\exists t < T : M(X_t)\ge s) \big ] \le ATe^{-\beta s} \, .
  \end{equation*}
Thus by Markov's inequality
  \begin{equation*}
    \pr\left(\mathcal{A}_{T,s}^c\right)
    \le A T^{-2} e^{-\beta s/2} \, .
  \end{equation*}
  By Borel-Cantelli we deduce that almost surely $\mathcal{A}_{T,s}$ occurs for all but
  finitely many pairs $T,s$. Conditioned on $(U,\rho,M)$, we may consider only finite subsets $B\subset U$ which contain $\rho$, since otherwise both probabilities in the statement of the lemma are $1$. Let $B$ be such a subset. By the commute time identity \cref{commute}, and since the maximum degree of $U$ is bounded,
  \begin{equation*}
    \EE_\rho(\tau_{U \setminus B}) \le C \reff(\rho \lr U\setminus B) |B|
    \le C |B|^2 \, ,
  \end{equation*}
for some constant $C>0$. The last inequality is since the resistance is bounded by $|B|$ since there is a path of length at most $|B|$ from $\rho$ to $U \setminus B$. By Markov's inequality,
  \begin{equation*}
    \PP_\rho(\tau_{U \setminus B} \ge T) \leq {C|B|^2 \over T} \, .
  \end{equation*}
  Write $S=\{v\in U :M(v)\ge s\}$. For every $T,s$ for which 
  $\mathcal{A}_{T,s}$ occurs we have
  \begin{equation*}
    \PP_\rho\left(\tau_S < \tau^+_{\{\rho\}\cup U\setminus B}\right) \le 
    \PP_\rho(\tau_{U\setminus B}\ge T) + \PP_\rho(\exists t < T:M(X_t)\ge s)
    \le {C|B|^2 \over T} + T^3 e^{-\beta s/2}.
  \end{equation*}
  We now choose $T=2C|B|^3$ and $s=21\beta^{-1}\log |B|$ so that the right hand side of the last inequality is at most $|B|^{-1}$ when $|B|$ is sufficiently large. It is clear that we can couple two random walks starting from $\rho$, one walking on $U$ with unit resistances and the other on $(U,R_e)$, so that they remain together until they visit a vertex of $S$. Hence, when $|B|$ is large enough so that the chosen $T,s$ are such that $\mathcal{A}_{T,s}$ holds we deduce from the last inequality that with probability at least $1-|B|^{-1}$ the simple random walk on $U$ visits $\{\rho\} \cup U \setminus B$ before visiting $S$, concluding our proof.
\end{proof}

\section{Proof of Theorem 6.1}


We now proceed to wrapping up the proof of \cref{thm:lim:rec}. Recall that we have a sequence of finite planar graphs $\{G_n\}$ such that $G_n \convl (U,\rho)$ such that $\pr ( \deg(\rho) \geq k) \leq Ce^{-ck}$. Our goal is to prove that $(U,\rho)$ is almost surely recurrent.


By Lemma \ref{lem:truncate} and \cref{cor:switch} we may truncate and degree bias $G_n$ and $(U,\rho)$ so that we may assume without loss of generality that $G_n \convlpi (U,\rho)$. It is an easy computation using \cref{def:degbias} that we still have $\pr ( \deg(\rho) \geq k) \leq Ce^{-ck}$ (possibly for some other positive constants $C,c$). Thus, from now on we assume this that $G_n \convlpi (U,\rho)$ and that $\deg(\rho)$ has exponential tails.


Recall now the definitions and notations of \cref{sec:startree}. 
 Consider the star-tree transform $G_n^*$ of $G_n$ and let $\rho_n^*$ be a random vertex of $T_{\rho_n}$ drawn according to the stationary distribution of $T_{\rho_n}$. Similarly, conditioned on $(U,\rho)$, let $U^*$ be the star-tree transform of $U$ and $\rho^*$ be a random vertex of $T_\rho$ drawn according to the stationary distribution of $T_{\rho}$. Furthermore, we put markings on $G_n^*$ and $U^*$ by marking each edge $e$ of $G_n^*$ or $U^*$ with $\deg(v)$ whenever $e$ is in the tree $T_v$ and $\deg(v)$ is the degree of $v$ in $G_n$ or $U$, respectively. Denote these markings by $M_n$ and $M$, respectively.


\begin{claim} We have that $(G_n^*,\rho_n^*, M_n)$ for each $n$ and $(U^*,\rho^*, M)$ are stationary, and,
  $$(G_n^*, \rho_n^*, M_n) \convd (U^*, \rho^*, M) \, .$$
\end{claim}
\begin{proof} Since for any fixed integer $r>0$, the laws of $B_{G_n^*}(\rho_n^*,r)$ and $B_{U^*}(\rho^*,r)$ are determined by $B_{G_n}(\rho_n,r)$ and $B_{U}(\rho,r)$, respectively, we obtain that 
$$ (G_n^*,\rho_n^*, M_n) \convd (U^*,\rho^*, M) \, .$$

Secondly, it is immediate to check that for each $v\in G_n$ we have that the number of edges in $T_v$ is precisely $2\deg_{G_n}(v)$. This is the reason why we added the two ``extra'' neighbors to the root of $T_v$ in the star tree transform described in \cref{sec:startree}. Thus, conditioned on $G_n$ for any $x\in G_n^*$ such that $x\in T_v$ for some $v\in G_n$ we have that
$$ \pr ( \rho_n^*=x \mid G_n) = {\deg_{G_n}(v) \over 2|E(G_n)|} \cdot {\deg_{T_v}(x) \over 2 |E(T_v)|} = {\deg_{T_v}(x) \over 2|E(G_n^*)|} \, ,$$
or in other words, $(G_n^*,\rho_n^*, M_n)$ is a stationary random graph and since it converges to $(U^*,\rho^*, M)$, the latter is also stationary.
%
\end{proof}

\begin{lemma}\label{lem:exptail} The triplet $(U^*, \rho^*, M)$ has an exponential tail.
\end{lemma}
\begin{proof}
We observe that $M(\rho^*)=\deg(v)$ where $v$ is either $\rho$ or one of its neighbors (the latter can happen if $\rho^*$ was chosen to be a leaf of $T_\rho$). Hence it suffices to show that if $(U,\rho)$ is a stationary local limit such that $\deg(\rho)$ has an exponential tail, then the random variable $D(\rho)=\max_{v:\{\rho,v\}\in E(U)} \deg(v)$ has an exponential tail. We have
\begin{equation}\label{eq:exptailsplit} \pr(D(\rho)\geq k) \leq \pr(\deg(\rho)\geq k) + \pr(\deg(\rho)\leq k \text{ and } D(\rho)\geq k) \, .\end{equation}
The probability of the first term on the right hand side decays exponentially in $k$ due to our assumption on $(U,\rho)$. Conditioned on $(U,\rho)$, let $X_1$ be a uniformly chosen random neighbor of $\rho$. Then clearly
$$ \pr(\deg(X_1)\geq k \, \mid \, \deg(\rho)\leq k \text{ and } D(\rho)\geq k ) \geq k^{-1} \, .$$
However, by stationarity $\pr(\deg(X_1) \geq k) = \pr(\deg(\rho) \geq k)$, which decays exponentially. We conclude that the second term on the right hand side of \eqref{eq:exptailsplit} decays exponentially as well.
\end{proof}

\newcommand{\Runit}{R^{\mathrm{unit}}}
\newcommand{\RM}{R^{\mathrm{mark}}}
\newcommand{\Rmid}{R^{\mathrm{mid}}}
\newcommand{\PPM}{\PP^{\mathrm{mark}}}
\newcommand{\PPmid}{\PP^{\mathrm{mid}}}

Consider the stationary random graph $(U^*,\rho^{*}, M)$. By \cref{lem:exptail} it has an exponential tail.  Consider the edge resistances 
\begin{equation*}
  \Runit_e \equiv 1 \, ,\qquad
  \RM_e = {1 \over M(e)} \, .
\end{equation*}
In view of \cref{lem:startree}, it suffices to show that the network $(U^*, \RM)$ is almost surely recurrent, for then it will follow that $U$ is almost surely recurrent. To prove the former, we apply the second assertion of \cref{cor:switch} which allows us to assume without loss of generality that $(U^*,\rho^*)$ is a local limit of finite planar maps (rather than a stationary local limit). Since $(U^*,\rho^*)$ is now a local limit of finite planar maps with degrees bounded by $3$ we may apply
\cref{theorem:loc:lim:recurrence} to almost surely obtain a constant $c>0$ and a sequence of sets $B_k \subset U^*$ such that 
\begin{enumerate}
  \item $ck \leq |B_k| \leq c^{-1}k$, and
  \item $\reff(\rho^{*} \lr U^* \setminus B_k \,\, ; \,\, \{\Runit_e\}) \geq c \log k$,
\end{enumerate}
where we added to the conclusion of \cref{theorem:loc:lim:recurrence} that $B_k \geq ck$ since adding vertices to $B_k$ makes the lower bound on the resistance even better.

We now define one extra set of edge resistances on $U^*$ which will allow us to interpolate between the edge resistances $\Runit$ and $\RM$. For each integer $k \geq 1$ we define
\begin{equation*}
  \Rmid_e =
  \begin{cases}
    1 & M(e) \le C\log k,\\
    M^{-1}(e) & \text{otherwise} \, ,
  \end{cases}
\end{equation*}
where $C>0$ is some large constant that will be chosen later. We will use $\PP$, $\PPM$ and $\PPmid$ to denote the probability measures, conditioned on $(U^*,\rho^{*},M)$, of random walks on $U^*$ with edge resistances $\{\Runit_e\}$, $\{\RM_e\}$ and $\{\Rmid_e\}$, respectively.

\begin{lemma}\label{lem:rmidres} For some other constant $c>0$ we have
  $$ \reff(\rho^{*}\lr U^*\setminus B_k \,\, ; \{\Rmid_e\})
  \geq c\log k \, .$$
\end{lemma}
\begin{proof}
We may assume $k$ is large enough so that $M(e)\le C\log k$ for every edge $e$ incident to $\rho^{*}$. 
By \cref{effresprob} we have
 $$ \reff(\rho^{*} \lr U^* \setminus B_k \,\, ;\,\, \{\Runit_e\}) \leq  {1 \over \PP_{\rho^{*}} ( \tau_{U^*\setminus B_k} < \tau^+_{\rho^{*}} ) } \, ,$$
 hence 
 $$ \PP_{\rho^{*}} ( \tau_{U^*\setminus B_k} < \tau^+_{\rho^{*}} ) \leq  {1 \over c \log k} \, ,$$
 by our assumption on $B_k$ above. By \cref{lem:comparerw} it follows that 
$$ \PPmid_{\rho^{*}} ( \tau_{U^*\setminus B_k} < \tau^+_{\rho^{*}} ) \leq {1 \over 2c \log k} \, ,$$
when $k$ is large enough and the constant $C>0$ in the definition of $\{\Rmid_e\}$ is chosen large enough with respect to $\beta$. Using \cref{effresprob} again and the fact that $U^*$ has degrees bounded by $3$ concludes the proof.
\end{proof}


We need yet another easy general fact about electric networks.

\begin{claim}\label{claim:resvertexball} Consider a finite network $G$ in which all resistances are bounded above by $1$. Then for any integer $m \geq 1$ and any two vertices $a\neq z$ we have
$$ \reff ( B_G(a,m) \lr z) \geq \reff(a \lr z) - m \, .$$
\end{claim}
\begin{proof}
  Let $\theta^m$ be the unit current flow from $B(a,m)$ to $z$.
  For a vertex $v\in B(a,m)$ denote 
  $$ \alpha_v = \sum_{u \not \in B(a,m): u \sim v}\theta^m(vu) $$
so that $\alpha_v \geq 0$ for all $v\in B(a,m)$ and $\sum_{v\in B(a,m)}\alpha_v = 1$. For a vertex $v\in B(a,m)$ let $\theta^{a,v}$ be a unit flow putting flow $1$ on some shortest path from $a$ to $v$ in $B(a,m)$. Set
  \begin{equation*}
    \theta = \sum_{v\in B(a,m)}\alpha_v(\theta^m + \theta^{a,v}) \, .
  \end{equation*}

By Thomson's principle (\cref{thomson:principle}), Jensen's inequality and since $\sum_v \alpha_v=1$ we have 
\begin{align*}
    \reff(a\lr z) &\le
    \energy(\theta)
    = \energy(\theta^m)
    + \sum_e r_e \big [
      \sum_{v\in B(a,m)} \alpha_v \theta^{a,v}(e)
    \big ]^2
    \le \energy(\theta^m)
    + \sum_{v\in B(a,m)} \alpha_v \sum_e r_e \left(\theta^{a,v}(e)\right)^2 \\
    &\leq \energy(\theta^m) + \sum_{v\in B(a,m)} \alpha_v \cdot m
    = \reff(B(a,m)\lr z) + m \, .\qedhere
  \end{align*}
\end{proof}

We are finally ready to conclude the proof of the main theorem of this chapter.

\begin{proof}[Proof of \cref{thm:lim:rec}] By \cref{lem:rmidres} and \cref{claim:resvertexball} we have that the sets $B_k$ obtained earlier satisfy that for any $m\geq 0$
  \begin{equation*}
    \reff(B_{U^*}(\rho^{*},m) \lr U^*\setminus B_k \,\, ; \{\Rmid_e\})
      \ge c\log k-m \, .
  \end{equation*}
Moreover, for every edge $e$,
  \begin{equation*}
    \RM_e \geq {\Rmid_e  \over C \log k} \, ,
  \end{equation*}
  hence
  \begin{equation*}
  \reff(B_{U^*}(\rho^{*},m) \lr U^*\setminus B_k \,\, ; \{\RM_e\})
      \ge c/C - m/C \log k \, .
  \end{equation*}
  By taking $k\to \infty$ we deduce that there exists $c>0$ such that for any $m \geq 1$
  \begin{equation*}
    \reff(B_{U^*}(\rho^{*},m) \lr \infty;\{\RM_e\})\ge c \, .
  \end{equation*}
  Consider the current unit flow from $\rho^{*}$ to $\infty$ in $(U^*, \{\RM_e\})$. If this flow had finite energy, then it follows that for any $\eps>0$ there exists $m \geq 1$ such that $\reff(B_{U^*}(\rho^{*},m) \lr \infty;\{\RM_e\})\leq \eps$, which is a contradiction to the above. Hence
  \begin{equation*}
    \reff(\rho^{*} \lr \infty;\{\RM_e\})=\infty \, ,
  \end{equation*}
  that is, $(U^*, \{\RM_e\})$ is almost surely recurrent. The theorem now follows by \cref{lem:startree}.
\end{proof}

\newcommand{\WUSF}{\mathsf{WUSF}}
\newcommand{\FUSF}{\mathsf{FUSF}}
\newcommand{\UST}{\mathsf{UST}}
\newcommand{\F}{\mathfrak F}

\chapter{Uniform spanning trees of planar graphs}\label{chp:planarusf}

\section{Introduction}


Let $G$ be a finite connected graph. A {\bf spanning tree} \index{spanning tree} $T$ of $G$ is a connected subgraph of $G$ that contains no cycles and such that every vertex of $G$ is incident to at least one edge of $T$. The set of spanning trees of a given finite connected graph is obviously finite and hence we may draw one uniformly at random. This random tree is called the {\bf uniform spanning tree} (UST) of $G$. \index{uniform spanning tree} This model was first studied by Kirchhoff \cite{kirchhoff1847ueber} who gave a formula for the number of spanning trees of a given graph and provided a beautiful connection with the theory of electric networks. In particular, he showed that the probability that a given edge $\{x,y\}$ of $G$ is contained in the UST equals $\reff(x \lr y; G)$; we prove this fundamental equality in \cref{sec:basicust} (see \cref{thm:kirchhoff}).

Is there a natural way of defining a UST probability measure on an infinite connected graph? It will soon become clear that we have set the framework to answer this positively in \cref{sec:effectiveresistance}. Let $G=(V,E)$ be an infinite connected graph and assume that $\{G_n\}$ is a finite exhaustion of $G$ as defined in \cref{sec:infinitegraphs}. That is, $\{G_n\}$ is a sequence of finite graphs, $G_n \subset G_{n+1}$ for all $n$, and $\cup G_n = G$. Russell Lyons conjectured that the UST probability measure on $G_n$ converges weakly to some probability measure on subsets of $E$ and in his pioneering work Pemantle \cite{Pem91} showed that is indeed the case. 

More precisely, denote by $\T_n$ a UST of $G_n$, then it is shown in \cite{Pem91} that for any two finite subset of edges $A,B$ of $G$ the limit
\begin{equation} \label{eq:fusfexists} \lim_{n \to \infty} \PP(A \subset \T_n \,\, , \,\, B \cap \T_n = \emptyset) \, ,\end{equation}
exists and does not depend on the exhaustion $\{G_n\}$.
The proof is a consequence of Rayleigh's monotonicity (\cref{rayleigh}) and will be presented in \cref{sec:ustlimits}. This together with Kolmogorov's extension theorem \cite[Theorem A.3.1]{Durrett} implies that there exists a unique probability measure on infinite subsets of $E$ for which a sample of $\F$ satisfies
$$ \PP( A \subset \F \, \, , \,\, B \cap \F = \emptyset ) = \lim_{n \to \infty} \PP(A \subset \T_n \,\, , \,\, B \cap \T_n = \emptyset) \, ,$$
for any two finite subsets of edges of $G$. Thus, the law of $\F$ is determined and we denote it by $\mu^F$. The superscript $F$ stands for \emph{free} and will be explained momentarily. Let us explore some properties of $\mu^F$ that are immediate from its definition.


Since every vertex of $G$ is touched by at least one edge of $\T_n$ with probability $1$ when $n$ is large enough (so that $G_n$ contains the vertex), we learn that the edges of $\F$ almost surely touch every vertex of $G$, that is, $\F$ is almost surely \emph{spanning}. Similarly, the probability that the edges of a given cycle in $G$ are contained in $\T_n$ (once $n$ is large enough so that $G_n$ contains the cycle) is $0$. Since $G$ has countably many cycles we deduce that almost surely there are no cycles in $\F$. By a similar reasoning we deduce that almost surely any connected component of $\F$ is infinite. However, a moment's reflection shows that this kind of reasoning cannot be used to determine that $\F$ is almost surely connected.



It turns out, perhaps surprisingly, that $\F$ need not be connected almost surely. A remarkable result of Pemantle \cite{Pem91} shows that a sample of $\mu^F$ on $\ZZ^d$ is almost surely connected when $d=1,2,3,4$ and almost surely disconnected when $dֿ\geq 5$. Since it may be the case that a sample of $\mu^F$ is disconnected with positive probability, we call $\mu^F$ the \defn{free uniform spanning forest} (rather than tree) of $G$, denoted henceforth $\FUSF_G$. The term \emph{free} corresponds to the fact that we have not imposed any boundary conditions when taking a limit. It will be very useful to take other boundary conditions, such as the \emph{wired} boundary condition, see \cref{sec:ustlimits}. The seminal paper of Benjamini, Lyons, Peres and Schramm \cite{BLPS} explores many properties of these infinite random trees (properties such as number of components and connectivity in particular, size of the trees, recurrence or transience of the trees and many others) on various underlying graphs with an emphasis on Cayley graphs. We refer the reader to \cite{BLPS} and to \cite[Chapters 4 and 10]{LyonsPeres} for a comprehensive treatment.

The question of connectivity of the $\FUSF$ is therefore fundamental and unfortunately it is not even known  that connectivity is an event of probability $0$ or $1$ on any graph $G$, see \cite[Question 15.7]{BLPS}. In \cite{HNPlanarUST} the circle packing theorem (\cref{thm:cp}) is used to prove that $\FUSF_G$ is almost surely connected when $G$ is a bounded degree proper planar map, answering a question of \cite[Question 15.2]{BLPS}. Our goal in this chapter is to present a proof for a specific case where $G$ is a bounded degree, transient, one-ended planar triangulations. Even though this is a particular case of a general theorem, the argument we present here contains most of the key ideas. We refer the interested reader to \cite{HNPlanarUST} for the general statement. 

\begin{theorem}[\cite{HNPlanarUST}]\label{thm:ustconnected} Let $G$ be a simple, bounded degree, transient, one-ended planar triangulation. Then $\FUSF_G$ is almost surely connected. 
\end{theorem}

The rest of this chapter is organized as follows. In \cref{sec:basicust} we discuss two basic properties of USTs on finite graphs. Namely, Kirchhoff's effective resistance formula mentioned earlier and the spatial Markov property for the UST. In \cref{sec:ustlimits} we prove Pemantle's \cite{Pem91} result \eqref{eq:fusfexists} showing that $\FUSF_G$ exists. We will also define there the \emph{wired uniform spanning forest} which is obtained by taking a limit of the UST probability measures over exhaustions with wired boundary. We will also need some fairly basic notions of electric networks on infinite graphs that we have not discussed in \cref{sec:infinitegraphs}. Next, in \cref{sec:planarduality} we will restrict to the setting of planar graph and employ planar duality to obtain an extremely useful connection between the free and wired spanning forests which will be useful later. Using these tools we have collected we will prove \cref{thm:ustconnected} in \cref{sec:mainproof}.


\section{Basic properties of the UST}\label{sec:basicust}

\subsection{Kirchhoff's effective resistance formula} 
\index{Kirchhof's effective resistance formula}
\begin{theorem}[Kirchoff \cite{kirchhoff1847ueber}]\label{thm:kirchhoff} Let $G$ be a finite connected graph and denote by $\T$ a uniformly drawn spanning tree of $G$. Then for any edge $e=(x,y)$ we have
$$ \PP(e \in \T) = \reff(x \lr y) \, .$$
\end{theorem}
\begin{proof}
Let $a\neq z$ be two distinct vertices of $G$ (later we will take $a=x$ and $z=y$) and note that any spanning tree of $G$ contains precisely one path connecting $a$ and $z$. Thus, a uniformly drawn spanning tree induces a random path from $a$ to $z$. By \cref{claim:randompath} we obtain a unit flow $\theta$ from $a$ to $z$. To be concrete, for each edge $e$ we have that $\theta(\vec{e})$ is the probability that the random path from $a$ to $z$ traverses $\vec{e}$ minus the probability that it traverses $\cev{e}$. We will now show that $\theta$ satisfies the cycle law (see \cref{claim:cyclelaw}), so it is in fact the unit current flow (see \cref{def:unitcurrentflow}).

Let $\vec{e_1},\ldots, \vec{e_m}$ be a directed cycle in $G$. Our goal is to show that
\begin{equation}\label{eq:usttoshow}
\sum _{i=1}^m \theta(\vec{e_i}) = 0 \, .
\end{equation}
Denote by $T(G)$ the set of spanning trees of $G$. Expanding the sum on the left hand side with the definition of $\theta$ we get that it equals
$$ |T(G)|^{-1} \sum_{t \in T(G)} \sum _{i=1}^m f_i^+(t) - |T(G)|^{-1} \sum_{t \in T(G)} \sum_{j=1}^m f_j^-(t) \, ,$$
where $f_i^+(t)$ equals $1$ if the unique path from $a$ to $z$ in $t$ traverses $\vec{e_i}$ and $0$ otherwise, and similarly, $f_j^-(t)$ equals $1$ if this path traverses $\cev{e_j}$ and $0$ otherwise.

For $1 \leq i \leq m$ we denote by $T^+_i$ the set of pairs $(t,i)$ for which $f^+_i(t)=1$. Similarly define $T^-_j$ as the set of pairs $(t,j)$ for which $f^-_j(t) = 1$. To prove \eqref{eq:usttoshow} it suffices to show that 
$$ |\uplus_{i \in \{1,\ldots m\}} T^+_i| = |\uplus_{j \in \{1,\ldots m\}} T^-_j| \, .$$
Let $(t,i)\in T^+_i$. The graph $t \setminus \{e_i\}$ has two connected components. Let $\vec{e_j}$ be the first edge after $\vec{e_i}$, in the order of the cycle $\vec{e_1},\ldots, \vec{e_m}$, that is incident to both connected components and consider the spanning tree  $t' = t \cup \{e_j\} \setminus \{e_i\}$. Note that the unique path in $t'$ from $a$ to $z$ traverses $\cev{e_j}$, so $(t',j) \in T^-_j$. This procedure defines a bijection from $\uplus_i T^+_i$ to $\uplus_j T^-_j$. Indeed, given $(t',j)$ from before, we can erase $e_j$ and go on the cycle in the opposite order until we reach $e_i$ which has to be the first edge incident to the two connected components of $t' \setminus \{e_j\}$. This shows \eqref{eq:usttoshow} and concludes the proof.
\end{proof}

\subsection{Spatial Markov property of the UST}
\index{spatial Markov property}
We would like to study the UST probability measure conditioned on the event that some edges are present in the UST and others not. It turns out that sampling from this conditional distribution amounts to drawing a UST on a modified graph. 

Let $G=(V,E)$ be a finite connected graph and let $A$ and $B$ be two disjoint subsets of edges. We write $(G-B)/A$ for the graph obtained from $G$ by erasing the edges of $B$ and contracting the edges of $A$. We identify the edges of $(G-B)/A$ with the edges $E\setminus B$. Denote by $\T_G$ and $\T_{(G-B)/A}$ a UST on $G$ and $(G-B)/A$, respectively, and assume that 
$$ \PP(A \subset \T_G \, , \, B \cap \T_G = \emptyset) > 0 \, .$$
This assumption is equivalent to $G-B$ being connected and that $A$ contains no cycles.

Then, conditioned on the event that $\T_G$ contains the edges $A$ and does not contain any edge of $B$ the distribution of $\T_G$ is equal to the union of $A$ with $\T_{(G-B)/A}$. In other words, for a set $\sA$ of spanning trees of $G$ we have that
\begin{equation}\label{eq:ustmarkov} \PP(\T_G \in \sA \, \mid \, A \subset \T_G \, , \, B \cap \T_G = \emptyset) = \PP(A \cup \T_{(G-B)/A} \in \sA) \, . \end{equation}
The proof of \eqref{eq:ustmarkov} follows immediately from the observation that the set of spanning trees of $G$ not containing any edge of $B$ is simply the set of spanning trees of $G-B$. Similarly, the set of spanning trees of $G$ containing all the edges of $A$ is simply the union of $A$ to each spanning tree of $G/A$, and \eqref{eq:ustmarkov} follows.

\section{Limits over exhaustions: the free and wired USF} \label{sec:ustlimits}

Let $G$ be an infinite connected graph and let $\{G_n\}$ be a finite exhaustion of it. In this section we will show that \eqref{eq:fusfexists} holds and that the UST measures with \emph{wired} boundary conditions also converge. Let us first explain the latter. Denote by $G_n^*$ the graph obtained from $G$ by identifying the infinite set of vertices $G \setminus G_n$ to a single vertex $z_n$ and erasing the loops at $z_n$ formed by this identification. We say that $\{G_n^*\}$ is a \defn{wired finite exhaustion} of $G$. 

\begin{theorem}[Pemantle \cite{Pem91}] \label{thm:usfsexist} Let $G$ be an infinite connected graph, $\{G_n\}$ a finite exhaustion and $\{G_n^*\}$ the corresponding wired finite exhaustion. Denote by $\T_n$ and $\T_n^*$ USTs on $G_n$ and $G_n^*$, respectively. Then for any two finite disjoint subsets $A, B \subset E(G)$ of edges of $G$ we have that the limits
$$\lim_{n \to \infty} \PP(A \subset \T_n \,\, , \,\, B \cap \T_n = \emptyset) \, ,$$
and 
$$\lim_{n \to \infty} \PP(A \subset \T_n^* \,\, , \,\, B \cap \T_n^* = \emptyset) \, ,$$
exist and do not depend on the exhaustion $\{G_n\}$.
\end{theorem}

We postpone the proof for a little longer and first discuss some of its implications. As mentioned earlier, \cref{thm:usfsexist} together with Kolmogorov's extension theorem \cite[Theorem A.3.1]{Durrett} implies that there exists two probability measures $\mu^F$ and $\mu^W$ on infinite subsets of the edges of $E$ arising as the unique limits of the laws $\T_n$ and $\T_n^*$. That is, the samples $\F^f$ and $\F^w$ of $\mu^F$ and $\mu^W$ satisfy
$$ \PP( A \subset \F^f \, \, , \,\, B \cap \F^f = \emptyset ) = \lim_{n \to \infty} \PP(A \subset \T_n \,\, , \,\, B \cap \T_n = \emptyset) \, ,$$
and
$$ \PP( A \subset \F^w \, \, , \,\, B \cap \F^w = \emptyset ) = \lim_{n \to \infty} \PP(A \subset \T_n^* \,\, , \,\, B \cap \T_n^* = \emptyset) \, .$$

We call $\mu^F$ and $\mu^W$ the \defn{free uniform spanning forest} and the \defn{wired uniform spanning forest} and denote them by $\FUSF_G$ and $\WUSF_G$ respectively. We have seen earlier (one paragraph below \eqref{eq:fusfexists}) that both $\F^f$ and $\F^w$ are almost surely spanning forests, that is, spanning graphs of $G$ with no cycles and that every connected component of them is infinite.
Thus $\mu^F$ and $\mu^W$ are supported on what are known as \defn{essential spanning forests} of $G$, that is, spanning forests of $G$ in which every component is infinite. 

Are the probability measures $\FUSF_G$ and $\WUSF_G$ equal? Not necessarily. It is easy to see that on the infinite path $\ZZ$ the $\WUSF_\ZZ$ and the $\FUSF_\ZZ$ are equal and are the entire graph $\ZZ$ with probability $1$. Conversely, it is not very difficult to see that they are different on a $3$-regular tree, see the exercise below \cref{thm:kirchhoffinfinite}. Pemantle \cite{Pem91} has shown that $\FUSF_{\ZZ^d} = \WUSF_{\ZZ^d}$ for any $d\geq 1$ and a very useful criterion for determining whether there is equality was developed in \cite{BLPS}. We refer the reader to \cite[Chapter 10]{LyonsPeres} for further reading.

Before presenting the proof of \cref{thm:usfsexist} let us make a few short observations regarding the effective resistance between two vertices in an infinite graph, extending what we proved in \cref{sec:infinitegraphs}.

\subsubsection{Effective resistance in infinite networks}

Let $G$ be an infinite connected graph. We have seen in \cref{sec:infinitegraphs} that for any vertex $v$ the electric resistance $\reff(v \lr \infty)$ from $v$ to $\infty$ is well defined as the limit of $\reff(a \lr z_n ; G_n^*)$ where $\{G_n^*\}$ is a wired finite exhaustion and $z_n$ is the vertex resulting in the identification of the vertices $G \setminus G_n$. 

To define the electric resistance between two vertices $v,u$ of an infinite graph, one has to take exhaustions and specify boundary conditions since the limits may differ depending on them. 

\begin{claim}\label{claim:infinitereff} Let $G$ be an infinite connected graph, $\{G_n\}$ a finite exhaustion and $\{G_n^*\}$ a wired finite exhaustion. Then for any two vertices $u,v$ of $G$ we have that the limits
$$ \reff^F(u \lr v; G) := \lim_n \reff(u \lr v ; G_n) \, ,$$
and 
$$ \reff^W(u \lr v; G) := \lim_n \reff(u \lr v ; G_n^*) \, ,$$
exist and do not depend on the exhaustion $\{G_n\}$.
\end{claim}
\begin{proof} For the first limit we note that by Rayleigh's monotonicity (\cref{rayleigh}), the sequence $\reff(u \lr v ; G_n)$ is non-increasing and non-negative since $G_n \subset G_{n+1}$, hence it converges. A sandwiching argument as in the proof of \cref{claim:infinitereff} shows that the limit does not depend on the exhaustion $\{G_n\}$.   

For the second limit, since $G_{n}$ can be obtained by gluing vertices of $G_{n+1}$ we deduce by \cref{gluing:monotone} that the sequence $\reff(u \lr v ; G_n^*)$ is non-decreasing and bounded (by the graph distance in $G$ between $u$ and $v$ for instance), hence it converges. The limit does not depend on the exhaustion by an identical sandwiching argument. 
\end{proof}

We call $\reff^F(u \lr v; G)$ and $\reff^W(u \lr v; G)$ the \defn{free effective resistance} and \defn{wired effective resistance} between $u$ and $v$ respectively. 

\subsection{Proof of \cref{thm:usfsexist}}

We will prove the assertion regarding the first limit; the second is almost identical. Write $A=\{e_1, \ldots, e_k\}$ and $e_i=(x_i,y_i)$ for each $1 \leq i \leq k$. Assume without loss of generality that $G_n$ contains $A$  for all $n$. As before, denote by $\T_n$ a UST of $G_n$. By \eqref{eq:ustmarkov} and \cref{thm:kirchhoff} we have that
\begin{eqnarray*} \PP(A \subset \T_n) &=& \prod_{i=1}^k \PP(e_i \in \T_n \, \mid \, e_j\in \T_n \quad \forall \, j < i ) 
= \prod_{i=1}^k \reff(x_i \lr y_i ; G_n / \{e_1,\ldots, e_{i-1}\}) \, .
\end{eqnarray*}
Note that $\big \{G_n / \{e_1, \ldots, e_{i-1}\} \big\}$ is a finite exhaustion of the infinite graph $G / \{e_1, \ldots, e_{i-1}\}$ and so by \cref{claim:infinitereff} we obtain that the limit
$$ \lim_n \PP(A \subset \T_n) = \prod_{i=1}^k \reff(x_i \lr y_i ; G / \{e_1,\ldots, e_{i-1}\}) \, ,$$
exists and does not depend on the exhaustion. 

Since we know this limit exists for all finite edge sets $A$, it follows by the inclusion-exclusion formula that $\PP(A \subset \T_n, B \cap \T_n = \emptyset)$ converges for any finite sets $A, B$, concluding our proof. \qed

\medskip


It is now quite pleasant to see that the symbiotic relationship between electric network and UST theories continues to flourish in the infinite setting. Indeed, by combining \cref{thm:usfsexist}, \cref{thm:kirchhoff} and \cref{claim:infinitereff} we obtain the extension of Kirchhoff's formula for infinite connected graphs.

\begin{theorem}\label{thm:kirchhoffinfinite} Let $G$ be an infinite connected graph and denote by $\F^F$ and $\F^W$ a sample from $\FUSF_G$ and $\WUSF_G$ respectively. Then for any edge $e=(x,y)$ of $G$ we have that 
$$ \PP( e \in \F^F) = \reff^F(x \lr y; G) \, ,$$
and 
$$ \PP( e \in \F^W) = \reff^W(x \lr y; G) \, .$$
\end{theorem}

\noindent {\bf Exercise:} Use this to show that on the $3$-regular tree $\mathbb{T}_3$ the probability measures $\FUSF_{\mathbb{T}_3}$ and the $\WUSF_{\mathbb{T}_3}$ are distinct.

\section{Planar duality}\label{sec:planarduality}

When $G$ is planar there is a very useful relationship between $\FUSF_G$ and $\WUSF_G$. Recall that given a planar map $G$, the \defn{dual graph} of $G$ is the graph $G^\dagger$ whose vertex set is the set of faces of $G$ and two faces are adjacent in $G^\dagger$ if they share an edge in $G$. Thus, $G^\dagger$ is locally-finite if and only if every face of $G$ has finitely many edges. To each edge $e \in E(G)$ corresponds a dual edge $e^\dagger \in E(G^\dagger)$ which is the pair of faces of $G$ incident to $e$; this is clearly a one-to-one correspondence. 

When $G$ is a finite planar graph, this correspondence induces a one-to-one correspondence between the set of spanning trees of $G$ and the set of spanning trees of $G^\dagger$. Given a spanning tree of $t$ of $G$ we slightly abuse the notation and write $t^\dagger$ for the set of edges $\{ e^\dagger : e \in G \setminus t \}$, that is
$$ e \in t \iff e^\dagger \not \in t^\dagger \, .$$
If $t^\dagger$ has a cycle, then $t$ is disconnected. Furthermore, if there is a vertex $G^\dagger$ not incident to any edge of $t^\dagger$, then all the edges of the corresponding face in $G$ are present in $t$ hence $t$ contains a cycle. We deduce that if $t$ is a spanning tree of $G$, then $t^\dagger$ is a spanning tree of $G^\dagger$. The converse also holds since $(t^\dagger)^\dagger = t$ and $(G^\dagger)^\dagger=G$.

Now assume that $G$ is an infinite planar maps such that $G^\dagger$ is locally finite. Given an essential spanning forest $\F$ of $G$ we similarly define $\F^\dagger$ as the set of edges $\{e ^\dagger : e \in G \setminus \F\}$. A similar argument shows that $\F^\dagger$ is an essential spanning forest of $G^\dagger$. This raises the natural question: when $\F$ is a sample of $\FUSF_G$, what is the law of $\F^\dagger$? The answer in general is an object known as the \emph{transboundary uniform spanning forest} \cite[Proposition 5.1]{HNPlanarUST}. However, when $G$ is additionally assumed to be one-ended (in particular, in the setting of \cref{thm:ustconnected}) it turns out that $\F^\dagger$ is distributed as $\WUSF_{G^\dagger}$:

\begin{proposition} \label{prop:ustduality} Let $G$ be an infinite, one-ended planar map with a locally finite dual $G^\dagger$ and let $\F$ be a sample of $\FUSF_G$. Then the law of $\F^\dagger$ is $\WUSF_{G^\dagger}$.
\end{proposition}
\begin{proof} Let $G_n$ be a finite exhaustion of $G$. Let $F_n$ be a finite exhaustion $G^\dagger$ defined by letting $f\in F_n$ if and only if every vertex of $f$ in $G$ belongs to $G_n$. Then $G_n^\dagger$ is obtained from $G^\dagger$ by contracting $G^\dagger \setminus F_n$ into a single vertex which corresponds to the outer face of $G_n$. Thus, $G_n^\dagger$ is a wired exhaustion of $G^\dagger$ and the statement follows.
\end{proof}

We use to obtain an important criterion of connectivity of $\FUSF_G$ in the planar case.

\begin{proposition} \label{prop:ustduality2} Let $G$ be an infinite, one-ended planar map with a locally finite dual $G^\dagger$. Then a sample of $\FUSF_G$ is connected almost surely if and only if each component of a sample of $\WUSF_G$ is one-ended almost surely.
\end{proposition}
\begin{proof} By \cref{prop:ustduality} it suffices to show that if $\F$ is an essential spanning forest of $G$, then $\F$ is connected if and only if every component of $\F^\dagger$ is one-ended. Indeed, if $\F$ is disconnected, then the boundary of a connected component of $\F$ induces an bi-infinite path in $\F^\dagger$. Conversely, if $\F^\dagger$ contains a bi-infinite path, then by the Jordan curve theorem $\F$ is disconnected.
\end{proof}

\section{Connectivity of the free forest} \label{sec:mainproof}

\subsection{Last note on infinite networks} \index{effective resistance in an infinite network}

We make two more useful and natural definitions. Given two disjoint finite sets $A$ and $B$ in an infinite connected graph $G$ we define the free and wired effective resistance between them $\reff^W(A \lr B; G)$ and $\reff^F(A \lr B; G)$ as the free and wired effective resistance between $a$ and $b$ in the graph obtained from $G$ by identifying $A$ and $B$ to the vertices $a$ and $b$. 

Lastly, given a graph $G$, a wired finite exhaustion $\{G_n^*\}$ of $G$ and two disjoint finite sets $A$ and $B$ we define
\begin{equation}\label{def:abinftyreff} \reff(A \lr B \cup \{\infty\}; G) := \lim_{n \to \infty} \reff(A \lr B \cup \{z_n\}; G_n^*) \, ,\end{equation}
where the last limit exists since the sequence is non-increasing from $n$ that is large enough so that $G_n$ contains $A$ and $B$. In the proof of \cref{thm:ustconnected} we will require the following estimate.

\begin{lemma}\label{lem:trianglewired} Let $A$ and $B$ be two finite sets of vertices in an infinite connected graph $G$. Then
\[ \reff^W(A \lr B;\, G) \leq 3\max\left\{\reff \left(A \lr B \cup \{\infty\};\, G \right),\, \reff\left(B \lr A \cup \{\infty\};\, G\right)\right\}.\]
\end{lemma}

\begin{proof} 
For any three distinct vertices $u,v,w$ in a finite network we have by the union bound that $\PP_u(\tau_{\{v,w\}} < \tau_u^+) \leq \PP_u(\tau_v < \tau_u^+) + \PP_u(\tau_w < \tau_u^+)$. Hence by \cref{effresprob} we get that
$$ \reff(u \lr \{v,w\})^{-1} \leq \reff(u \lr v)^{-1} + \reff(u \lr w)^{-1} \, .$$
Let $\{G_n^*\}$ be a wired finite exhaustion of $G$ and assume without loss of generality that $A$ and $B$ are contained in $G_n^*$ for all $n$. Then by the previous estimate
$$ \reff(A \lr B \cup \{z_n\}; G_n^*)^{-1} \leq \reff(A \lr B; G_n^*)^{-1} + \reff(A \lr z_n; G_n^*)^{-1} \, .$$
Denote by $M$ the maximum in the statement of the lemma and take $n\to \infty$ in the last inequality. We obtain that
$$ M^{-1} \leq \reff(A \lr B \cup \{\infty\}; G)^{-1} \leq \reff^W(A \lr B; G)^{-1} + \reff(A \lr \infty; G)^{-1} \, .$$
Rearranging gives that 
$$ \reff(A \lr \infty; G) \leq \frac{M \reff^W(A \lr B; G)}{\reff^W(A \lr B; G) - M} \, .$$
By symmetry, the same inequality holds when we replace the roles of $A$ and $B$. We put this together with the triangle inequality for effective resistances \eqref{exercise:triangle} and get that
$$ \reff^W(A \lr B;\, G) \leq \reff(A \lr \infty; G) + \reff(B \lr \infty; G) \leq \frac{2M \reff^W(A \lr B; G)}{\reff^W(A \lr B; G) - M} \, ,$$
which by rearranging gives the desired inequality.
\end{proof}

\subsection{Method of random sets} \index{method of random sets} We present the following weakening of the method of random paths as in \cref{sec:randompaths}. Let $\mu$ be the law of a random subset $W$ of vertices of $G$. Define the \emph{energy} of $\mu$ as
 \[\energy(\mu) = \sum_{v\in V}\mu(v\in W)^2.\]

\begin{lemma}[Method of random sets]\label{lem:randomsets} Let $A, B$ be two disjoint finite sets of vertices in an infinite graph $G$. Let $W$ be a random subset of vertices of $G$ and denote by $\mu$ its law. Assume that the subgraph of $G$ induced by $W$ almost surely contains a simple path starting at $A$ that is either infinite or finite and ends at $B$. Then
\begin{equation}\label{eq:randomsets}\reff(A \leftrightarrow B \cup \{\infty\}; G) \leq \energy(\mu). \end{equation}
\end{lemma}

\begin{proof}
Given $W$ let $\gamma$ be a simple path, contained in $W$, connecting $A$ to $B$ or an infinite path starting at $A$. We choose $\gamma$ according to some prescribed lexicographical ordering. Then, letting $\nu$ be the law of $\gamma$,
\begin{align*}
  \energy(\nu)
  & \leq \sum_{\vec{e}\in E} \nu(\vec{e} \in \gamma)^2 ,
  \end{align*}
where by $\vec{e} \in \gamma$ we mean that the directed edge $\vec{e}$ is traversed (in its direction) by $\gamma$, and by $\energy(\nu)$ we mean the energy of the flow induced by $\gamma$, as in \cref{claim:randompath}.

Let $\gamma'$ be an independent random path having the same law as $\gamma$. Then the sum above is precisely the expected number of directed edges traversed both by $\gamma$ and $\gamma'$. Since these are simple paths, they each contain at most one directed edge emanating from each vertex $v\in W$. Thus, the expected number of directed edges used by both paths is at most the number of vertices used by both paths. Hence, 
$$\energy(\nu)\leq \sum_{v \in V(G)} \nu(v \in \gamma)^2 \leq \sum_{v \in V(G)} \mu(v \in W)^2 = \energy(\mu) \, ,$$
and the proof is concluded by Thomson's principle (\cref{thomson:principle}).
\end{proof}

\subsection{Proof of \cref{thm:ustconnected}}


In \cref{thm:ustconnected} we assume that $G=(V,E)$ is a bounded-degree, one-ended triangulation. Hence $G^\dagger$ is a bounded degree (in fact, $3$-regular), one-ended and transient planar map with faces of uniformly bounded size. We leave this verification as an exercise for the reader. To avoid carrying the $\dagger$ symbol around, and with a slight abuse of notation, let $G=(V,E)$ be a graph satisfying these assumptions on $G^\dagger$, that is, we assume that $G$ is a one-ended, transient, infinite planar map with bounded degrees and face sizes. We will prove under these assumptions that every component of $\WUSF_{G}$ is one ended almost surely which implies \cref{thm:ustconnected} by \cref{prop:ustduality2}.

Let $T$ be the bounded-degree one-ended triangulation obtained from $G$ by adding a vertex inside each face of $G$ and connecting it by edges to the vertices of that face according to their cyclic ordering. By \cref{thm:hs} there exists a circle packing of $T$ in the unit disc $\UU$. We identify the vertices of $T$ as the vertices $V(G)$ and faces $F(G)$ of $G$, and denote this circle packing as $P=\{P(v) : v \in V(G)\} \cup \{P(f) : f \in F(G)\}$. 

Given $z \in \UU$ and $r' \geq r > 0$ denote by $A_z(r,r')$ the annulus $\{w\in \CC : r\leq |w-z| \leq r'\}$.

\begin{definition} \label{def:Vz}
 Write $V_z(r,r')$ for the set of vertices $v$ of $G$ such that either 
\begin{itemize}
\item $P(v) \cap A_z(r,r') \neq \emptyset$, or
\item $P(v) \subset \{ w \in \CC: |w|\leq  r \}$ and there is a face $f$ of $G$ such that $v\in f$ with $P(f) \cap A_z(r,r') \neq \emptyset$.
\end{itemize}
We emphasize that $V_z(r,r')$ contains only vertices of $G$ (that is, no vertices of $T$ that correspond to faces of $G$ belong to it).
\end{definition}

\begin{lemma}\label{lem:sepannuli} There exists a constant $C<\infty$ depending only on the maximal degree such that for any $z \in \UU$ and any positive integer $n$ satisfying $|z| \geq 1 - C^{-n}$ the sets
$$ V_z(C^{-i}, 2C^{-i}) \qquad 1 \leq i \leq n \, ,$$
are disjoint.
\end{lemma}
\begin{proof} By the Ring Lemma (\cref{lem:ring}) there exists a constant $B<\infty$ such that for any $C > 1$, any $z$ satisfying $z \geq 1- C^{-n}$ and any $1 \leq i \leq n$, if a circle of $P$ intersects $A_z(C^{-i}, 2C^{-i})$ or is tangent to a circle that intersects $A_z(C^{-i}, 2C^{-i})$, then its radius is at most $B C^{-i}$. Hence, this set of circles is contained in the disc of radius $(2+4B)C^{-i}$ around $z$. Furthermore, since $|z| \geq 1- C^{-n}$, by the Ring Lemma again there exists $b>0$ such that any such circle must be of distance at least $b C^{-i}$ from $z$. Hence, any fixed $C > {4+4B \over b}$ satisfies the assertion of the lemma.
\end{proof}

\begin{lemma}\label{lem:energyest} Let $z\in \UU$ and $r>0$. Let $U$ be a uniform random variable in $[1,2]$ and denote by $\mu_r$ the law of the random set $V_z(Ur,Ur)$ (as defined in \cref{def:Vz}). Then there exists a constant $C<\infty$ depending only on the maximal degree such that
$$ \energy(\mu_r) \leq C \, .$$
\end{lemma}

\begin{proof} For each vertex $v$, the event $v \in V_z(Ur,Ur)$ implies that the circle $\{w\in\CC: |w-z|=Ur\}$ intersects the circle $P(v)$ or intersects $P(f)$ for some face $f$ incident to $v$. The union of $P(v)$ and $P(f)$ over all such faces $f$ is contained in the Euclidean ball around the center of $P(v)$ of radius $r(v) + 2 \max_{f : v\in f} r(f)$. Since $T$ has finite maximal degree we have that $r(f) \leq C r(v)$ for all $f$ with $v\in f$ where $C<\infty$ depends only on the maximal degree by the Ring Lemma \cref{lem:ring}. Hence,
\begin{align}
  \mu_r(v \in V_z(Ur, Ur)) &\leq {1 \over r}\min\left(2r(v) + 4\max_{f : f \ni v}r(f),\,r\right) \leq {C \over r}\min\{r(v),r\}.\label{eq:egy1}
\end{align}
We claim that
\begin{equation}\label{eq:egy2}\sum_{v\in V_z(r,2r)} \min\{r(v),\, r\}^2 \leq 16r^2.\end{equation}
Indeed, consider a vertex $v \in V_z(r,2r)$ for which the corresponding circle $P(v)$ has radius larger than $r$. By \cref{def:Vz} this circle must intersect $\{w \in \CC : |w-z| \leq 2r\}$. We replace each such $P(v)$ with a circle of radius $r$ that is contained in the original circle and intersects $\{w \in \CC : |w-z| \leq 2r\}$. The circles in this new set still have disjoint interiors and are contained in $\{w \in \CC : |w-z| \leq 4r\}$. Therefore their area is at most $\pi 16r^2$ and \eqref{eq:egy2} follows. The proof of lemma is now concluded by combining \eqref{eq:egy1} and \eqref{eq:egy2}.
\end{proof}

\begin{proof}[Proof of \cref{thm:ustconnected}] Let $\F$ be a sample of $\WUSF_G$ and given an edge $e=(x,y)$ we define $\sA^e$ to be the event that $x$ and $y$ are in two distinct infinite connected components of $\F\setminus\{e\}$. It is clear that every component of $\F$ is one-ended almost surely if and only if 
\begin{equation}\label{eq:usttoprove}
\PP(e \in \F \, , \, \sA^e ) = 0 
\end{equation}
for every edge $e$ of $G$. Consider the triangulation $T$ described above \cref{def:Vz} and its circle packing $P$ in $\UU$. By choosing the proper M\"obius transformation we may assume that the tangency point between $P(x)$ and $P(y)$ is the origin, and that the centers of $P(x)$ and $P(y)$ lie on the negative and positive real axis, respectively. 

Fix now an arbitrary $\eps>0$ and let $V_\eps$ be all the vertices of $G$ such that the center $z(v)$ of $P(v)$ satisfies $|z(v)|\leq 1-\eps$. Denote by $\sB^e_\eps$ the event that every connected component of $\F \setminus \{e\}$ intersects $V\setminus V_\eps$. Note that $\displaystyle \A^e \subset \cap_{\eps>0} \sB^e_\eps$ but this containment is strict since it is possible that $e \not \in \F$ and $x$ is connected to $y$ in $\F$ inside $V_\eps$.

Assume that $\sB_\eps^e$ holds. Let $\eta^x$ be the rightmost path in $\F\setminus\{e\}$ from $x$ to $V\setminus V_\eps$ when looking at $x$ from $y$, and let $\eta^y$ be the leftmost path in $\F\setminus\{e\}$ from $y$ to $V\setminus V_\eps$ when looking at $y$ from $x$. As mentioned above, the paths $\eta_x$ and $\eta_y$ are not necessarily disjoint. Nonetheless, concatenating the reversal of $\eta^x$ with $e$ and $\eta^y$ separates $V_\eps$ into two sets of vertices, $\cL$ and $\cR$, which are to the left and right of $e$ (when viewed from $x$ to $y$) respectively. See Figure \ref{fig:exploration} for an illustration of the case when $\eta_x$ and $\eta_y$ are disjoint (when they are not, $\cR$ is a ``bubble'' separated from $V \setminus V_\eps$).

On the event $\sB^e_\eps$, let $K$ be the set of edges that are either incident to a vertex in $\cL$ or belong to the path $\eta_x \cup \eta_y$, and set $K=E$ off of this event. Note that the edges of $K$ do not touch the vertices of $\cR$. The condition that $\eta^x$ and $\eta^y$ are the rightmost and leftmost paths to $V\setminus V_\eps$ from $x$ and $y$ is equivalent to the condition that $K$ does not contain any open path from $x$ to $V\setminus V_\eps$ other than $\eta^x$, and does not contain any open path from $y$ to $V\setminus V_\eps$ other than $\eta^y$. We note that $K$ can be explored algorithmically, without querying the status of any edge in $E\setminus K$, by performing a right-directed depth-first search of $x$'s component in $\F$ and a left-directed depth-first search of $y$'s component in $\F$, stopping each search when it first leaves $V_\eps$.

\begin{figure}[t]
\centering
\includegraphics[width=0.44\textwidth]{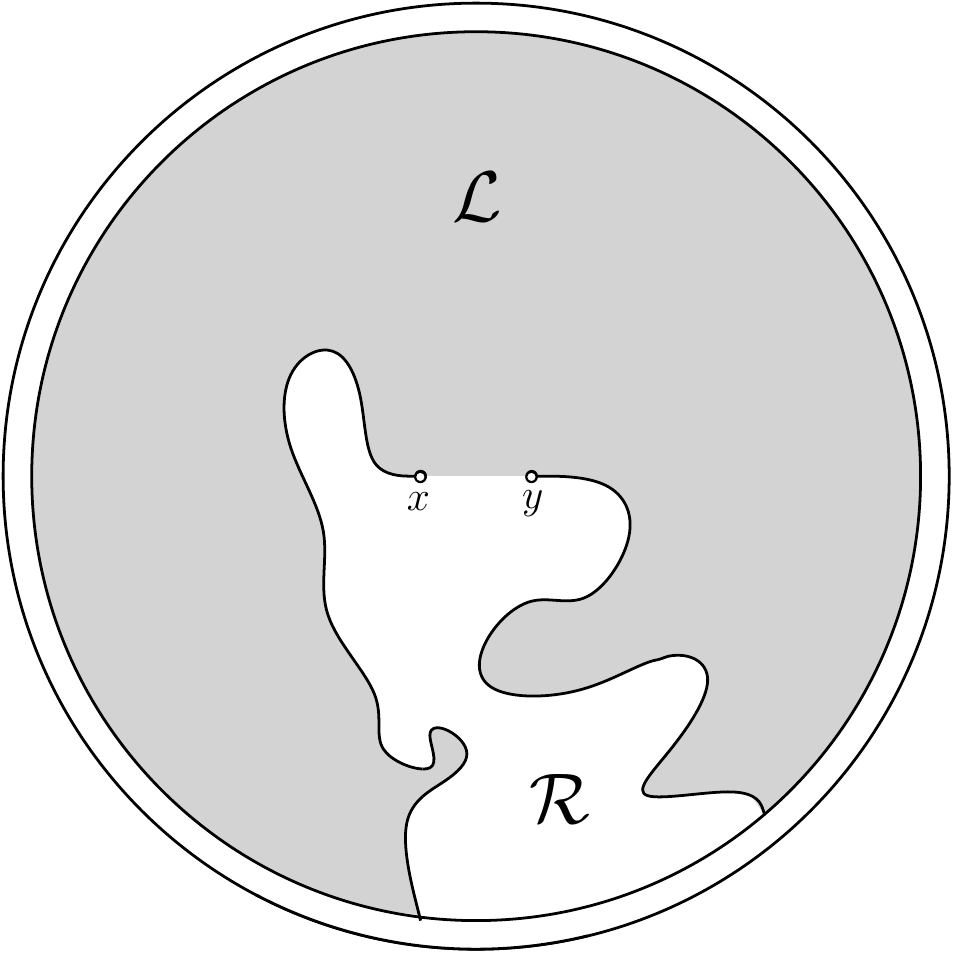} \qquad \includegraphics[width=0.44\textwidth]{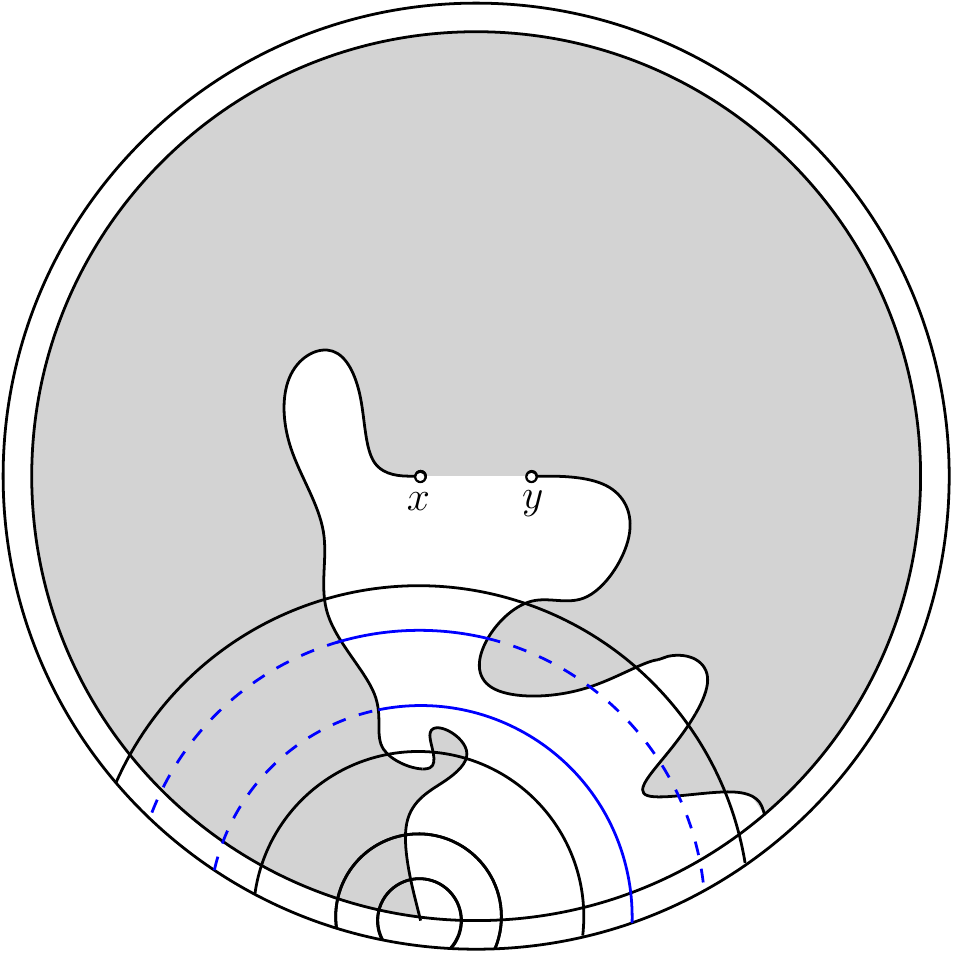}
\caption{Illustration of the proof. Left: On the event $\sA^e_\eps$, the paths $\eta^x$ and $\eta^y$ split $V_\eps$ into two pieces, $\cL$ and $\cR$. Right: We define a random set containing a path (solid blue) from $\eta^x$ to $\eta^y\cup\{\infty\}$ in $G\setminus K_c$ using a random circle (dashed blue). Here we see two examples, one in which the path ends at $\eta^y$, and the other in which the path ends at the boundary (i.e., at infinity).}\label{fig:exploration}
\end{figure}

Denote by $\sA^e_\eps$ the event that $\eta_x$ and $\eta_y$ are disjoint, or equivalently, that $K$ does not contain an open path from $x$ to $y$ (and in particular, no path starting at $\eta_x$ and ending at $\eta_y$). The event $\sA^e_\eps$ is measurable with respect to the random set $K$ and $\sA^e = \cap _{\eps >0} \sA^e_\eps$. Hence
\begin{equation}\label{eq:onebefore} \PP(e \in \F \, , \, \sA^e ) \leq \PP( e\in \F \mid \sA^e_\eps) = \EE [ \PP( e\in \F \mid \sA^e_\eps, K)] \, .\end{equation}
Denote by $K_o$ the open edge of $K$ (that is, the edge of $K$ in $\F$) and by $K_c$ the closed edges of $K$ (that is, the edges of $K$ not belonging to $\F$). In particular, $\eta_x$ and $\eta_y$ are contained in $K_o$. Then by the UST Markov property \eqref{eq:ustmarkov}, conditioned on $K$ and the event $\sA^e_\eps$, the law of $\F$ is equal to the union of $K_o$ with a sample of the $\WUSF$ on $(G-K_c)/K_o$. In particular, by Kirchhoff's formula \cref{thm:kirchhoffinfinite} we have that
\begin{equation}\label{eq:applykirchhoff} \PP( e\in \F \mid \sA^e_\eps, K) \leq \reff^W ( \eta_x \lr \eta_y ; G - K_c) \, ,\end{equation}
where in the last inequality we used the fact that gluing cannot increase the resistance (\cref{gluing:monotone}). 

We will show that the last quantity tends to $0$ as $\eps\to 0$ which gives \eqref{eq:usttoprove}. To that aim, let let $v^x$ be the endpoint of the path $\eta^x$ and let $z_0$ be the center of the $P(v_x)$. 
On the event $\sA_\eps^e$, for each $1-|z_0|\leq r\leq 1/4$, we claim that the set $V_{z_0}(r,r)$, as defined in \cref{def:Vz}, contains a path in $G$ from $\eta^x$ to $\eta^y$ that is contained in $\cR \cup \eta^x \cup \eta^y$ or an infinite simple path starting at $\eta_x$ that is contained in $\cR \cup \eta^x$. Either of these paths are therefore a path in $G - K_c$. 

\newcommand{\fA}{\mathsf{A}}

To see this, consider the arc $\fA'(z_0,r)=\{z\in \overline{\UU} : |z-z_0|=r\}$ viewed in the clockwise direction and let $\fA(z_0,r)$ be the subarc beginning at the last intersection of $\fA'(z_0,r)$ with a circle corresponding to a vertex in the trace of $\eta^x$, and ending at the first intersection after this time of $\fA'(z_0,r)$ with either $\partial \UU$ or a circle corresponding to a vertex in the trace of $\eta^y$ (see \cref{fig:exploration}).
Hence, if $\sA^e_\eps$ holds, then the set of vertices of $T$ whose circles in $P$ intersect $\fA(z_0,r)$ contains a path in $T$ starting at $\eta^x$ and ending $\eta^y$ or does not end at all, for every $1-|z_0|\leq r\leq 1/4$.
To obtain a path in $G$ rather than $T$ we  divert the path counterclockwise around each face of $G$. That is, whenever the path passes from a vertex $u$ of $G$ to a face $f$ of $G$ and then to a vertex $v$ of $G$, we replace this section of the path with the list of vertices of $G$ incident to $f$ that are between $u$ and $v$ in the counterclockwise order. By \cref{def:Vz} this diverted path is in $V_{z_0}(r,r)$ and so this construction shows that the subgraph of $G - K_c$ induced by the set $V_{z_0}(r,r)$ contains a path from $\eta^x$ to $\eta^y$ or an infinite path from $\eta^x$, as claimed.

Let $r_i = C^{-i}$ for $i=1,\ldots,N$ where $C<\infty$ the constant from \cref{lem:sepannuli} and $N=\lfloor \log_C(\eps) \rfloor$. Assume without loss of generality that $C\geq 4$ so that $\eps \leq r_i \leq 1/4$ for all $i=1,\ldots, N$. By \cref{lem:sepannuli} the measures $\mu_{r_i}$ defined in \cref{lem:energyest} are supported on sets that are contained in the disjoint sets $V_z(r_i,2r_i)$. Thus, by \cref{lem:randomsets} and \cref{lem:energyest} we have
 \begin{align*} \reff^W\Big(\eta^x \lr \eta^y \cup \{\infty\};\, G \setminus K_c\Big)  \leq \energy \left ({1 \over N}\sum_{i=1}^N \mu_{r_i} \right )={1 \over N^2}\sum_{i=1}^N \energy (\mu_{r_i}) \leq {B \over \log(1/\eps)} \, ,  \end{align*}
 where $B<\infty$ is a constant depending only on the maximum degree. By symmetry we also have
 \[ \reff^W(\eta^y \leftrightarrow \eta^x \cup \{\infty\};\, G - K_c) \leq {B \over \log(1/\eps)}.  \]
 Applying \cref{lem:trianglewired} and \eqref{eq:applykirchhoff} gives
 $$ \PP(e\in\F \, |\, \cF_K,\,\sA_\eps^e) \leq {3B \over \log(1/\eps) } \, .$$
We plug this estimate into \eqref{eq:applykirchhoff} and take $\eps \to 0$, which together with \eqref{eq:onebefore} shows that \eqref{eq:usttoprove} holds, concluding our proof.
\end{proof}

\chapter{Related topics}\label{chp:related}

In this chapter we briefly review some aspects of the literature on circle packing that unfortunately we do not have space to get into in depth in this course. We hope this will be useful as a guide to further reading. 

\begin{enumerate}
\item \textbf{Double circle packing.} \index{double circle packing}
If one wishes to study planar graphs that are \emph{not} triangulations, it is often convenient to work with \emph{double circle packings}, which enjoy similar rigidity properties to usual circle packings, but for the larger class of \textbf{polyhedral} \index{polyhedral graph} planar graphs. Here, a planar graph is polyhedral if it is both simple and \textbf{$3$-connected}, meaning that the removal of any two vertices cannot disconnect the graph. Double circle packings also satisfy a version of the ring lemma \cite[Theorem 4.1]{HutNach15b}, which means that they can be used to produce good straight-line embeddings of polyhedral planar graphs that have bounded face degrees but which are not necessarily triangulations. 

\begin{figure}[h!]
\centering
\includegraphics[height=0.325\textwidth]{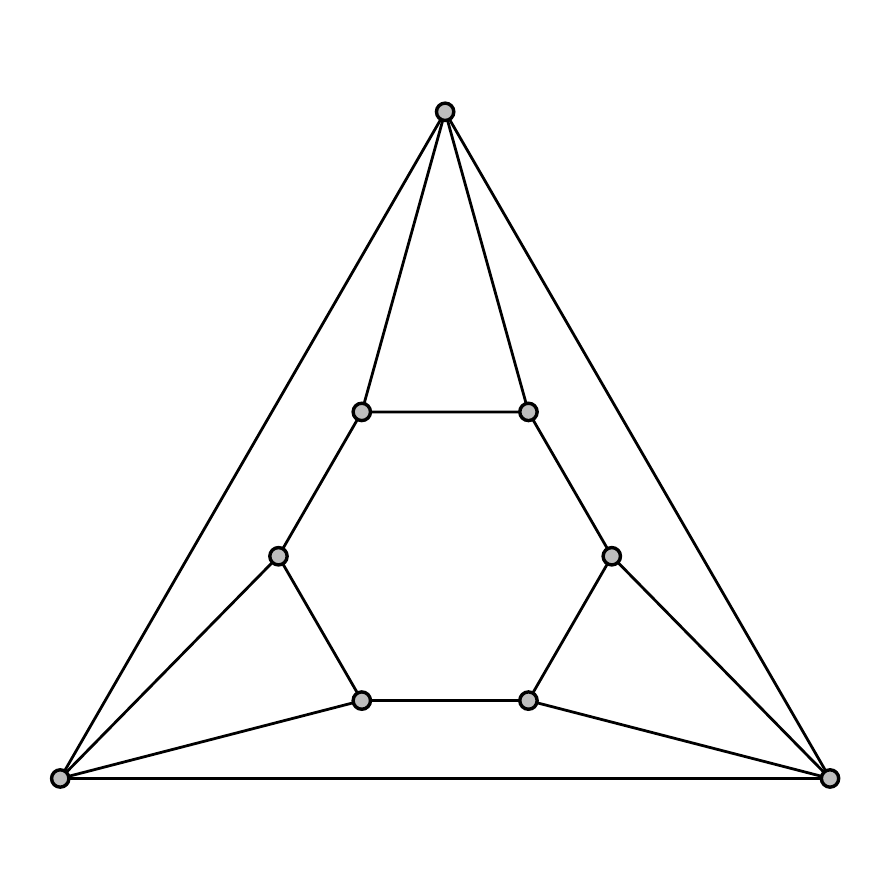} \hspace{1.3cm} \includegraphics[height=0.325\textwidth]{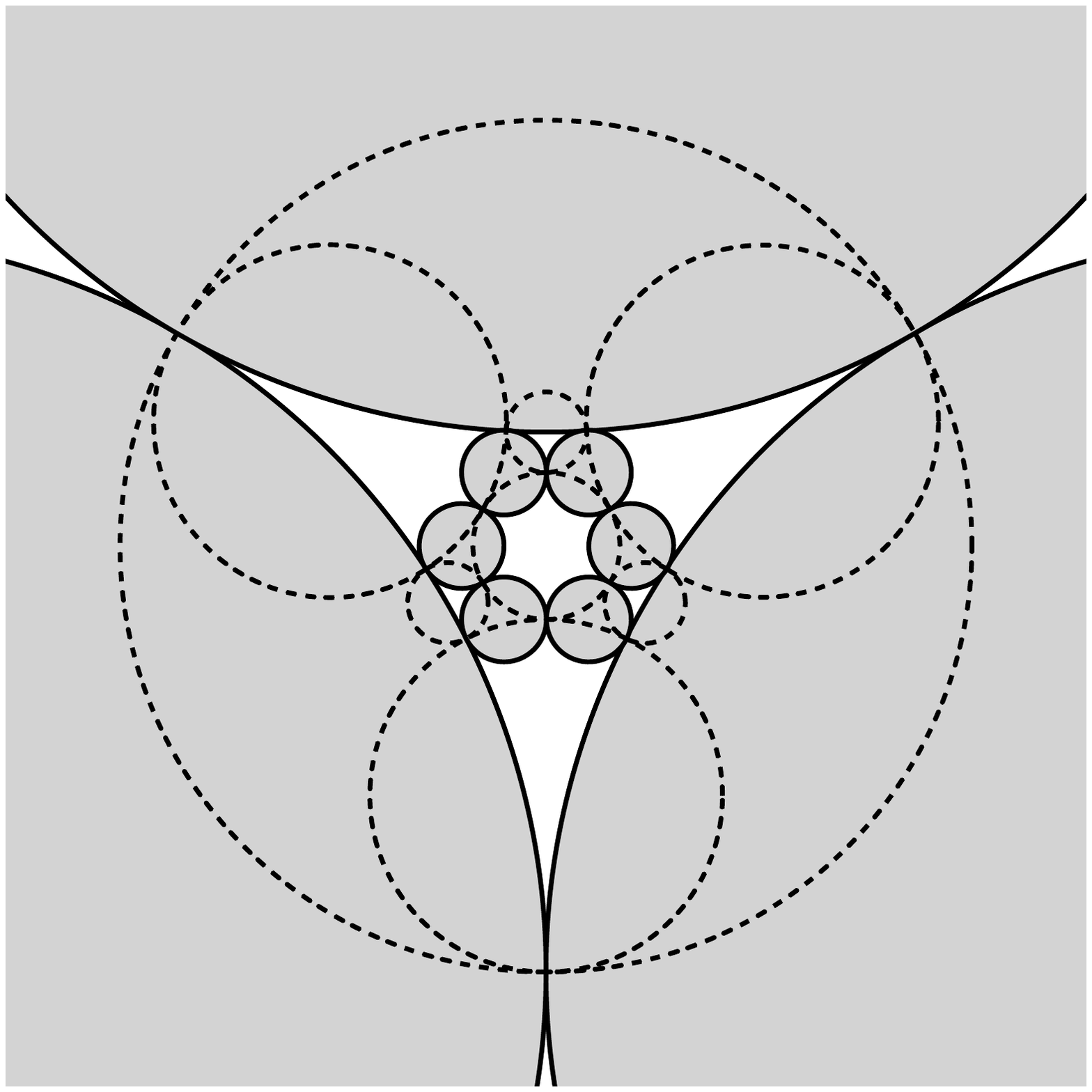}
\caption{
A finite polyhedral plane graph (left) and its double circle packing (right). Primal circles are filled and have solid boundaries, dual circles have dashed boundaries.}
\label{fig.dcp}
\end{figure}

Let $G$ be a planar graph with vertex set $V$ and face set $F$. A double circle packing of $G$ is a pair of circle packings $P=\{ P_v : v \in V\}$ and $P^\dagger=\{P_f:f\in F\}$ satisfying the following conditions:
\begin{enumerate}
\item (\textbf{$G$ is the tangency graph of $P$}.) For each pair of vertices $u$ and $v$ of $G$, the discs $P_u$ and $P_v$ are tangent if and only if $u$ and $v$ are adjacent in $G$.
\item (\textbf{$G^\dagger$ is the tangency graph of $P^\dagger$}.)  For each pair of faces $f$ and $g$ of $G$, the discs $P_f$ and $P_g$ are tangent if and only if $f$ and $g$ are adjacent in $G^\dagger$.
\item (\textbf{Primal and dual circles are perpendicular}.) For each vertex $v$ and face $f$ of $G$, the discs $P_f$ and $P_v$ have non-empty intersection if and only if $f$ is incident to $v$, and in this case the boundary circles of $P_f$ and $P_v$ intersect at right angles.
\end{enumerate}
See \cref{fig.dcp} for an illustration. 

Thurston's proof of the circle packing theorem also implies that every finite polyhedral planar graph admits a double circle packing. This was also shown by Brightwell and Scheinerman \cite{BriSch93}. As with circle packings of triangulations, the double circle packing of any finite polyhedral planar map is unique up to M\"obius transformations or reflections.  The theory of double circle packings in the infinite setting follows from the work of He \cite{he1999rigidity}, and is exactly analogous to the corresponding theory for triangulations. Indeed, essentially everything we have to say in these notes about circle packings of simple triangulations can be generalized to double circle packings of polyhedral planar maps (sometimes under the additional assumption that the faces are of bounded degree).

\item \textbf{Packing with other shapes}. A very powerful generalization of the circle packing theorem known as the \emph{monster packing theorem} was proven by Oded Schramm in his PhD thesis \cite{schramm2007combinatorically}. One consequence of this theorem is as follows: Let $T=(V,E)$ be a finite planar triangulation with a distinguished boundary vertex $\partial$. Specify a bounded, simply connected domain $D \subset \CC$ with smooth boundary, and for each $v\in V \setminus \{\partial\}$ specify a strictly convex, bounded  domain $D_v$ with smooth boundary. Then there exists a collection of homotheties (compositions of translations and dilations) $\{h_v : v\in V\}$ such that 
\begin{itemize}
  \item
If $u,v\in V \setminus \{\partial\}$ are distinct, then $h_v D_v$ and $h_u D_u$ have disjoint interiors, and intersect if and only if $v$ and $u$ are adjacent in $T$.
\item 
If $v\in V \setminus \{\partial\}$, then $h_v D_v$ and $\CC \setminus D$ have disjoint interiors, and intersect if and only if $v$ is adjacent to $\partial$ in $T$.
\end{itemize}
In other words, we can represent the triangulation of $T$ by a packing with arbitrary smooth convex shapes that are specified up to homothety (it is quite surprising at first that rotations are not needed). The full monster packing theorem also allows one to relax the smoothness and convexity assumptions above in various ways. The proof of the monster packing theorem is based upon Brouwer's fixed point theorem, and does not give an algorithm for computing the packing.

\begin{figure}[t]
\centering
\includegraphics[width=0.4\textwidth]{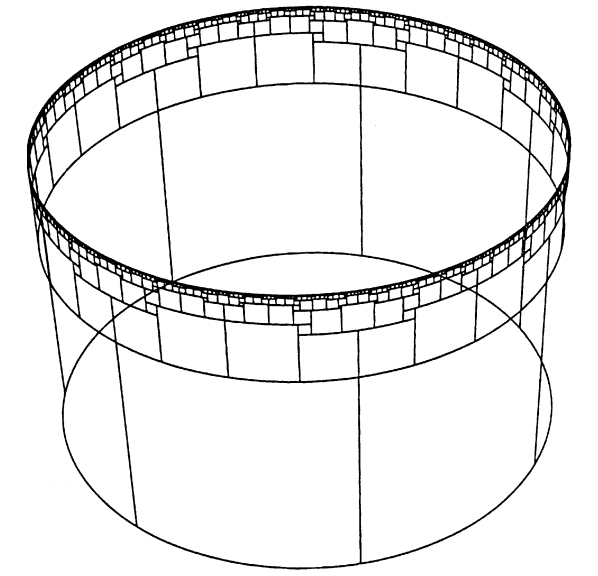}
\caption{The square tiling of the $7$-regular triangulation.}
\end{figure}

\item \textbf{Square tiling.} \index{square tiling} Another popular method of embedding planar graphs is the \emph{square tiling}, in which vertices are represented by horizontal line segments and edges by squares; such square tilings can take place either in a rectangle, the plane, or a cylinder.  Square tiling was introduced by Brooks, Smith, Stone, and Tutte \cite{BSST40}, and generalized to infinite planar graphs by Benjamini and Schramm \cite{BS96b}.  Like circle packing, square tiling can be thought of as a discrete version of conformal mapping, and in particular can be used to approximate the uniformizing map from a simply connected domain with four marked boundary points to a rectangle. For studying the random walk, a very nice feature of the square tiling that is not enjoyed by the circle packing is that the height of a vertex in the cylinder is a harmonic function, so that the height of a random walk is a martingale. Furthermore, Georgakopoulos \cite{G13} observed that if one stops the random walk at the first time it hits some height, then its horizontal coordinate at this time is uniform on the circle (this takes some interpretation to make precise). Further works on square tiling include \cite{G13,MR3498002,hutchcroft2015boundaries}.

Unlike circle packing, however, square tilings do not enjoy an analogue of the ring lemma, and can be geometrically very degenerate. Indeed, it is possible for edges to be represented by squares of zero area, and is also possible for two distinct planar graphs to have the same square tiling. Furthermore, square tilings are typically defined with reference to a specified root vertex, and it is difficult to compare the two different square tilings of the same graph that are computed with respect to different root vertices. These differences tend to mean that square tilings are best suited to quite different problems than circle packing.

 We also remark that a different sort of square tiling in which \emph{vertices} are represented by squares was introduced independently by Cannon, Floyd, and Parry \cite{MR1292901} and Schramm \cite{MR1244661}.

\item \textbf{Multiply-connected triangulations.} 
Several works have studied generalizations of the circle packing theorem to triangulations that are either not simply connected or not planar. Most notably, He and Schramm \cite{HS93} proved that every triangulation of a domain with countably many boundary components can be circle packed in a circle domain, that is, a domain all of whose boundary components are either circles or points: see \cref{fig:multiplyended} for examples. The corresponding statement for a triangulation of an \emph{arbitrary} domain is a major open problem, and is closely related to the Koebe conjecture. 

\begin{figure}[t]
\centering
\includegraphics[
 trim = 3.24em 10em 3em 10em, clip,
width=0.35\textwidth]{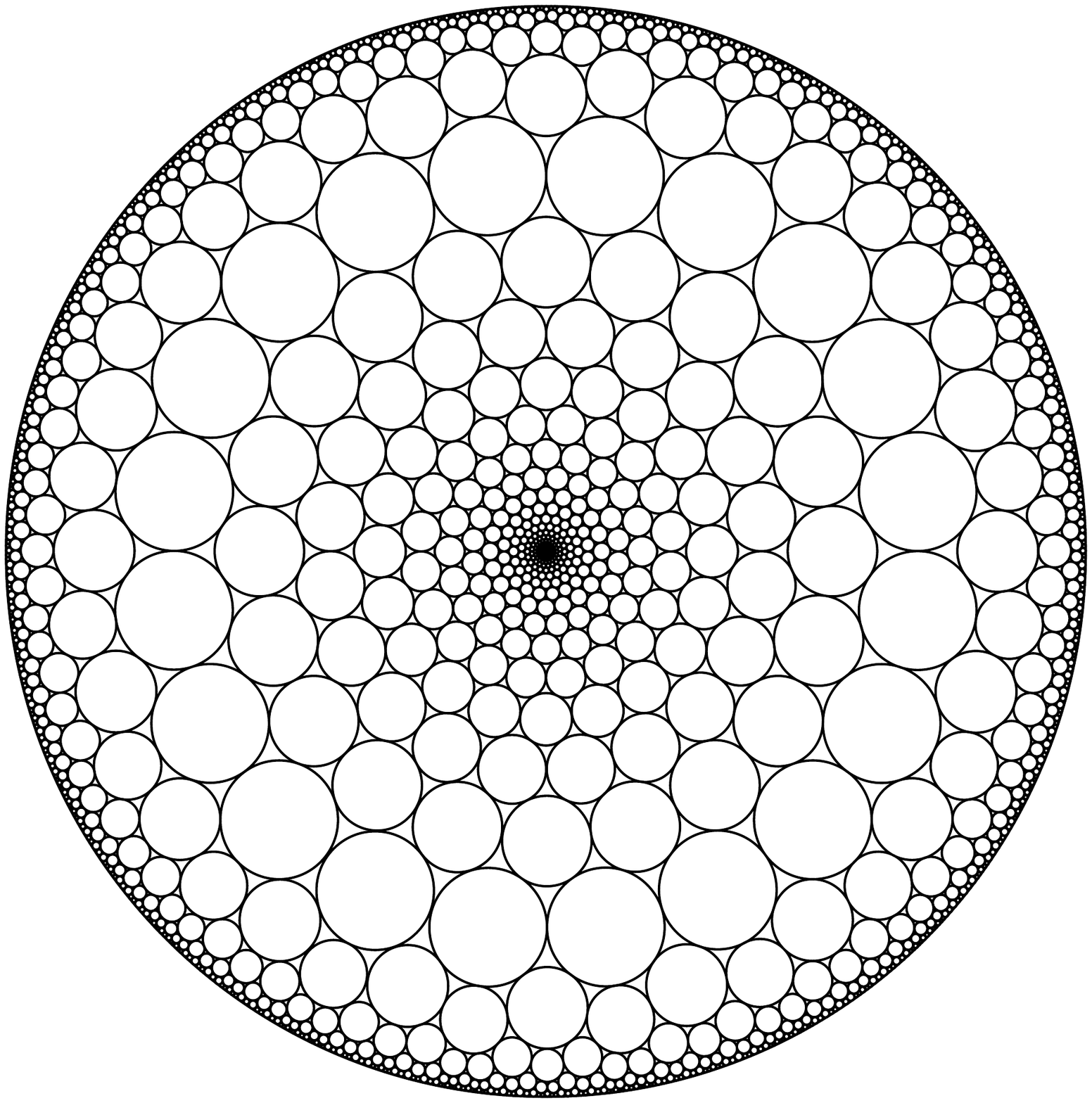}
\hspace{1.5cm}
\includegraphics[trim = 4cm 3.7cm 4cm 3.7cm, clip, height=0.35\textwidth]{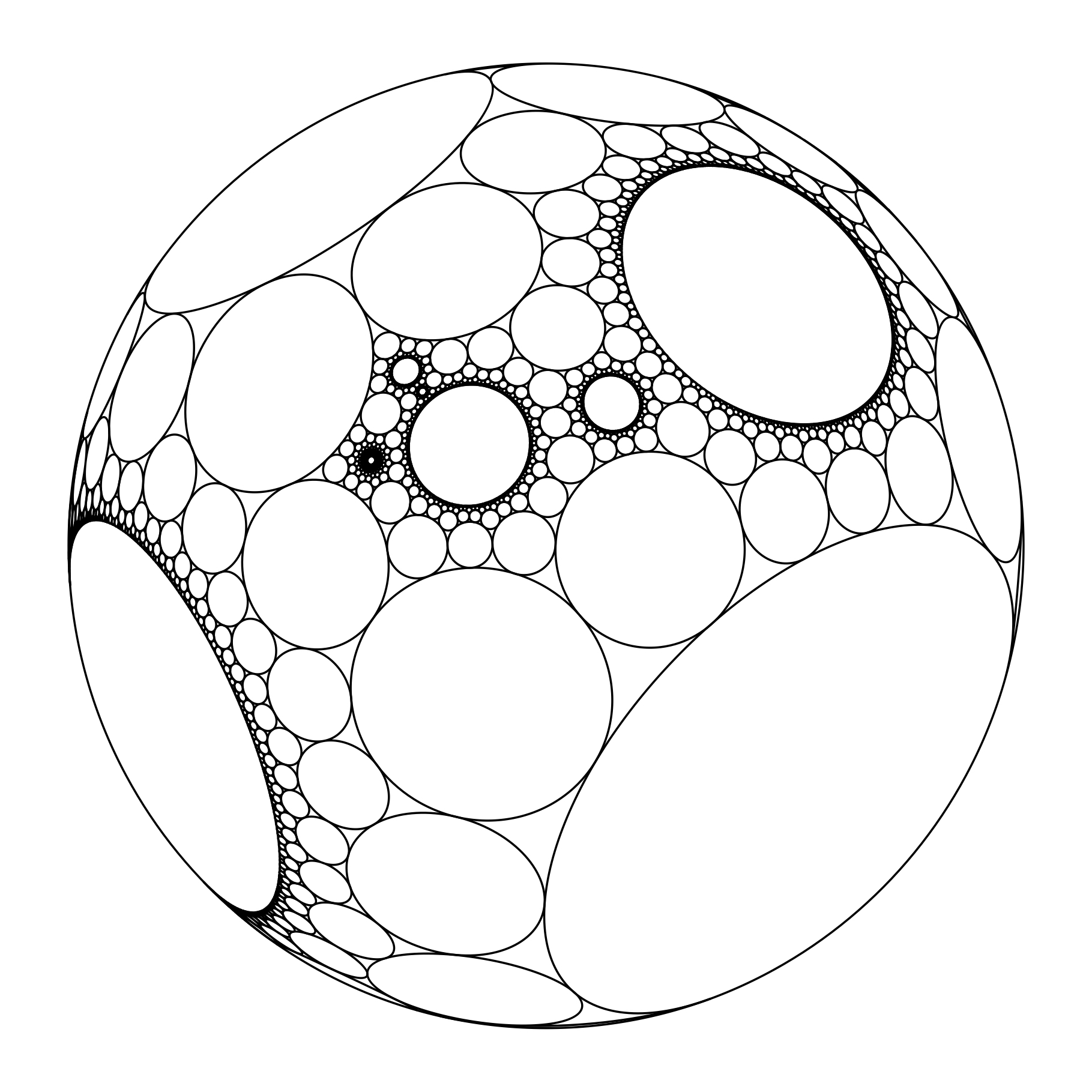}
\caption{Left: A circle packing in the multiply-connected circle domain $\UU \setminus\{0\}$. Right: A circle packing in a circle domain domain  with several boundary components.}
\label{fig:multiplyended}
\end{figure}

Gurel-Gurevich, the current author, and Suoto \cite{MR3607800} generalized the part of the He-Schramm Theorem concerning recurrence of the random walk as follows: A bounded degree triangulation circle packed in a domain $D$ is transient if and only if Brownian motion on $D$ is transient, i.e. leaves $D$ in finite time almost surely. 

\item \textbf{Isoperimetry of planar graphs.} \index{planar separator theorem}
In \cite{MTTV97}, Miller, Teng, Thurston, and Vavasis used circle packing to give a new proof of the \emph{Lipton-Tarjan planar separator theorem} \cite{LT80}, which concerns sparse cuts in planar graphs. Precisely, the theorem states that for any $n$-vertex planar graph, one can find a set of vertices of size at most $O(\sqrt{n})$ such that if this vertex set is deleted from the graph then every connected component that remains has size at most $3 n/4$. More precisely, the authors of \cite{MTTV97} showed that if one circle packs a planar graph in the unit sphere of $\RR^3$, normalizes by applying an appropriate M\"obius transformation, and takes a random plane passing through the origin in $\RR^3$, then the set of vertices whose corresponding discs intersect the plane will have the desired properties with high probability.

A related result of Jonasson and Schramm \cite{JS00} concerns the \emph{cover time} of planar graphs, i.e., the expected number of steps for a random walk on the graph to visit every vertex of the graph. They used circle packing to prove that the cover time of an $n$-vertex planar graph with maximum degree $M$ is always at least $c_M n \log^2 n$ for some positive constant $c_M$ depending only on $M$. This bound is attained (up to the constant) for large boxes $[-n,n]^2$ in $\ZZ^2$. In general, it is possible for $n$-vertex graphs to have cover time as small as $(1+o(1))n \log n$.
\item \textbf{Boundary theory.} 
Benjamini and Schramm \cite{BS96a} proved that if $P$ is a circle packing of a bounded degree triangulation in the unit disc $\UU$, then the simple random walk on the circle packed triangulation converges to a point in the boundary of $\UU$, and that the law of the limit point is non-atomic and has full support. (That is, the walk has probability zero of converging to any specific boundary point, and has positive probability of converging to any positive-length interval.) They used this result to deduce that a bounded degree planar graph admits non-constant bounded harmonic functions if and only if it is transient (equivalently, the invariant sigma-algebra of the random walk on the triangulation is non-trivial if and only if the walk is transient), and in this case it also admits non-constant bounded harmonic functions of finite Dirichlet energy. They also gave an alternative proof of the same result using square tiling instead of circle packing in \cite{BS96b}.

Indeed, given the result of Benjamini and Schramm, one may construct a non-constant bounded harmonic function $h$ on $T$ by taking any bounded, measurable function $f:\partial \UU \to \RR$ and defining $h$ to be the \emph{harmonic extension} of $f$, that is,
\[
h(v)= \mathbf{E}_v\left[ f\left(\lim_{n\to\infty} z(X_n) \right) \right],
\]
where $\mathbf{E}_v$ denotes expectation taken with respect to the random walk $X$ started at $v$, and $z(u)$ denotes the center of the circle in $P$ corresponding to $u$. Angel, Barlow, Gurel-Gurevich, and the current author \cite{ABGN14} proved that, in fact, \emph{every} bounded harmonic function on a bounded degree triangulation can be represented in this way. In other words, the boundary $\partial \UU$ can be identified with the \defn{Poisson boundary} of the triangulation. Probabilistically, this means that the entire invariant $\sigma$-algebra of the random walk coincides with the $\sigma$-algebra generated by the limit point.  They also proved the stronger result that $\partial \UU$ can be identified with the \defn{Martin boundary} of the triangulation. Roughly speaking, this means that every \emph{positive} harmonic function on the triangulation admits a representation as the harmonic extension of some \emph{measure} on $\partial \UU$. A related representation theorem for harmonic functions of \emph{finite Dirichlet energy} on bounded degree triangulations was established by Hutchcroft \cite{Hutch2017a}.

The results of \cite{ABGN14} regarding the Poisson boundary followed earlier work by Georgakopoulos \cite{G13}, which established a corresponding result for square tilings. Both results were revisited in the work of Hutchcroft and Peres \cite{hutchcroft2015boundaries}, which gave a simplified and unified proof that works for both embeddings.

A parallel boundary theory for circle packings of \textbf{unimodular random triangulations} of \emph{unbounded} degree was developed by Angel, Hutchcroft, the current author, and Ray in \cite{AHNR15}. 

\item \textbf{Harnack inequalities, Poincar\'e inequalities, and comparison to Brownian motion.}
The work of Angel, Barlow, Gurel-Gurevich and the current author \cite{ABGN14} also established various quite strong estimates for random walk on circle packings of bounded degree  triangulations. Roughly speaking, these estimates show that the random walk behaves similarly to the image of a Brownian motion under a \emph{quasi-conformal map}, that is, a bijective map that distorts angles by at most a bounded amount (i.e., maps infinitesimal circles to infinitesimal ellipses of bounded eccentricity). These estimates were central to their result concerning the Martin boundary of the triangulation, and are also interesting in their own right. Further related estimates have also been established by Chelkak \cite{Chelkak}. 

Recent work of Murugan \cite{murugan2018quasisymmetric} has built further upon these methods to establish very precise control of the random walk on (graphical approximations of) various deterministic self-similar fractal surfaces.

\item \textbf{Liouville quantum gravity and the KPZ correspondence.} Statistical physics in two dimensions has been one of the hottest areas of probability theory in recent years. The introduction of Schramm's SLE \cite{SchrammSLE} and further breakthrough developments by Lawler, Schramm and Werner (see \cite{LSW1,LSW2} and the references within) on the one hand, and the application of discrete complex analysis, pioneered by Smirnov \cite{Smirnov}, on the other, have led to several breakthroughs and to the resolution of a number of long-standing conjectures. These include the conformally invariant scaling limits of critical percolation \cite{SmirnovPerc} and Ising models \cite{SmirnovIsing}, and the determination of critical exponents and dimensions of sets associated with
planar Brownian motion \cite{LSW1} (such as the frontier and the set of cut points). It is manifest that much progress will follow, possibly including the treatment of self-avoiding walk (the connective constant of the hexagonal lattice was calculated in the breakthrough work \cite{DCSmirnovSAW}), the $O(n)$ loop model and the Potts model. While the bulk of this body of work applies to specific lattices, there are many fascinating problems in extending results to arbitrary
planar graphs.

The next natural step is to study the classical models of statistical physics in the context of random planar maps (see Le Gall's 2014 ICM proceedings \cite{LeG14}). There are deep conjectured connections between the behaviour of the models in the random setting versus the Euclidean setting, most significantly the KPZ formula of Knizhnik, Polyakov and Zamolodchikov \cite{KPZ88} from conformal field theory. This
formula relates the dimensions of certain sets in Euclidean geometry to the dimensions of corresponding sets in the random geometry. It may provide a systematic way to analyze models on the two dimensional Euclidean lattice: first study the model in the random geometry setting, where the Markovian properties of the underlying space make the model tractable; then use the
KPZ formula to translate the critical exponents from the random setting to the Euclidean one. 

Much of this picture is conjectural but a definite step towards this goal was taken in the influential paper of Duplantier and Sheffield \cite{DS}. Let us describe their formulation. Let $G_n$ be a random triangulation on $n$ vertices and consider its circle packing (or any other ``natural'' embedding) in the unit sphere. The embedding induces a random measure $\mu_n$ on the sphere by putting $\mu_n(A)$ to be the proportion of circle centers that are in $A$. The Duplantier-Sheffield conjecture asserts that the measures $\mu_n$ converge in distribution to a random measure $\mu$ on the sphere that has density given by an exponential of the Gaussian free field --- the latter is carefully defined and constructed in \cite{DS}. This measure is what is known as \emph{Liouville quantum gravity} (LQG).

Next, given a deterministic or random set $K$ on the sphere, one can calculate its expected dimension using the random measure given by LQG, and using the usual Lebesgue measure --- one gets two different numbers. Duplantier and Sheffield \cite{DS} obtain a quadratic formula allowing to compute one number from the other in the spirit of \cite{KPZ88}; this is the first rigorous instance of the KPZ correspondence. It allows one to compute the dimension of random sets in the $\ZZ^2$ lattice (corresponding to Lebesgue measure) by first calculating the corresponding dimension in the random geometry setting and then appealing to the KPZ formula. 

Many difficult models of statistical physics are tractable on a random planar map due to the inherent randomness of the space. For instance, it can be shown that the self avoiding walk on the UIPT behaves diffusively, that is, the endpoint of a self avoiding walk of length $n$ is typically of distance $n^{1/2+o(1)}$ from the origin \cite{CurienSAW, GwynneSAW}. A straightforward calculation with the KPZ formula allows one to predict that the typical displacement of the self-avoiding walk of length $n$ on the lattice $\ZZ^2$ is $n^{3/4+o(1)}$ --- a notoriously hard open problem with endless simulations supporting it. 

LQG and the KPZ correspondence thus pose a path to solving many difficult problems in classical two-dimensional statistical physics. We refer the interested reader to Garban's excellent survey \cite{GarbanKPZ} of the topic.


\end{enumerate}

\backmatter
\clearpage
\phantomsection
\addcontentsline{toc}{chapter}{Index}
\printindex

\cleardoublepage
\phantomsection
\addcontentsline{toc}{chapter}{References}
\bibliography{library}

\end{document}